\documentclass{amsart}
\chardef\bslash=`\\

\usepackage{amssymb,amsmath,amsfonts,amsthm,mathdots,epsfig,amscd,stmaryrd,enumitem}
\usepackage{stmaryrd}
\usepackage[all,cmtip,poly]{xy}
\usepackage{xcolor}
\usepackage{tikz}
\usepackage[hidelinks]{hyperref}

\usepackage{silence}
\usepackage{microtype}

\newtheorem{theorem}[subsection]{Theorem}
\newtheorem{thm}[subsection]{Theorem}
\newtheorem{lemma}[subsection]{Lemma}

\newtheorem{example}[subsection]{Example}
\newtheorem{introthm}{Theorem}

\newtheorem{introcor}[introthm]{Corollary}

\newtheorem{cor}[subsection]{Corollary}

\newtheorem{prop}[subsection]{Proposition}
\newtheorem{proposition}[subsection]{Proposition}

\newtheorem{defn}[subsection]{Definition}
\theoremstyle{remark}
\newtheorem{remark}[subsection]{Remark}

\newtheorem*{remark-unnumbered}{Remark}

\numberwithin{equation}{subsection}

\newif\iffinalrun
\iffinalrun
\else
\fi

\iffinalrun
\newcommand{\need}[1]{}
\newcommand{\mar}[1]{}
\else
\newcommand{\need}[1]{{\tiny *** #1}}
\newcommand{\mar}[1]{\marginpar{\raggedright\tiny  #1}}\fi

\renewcommand\mathbb{\mathbf}

\newcommand{\Lie}{{\operatorname{Lie}\,}}

\newcommand{\gog}{{\mathfrak{g}}}

\newcommand{\wt}{{\operatorname{wt}}}

\newcommand{\rec}{{\operatorname{rec}}}
 
\newcommand{\barepsilon}{{\overline{\epsilon}}}

\newcommand{\rbar}{\overline{r}}

\newcommand{\wv}{{\widetilde{v}}}
\newcommand{\ww}{{\widetilde{w}}}

\renewcommand{\ell}{l}

\def\PGL{\mathrm{PGL}}

\def\Iw{\mathrm{Iw}}

\newcommand{\cInd}{\operatorname{c-Ind}}
\newcommand{\St}{\operatorname{St}}
\newcommand{\ad}{\operatorname{ad}}
\newcommand{\diag}{\operatorname{diag}}
\newcommand{\tr}{\operatorname{tr}}
\newcommand{\Sp}{\operatorname{Sp}}

\newcommand{\D}{\cD}
\newcommand{\G}{\cG}
\newcommand{\M}{\cM}

\newcommand{\A}{\mathbf{A}}
\newcommand{\bA}{\ensuremath{\mathbf{A}}}
\newcommand{\CC}{{\mathbb C}}
\newcommand{\C}{\CC}
\newcommand{\bC}{\ensuremath{\mathbf{C}}}

\newcommand{\F}{\FF}
\newcommand{\FF}{{\mathbb F}}
\newcommand{\bF}{\ensuremath{\mathbf{F}}}
\newcommand{\bG}{\ensuremath{\mathbf{G}}}

\newcommand{\bQ}{\ensuremath{\mathbf{Q}}}
\newcommand{\Q}{\QQ}
\newcommand{\QQ}{{\mathbb Q}}

\newcommand{\bR}{\ensuremath{\mathbf{R}}}
\newcommand{\R}{\RR}
\newcommand{\RR}{{\mathbb R}}
\newcommand{\bT}{\ensuremath{\mathbf{T}}}
\newcommand{\TT}{{\mathbb T}}
\newcommand{\T}{{\mathbb T}}
\newcommand{\Z}{\ZZ}
\newcommand{\ZZ}{{\mathbb Z}}
\newcommand{\bZ}{\ensuremath{\mathbf{Z}}}

\newcommand{\bbZ}{\ensuremath{\mathbf{Z}}}
\newcommand{\bbQ}{\ensuremath{\mathbf{Q}}}

\newcommand{\cC}{{\mathcal C}}
\newcommand{\cD}{{\mathcal D}}
\newcommand{\cE}{{\mathcal E}}
\newcommand{\cF}{{\mathcal F}}
\newcommand{\cG}{{\mathcal G}}

\newcommand{\cL}{{\mathcal L}}

\newcommand{\cM}{{\mathcal M}}
\newcommand{\cO}{{\mathcal O}}
\newcommand{\cR}{{\mathcal R}}
\newcommand{\cS}{{\mathcal S}}
\newcommand{\cT}{{\mathcal T}}
\newcommand{\cV}{{\mathcal V}}
\newcommand{\cW}{{\mathcal W}}

\newcommand{\cU}{{\mathcal{U}}}
\newcommand{\cZ}{{\mathcal{Z}}}

\newcommand{\cX}{{\mathcal{X}}}
\newcommand{\cY}{{\mathcal{Y}}}

\newcommand{\m}{\frakm}
\newcommand{\ffrm}{{\mathfrak m}}
\newcommand{\frakm}{\mathfrak{m}}

\newcommand{\frakp}{\mathfrak{p}}
\newcommand{\p}{\frakp}
\newcommand{\frakq}{\mathfrak{q}}
\newcommand{\q}{\frakq}

\newcommand{\Qbar}{\overline{\Q}}
\newcommand{\Zbar}{\overline{\Z}}

\newcommand{\Zp}{\Z_p}

\newcommand{\Zpbar}{\Zbar_p}

\newcommand{\Zpx}{\Zp^{\times}}

\newcommand{\Zpbarx}{\Zpbar^{\times}}

\newcommand{\Qp}{\Q_p}
\newcommand{\Cp}{\C_p}

\newcommand{\Qpbar}{\Qbar_p}

\newcommand{\Qpbarx}{\Qpbar^{\times}}

\DeclareMathOperator{\End}{End}

\DeclareMathOperator{\Fil}{Fil}
\DeclareMathOperator{\gr}{gr}
\DeclareMathOperator{\Gal}{Gal}
\newcommand{\GL}{\mathrm{GL}}
\newcommand{\GSp}{\mathrm{GSp}}
\DeclareMathOperator{\Hom}{Hom}

\DeclareMathOperator{\Ind}{Ind}

\DeclareMathOperator{\ord}{ord}
\DeclareMathOperator{\SL}{SL}
\DeclareMathOperator{\Spec}{Spec}

\DeclareMathOperator{\WD}{WD}
\DeclareMathOperator{\Sym}{Sym}

\newcommand{\Frob}{\mathrm{Frob}}

\newcommand{\rhobar}{\overline{\rho}}

\newcommand{\Art}{{\operatorname{Art}}}
\newcommand{\Res}{\operatorname{Res}}

\newcommand{\doubleslash}{/\kern-0.2em{/}}

\begin{document}
\author{James Newton and Jack A. Thorne}
\title[Symmetric power functoriality]{Symmetric power functoriality for 
holomorphic modular forms}
\begin{abstract} 
	Let $f$ be a cuspidal Hecke eigenform of level 1. We prove the automorphy of the symmetric power 
	lifting $\Sym^n f$ for every $n \geq 1$. 
	
	We establish the same result for a more general class of cuspidal Hecke 
	eigenforms, including all those associated to semistable elliptic curves 
	over $\bQ$.
\end{abstract}
\maketitle
\setcounter{tocdepth}{1}
\tableofcontents
\section*{Introduction}

\textbf{Context.} Let $F$ be a number field, and let $\pi$ be a cuspidal 
automorphic representation of $\GL_n(\A_F)$. Langlands's functoriality 
principle \cite[Question 5]{Lan70} predicts the existence, for any algebraic 
representation $R : \GL_n \to \GL_N$, of a functorial lift of $\pi$ along $R$; 
more precisely, an automorphic representation $R(\pi)$ of $\GL_N(\A_F)$ which 
may be characterized by the following property: for any place $v$ of $F$, the 
Langlands parameter of $R(\pi)_v$ is the image, under $R$, of the Langlands 
parameter of $\pi_v$. The Langlands parameter is defined for each place $v$ of 
$F$ using the local Langlands correspondence for $\GL_n(F_v)$ (see 
\cite{langlandsrg, ht, MR1738446}).

The simplest interesting case is when $n = 2$ and $R = \Sym^m$ is the $m$\textsuperscript{th}  symmetric power of the standard representation of $\GL_2$. In this case the automorphy of $\Sym^m\pi$ 
was proved for $m=2$ by Gelbart and Jacquet \cite{gjsym2} and for $m = 3,4$ by 
Kim and Shahidi \cite{Kim-Shahidi,Kim}. 

More recently, Clozel and the second author
have proved the automorphy of $\Sym^m\pi$ for $m \le 8$ under the assumption 
that $\pi$ can be realised in a space of Hilbert modular forms of regular 
weight \cite{Clo14,ctii,ctiii}; equivalently, that the number field $F$ is 
totally real and the automorphic representation $\pi$ is regular algebraic, in 
the sense of \cite{Clo90}. This includes the most 
classical case of automorphic representations arising from holomorphic modular 
forms of weight $k \ge 2$. We also mention the work of Dieulefait 
\cite{Dieulefait}, which shows automorphy of the $5$\textsuperscript{th} 
symmetric power for cuspidal Hecke eigenforms of level $1$ and weight $k \ge 
2$. 

On the other hand, the \emph{potential} automorphy (i.e.~the existence 
of the symmetric power lifting after making some unspecified Galois base 
change) of \emph{all} symmetric powers for automorphic representations $\pi$ 
associated to Hilbert modular forms was obtained by Barnet-Lamb, Gee and 
Geraghty \cite{blgg} (the case of elliptic modular forms is due to Barnet-Lamb, 
Geraghty, Harris and Taylor \cite{cy2}). 

\textbf{Results of this paper.} In this paper, we prove the automorphy of all symmetric powers for cuspidal Hecke eigenforms of level $1$ and weight $k 
\ge 2$. More precisely:
\begin{introthm}\label{thm_intro_main_theorem}
	Let $\pi$ be a regular algebraic cuspidal automorphic representation of 
	$\GL_2(\bA_\Q)$ of level $1$ (i.e.~which is everywhere unramified). Then 
	for each 
	integer $n \ge 2$, the symmetric power 
	lifting $\Sym^{n-1}\pi$ exists, 
	as a regular algebraic cuspidal automorphic representation of $\GL_{n}(\bA_\Q)$. 
\end{introthm}
In fact, we establish a more general result in which ramification is allowed:
\begin{introthm}\label{intro_ramified_main_theorem}
	Let $\pi$ be a regular algebraic cuspidal automorphic representation of 
	$\GL_2(\bA_\Q)$ of conductor $N \geq 1$, which does not have CM.\footnote{In other words, there is no quadratic Hecke character $\chi$ such that $\pi \cong \pi \otimes \chi$.} Suppose that for each prime $l | N$, the Jacquet module of $\pi_l$ is non-trivial; equivalently, that $\pi_l$ is not supercuspidal. Then 
	for each 
	integer $n \ge 2$, the symmetric power 
	lifting $\Sym^{n-1}\pi$ exists, 
	as a regular algebraic cuspidal automorphic representation of $\GL_{n}(\bA_\Q)$.
\end{introthm}
The class of automorphic representations described by Theorem \ref{intro_ramified_main_theorem} includes all those associated to holomorphic newforms of level $\Gamma_0(N)$, for some squarefree integer $N \geq 1$; in particular those associated to semistable elliptic curves over $\Q$. We can therefore offer the following corollary in more classical language:
\begin{introcor}
	Let $E$ be a semistable elliptic curve over $\bQ$. Then, for each integer $n \geq 2$, the completed symmetric power $L$-function $\Lambda(\Sym^n E, s)$ as defined in e.g.\ \cite{Dum09}, admits an analytic continuation to the entire complex plane.
\end{introcor}
We remark that the meromorphic, as opposed to analytic, continuation of the 
completed $L$-function $\Lambda(\Sym^n E, s)$ was already known, as a 
consequence of the potential automorphy results mentioned above. Potential automorphy results were sufficient to prove the Sato--Tate 
conjecture, but our automorphy results make it possible to establish 
\emph{effective} versions of Sato--Tate (we thank Ana Caraiani and Peter Sarnak 
for pointing this out to us). See, for example, \cite{Thorner} for 
an unconditional result and \cite{KMurty, BucurKedlaya, RouseThorner} for 
results conditional on the Riemann Hypothesis for the symmetric power 
$L$-functions. 

\textbf{Strategy.}
Algebraic automorphic representations of $\GL_n(\A_F)$ are conjectured to admit 
associated Galois representations \cite{Clo90}. When $F$ is totally real and 
$\pi$ is a self-dual regular algebraic automorphic representation, these Galois 
representations are known to exist; their Galois deformation theory is 
particularly well-developed; and they admit $p$-adic avatars, which fit into 
$p$-adic families of overconvergent automorphic forms. We make use of all of 
these tools. We begin by proving the following theorem:
\begin{introthm}\label{thm:intro_propagation}
	Let $n \ge 2$ be an integer and suppose that the $n$\textsuperscript{th} 
	symmetric power lifting exists for one regular algebraic 
	cuspidal 
	automorphic 
	representation of $\GL_2(\bA_\bQ)$ of level $1$. Then the 
	$n$\textsuperscript{th} symmetric power lifting exists for every 
	regular algebraic cuspidal automorphic 
	representation of $\GL_2(\bA_\bQ)$ of level $1$.
\end{introthm}
We sketch the proof of Theorem \ref{thm:intro_propagation}, which is based on 
the properties of the Coleman--Mazur eigencurve $\cE_p$. We recall that if $p$ is a 
prime, the eigencurve $\cE_p$ is a $p$-adic rigid analytic space that admits a 
Zariski dense set of classical points corresponding to pairs $(f, \alpha)$ 
where $f$ is a cuspidal eigenform of level $1$ and some weight $k \geq 2$ and 
$\alpha$ is a root of the Hecke polynomial $X^2 - a_p(f) X + p^{k-1}$. The 
eigencurve admits a map $\kappa : \cE_p \to \cW_p = \Hom(\Z_p^\times, \bG_m)$ 
to weight space with discrete fibres; the image of $(f, \alpha)$ is the 
character $x \mapsto x^{k-2}$.

We first show that for a fixed $n \geq 1$, the automorphy of $\Sym^{n} f$ is a property which is ``constant on irreducible components of $\cE_p$''. (Here we confuse $f$ and the automorphic representation $\pi$ that it generates in order to simplify notation.) More precisely, if $(f, \alpha)$ and $(f', \alpha')$ determine points on the same irreducible component of $\cE_p$, then the automorphy of $\Sym^n f$ is equivalent to that of $\Sym^n f'$. This part of the argument, which occupies \S \ref{sec_rigid_geometry} of this paper, does not require a restriction to cusp forms of level 1 -- see Theorem \ref{thm_propogration_along_components_of_eigencurve}. It is based on an infinitesimal $R = \bT$ theorem on the eigenvariety associated to a definite unitary group in $n$ variables. Kisin (for $\GL_2$) \cite{kisin-ocfmc} and 
Bella\"iche--Chenevier (for higher rank) \cite{bellaiche_chenevier_pseudobook} have observed that such theorems are often implied by the vanishing of adjoint Bloch--Kato Selmer groups. We are able to argue in this fashion here because we have proved the necessary vanishing results in \cite{newton2020adjoint}.

To exploit this geometric property, we need to understand the irreducible components of $\cE_p$. This is a notorious problem. However, conjectures predict that $\cE_p$ has a simple structure over a suitably thin boundary annulus of a connected component of weight space $\cW_p$ (see e.g.\ \cite[Conjecture 1.2]{Liu17}). We specialise to the case $p = 2$, in which case Buzzard--Kilford give a beautifully simple and explicit description of the geometry of $\cE_p$ ``close to the boundary of weight space'' \cite{Buz05}. 

More precisely, $\cE_2$ is supported above a single connected component 
$\cW_2^+ \subset \cW_2$, which we may identify with the rigid unit disc $\{ | w 
| < 1\}$. The main theorem of \cite{Buz05} is that the pre-image $\kappa^{-1}( 
\{ | 8 | < |w| < 1 \})$ 
decomposes as a disjoint union $\sqcup_{i=1}^\infty X_i$ of rigid annuli, each 
of which maps isomorphically onto $\{ | 8 | < |w| < 1 \}$. Moreover, $X_i$ has 
the following remarkable property: if $(f, \alpha) \in X_i$ is a point 
corresponding to a classical modular form, then the $p$-adic valuation
$v_p(\alpha)$ (otherwise known as the slope of the pair $(f, \alpha)$) equals 
$i v_p(w(\kappa(f, \alpha)))$. 

We can now explain the second part of the proof of Theorem 
\ref{thm:intro_propagation}, which occupies \S \ref{sec_ping_pong} of the 
paper. Since each irreducible component of $\cE_2$ meets $\kappa^{-1}( \{ | 8 | 
< |w| < 1 \})$, it is enough to show that each $X_i$ contains a point $(f, 
\alpha)$ such that $\Sym^n f$ is automorphic. This property only depends on $f$ 
and not on the pair $(f, \alpha)$! Moreover, the level 1 form $f$ determines 
two points $(f, \alpha)$, $(f, \beta)$ of $\kappa^{-1}( \{ | 8 | < |w| < 1 
\})$, which lie on components $X_i$ and $X_{i'}$ satisfying $i + i' = (k-1) / 
v_p(w(\kappa(f, \alpha)))$. Starting with a well-chosen initial point on a given annulus $X_i$, 
we can jump to any other $X_{i'}$ in a finite series of swaps between pairs 
$(f', \alpha')$, $(f', \beta')$ and moves within an annulus. We call this 
procedure playing ping pong, and it leads to a complete proof of Theorem 
\ref{thm:intro_propagation}.

We remark that for this second step of the proof it is essential that we work with level 1 forms, since it is only in the level 1, $p = 2$ case that the eigencurve $\cE_p$ admits such a simple structure (in particular, the eigencurve is supported above a single connected component of weight space and every Galois representation appearing in $\cE_2$ admits the same residual representation, namely the trivial 2-dimensional representation of $\Gal(\overline{\Q} / \Q)$ over $\bF_2$). We note as well that it is necessary to work with classical forms which may be ramified at the prime 2 in order for their weight characters to lie in the boundary annulus of $\cW_2^+$. We have suppressed this minor detail here. 

Theorem \ref{thm:intro_propagation} implies that to prove Theorem \ref{thm_intro_main_theorem}, it is enough to prove the following result:
\begin{introthm}\label{introthm_existence_of_single_symmetric_power}
	For each integer $n \geq 2$, there is a regular algebraic cuspidal 
	automorphic representation $\pi$ of $\GL_2(\A_\Q)$ of level $1$ such that 
	$\Sym^{n-1} 
	\pi$ exists.
\end{introthm}
As in the previous works 
of Clozel and the second author \cite{Clo14,ctii,ctiii}, we achieve this by combining an automorphy lifting theorem with the construction of level-raising congruences. We aim to find $f$ and an isomorphism $\iota : \overline{\QQ}_p \to \C$ such that (writing $r_{f, \iota} : G_\Q \to \GL_2(\overline{\Q}_p)$ for the $p$-adic Galois representation associated to $f$) the residual representation
\[ \Sym^{n-1} \overline{r}_{f, \iota} : G_\Q \to \GL_{n}(\overline{\F}_p) \]
is automorphic; then we hope to use an automorphy lifting theorem to verify 
that $\Sym^{n-1} r_{f, \iota}$ is automorphic, and hence that $\Sym^{n-1} f$ is 
automorphic. In contrast to the papers just cited, where we chose 
$\overline{r}_{f, \iota}$ to have large image but $p$ to be small, in order to 
exploit the reducibility of the symmetric power representations of $\GL_2$ in 
small characteristic, here we choose $\overline{r}_{f, \iota}$ to have small 
image, and $p$ to be large.

More precisely, we choose $f$ to be congruent modulo $p$ to a theta series, so that $\overline{r}_{f, \iota} \cong \Ind_{G_K}^{G_\Q} \overline{\psi}$ is induced. In this case $\Sym^{n-1} \overline{r}_{f, \iota}|_{G_K}$ is a sum of characters, so its residual automorphy can be verified using the endoscopic classification for unitary groups in $n$ variables. The wrinkle is that the automorphy lifting theorems proved in \cite{All19} (generalizing those of \cite{jackreducible}) require the automorphic representation $\pi$ of $\GL_{n}(\A_K)$ (say) verifying residual automorphy to have a local component which is a twist of the Steinberg representation. To find such a $\pi$ we need to combine the endoscopic classification with the existence of level-raising congruences.

In fact, we combine two different level-raising results in order to construct the desired congruences. The first of these, original to this paper, suffices to prove Theorem \ref{introthm_existence_of_single_symmetric_power} in the case that $n$ is odd. The argument is based on a generalization of the following simple observation, which suffices to prove Theorem \ref{introthm_existence_of_single_symmetric_power} in the case $n = 3$: let $q$ be an odd prime power, and let $U_3(q)$ denote the finite group of Lie type associated to the outer form of $\GL_3$ over $\bF_q$. Let $p$ be a prime such that $q \pmod{p}$ is a primitive $6^\text{th}$ root of unity. Then the unique cuspidal unipotent representation of $U_3(q)$ remains irreducible on reduction modulo $p$, and this reduction occurs as a constituent of the reduction modulo $p$ of a generic cuspidal representation of $U_3(q)$ (see Proposition \ref{prop_congruence_of_types}). Using the theory of depth zero types, this observation has direct consequences for the existence of congruences between automorphic representations of $U_3$. Similar arguments work for general odd $n$, for carefully chosen global data. We leave a discussion of the (quite intricate) details to \S \ref{sec:automorphic_level_raising}.

The second level-raising result, proved by Anastassiades in his thesis, allows us to pass from the existence of $\Sym^{n-1} f$ to the existence of $\Sym^{2n-1} f$. We refer to the paper \cite{Ana19} for a more detailed discussion.

It remains to extend Theorem \ref{thm_intro_main_theorem} to the ramified case, 
and prove Theorem \ref{intro_ramified_main_theorem}. For this we induct on the 
number of primes dividing the conductor, and use an argument of `killing 
ramification' as in the proof of Serre's conjecture \cite{Kha09}. Thus to 
remove a prime $l$ from the level we need to be able to move within a family of 
$l$-adic overconvergent modular forms to a classical form of the same tame 
level, but now unramified at $l$. This explains our assumption in Theorem 
\ref{intro_ramified_main_theorem} that the Jacquet module of $\pi_l$ is 
non-trivial for every prime $l$: it implies the existence of a point associated 
to (a twist of) $\pi$ on an $l$-adic eigencurve for every prime $l$. 

In a sequel to this paper \cite{New20}, we prove a new kind of automorphy lifting theorem for symmetric power Galois representations. This allows us to finally prove a version of Theorem \ref{intro_ramified_main_theorem} where the hypothesis that no local component $\pi_l$ is supercuspidal is removed. The arguments of \cite{New20} use only fixed weight classical automorphic forms (as opposed to overconvergent automorphic forms) but do require the results of this paper (in particular, Theorem \ref{intro_ramified_main_theorem}) as a starting point.

\textbf{Organization of this paper.} We begin in \S 
\ref{sec_definite_unitary_groups} by recalling known results on the 
classification of automorphic representations of definite unitary groups. We 
make particular use of the construction of $L$-packets of discrete series 
representations of $p$-adic unitary groups given by M{\oe}glin \cite{ Moe07, 
Moe14}, the application of  Arthur's simple trace formula for definite unitary 
groups as explicated in \cite{labesse}, and Kaletha's results on the 
normalisation of transfer factors (in the simplest case of pure inner forms) 
\cite{Kal16}.

In \S \ref{sec_rigid_geometry} we study the interaction between the existence of symmetric power liftings of degree $n$ with the geometry of the eigenvariety associated to a definite unitary group in $n$ variables. The basic geometric idea is described in \S \ref{subsec_prototype_argument}. In \S \ref{sec_ping_pong} we combine these results with the explicit description of the tame level 1, $p = 2$ Coleman--Mazur eigencurve to complete the proof of Theorem \ref{thm:intro_propagation}.

We then turn to the proof of Theorem 
\ref{introthm_existence_of_single_symmetric_power}, which rests upon two 
level-raising results, only the first of which is proved here. The proof of 
this result is in turn split into two halves; first we give in \S 
\ref{sec:automorphic_level_raising} an automorphic construction of 
level-raising congruences using types, in the manner sketched above. Then in \S 
\ref{sec:galois_level_raising} we establish level-raising congruences of a 
different kind using deformation theory for residually reducible 
representations, as developed in \cite{jackreducible, All19}. These 
two results are applied in turn to construct our desired level-raising 
congruences for odd $n$ (Proposition \ref{prop:levelraisingoddn}). A key 
intermediate result is a finiteness result for certain Galois deformation 
rings, established in \S \ref{sec_finiteness_result}, and which may be of 
independent interest. We use this to control the dimension of the locus of 
reducible deformations. 

Finally, we are in a position to prove our main theorems. In \S \ref{sec_existence_of_seed_points} we combine the preceding constructions with the main theorem of \cite{Ana19} in order to prove Theorem \ref{introthm_existence_of_single_symmetric_power} and therefore Theorem \ref{thm_intro_main_theorem}. In \S \ref{sec_higher_levels}, we carry out the argument of `killing ramification' in order to obtain Theorem \ref{intro_ramified_main_theorem}. The main technical challenge is to manage the hypothesis of `$n$-regularity' which appears in our analytic continuation results (see especially Theorem \ref{thm_propogration_along_components_of_eigencurve}). To do this we prove a result (Proposition \ref{prop_TW_congruence_to_n_regular}) which takes a given automorphic representation $\pi$ and constructs a congruence to an $n$-regular one $\pi'$. This may also be of independent interest.

\textbf{Acknowledgements.} The influence of the papers \cite{Clo14, ctii, 
ctiii} by Laurent Clozel and J.T. on the current work will be clear. We would 
like to express our warmest thanks to L.C. for many helpful discussions over 
the years. In particular, the idea of exploiting congruences to a theta series 
was first discussed during a visit to Orsay by  J.T. during the preparation of 
these earlier papers. We also thank L.C. for helpful comments on an earlier 
draft of this paper. We thank an anonymous referee for their careful reading and many helpful comments.

 J.T.'s work received funding from the European 
Research Council (ERC) under the European Union's Horizon 2020 research and 
innovation programme (grant agreement No 714405). 

\textbf{Notation.}
If $F$ is a perfect field, we generally fix an algebraic closure $\overline{F} / F$ and write $G_F$ for the absolute Galois group of $F$ with respect to this choice. We make the convention that a soluble extension $F' / F$ is a (finite) \emph{Galois} extension with  soluble Galois group $\Gal(F' / F)$.

When the characteristic of $F$ is not equal to $p$, we write $\epsilon : G_F \to \bbZ_p^\times$ for the $p$-adic cyclotomic character. We write $\zeta_n \in \overline{F}$ for a fixed choice of primitive $n^\text{th}$ root of unity (when this exists). If $F$ is a number field, then we will also fix embeddings $\overline{F} \to \overline{F}_v$ extending the map $F\to F_v$ for each place $v$ of $F$; this choice determines a homomorphism $G_{F_v} \to G_F$. When $v$ is a finite place, we will write $\cO_{F_v} \subset F_v$ for the valuation ring, $\varpi_v \in \cO_{F_v}$ for a fixed choice of uniformizer, $\Frob_v \in G_{F_v}$ for a fixed choice of (geometric) Frobenius lift, $k(v) = \cO_{F_v} / (\varpi_v)$ for the residue field, and $q_v = \# k(v)$ for the cardinality of the residue field. When $v$ is a real place, we write $c_v \in G_{F_v}$ for complex conjugation. If $S$ is a finite set of finite places of $F$ then we write $F_S / F$ for the maximal subextension of $\overline{F}$ unramified outside $S$ and $G_{F, S} = \Gal(F_S / F)$.

If $p$ is a prime, then we call a coefficient field a finite extension $E / \bbQ_p$ contained inside our fixed algebraic closure $\overline{\bbQ}_p$, and write $\cO$ for the valuation ring of $E$, $\varpi \in \cO$ for a fixed choice of uniformizer, and $k = \cO / (\varpi)$ for the residue field. If $A$ is a local ring, we write $\cC_A$ for the category of complete Noetherian local $A$-algebras with residue field $A / \ffrm_A$. We will use this category mostly with $A = E$ or $A = \cO$. If $G$ is a profinite group and $\rho : G \to \GL_n(\overline{\bQ}_p)$ is a continuous representation, then we write $\overline{\rho} : G \to \GL_n(\overline{\bF}_p)$ for the associated semisimple residual representation (which is well-defined up to conjugacy).

If $F$ is a CM number field (i.e. a totally imaginary quadratic extension of a totally real number field), then we write $F^+$ for its maximal totally real subfield, $c \in \Gal(F / F^+)$ for the unique non-trivial element, and $\delta_{F / F^+} : \Gal(F / F^+) \to \{ \pm 1 \}$ for the unique non-trivial character. If $S$ is a finite set of finite places of $F^+$, containing the places at which $F / F^+$ is ramified, we set $F_S = F^+_S$ and $G_{F, S} = \Gal(F_S / F)$.

We write $T_n \subset B_n \subset \GL_n$ for the standard diagonal maximal torus and upper-triangular Borel subgroup. Let $K$ be a non-archimedean characteristic $0$ local field, 
and let $\Omega$ 
be an algebraically 
closed field of characteristic 0. If $\rho : G_K \to \GL_n(\overline{\bbQ}_p)$ is a continuous 
representation (which is de Rham if $p$ equals the residue characteristic of 
$K$), then we write $\mathrm{WD}(\rho) = (r, N)$ for the associated 
Weil--Deligne representation of $\rho$, and $\mathrm{WD}(\rho)^{F-ss}$ for its 
Frobenius semisimplification. We use the cohomological normalisation of 
class field theory: it is the isomorphism $\Art_K: K^\times \to W_K^{ab}$ which 
sends uniformizers to geometric Frobenius elements. When $\Omega = \bC$, we 
have the local Langlands correspondence $\rec_{K}$ for $\GL_n(K)$: a bijection 
between the sets of isomorphism classes of irreducible, admissible 
$\bC[\GL_n(K)]$-modules and Frobenius-semisimple Weil--Deligne representations 
over $\bC$ of rank $n$. In general, we have the Tate normalisation $\rec_K^T$ of the local 
Langlands correspondence 
for $\GL_n$ as described in \cite[\S 
2.1]{Clo14}. When $\Omega = 
\bC$, we have $\rec^T_K(\pi) = \rec_K(\pi \otimes | \cdot |^{(1-n)/2})$. 

If $G$ is a reductive group over $K$ and $P \subset G$ is a 
parabolic subgroup and $\pi$ is an admissible $\Omega[P(K)]$-module, then we 
write $\Ind_{P(K)}^{G(K)} \pi$ for the usual smooth induction.  If $\Omega = \bC$ then we 
write $i_P^G \pi$ for the normalised induction, defined as $i_P^G \pi = 
\Ind_{P(K)}^{G(K)} \pi \otimes \delta_P^{1/2}$, where $\delta_P : P(K) \to 
\bR_{>0}$ is the character $\delta_P(x) = | \det (\mathrm{Ad}(x)|_{\Lie N_P}) |_K$ 
(and $N_P$ is the unipotent radical of $P$).

If $\psi : K^\times \to \C^\times$ is a smooth character, then we write 
$\operatorname{Sp}_n(\psi) = (r, N)$ for the Weil--Deligne representation on 
$\C^n = \oplus_{i=1}^n \C \cdot e_i$ given by $r = (\psi \circ \Art_{K}^{-1}) 
\oplus 
(\psi| \cdot |^{-1} \circ \Art_{K}^{-1}) \oplus \dots \oplus (\psi| \cdot |^{1-n} \circ 
\Art_{K}^{-1})$ and $N e_1 = 0$, $N e_{i+1} = e_i$ ($1 \leq 
i 
\leq n-1$). We write $\operatorname{St}_{n}(\psi)$ for the unique 
irreducible quotient of $i_{B_n}^{\GL_n} 
(\psi \circ \det)\delta_{B_n}^{-1/2} = \Ind_{B_n(K)}^{\GL_n(K)} \psi\circ\det$. We 
have $\rec_{K}^T(\operatorname{St}_{n}(\psi)) = \operatorname{Sp}_{n}(\psi)$.

If $F$ is a number field and $\chi : F^\times \backslash \A_F^\times \to \bC^\times$ is a Hecke character of type $A_0$ (equivalently: algebraic), then for any isomorphism $\iota : \overline{\bQ}_p \to \bC$ there is a continuous character $r_{\chi, \iota} : G_F \to \overline{\bQ}_p^\times$ which is de Rham at the places $v | p$ of $F$ and such that for each finite place $v$ of $F$, $\mathrm{WD}(r_{\chi, \iota}) \circ \Art_{F_v} = \iota^{-1} \chi|_{F_v^\times}$. Conversely, if $\chi' : G_F \to \overline{\bQ}_p^\times$ is a continuous character which is de Rham and unramified at all but finitely many places, then there exists a Hecke character $\chi : F^\times \backslash \A_F^\times \to \bC^\times$ of type $A_0$ such that $r_{\chi, \iota} = \chi'$. In this situation we abuse notation slightly by writing $\chi = \iota \chi'$.

If $F$ is a CM or totally real number field and $\pi$ is an automorphic representation of $\GL_n(\A_F)$, we say that $\pi$ is regular algebraic if $\pi_\infty$ has the same infinitesimal character as an irreducible algebraic representation $W$ of $(\Res_{F/ \bQ} \GL_n)_\bC$. We identify $X^\ast(T_n)$ with $\Z^n$ in the usual way, and write $\Z^n_+ \subset \Z^n$ for the subset of weights which are $B_n$-dominant. If $W^\vee$ has highest weight $\lambda = (\lambda_\tau)_{\tau \in \Hom(F, \C)} \in (\Z^n_+)^{\Hom(F, \C)}$, then we say that $\pi$ has weight $\lambda$.

When $F$ is CM, the automorphic representation $\pi$ is said to be conjugate self-dual if $\pi^c \cong \pi^\vee$. We refer to \cite[\S2.1]{BLGGT} for the more general notion of a polarizable automorphic representation. Note that if $\pi$ is conjugate self-dual, then $(\pi,\delta_{F/F^+}^{n})$ is polarized and therefore $\pi$ is polarizable.

If $\pi$ is cuspidal, regular algebraic, and polarizable, then for any isomorphism $\iota : \overline{\bQ}_p \to \bC$ there exists a continuous, semisimple representation $r_{\pi, \iota} : G_F \to \GL_n(\overline{\bQ}_p)$ such that for each finite place $v$ of $F$, $\mathrm{WD}(r_{\pi, \iota}|_{G_{F_v}})^{F-ss} \cong \rec_{F_v}^T(\iota^{-1} \pi_v)$ (see e.g.\ \cite{Caraianilp}). (When $n = 1$, this is compatible with our existing notation.) We use the convention that the Hodge--Tate weight of the cyclotomic character is $-1$. Thus if $\pi$ is of weight $\lambda$, then for any embedding $\tau : F \to \overline{\Q}_p$ the $\tau$-Hodge--Tate weights of $r_{\pi, \iota}$ are given by 
\[ \mathrm{HT}_\tau(r_{\pi, \iota}) = \{ \lambda_{\iota \tau, 1} + (n-1), \lambda_{\iota \tau, 2} + (n-2), \dots, \lambda_{\iota \tau, n} \}. \]

For 
$n \geq 1$, we define a matrix
\[ \Phi_n = \left( \begin{array}{cccc} & & & 1 \\ & & -1 & \\ & \iddots & & \\
\\ (-1)^{n-1} \end{array}\right). \]
If $E / F$ is a quadratic extension of fields of characteristic 0 then we write 
$\theta = \theta_n : \Res_{E / F} \GL_n \to \Res_{E / F} \GL_n$ for the 
involution given by the formula $\theta(g) = \Phi_n c(g)^{-t} \Phi_n^{-1}$. We 
write $U_n \subset \Res_{E / F} \GL_n$ for the fixed subgroup of $\theta_n$. 
Then $U_n$ is a quasi-split unitary group. The standard pinning of $\GL_n$ 
(consisting of the maximal torus of diagonal matrices, Borel subgroup of 
upper-triangular matrices, and set $\{ E_{i, i+1} \mid i = 1, \dots, n-1 \}$ of 
root vectors)  is invariant under the action of $\theta$ and defines an 
$F$-pinning of $U_n$, that we call its standard pinning. If $F$ is a number 
field or a non-archimedean local field, then we also write $U_n$ for the 
extension of $U_n$ to a group scheme over $\cO_F$ with functor of points
\[ U_n(R) = \{ g \in \GL_n(R \otimes_{\cO_F} \cO_E) \mid g = \Phi_n (1 \otimes c)(g)^{-t} \Phi_n^{-1} \}. \]
When $F$ is a number field or a local field, we identify the dual group ${}^L 
U_n = \GL_n(\bC) \rtimes W_F$, where $W_E$ acts trivially on $\GL_n(\bC)$ and 
an element $w_c \in W_F - W_E$ acts by the formula $w_c \cdot g = \Phi_n g^{-t} 
\Phi_n^{-1}$ (therefore preserving the standard pinning of $\GL_n(\bC)$).

Given a partition of $n$ (i.e.\ a tuple $(n_1, n_2, \dots, n_k)$ of natural numbers such that $n_1 + n_2 + \dots + n_k = n$), we write $L_{(n_1, \dots, n_k)}$ for the corresponding standard Levi subgroup of $\GL_n$ (i.e.\ the block diagonal subgroup $\GL_{n_1} \times \dots \times \GL_{n_k} \subset \GL_n$), and $P_{(n_1, \dots, n_k)}$ for the corresponding standard parabolic subgroup (i.e.\ block upper-triangular matrices with blocks of sizes $n_1, \dots, n_k$). If $E$ is a non-archimedean characteristic 0 local field and $\pi_1, \dots, \pi_k$ are admissible representations of $\GL_{n_1}(E), \dots, \GL_{n_k}(E)$, respectively, then we write $\pi_1 \times \pi_2 \times \dots \times \pi_k = i_{P_{(n_1, \dots, n_k)}}^{\GL_n} \pi_1 \otimes \dots \otimes \pi_k$. We write $\pi_1 \boxplus \dots \boxplus \pi_k$ for the irreducible admissible representation of $\GL_n(E)$ defined by $\rec_E(\boxplus_{i=1}^k \pi_i) \cong \oplus_{i=1}^k \rec_E(\pi_i)$; it is a subquotient of $\pi_1 \times \dots \times \pi_k$.

Given a tuple $(n_1, n_2, \dots, n_k)$ of natural numbers such that $2(n_1 + \dots + n_{k-1}) + n_k = n$, we write $M_{(n_1, \dots, n_k)}$ for the Levi subgroup of $U_n$ given by block diagonal matrices with blocks of size $n_1, n_2, \dots, n_{k-1}, n_k, n_{k-1}, \dots, n_1$. Then $M_{(n_1, \dots, n_k)}$ is a standard Levi subgroup (with respect to the diagonal maximal torus of $U_n$), and projection to the first $k$ blocks gives an isomorphism $M_{(n_1, \dots, n_k)} \cong (\Res_{E / F} \GL_{n_1} \times \dots \times \Res_{E / F}  \GL_{n_{k-1}}) \times U_{n_k}$. We write $Q_{(n_1, \dots, n_k)}$ for the parabolic subgroup given by block upper triangular matrices (with blocks of the same sizes). If $F$ is a non-archimedean characteristic 0 local field and $\pi_1, \dots, \pi_{k-1}, \pi_k$ are admissible representations of $\GL_{n_1}(E), \dots, \GL_{n_{k-1}}(E), U_{n_k}(F)$, respectively, then we write $\pi_1 \times \pi_2 \times \dots \times \pi_k = i_{Q_{(n_1, \dots, n_k)}}^{U_n} \pi_1 \otimes \dots \otimes \pi_k$.

\section{Definite unitary groups}\label{sec_definite_unitary_groups}

In this paper we will often use the following assumptions and notation, which 
we call the ``standard assumptions'':
\begin{itemize}
	\item $F$ is a CM number field  such that $F / F^+$ is everywhere 
	unramified. We note this implies that $[F^+ : \Q]$ is even (the 
	quadratic character of 
	$(F^+)^\times\backslash\A_{F^+}^\times/\widehat{\cO}_{F^+}^\times$ cutting out $F$ 
	has non-trivial restriction to $F^+_v$ for each $v|\infty$ but is trivial 
	on $(-1)_{v|\infty} \in (F^+_\infty)^\times$).
	\item $p$ is a prime. We write $S_p$ for the set of $p$-adic places of 
	$F^+$. 
	\item $S$ is a finite set of finite places of $F^+$, all of which split in 
	$F$. $S$ contains $S_p$.
	\item For each $v \in S$, we suppose fixed a factorization $v = \wv \wv^c$ in 
	$F$, and write $\widetilde{S} = \{ \wv \mid v \in S \}$. 
\end{itemize}
Let $n \geq 1$ be an integer. Under the above assumptions we can fix the following data:
\begin{itemize}
	\item The unitary group $G_n = G$ over $F^+$ with $R$-points given by the formula
	\begin{equation}\label{eqn_definite_unitary_group} G(R) = \{ g \in \GL_n(R \otimes_{F^+} F) \mid g =  (1 \otimes c)(g)^{-t} \}. \end{equation}
	We observe that for each finite place $v$ of $F^+$, $G_{F^+_v}$ is 
	quasi-split, while for each place $v | \infty$ of $F^+$, $G(F^+_v)$ is 
	compact. We use the same formula to extend $G$ to a reductive group scheme over $\cO_{F^+}$ (this uses that $F / F^+$ is everywhere unramified).\footnote{The authors apologize for using the same notation $G_L$ to denote both an extension of scalars of the algebraic group $G$ and an absolute Galois group. We hope no confusion will arise.}
	\item The inner twist $\xi : U_{n, F} \to G_F$, given by the formula 
	\[ \xi(g_1, g_2) = (g_1, \Phi_n^{-1} g_2 \Phi_n) \]
	 with respect to the identifications
	\[ U_{n, F} = \{ (g_1, g_2) \in \GL_n \times \GL_n \mid g_2 = \Phi_n g_1^{-t} \Phi_n^{-1} \}\]
	and
	\[ G_F = \{ (g_1, g_2) \in \GL_n \times \GL_n \mid g_2 = g_1^{-t} \}. \]
	\item A lift of $\xi$ to a pure inner twist $(\xi, u) : U_{n, F}\to G_F$. 
	We recall (see e.g.\ \cite{Kal11}) that by definition, this means that $u 
	\in Z^1(F^+, U_n)$ is a cocycle such that for all $\sigma \in G_{F^+}$, we 
	have $\xi^{-1} {}^\sigma \xi = \operatorname{Ad}(u_\sigma)$. When $n$ is 
	odd, we define $u$ to be the cocycle inflated from $Z^1(\Gal(F / F^+), 
	U_n(F))$, defined by the formula $u_1 = 1$, $u_c = (\Phi_n, \Phi_n)$. When 
	$n$ is even, we choose an element $\zeta \in F^\times$ with $\tr_{F/F^+}(\zeta)=0$ 
	and define  $u$ to be the 
	cocycle inflated from $Z^1(\Gal(F / F^+), U_n(F))$, defined by the formula 
	$u_1 = 1$, $u_c = (\zeta \Phi_n, \zeta^{-1} \Phi_n)$. (In fact, we will 
	make 
	essential use of this structure only when $n$ is odd.)
	\item We also fix a choice of continuous character $\mu_F = \mu : F^\times \backslash \A_F^\times \to \bC^\times$ such that $\mu|_{\A_{F^+}^\times} = \delta_{F / F^+} \circ \Art_{F^+}$ and such that if $v$ is any place of $F$ which is inert over $F^+$, then $\mu|_{F_v^\times}$ is unramified.
\end{itemize}
If $v$ is a finite place of $F^+$, then the image of the cocycle $u$ in 
$H^1(F^+_v, U_n)$ is trivial (this is true by Hilbert 90 if $v$ splits in $F$, 
and true because $\det u_c \in \mathbf{N}_{F_\wv / F_v^+} F_\wv^\times$ if $v$ 
is inert in $F$, cf. \cite[\S 1.9]{Rog90}). Our choice of pure inner twist 
$(\xi, u)$ therefore determines a $U_n(F^+_v)$-conjugacy class of isomorphisms 
$\iota_v : G(F^+_v) \to U_n(F^+_v)$ (choose $g \in U_n(\overline{F}^+_v)$ such 
that $g^{-1} {}^c g = u_c$; then $\iota_v$ is the map induced on $F^+_v$-points 
by the map $\operatorname{Ad}(g) \circ \xi^{-1} : G_{\overline{F}^+_v} \to 
U_{n, \overline{F}^+_v}$, which descends to $F^+_v$). If $v$ splits $v = w w^c$ 
in $F$, then we have an isomorphism $\iota_w : G(F^+_v) \to \GL_n(F_w)$ 
(composite of inclusion $G(F^+_v) \subset (\Res_{F / F^+} \GL_n)(F_v^+)$ and 
canonical projection $(\Res_{F / F^+} \GL_n)(F_v^+) \to \GL_n(F_w)$).

If $L^+ / F^+$ is a finite totally real extension, then we will use the following standard notation:
\begin{itemize}
\item We set $L = L^+ F$.
\item If $T$ is a set of places of $F^+$ then we write $T_L$ for the set of 
places of $L^+$ lying above $T$. If $w \in T_L$ lies above $v \in T$ and $v$ 
splits $v = \wv \wv^c$ in $F$ (in particular, we suppose that we have made a choice of $\wv|v$), then we will write $\ww$ for the unique place of 
$L$ which lies above both $w$ and $\wv$ (in which case $w$ splits $w = 
\ww \ww^c$ in $L$). We write e.g.\ $\widetilde{S}_L$ for the set of places of the form $\ww$ ($w \in S_L$).
\end{itemize}
We note that formation of $G$ is compatible with base change, in the sense that the group $G_{L^+}$ is the same as the one given by formula (\ref{eqn_definite_unitary_group}) relative to the quadratic extension $L / L^+$. The same remark applies to the pure inner twist $(\xi, u)$. When we need to compare trace formulae over $F^+$ and its extension $L^+ / F^+$ (a situation that arises in \S \ref{sec:automorphic_level_raising}), we will use the character $\mu_L = \mu_F \circ \mathbf{N}_{L / F}$.

\subsection{Base change and descent -- first cases}
In the next few sections we summarise some results from the literature concerning automorphic representations of the group $G(\bA_{F^+})$. We first give some results which do not rely on an understanding of the finer properties of $L$-packets for $p$-adic unitary groups at inert places of the extension $F / F^+$.
\begin{theorem}\label{thm_base_change}
	Let $\sigma$ be an automorphic representation of $G(\A_{F^+})$. Then there 
	exist a partition $n = n_1 + \dots + n_k$ and discrete, conjugate self-dual automorphic representations 
	\[ \pi_1, 
	\dots, \pi_k \] of 
	\[ \GL_{n_1}(\A_F), \dots, \GL_{n_k}(\A_F), \]
	 respectively, with the 
	following properties:
	\begin{enumerate}
		\item Let $\pi = \pi_1 \boxplus \dots \boxplus \pi_k$. Then for each  finite place $w$ of $F$ below which $\sigma$ is 
		unramified, $\pi_w$ is unramified and is the unramified base change of 
		$\sigma_{w|_{F^+}}$.
		\item For each place $v=ww^c$ of $F^+$ which splits in $F$, $\pi_w\cong \sigma_v \circ 
		\iota_w^{-1}$.
		\item For each place $v | \infty$ of $F$, $\pi_v$ has the same 
		infinitesimal character as $\otimes_{\tau : F_v \to \C} W_\tau$, where 
		$W_\tau$ is the algebraic representation of $\GL_n(F_v) \cong 
		\GL_n(\C)$ such that $\sigma_v \cong W_\tau|_{G(F_v^+)}$.
	\end{enumerate}
\end{theorem}
\begin{proof}
	This follows from \cite[Corollaire 5.3]{labesse}. 
\end{proof}
We call $\pi$ the base change of $\sigma$. If $\iota: \overline{\Q}_p \to \C$ is an isomorphism, we say that 
$\sigma$ is $\iota$-ordinary if $\pi$ is $\iota$-ordinary at all places $w|p$ 
in the sense of 
\cite[Definition 5.3]{ger}. We note that this depends only on $\pi_p$ and the 
weight of $\pi$ (equivalently, on $\sigma_p$ and $\sigma_\infty$).
\begin{cor}\label{cor:galois_existence}  
	Let $\iota: \overline{\Q}_p \to \C$ be an isomorphism. Then there exists a 
	unique continuous semisimple representation $r_{\sigma, \iota} : G_F \to 
	\GL_n(\overline{\Q}_p)$ with the following properties:
	\begin{enumerate}
		\item For each prime-to-$p$ place $w$ of $F$ below which $\sigma$ is 
		unramified, $r_{\sigma, \iota}|_{G_{F_w}}$ is unramified.
		\item For each place $v \in S_p$, $r_{\sigma, \iota}|_{G_{F_\wv}}$ is de Rham.
		\item For each place $v=ww^c$ of $F^+$ which splits in $F$, 
		$\mathrm{WD}(r_{\sigma, 
		\iota}|_{G_{F_w}})^{F-ss} \cong \rec^T_{F_w}(\sigma_v \circ 
		\iota_w^{-1})$.
	\end{enumerate}
\end{cor}
\begin{proof}
	This follows from the classification of discrete automorphic 
	representations of $\GL_{n_i}(\A_F)$ \cite{Moe89}, together with the known existence of 
	Galois representations attached to RACSDC (regular algebraic, conjugate self-dual, cuspidal) automorphic representations of 
	$\GL_{n_i}(\A_F)$ (cf.~\cite[Corollary 3.4]{Ana19}).
\end{proof}
We remark that if $\overline{r}_{\sigma, \iota}|_{G_{F(\zeta_p)}}$ is 
multiplicity free, then the base change of $\sigma$ is $\pi = \pi_1 \boxplus \dots \boxplus 
\pi_k$, where each $\pi_i$ is a \emph{cuspidal} automorphic representation of 
$\GL_{n_i}(\A_F)$. Indeed, \cite{Moe89} shows that a non-cuspidal $\pi_i$ would contribute a direct sum of copies of a single Galois representation twisted by powers of the cyclotomic character to $r_{\sigma,\iota}$, which gives a factor with multiplicity $>1$ in $\overline{r}_{\sigma, \iota}|_{G_{F(\zeta_p)}}$. In particular, if $\overline{r}_{\sigma, \iota}|_{G_{F(\zeta_p)}}$ is 
multiplicity free then $\pi$ is tempered (as each $\pi_i$ is, by the results of \cite{shin,MR3010691,Caraianilnotp}).
\begin{theorem}\label{thm_cuspidal_descent}
Let $\pi$ be a RACSDC  automorphic representation of $\GL_n(\bA_F)$. Suppose that $\pi$ is unramified outside $S$. Then there exists an automorphic representation $\sigma$ of $G(\bA_{F^+})$  with the following properties:
\begin{enumerate}
\item For each finite place $v \not\in S$ of $F$, $\sigma_v^{\iota_v^{-1}(U_n(\cO_{F_v^+}))} \neq 0$.
\item $\pi$ is the base change of $\sigma$.
\end{enumerate}
\end{theorem}
\begin{proof}
This follows from \cite[Th\'eor\`eme 5.4]{labesse}.
\end{proof} 

\subsection{Endoscopic data and normalisation of transfer factors}
\label{subsec:endo}
To go further we need to use some ideas from the theory of endoscopy, both for 
the unitary group $G$ and for the twisted group $\Res_{F / F^+} \GL_{n} \rtimes 
\theta$. We begin by describing endoscopic data for $G$ (cf. \cite[\S 
4.2]{labesse}, \cite[\S 4.6]{Rog90}). The equivalence classes of endoscopic 
data for $G$ are in bijection with pairs $(p, q)$ of integers such that $p + q 
= n$ and $p \geq q \geq 0$. Define $\mu_+ = 1$, $\mu_- = \mu$. We identify 
$\mu_{\pm}$ with characters of the global Weil group $W_F$ using $\Art_F$. Then 
we can write down an extended endoscopic triple $\cE = (H, s, \eta)$ giving 
rise to each equivalence class as follows:
\begin{itemize}
\item The group $H$ is $U_p \times U_q$.
\item $s = \diag(1, \dots, 1, -1, \dots, -1)$ (with $p$ occurrences of $1$ and $q$ occurences of $-1$).
\item $\eta : {}^L H \to {}^L G$ is given by the formulae:
\[ \begin{split} \eta : (g_1, g_2) \rtimes 1 & \mapsto \diag(g_1, g_2) \rtimes 1 \in \GL_n(\bC) = \widehat{G}, \, \, (g_1, g_2) \in \GL_p(\bC) \times \GL_q(\bC) = \widehat{H};\\  
(1_p, 1_q) \rtimes w & \mapsto \diag(\mu_{(-1)^{q}}(w) 1_p, 
\mu_{(-1)^{p}}(w) 1_q) \rtimes w\,\,(w \in W_F) \\ 
(1_p, 1_q) \rtimes w_c & \mapsto \diag(\Phi_p,\Phi_q) 
\Phi_n^{-1} \rtimes w_c,\end{split} \]
where $w_c \in W_{F^+} - W_F$ is any fixed element.
\end{itemize}
As described in \cite[\S 4.5]{labesse},
a choice of extended endoscopic triple $\cE$ determines a 
normalisation of the local transfer factor $\Delta_v^\cE$ ($v$ a place of 
$F^+$) up to non-zero scalar. We will fix a normalisation of local transfer 
factors only when $n$ is odd, using the following observations:
\begin{itemize}
\item The quasi-split group $U_n$, with its standard pinning, has a canonical 
normalisation of transfer factors. Indeed, in this case the Whittaker 
normalisation of transfer factors defined in \cite[\S 5]{Kot99} is independent 
of the choice of additive character and coincides with the transfer factor 
denoted $\Delta_0$ in \cite{Lan87}.
\item Our choice of pure inner twist $(\xi, u) : U_n \to G$ defines a normalisation of the local transfer factors for $G$. This normalisation of local transfer factors satisfies the adelic product formula (a very special case of \cite[Proposition 4.4.1]{Kal18}).
\end{itemize}
A local transfer factor having been fixed, one can define what it means for a 
function $f^H \in C_c^\infty(H(F^+_v))$ (resp. $f^H \in 
C_c^\infty(H(\A_{F^+}))$) to be an endoscopic transfer of a function $f \in 
C_c^\infty(G(F_v^+))$ (resp. $f \in C_c^\infty(G(\A_{F^+}))$). After the work of 
Waldspurger, Laumon, and Ng\^o, any function $f \in C_c^\infty(G(F_v^+))$ 
(resp. $C_c^\infty(G(\A_{F^+}))$) admits an endoscopic transfer (see  
\cite[Th\'eor\`eme 4.3]{labesse} for detailed references). 

We next discuss base change, or in other words, endoscopy for the twisted group 
$\Res_{F / F^+}  \GL_n \rtimes \theta_n$. We will only require the principal 
extended endoscopic triple $(U_n, 1_n, \eta)$, where $\eta : {}^L U_n \to {}^L 
\Res_{F / F^+}  \GL_n$ is defined as follows: first, identify ${}^L \Res_{F / 
F^+} \GL_n = (\GL_n(\bC) \times \GL_n(\bC)) \rtimes W_{F^+}$, where $W_{F^+}$ 
acts through its quotient $\Gal(F / F^+)$ and an element $w_c \in W_{F^+} - 
W_F$ acts by the automorphism $(g_1, g_2) \mapsto (\Phi_n g_2^{-t} \Phi_n^{-1}, 
\Phi_n g_1^{-t} \Phi_n^{-1})$. Then $\eta : {}^L U_n \to {}^L \Res_{F / F^+}  
\GL_n$ 
is given by the formulae:
\[ \begin{split} \eta : (g) \rtimes 1 & \mapsto \diag(g, {}^t g^{-1}) \rtimes 1 \in \GL_n(\bC) \times \GL_n(\bC);\\  
(1_n) \rtimes w & \mapsto \diag(1_n, 1_n) \rtimes w, \, \, (w \in W_F); \\ 
(1_n) \rtimes w_c & \mapsto \diag(\Phi_n, \Phi_n^{-1}) \rtimes w_c.\end{split} 
\]
Following \cite[\S 4.5]{labesse}, we fix the trivial transfer factors in this 
case. By \cite[Lemme 4.1]{labesse}, each function 
$\phi \in C_c^\infty(\Res_{F / F^+}  \GL_n(F_v^+) \rtimes \theta_n)$ admits an 
endoscopic transfer $\phi^{U_n} \in C_c^\infty(U_n(F_v^+))$, and every function 
in $C_c^\infty(U_n(F_v^+))$ arises this way. We will follow \emph{op. cit.} in 
using the following notation: if $f \in C_c^\infty(U_n(F_v^+))$ (or more 
generally, if $U_n$ is replaced by a product of unitary groups) then we write 
$\widetilde{f} \in C_c^\infty(\Res_{F / F^+} \GL_n(F_v^+) \rtimes \theta_n)$ 
for any function that admits $f$ as endoscopic transfer (with respect to the 
principal extended endoscopic triple defined above). 

If $\cE = (H, s, \eta)$ is one of the extended endoscopic triples for $G$ as 
above then, following \cite[\S 4.7]{labesse}, we set $M^H = \Res_{F / F^+} 
H_F$, and write $\widetilde{M}^H$ for the twisted space on $M^H$ associated to 
the non-trivial element of $\Gal(F / F^+)$. Then we may canonically identify 
$M^H = \Res_{F / F^+} \GL_p \times \GL_q$ and $\widetilde{M}^H =  (\Res_{F / 
F^+} \GL_p \times \GL_q) \rtimes ( \theta_p \times \theta_q )$. We will use the 
same notation to describe stable base change for $M^H$. In particular, if $f 
\in C_c^\infty(H(\A_{F^+}))$, then we will use $\widetilde{f} \in 
C_c^\infty(\widetilde{M}^H(\A_{F^+}))$ to denote a function whose endoscopic 
transfer (with respect to the principal extended endoscopic triple for $M^H$, 
defined as above) with respect to the trivial transfer factors is $f$ (cf. 
\cite[Proposition 4.9]{labesse}).

Having fixed the above normalisations, we can now formulate some simple propositions.
\begin{proposition}\label{prop_archimedean_signs}
Let $n \geq 1$ be odd, and let $v$ be an infinite place of $F^+$. Suppose given 
an extended endoscopic triple $\cE = (H, s, \eta)$ as above and a Langlands 
parameter $\varphi_H : W_{F^+_v} \to {}^L H$ such that $\eta \circ \varphi_H$ 
is the Langlands parameter of an irreducible representation $\sigma_v$ of 
$G(F^+_v)$. Let $\pi$ be the (necessarily tempered, $\theta$-invariant) 
irreducible admissible representation of $H(F_\wv)$ associated to the Langlands 
parameter $\varphi_H|_{W_{F_v}}$, and let $f_v \in C^\infty(G(F^+_v))$ be a 
coefficient for $\sigma_v$. Then there is a sign $\epsilon(v, \cE, \varphi_H) 
\in \{ \pm 1 \}$ such that the identity $\widetilde{\pi}(\widetilde{f}_v^H) = 
\epsilon(v, \cE, \varphi_H) \sigma_v(f_v) = \epsilon(v, \cE, \varphi_H)$ holds, 
where the twisted trace is Whittaker normalised (cf.~\cite[\S3.6]{labesse}).
\end{proposition}
\begin{proof}
Let $\Pi(\varphi_H)$ be the $L$-packet of discrete series representations of 
$H(F_v^+)$ associated to $\varphi_H$. According to the main result of 
\cite{Clo82},
there 
is a sign $\epsilon_1 \in \{ \pm 1 \}$ such that 
$\widetilde{\pi}(\widetilde{f}_v^H) = \epsilon_1 \sum_{\sigma_{v, H} \in 
\Pi(\varphi_H)} \sigma_{v, H}(f_v^H)$. According to \cite[Proposition 
5.10]{Kal16}, there is a sign $\epsilon_2 \in  \{ \pm 1 \}$ such that 
$\epsilon_2 \sum_{\sigma_{v, H} \in \Pi(\varphi_H)} \sigma_{v, H}(f_v^H) = 
\sigma_v(f_v)$. We may take $\epsilon(v, \cE, \varphi_H) = \epsilon_1 
\epsilon_2$.
\end{proof}
The sign in Proposition \ref{prop_archimedean_signs} depends on our fixed 
choice of pure inner twist (because it depends on the normalisation of transfer 
factors). We make the following basic but important remark, which is used in 
the proof of Proposition \ref{prop_automorphic_descent}: let $L^+ / F^+$ be a 
finite totally real extension, and let $L = L^+ F$. Then $G_{L^+}$ satisfies 
our standard assumptions, and comes equipped with a pure inner twist by base 
extension. If $v$ is an infinite place of $L^+$, then we have the identity 
$\epsilon(v, \cE_{L^+}, \varphi_H|_{W_{L^+_v}}) = \epsilon(v|_{F^+}, \cE, 
\varphi_H)$. 
\begin{proposition}\label{prop_transfer_at_finite_places}
Let $n \geq 1$ be odd, let $v$ be a finite place of $F^+$, and let $f_v \in C_c^\infty(G(F^+_v))$. Suppose given an extended endoscopic triple $\cE = (H, s, \eta)$.
\begin{enumerate}
\item Suppose that $v$ is inert in $F$ and that $f_v$ is unramified (i.e.\ 
$G(\cO_{F^+_v})$-biinvariant). Suppose given an unramified Langlands parameter 
$\varphi_H : W_{F^+_v} \to {}^L H$ and let $\sigma_{v, H}$, $\sigma_v$ be the 
unramified irreducible representations of $H(F^+_v)$, $G(F^+_v)$ associated to 
the parameters $\varphi_H, \eta \circ \varphi_H$, respectively. Let $\pi$ be 
the unramified irreducible representation of $M^H(F^+_v)$ associated to 
$\varphi_H|_{W_{F_\wv}}$. Then there are identities 
$\widetilde{\pi}(\widetilde{f}_v^H) = \sigma_{v, H}(f_{v}^H) = \sigma_v(f_v)$, 
where the twisted trace is normalised so that $\theta$ fixes the unramified 
vector of $\pi$. (If $\pi$ is generic, this agrees with the Whittaker 
normalisation of the twisted trace.)
\item Suppose that $v = \wv \wv^c$ splits in $F$. Suppose given a bounded Langlands parameter $\varphi_H : W_{F^+_v} \to {}^L H$ and let $\sigma_{v, H}$, $\sigma_v$ be the representations of $H(F^+_v)$, $G(F^+_v)$ associated to the parameters $\varphi_H$, $\eta \circ \varphi_H$, respectively (by the local Langlands correspondence $\rec_{F_\wv}$ for general linear groups). Let $\pi_v$ be the irreducible representation of $M^H(F^+_v)$ associated to $\varphi_H|_{W_{F_\wv}}$. Then there is an identity $\widetilde{\pi}_v(\widetilde{f}_v^H) = \sigma_{v, H}(f_v^H) = \sigma_v(f_v)$, where the twisted trace is Whittaker normalised.  
\end{enumerate}
\end{proposition}
\begin{proof}
It is well-known that these identities hold up to non-zero scalar, which 
depends on the choice of transfer factor; the point here is that, with our 
choices, the scalar disappears. In the first part, the identity 
$\widetilde{\pi}(\widetilde{f}_v^H) = \sigma_{v, H}(f_{v}^H)$ is the 
fundamental lemma for stable base change \cite{Clo90a}. The identity 
$\sigma_{v, H}(f_{v}^H) = \sigma_v(f_v)$ is the fundamental lemma for standard 
endoscopy \cite{Lau08}, which holds on the nose because our transfer factors 
are identified, by the isomorphism $\iota_v : G(F^+_v) \to U_n(F^+_v)$, with 
those defined in \cite{Lan87} with respect to our fixed pinning of 
$U_{n, F^+_v}$; this is the `canonical normalisation' of \cite{Hal93}. If $\pi$ 
is generic then a Whittaker functional is 
non-zero on the unramified vector, which gives the final assertion of the first 
part of the proposition. 

In the second part, the equality $\widetilde{\pi}_v(\widetilde{f}_v^H) = 
\sigma_{v, H}(f_v^H)$ is the fundamental lemma for stable base change in the 
split case, cf. \cite[Proposition 4.13.2]{Rog90} (where the result is stated 
for $U(3)$ but the proof is valid in general). The equality $\sigma_{v, 
H}(f_v^H) = \sigma_v(f_v)$ holds because $\sigma_v$ can be expressed as the 
normalised induction of a character twist of $\sigma_{v, H}$ (after choosing an 
appropriate embedding $H_{F^+_v} \to G_{F^+_v}$ and a parabolic subgroup of 
$G_{F^+_v}$ containing $H_{F^+_v}$) and because the correspondence $f_v \mapsto 
f_v^H$ can in this case be taken to be the corresponding character twist of the 
constant term 
along $H_{F^+_v}$ (cf. 
\cite[Lemma 4.13.1]{Rog90} and \cite[\S \S 3.3--3.4]{shin}, noting that our 
normalisation of transfer factors at the place $v$ in this case agrees on the 
nose with the analogue of the factor written as $\Delta_v^0$ in \emph{loc. 
cit.}, as follows from the definition in \cite{Lan87}).
\end{proof}

\subsection{$L$-packets and types for $p$-adic unitary groups}\label{subsec_local_endoscopic_classification}

Let $v$ be a place of $F^+$ inert in $F$. In this section we follow M{\oe}glin \cite{Moe07, Moe14} in defining $L$-packets of tempered representations for the group $G(F_v^+)$ (equivalently, given our choice of pure inner twist, $U_n(F_v^+)$). 

We write $\mathcal{A}(\GL_n(F_\wv))$ for the set of  isomorphism classes of irreducible admissible representations of $\GL_n(F_\wv)$ over $\bC$,  and $\mathcal{A}_t(\GL_n(F_\wv))$ for its subset of tempered  representations. We define $\mathcal{A}(U_n(F_v^+))$ and $\mathcal{A}_t(U_n(F_v^+))$ similarly. We write $\mathcal{A}^\theta(\GL_n(F_\wv))$ and $\mathcal{A}_t^\theta(\GL_n(F_\wv))$ for the respective subsets of $\theta$-invariant representations (so e.g.\ $\mathcal{A}^\theta(\GL_n(F_\wv))$ is the set of irreducible representations of $\GL_n(F_\wv)$ such that $\pi^\theta := \pi \circ \theta \cong \pi$). Using the local Langlands correspondence $\rec_{F_\wv}$ for $\GL_n(F_\wv)$ (and the Jacobson--Morozov theorem), we can identify $\mathcal{A}(\GL_n(F_\wv))$ with the set of $\GL_n(\bC)$-conjugacy classes of Langlands parameters, i.e.\ the set of $\GL_n(\bC)$-conjugacy classes of continuous homomorphisms $\varphi : W_{F_\wv} \times \SL_2(\bC) \to \GL_n(\bC)$ satisfying the following conditions:
\begin{itemize}
\item $\varphi|_{W_{F_\wv}}$ is semisimple;
\item $\varphi|_{\SL_2(\bC)}$ is algebraic. 
\end{itemize}
Then $\mathcal{A}_t(\GL_n(F_\wv))$ is identified with the set of parameters 
$\varphi$ such that the $\varphi(W_{F_\wv})$ is relatively compact, and 
$\mathcal{A}^\theta(\GL_n(F_\wv))$ is identified with the set of conjugate 
self-dual parameters. We write $\mathcal{A}_t^\theta(\GL_n(F_\wv))_+ \subset 
\mathcal{A}_t^\theta(\GL_n(F_\wv))$ for the subset of parameters $\varphi$ 
which extend to a homomorphism $\varphi_{F_v^+} : W_{F_v^+} \times \SL_2(\bC) 
\to {}^L U_n$. Such an extension,  if it exists, is unique up to 
$\GL_n(\bC)$-conjugacy (see e.g.\ \cite[Theorem 8.1]{Gan12}). The existence of 
such an extension $\varphi_{F_v^+}$ can be equivalently phrased as follows: fix 
a decomposition $\varphi = \oplus_{i \in I} \rho_i^{l_i} \oplus_{j \in J} 
\sigma_j^{m_j} \oplus_{k \in K} (\tau_k \oplus \tau_k^{w_c \vee})^{n_k}$, where:
\begin{itemize}
\item The integers $l_i, m_j, n_k$ are all non-zero.
\item Each representation $\rho_i, \sigma_j, \tau_k$ is irreducible and no two are isomorphic.
\item For each $i$ we have $\rho_i \cong \rho_i^{w_c, \vee}$ and for each $j$ we have $\sigma_j \cong \sigma_j^{w_c, \vee}$. For each $k$ we have $\tau_k \not\cong \tau_k^{w_c, \vee}$.
\item For each $i$, $\rho_i$ is conjugate self-dual of sign $(-1)^{n-1}$ and for each $j$, $\sigma_j$ is conjugate self-dual of sign $(-1)^n$, in the sense of \cite[p.~10]{Gan12}.
\end{itemize}
Then an extension $\varphi_{F_v^+}$ exists if and only if each integer $m_j$ is 
even. If the extension $\varphi_{F_v^+}$ is discrete, in the sense that 
$\operatorname{Cent}(\GL_n(\bC), \operatorname{im} \varphi_{F^+_v})$ is finite, 
then $l_i = 1$ for each $i \in I$ and the sets $J, K$ are empty. If the 
parameter $\varphi_{F_v^+}$ corresponding to a representation $\pi \in 
\mathcal{A}^\theta(\GL_n(F_\wv))_+$ is discrete, then we say that $\pi$ is 
$\theta$-discrete.

Let $\cS_n$ denote the set of equivalence classes of pairs $( (n_1, \dots, n_k), (\pi_1, \dots, \pi_k) )$, where $(n_1, \dots, n_k)$ is a partition of $n$ and $\pi_1, \dots, \pi_k$ are supercuspidal representations of $\GL_{n_1}(F_\wv), \dots, \GL_{n_k}(F_\wv)$, respectively. Two such pairs are said to be equivalent if they are isomorphic after permutation of the indices $\{ 1, \dots, k \}$. Thus we may think of an element of $\cS_n$ as a formal sum of supercuspidal representations. We recall (see e.g.\ \cite{BZ77ENS}) that to any $\pi \in \mathcal{A}(\GL_n(F_\wv))$ we may associate the supercuspidal support $sc(\pi) \in \cS_n$, defined by the condition that $\pi$ occurs as an irreducible subquotient of the representation $\pi_1 \times \pi_2 \times \dots \times \pi_k$ (notation for induction as defined at the beginning of this paper).

M{\oe}glin associates to any element $\tau \in \mathcal{A}_t(U_n(F^+_v))$ its 
\emph{extended cuspidal support} $esc(\tau) \in \cS_n$. We do not recall the 
definition here but note that its definition can be reduced to the case where 
$\tau$ is supercuspidal, in the following sense: suppose that $\tau$ is a 
subquotient of a representation 
\[ \pi_1 \times \dots \times \pi_{k-1} \times \tau_0 = i_{Q_{(n_1, \dots, n_k)}}^{U_n} 
\pi_1 \otimes \dots \otimes \pi_{k-1} \otimes \tau_0, \]
where $\tau_0$ is a 
supercuspidal representation of $U_{n_k}(F_v^+)$. Then $esc(\tau) = sc(\pi_1) + 
\dots + sc(\pi_{k-1}) + esc(\tau_0) + sc(\pi_{k-1}^\theta) + \dots + 
sc(\pi_1^\theta)$.
\begin{proposition}
If $\tau \in  \mathcal{A}_t(U_n(F^+_v))$ then there is a unique element $\pi_\tau \in \mathcal{A}_t^\theta(\GL_n(F_\wv))_+$ such that $esc(\tau) = sc(\pi_\tau)$.
\end{proposition}
\begin{proof}
\cite[Lemme 5.4]{Moe07} states that there is a unique element $\pi = \pi_\tau 
\in \mathcal{A}_t^\theta(\GL_n(F_\wv))$ such that $esc(\tau) = sc(\pi)$. We 
need to explain why in fact $\pi \in \mathcal{A}_t^\theta(\GL_n(F_\wv))_+$. 
\cite[Th\'eor\`eme 5.7]{Moe07} states that this is true when $\tau$ is 
square-integrable. In general, we can find a Levi subgroup $M_{(n_1, \dots, 
n_k)} \subset U_n$ and an irreducible square-integrable representation $\pi_1 
\otimes \dots \otimes \pi_{k-1} \otimes \tau_0$ of $M_{(n_1, \dots, 
n_k)}(F^+_v)$ such that $\tau$ is a subquotient of $i_{Q_{(n_1, \dots, 
n_k)}}^{U_n} \pi_1 \otimes \dots \otimes \pi_{k-1} \otimes 
\tau_0$ (see \cite[Proposition III.4.1]{Wal03}). Then $\pi_\tau = (\pi_1 \times 
\pi_1^\theta) \times (\pi_2 \times \pi_2^\theta) \times \dots \times ( 
\pi_{k-1} \times \pi_{k-1}^\theta) \times \pi_{\tau_0}$, so the result follows 
from the square-integrable case.
\end{proof}
According to the proposition, there is a well-defined map 
\[ BC : \mathcal{A}_t(U_n(F_v^+)) \to \mathcal{A}_t^\theta(\GL_n(F_\wv))_+ \]
defined by $BC(\tau) = \pi_\tau$ (which might be called stable base change).
\begin{proposition}
The map $BC$ is surjective, and it has finite fibres.
\end{proposition}
\begin{proof}
The image of $BC$ contains the $\theta$-discrete representations,  and the fibres of $BC$ above such representations are finite, by \cite[Th\'eor\`eme 5.7]{Moe07}. The general case can again be reduced to this one.
\end{proof}
If $\pi \in \mathcal{A}_t^\theta(\GL_n(F_\wv))_+$,  then we define $\Pi(\pi) = BC^{-1}(\pi)$. By definition, the sets $\Pi(\pi)$ partition  $\mathcal{A}_t(U_n(F^+_v)) $ and therefore deserve to be called $L$-packets. The following proposition is further justification for this.
\begin{proposition}\label{prop_L-packets_via_character_identities}
Let $\pi \in \mathcal{A}_t^\theta(\GL_n(F_\wv))_+$, and fix an extension $\widetilde{\pi}$ to the twisted group $\GL_n(F_\wv) \rtimes \theta$. Then there are constants $c_\tau \in \bC^\times$ such that for   any $f  \in C_c^\infty(\GL_n(F_\wv) \rtimes \theta)$:
\[  \widetilde{\pi}(f) = \sum_{\tau \in \Pi(\pi)} c_\tau \tau(f^{U_n}). \]
\end{proposition}
\begin{proof}
When $\pi$ is $\theta$-discrete, this is the content of \cite[Proposition 
5.5]{Moe07}. In general, $\Pi(\pi)$ admits the following explicit description: 
decompose $\pi = \pi_1  \times \pi_2 \times \pi_1^\theta$, where $\pi_1 \in 
\mathcal{A}_t(\GL_{n_1}(F_\wv))$ and  $\pi_2 \in 
\mathcal{A}_t^\theta(\GL_{n-2n_1}(F_\wv))_+$ is $\theta$-discrete. Then 
$\Pi(\pi)$ is the set of Jordan--H\"older factors of the induced 
representations $\pi_1 \times \tau_2$ as $\tau_2$ varies over the set of 
elements of $\Pi(\pi_2)$. Using the compatibility of transfer with normalised 
constant terms along a parabolic (see \cite[Lemma 6.3.4]{Mor10}) we thus have 
an identity
\[ \widetilde{\pi}(f) = (\pi_1 \times \pi_2 \times \pi_1^\theta)^\sim(f) = 
\sum_{\tau_2\in \Pi(\pi_2)} c_{\tau_2} (\pi_1 \times \tau_2)(f^{U_n}) 
\]
for some constants 
$c_{\tau_2} \in \bC^\times$. To prove the proposition, it 
is  enough to show that if $\tau_2,  \tau_2' \in \Pi(\pi_2)$ are non-isomorphic 
then the induced representations $\pi_1 \times \tau_2$, $\pi_1 \times \tau'_2$ 
have no Jordan--H\"older  factors in common. This follows from 
\cite[Proposition III.4.1]{Wal03}.
\end{proof}
We now introduce some particular representations of $U_n$. These are built out 
of depth zero supercuspidal representations of $U_3$. Accordingly we first 
introduce some cuspidal representations of the finite group of Lie type 
$U_3(k(v))$:
\begin{itemize}
\item We write $\tau(v)$ for the unique cuspidal unipotent representation of $U_3(k(v))$  (see \cite[\S9]{lusztig}).
\item Let $k_3 / k(v)$ be a degree 3 extension, and define
\[ C = \ker(\mathbf{N}_{k_3 k(\wv) / k_3} : 
\Res_{k_3 k(\wv) /  k(v)} \bG_m \to \Res_{k_3 / k(v) } \bG_m). \]
 Then there is a 
unique $U_3(k(v))$-conjugacy class of embeddings $C \to U_{3, k(v)}$ (as can be proved using e.g.\  \cite[Corollary 1.14]{deligne-lusztig}). 

 Let $p$ 
be a prime such that $q_v$ is a primitive $6^\text{th}$ root of unity modulo 
$p$, and let $\theta : C(k(v)) \to \bC^\times$ be a character of order $p$. 
Then we write $\lambda(v, \theta)$ for the (negative of the) Deligne--Lusztig induction 
$-R_C^{U_{3, k(v)}} \theta$. Then $\lambda(v, \theta)$ is a cuspidal 
irreducible representation of $U_3(k(v))$ (note that $C$ is not contained in any proper $k(v)$-rational parabolic of $U_{3,k(v)}$).
\end{itemize}
We define $\widetilde{C} = \Res_{k_3 k(\wv) /  k(\wv)} \bG_m$. Then the 
homomorphism $\widetilde{C}(k(\wv)) \to C(k(v))$, $z \mapsto z / z^{c}$, is 
surjective, and we define a character $\widetilde{\theta} : 
\widetilde{C}(k(\wv)) \to \bC^\times$ of order $p$ by $\widetilde{\theta}(z) = 
\theta(z / z^c)$. There is a unique $\GL_3(k(\wv))$-conjugacy class of 
embeddings $\widetilde{C} \to \GL_{3, k(\wv)}$, and we write 
$\widetilde{\lambda}(\wv, \widetilde{\theta})$ for the Deligne--Lusztig 
induction $R_{\widetilde{C}}^{\GL_{3, k(\wv)}} \widetilde{\theta}$. Then 
$\widetilde{\lambda}(\wv, \widetilde{\theta})$ is a cuspidal irreducible 
representation of $\GL_3(k(\wv))$. 

We now assume that the residue characteristic of $k(v)$ is odd.
\begin{proposition}\label{prop_supercuspidals_on_U_3}
\begin{enumerate}
\item Let $\tau_v = \cInd_{U_3(\cO_{F^+_v})}^{U_3(F^+_v)} \tau(v)$ (compact induction). Then 
$\tau_v$ is a supercuspidal irreducible admissible representation of 
$U_3(F^+_v)$ and $BC(\tau_v) = \St_2(\chi) \boxplus \mathbf{1}$, where $\chi : 
F_{\wv}^\times \to \bC^\times$ is the unique non-trivial quadratic unramified 
character.
\item Let $\lambda_v(\theta) = \cInd_{U_3(\cO_{F^+_v})}^{U_3(F^+_v)} \lambda(v, 
\theta)$. Then $\lambda_v(\theta)$ is a supercuspidal irreducible admissible 
representation of $U_3(F^+_v)$.
\item Extend $\widetilde{\lambda}(\wv, \widetilde{\theta})$ to a representation 
of $F_\wv^\times \GL_3(\cO_{F_\wv})$ by making $F_\wv^\times$ act trivially, 
and let $\widetilde{\lambda}_\wv(\widetilde{\theta}) = \cInd_{F_\wv^\times 
\GL_3(\cO_{F_\wv})}^{\GL_3(F_\wv)} \widetilde{\lambda}(\wv, 
\widetilde{\theta})$. Then $\widetilde{\lambda}_\wv(\widetilde{\theta})$ is a 
supercuspidal irreducible admissible representation of $\GL_3(F_\wv)$, and 
$BC(\lambda_v(\theta)) = \widetilde{\lambda}_\wv(\widetilde{\theta})$. 
\end{enumerate}
\end{proposition}
\begin{proof}
If $\mu_0$ is a cuspidal irreducible representation of $U_3(k(v))$, then 
$\cInd_{U_3(\cO_{F^+_v})}^{U_3(F^+_v)} \mu_0$ is a supercuspidal, irreducible 
admissible representation of $U_3(F^+_v)$ (see \cite[Proposition 6.6]{Moy96} -- we will return 
to this theme shortly). The essential point therefore is to calculate the 
extended cuspidal support in each case, which can be done using the results of 
\cite{lust2016depth} (which require the assumption that $k(v)$ has odd 
characteristic). Indeed $\S8$ in \emph{op.~cit.}~explains how to compute the 
reducibility points $\mathrm{Red}(\pi)$ (defined in \cite[\S4]{Moe07}) of a 
depth $0$ supercuspidal representation, at least up to unramified twist. We 
compute that $\mathrm{Red}(\tau_v) = \{(\mathbf{1},3/2),(\chi,1)\}$ or 
$\{(\mathbf{1},1),(\chi,3/2)\}$ which corresponds to $BC(\tau_v) = 
\St_2(\mathbf{1}) \boxplus \chi$ or $BC(\tau_v) = \St_2(\chi) \boxplus 
\mathbf{1}$. Since $BC(\tau_v) \in \mathcal{A}_t^{\theta}(\GL_3(F_{\wv}))_+$, 
the second alternative holds. For $\lambda_v(\theta)$, we deduce that 
$\mathrm{Red}(\lambda_v(\theta)) =\{(\rho,1)\}$, where $\rho$ is a conjugate 
self-dual unramified twist of $\widetilde{\lambda}_\wv(\widetilde{\theta})$. We 
again conclude by sign considerations. 
\end{proof}
\begin{cor}\label{cor_descent_L_packet_contains_tau}
Let $n = 2k + 1$ be an odd integer, and consider a representation 
\[ \pi = \St_2(\chi) \boxplus \mathbf{1} \boxplus (\boxplus_{i=1}^{2k-2} \chi_i) \in \mathcal{A}_t^\theta(\GL_n(F_\wv))_+, \]
 where $\chi : F_\wv^\times \to \bC^\times$ is the unique non-trivial quadratic 
 unramified character and for each $i = 1, \ldots, 2k-2$, $\chi_i : 
 F_\wv^\times \to 
 \bC^\times$ is a character such that $\chi_i|_{\cO_{F_\wv}^\times}$ has order 
 2. We can assume, after relabelling, that $\chi_i = \chi_{2k-1-i}^{w_c, \vee}$ 
 ($i = 1, \dots, k-1$),
and then $\Pi(\pi)$ contains each irreducible 
 subquotient of the induced representation $\chi_1 \times \chi_2 \times 
 \dots \times \chi_{k-1} \times \tau_v$.
\end{cor}
\begin{proof} First we explain why we can relabel the characters so that $\chi_i = \chi_{2k-1-i}^{w_c, \vee}$. Considering the explicit description of $\mathcal{A}_t^\theta(\GL_n(F_\wv))_+$, we need to explain why conjugate self-dual characters must appear with even multiplicity amongst the $\chi_i$. Suppose $\chi_1$ is conjugate self-dual. We know that $\chi_1|_{\cO_{F_\wv}^\times}$ is the non-trivial quadratic character, so there are two possibilities for $\chi_1$ determined by $\chi_1(\varpi_v) = -1$ or $1$ (this value is also the sign of $\chi_1$). If the sign is $-1$, the multiplicity of $\chi_1$ is one of the even exponents $m_j$. If the sign is $+1$, dimension reasons force its multiplicity to be even. The rest of the Corollary follows from the definition of $\Pi(\pi)$ in terms of extended 
supercuspidal supports. Note that we do not claim that $\Pi(\pi)$ contains
\emph{only} the subquotients of this induced representation -- this is not true 
even when $k = 1$. 
\end{proof}
To exploit Corollary \ref{cor_descent_L_packet_contains_tau} we need to 
introduce some results from the theory of types. We state only the results we 
need, continuing to assume that $n = 2k + 1$ is odd. Let $\p_{v}$ denote the 
standard parahoric subgroup of $U_n(\cO_{F^+_v})$ associated to the partition 
$(1, 1, \dots, 1, 3)$; in other words, the pre-image under the reduction modulo 
$\varpi_v$ map $U_n(\cO_{F^+_v}) \to U_n(k(v))$ of $Q_{(1, 1, \dots, 1, 
3)}(k(v))$. Projection to the Levi factor gives a surjective homomorphism 
$\p_{v} 
\to M_{(1, 1, \dots, 1, 3)}(k(v)) \cong (k(\wv)^\times)^{k-1} \times 
U_3(k(v))$.  

Given a cuspidal representation $\sigma(v)$ of $M_{(1, 1, \dots, 1, 3)}(k(v)) \cong (k(\wv)^\times)^{k-1} \times U_3(k(v))$, the pair $(\p_v, \sigma(v))$ defines a depth zero unrefined minimal $K$-type in the sense of \cite{Moy96}. In this case we write $\cE(\sigma_v)$ for the set of irreducible representations of $(F_\wv^\times)^{k-1} \times U_3(\cO_{F^+_v}) \subset M_{(1, 1, \dots, 1, 3)}(F^+_v)$ whose restriction to $M_{(1, 1, \dots, 1, 3)}(\cO_{F^+_v}) = (\cO_{F_\wv}^\times)^{k-1} \times U_3(\cO_{F^+_v})$ is isomorphic to (the inflation of) $\sigma(v)$. We have the following result.
\begin{proposition}\label{prop_moy_prasad}
Let $\sigma(v)$ be a cuspidal irreducible representation of $M_{(1, 1, \dots, 1, 3)}(k(v))$. Then:
\begin{enumerate}
\item For any $\sigma' \in \cE(\sigma(v))$, the compact induction $\cInd_{(F_\wv^\times)^{k-1} \times U_3(\cO_{F^+_v})}^{M_{(1, 1, \dots, 1, 3)}(F^+_v)} \sigma'(v)$ is irreducible and supercuspidal.
\item Let $\pi$ be an irreducible admissible representation of $U_n(F^+_v)$. Then 
$\pi|_{\p_v}$ contains $\sigma(v)$ if and only if $\pi$ is a subquotient of an 
induced representation 
\[ i_{Q_{(1, 1, \dots, 1, 3)}}^{U_n} \cInd_{(F_\wv^\times)^{k-1} \times U_3(\cO_{F^+_v})}^{M_{(1, 1, \dots, 1, 3)}(F^+_v)} \sigma'(v) \]
 for some $\sigma' \in \cE(\sigma(v))$.
\end{enumerate}
\end{proposition}
\begin{proof}
See \cite[Proposition 6.6]{Moy96} and \cite[Theorem 6.11]{Moy96}. 
\end{proof}
We now describe explicitly the two types that we need. Recall that we are 
assuming that the characteristic of $k(v)$ is odd. Let $\omega(\wv) : k(\wv)^\times \to \{ \pm 1 \}$ 
denote the unique 
non-trivial quadratic character of $k(\wv)^\times$.
\begin{itemize}
\item The representation $\tau(v, n)$ of $\mathfrak{p}_v$ inflated from the representation
\[ \omega(\wv) \otimes \dots \otimes \omega(\wv) \otimes \tau(v) \]
of $M_{(1, 1, \dots, 1, 3)}(k(v))$.
\item The representation $\lambda(v, \theta, n)$ of $\mathfrak{p}_v$ inflated from the representation
\[ \omega(\wv) \otimes \dots \otimes \omega(\wv) \otimes \lambda(v, \theta) \]
of $M_{(1, 1, \dots, 1, 3)}(k(v))$ (where $\theta$ as above is a character of $C(k(v))$ of order $p$, and we assume $q_v \text{ mod }p$ is a primitive $6^\text{th}$ root of unity).
\end{itemize}
These types are introduced because they are related by a congruence modulo $p$, because of our assumption that $q_v \text{ mod }p$ is a primitive $6^\text{th}$ root of unity:
\begin{proposition}\label{prop_congruence_of_types}
Fix an isomorphism $\overline{\bQ}_p \to \bC$ and use this to view $\tau(v, n)$ and 
$\lambda(v, \theta, n)$ as representations with coefficients in $\Qpbar$. Then:
\begin{enumerate}
\item $\overline{\tau}(v, n)$ is irreducible.
\item $\overline{\lambda}(v, \theta, n)$ contains $\overline{\tau}(v, n)$ as a 
Jordan--H\"older factor with multiplicity 1. 
\end{enumerate}
\end{proposition}
(As usual, overline denotes semi-simplified residual representation over $\overline{\bF}_p$.)
\begin{proof}
The modular irreducibility of cuspidal unipotent representations is a general 
phenomenon (see \cite{Dud18}). The proposition is a statement about
representations of $U_3(k(v))$, which can be proved by explicit computation with 
Brauer characters; see \cite[Theorem 4.2]{Gec90} (although note that there is a 
typo in the proof: the right-hand side of the first displayed equation should 
have $\widehat{\chi}_1$ in place of $\widehat{\chi}_{q^2 - q}$).
\end{proof}
The following proposition will be a useful tool for exploiting the type $(\p_v, \lambda(v, \theta, n))$. We introduce an associated test function $\phi(v, \theta, n)\in C_c^\infty(U_n(F^+_v))$: it is the function supported on $\mathfrak{p}_v$ and inflated from the character of $\lambda(v, \theta, n)^\vee$. If $\pi$ is an admissible representation of $U_n(F^+_v)$, then $\pi(\phi(v, \theta, n))$ is (up to a positive real scalar depending on normalisation of measures) the dimension of the space $\Hom_{\p_v}(\lambda(v, \theta, n), \pi|_{\p_v})$.
\begin{proposition}\label{prop_kazhdan_varshavsky}
Assume that the characteristic of $k(v)$ is greater than $n$. Let $\phi = 
\phi(v, \theta, n)$, and let $\cE = (H, s, \eta)$ be one of our fixed 
endoscopic triples for $U_n$, with $H = U_p \times U_q$. Suppose given 
representations $\pi_p$, $\pi_q$ in $\mathcal{A}_t^\theta(\GL_p(F_\wv))_+$, 
$\mathcal{A}_t^\theta(\GL_q(F_\wv))_+$, respectively, such that $(\pi_p \otimes 
\pi_q)^\sim(\widetilde{\phi}^H) \neq 0$. Then $sc(\pi_p) + sc(\pi_q) = 
\lambda_\wv(\widetilde{\theta}) + \chi_1 + \dots + \chi_{2k-2}$, where $\chi_1, 
\dots, \chi_{2k-2} : F_\wv^\times \to \bC^\times$ are characters such that for 
each $i = 1, \dots, 2k-2$, $\chi_i|_{\cO_{F_\wv}^\times} = \omega(\wv)$.
\end{proposition}
\begin{proof}
By Proposition \ref{prop_L-packets_via_character_identities}, there is an identity 
\[ (\pi_p \otimes \pi_q)^\sim(\widetilde{\phi}^H) = \sum_{\substack{\tau_p \in \Pi(\pi_p)  \\ \tau_q \in \Pi(\pi_q)}} c_{\tau_p} c_{\tau_q} (\tau_p \otimes \tau_q)(\phi^H) \]
for some constants $c_{\tau_p}, c_{\tau_q} \in \bC^\times$. Now, \cite[Theorem 
2.2.6]{Kaz12} shows that $\phi^H$ can be taken to be a weighted sum of 
inflations to $H(\cO_{F^+_v})$ of characters $R^{H_{k(v)}}_{C_i} (\theta^{-1} 
\otimes \omega(\wv)^{\otimes (k-1)})$ associated to conjugacy classes of 
embeddings $C_i : C \times \Res_{k(\wv) / k(v)} \bG_m^{k-1} \to {H_{k(v)}}$. 
(Our appeal to this reference is the reason for the additional assumption on 
the characteristic of $k(v)$ in the statement of the theorem.) If  
$(\pi_p \otimes \pi_q)^\sim(\widetilde{\phi}^H) \neq 0$, then there exists a 
summand on the right-hand side such that $\tau_p \otimes \tau_q$ contains the 
inflation to $H(\cO_{F^+_v})$ of the (irreducible) representation with 
character $-R^{H_{k(v)}}_{C_i}(\theta 
\otimes \omega(\wv)^{\otimes (k-1)})$. 
Taking into account the compatibility between parabolic induction and 
Deligne--Lusztig induction, the transitivity of Deligne--Lusztig induction 
\cite{Lus76}, and Proposition \ref{prop_moy_prasad}, we see that for one of the 
representations $\tau_p, \tau_q$ (the one for the factor of even rank), the 
extended cuspidal support is a sum of 
characters of $F_\wv^\times$, each of which is the twist of an unramified 
character by a ramified quadratic character; and for the other of the 
representations $\tau_p, \tau_q$, the extended cuspidal support is a sum of 
such characters, together with $\lambda_\wv(\widetilde{\theta})$. This 
completes the proof.
\end{proof}

\subsection{Types for the general linear group}\label{subsec_local_theory_GL_n}

In this section we record some analogues of the results of the previous section for general linear groups. Let $2 \leq n_1 \leq n$ be an integer. Let $\wv$ be a finite place of $F$. We assume that the characteristic of $k(\wv)$ is odd. We have already introduced the notation $\omega(\wv)$ for the unique non-trivial quadratic character of $k(\wv)^\times$. We introduce a further representation of the finite group $\GL_{n_1}(k(\wv))$ of Lie type:
\begin{itemize}
\item Let $k_{n_1} / k(\wv)$ be an extension of degree $n_1$, and suppose that 
$q_\wv \text{ mod }p$ is a primitive $n_1^\text{th}$ root of unity modulo $p$. 
Let $\Theta : k_{n_1}^\times \to \bC^\times$ be a character of order $p$. Then 
$\Theta$ is distinct from its conjugates by $\Gal(k_{n_1} / k_\wv)$, and we write $\widetilde{\lambda}(\wv, \Theta) 
= (-1)^{n_1-1} R^{\GL_{n_1}}_{\Res_{k_{n_1} / k(\wv)}\bG_m} \Theta$ for the 
Deligne--Lusztig induction. Then $\widetilde{\lambda}(\wv, \Theta)$ is an 
irreducible representation of $\GL_{n_1}(k(\wv))$.
\end{itemize}
The notation $\widetilde{\lambda}(\wv, \Theta)$ thus generalises that introduced in the previous section (where $n_1 = 3$ and $\Theta = \widetilde{\theta}$). 
\begin{proposition}\label{prop_depth_zero_supercuspidal_on_gl_n}
Let $\pi$ be an irreducible admissible representation of $\GL_{n_1}(F_\wv)$, and let $F_{\wv, n_1} / F_\wv$ denote an unramified extension of degree $n_1$. Then the following are equivalent:
\begin{enumerate}
\item The restriction of $\pi$ to $\GL_{n_1}(\cO_{F_\wv})$ contains $\widetilde{\lambda}(\wv, \Theta)$. 
\item There exists a continuous character $\chi : F_{\wv, n_1}^\times \to \bC^\times$ such that $\chi|_{\cO_{F_{\wv, n_1}}^\times} = \Theta$ and $\rec_{F_\wv} \pi \cong \Ind_{W_{F_{\wv, n_1}}}^{W_{F_\wv}} (\chi \circ \Art_{F_{\wv, n_1}}^{-1})$. In particular, $\pi$ is supercuspidal.
\end{enumerate}
\end{proposition}
\begin{proof}
This follows from the results of \cite{Hen92} (see especially \S 3.4 of that paper) and \cite{Moy96}.
\end{proof}
Let $n_2 = n - n_1$. We write $\q_\wv \subset \GL_n(\cO_{F_\wv})$ for the standard parahoric subgroup associated to the partition $(n_1, n_2)$, i.e.\ the pre-image under the reduction modulo $\varpi_\wv$ map $\GL_n(\cO_{F_\wv}) \to \GL_n(k(\wv))$ of $P_{(n_1, n_2)}(k(\wv))$. We write $\widetilde{\lambda}(\wv, \Theta, n)$ for the irreducible representation of $\q_\wv$ inflated from the representation $ \widetilde{\lambda}(\wv, \Theta) \otimes (\omega(\wv) \circ \det)$ of $L_{(n_1, n_2)}(k(\wv))$. We write $\mathfrak{r}_\wv \subset \q_\wv$ for the standard parahoric subgroup associated to the partition $(n_1, 1, 1, \dots, 1)$. Then we have the following analogue of Proposition \ref{prop_moy_prasad}:
\begin{proposition}\label{prop_cuspidal_types_for_general_linear_group}
Let $\pi$ be an irreducible admissible representation of $\GL_n(F_\wv)$. Then the following are equivalent:
\begin{itemize}
\item $\pi|_{\mathfrak{r}_\wv}$ contains $\widetilde{\lambda}(\wv, \Theta, n)|_{\mathfrak{r}_\wv}$. 
\item $sc(\pi) = \pi_1 + \chi_1 + \dots + \chi_{n_2}$, where $\pi_1$ satisfies the equivalent conditions of Proposition \ref{prop_depth_zero_supercuspidal_on_gl_n} and $\chi_1, \dots, \chi_{n_2} : F_\wv^\times \to \bC^\times$ are characters such that for each $i = 1, \dots, n_2$, $\chi_i|_{\cO_{F_\wv}^\times} = \omega(\wv)$.
\end{itemize}
\end{proposition}
\begin{proof}
This once again follows from the results of \cite{Moy96}.
\end{proof}
The pair $(\q_\wv, \widetilde{\lambda}(\wv, \Theta, n))$ is not in general a type (because $\widetilde{\lambda}(\wv, \Theta, n)$ is not a cuspidal representation of $L_{(n_1, n_2)}(k(\wv))$ unless $n_2 = 1$). Nevertheless, we have the following proposition:
\begin{proposition}\label{prop_types_for_general_linear_group}
Let $\pi$ be an irreducible admissible representation of $\GL_n(F_\wv)$. Then the following are equivalent:
\begin{enumerate}
\item The restriction of $\pi$ to $\q_\wv$ contains $\widetilde{\lambda}(\wv, \Theta, n)$.
\item There exist irreducible admissible representations $\pi_i$ of $\GL_{n_i}(F_\wv)$ ($i = 1, 2$) such that $\pi = \pi_1 \boxplus \pi_2$, the restriction of $\pi_1$ to  $\GL_{n_1}(\cO_{F_\wv})$ contains $\widetilde{\lambda}(\wv, \Theta)$, and the restriction of $\pi_2$  to $\GL_{n_2}(\cO_{F_\wv})$ contains $\omega(\wv) \circ \det$.
\end{enumerate}
\end{proposition}
We note that in the situation of the proposition, $\pi_2$ is the twist of an unramified representation by a quadratic ramified character.
\begin{proof}
Let $P = P_{(n_1, n_2)}$, $L = L_{(n_1, n_2)}$, and let $N_P$ denote the unipotent radical of $P$. Abbreviate $\widetilde{\lambda} = \widetilde{\lambda}(\wv, \Theta, n)$ and $\widetilde{\lambda}_{N_P} = \widetilde{\lambda}|_{L(\cO_{F_\wv})}$. If $\pi$ is an irreducible admissible representation of $\GL_n(F_\wv)$ then we define $\pi^{\widetilde{\lambda}} = \Hom_{\q_\wv}(\widetilde{\lambda}, \pi|_{\q_\wv})$. We first show that for any  admissible representation $\pi$ of  $\GL_n(F_\wv)$, the natural projection $\pi^{\widetilde{\lambda}} \to \pi_{N_P}^{\widetilde{\lambda}_{N_P}}$ (restriction of projection to unnormalised Jacquet module) is an isomorphism. Indeed, it is surjective by \cite[II.10.1, 1)]{Vig98}. To show that it is injective, let $\widetilde{\mu} = \widetilde{\lambda}|_{\mathfrak{r}_\wv}$ and let $R = P_{(n_1, 1, 1, \dots, 1)}$, $N_R$ the unipotent radical of $R$. Then the pair $(\mathfrak{r}_\wv,  \widetilde{\mu})$ is a depth zero unrefined minimal $K$-type  in the sense of \cite{Moy96}.  We  now  have  a commutative diagram
\[ \xymatrix{ \pi^{\widetilde{\lambda}} \ar[r] \ar[d] & \pi_{N_P}^{\widetilde{\lambda}_{N_P}} \ar[d] \\ \pi^{\widetilde{\mu}} \ar[r]  & \pi_{N_R}^{\widetilde{\mu}_{N_R}}, } \]
where the left vertical arrow is  the natural inclusion and the right vertical 
arrow is the natural projection to co-invariants. The lower horizontal arrow is 
an isomorphism, by \cite[Lemma 3.6]{Mor99}. We conclude that the top horizontal 
arrow is injective,  and therefore an isomorphism.

Suppose now that $\pi$ is an irreducible admissible representation of $\GL_n(F_\wv)$ and that $\pi^{\widetilde{\lambda}} \neq 0$. Then $\pi^{\widetilde{\mu}} \neq 0$, so  by Proposition \ref{prop_cuspidal_types_for_general_linear_group},  $\pi$ is an irreducible subquotient of an induced representation $\pi' = \pi_1 \times \chi_1 \times \dots \times \chi_{n_2}$, where the inducing data is as in the statement of that proposition. Computation of the Jacquet module (using the geometric lemma \cite[Lemma 2.12]{BZ77ENS}) shows that $(\pi')^{\widetilde{\lambda}}$ has dimension 1; therefore $\pi$ must be isomorphic to the unique irreducible subquotient of $\pi'$ which contains $\widetilde{\lambda}$. This is $\pi_1 \times \pi_2$, where $\pi_2$ is the unique irreducible subquotient of $\chi_1 \times \dots \times \chi_{n_2}$ such that $\pi_2|_{\GL_{n_2}(\cO_{F_\wv})}$ contains $\omega(\wv) \circ \det$ (note that $\pi_1 \times \pi_2$ is irreducible, by \cite[Proposition 8.5]{Zel80}).

Suppose instead that $\pi = \pi_1 \boxplus \pi_2 = \pi_1 \times \pi_2$, with $\pi_1, \pi_2$ as in the statement of the proposition. Then the geometric lemma shows that $\pi_{N_P}^{\widetilde{\lambda}_{N_P}} \neq 0$, hence $\pi^{\widetilde{\lambda}} \neq 0$.
\end{proof}
We now introduce the local lifting ring associated to the inertial type which 
is the analogue, on the Galois side, of the pair $(\q_\wv, 
\widetilde{\lambda}(\wv, \Theta, n))$ introduced above. We recall that $k_{n_1} / k(\wv)$ is an extension of degree $n_1$, $q_\wv \text{ mod }p$ is a primitive $n_1^\text{th}$ root of unity modulo $p$, and $\Theta : k_{n_1}^\times \to \bC^\times$ is a character of order $p$.  Let $\iota : 
\overline{\bQ}_p \to \bC$ be an isomorphism, so that $\iota^{-1} \Theta : 
k_{n_1}^\times \to \overline{\bQ}_p^\times$ is a character with trivial 
reduction modulo $p$. Fix a coefficient field $E$ and suppose given a 
representation $\overline{\rho}_\wv : G_{F_\wv} \to \GL_n(k)$ of the form 
$\overline{\rho}_v = \overline{\sigma}_{\wv, 1} \oplus \overline{\sigma}_{\wv, 
2}$, where:
\begin{itemize}
\item Let $F_{\wv, n_1} / F_{\wv}$ be the unramified extension of degree $n_1$ and residue field $k_{n_1}$. Then there is an unramified character $\overline{\psi}_\wv : G_{F_{\wv, n_1}} \to k^\times$ and an isomorphism $\overline{\sigma}_{\wv, 1} \cong \Ind_{G_{F_{\wv, n_1}}}^{G_{F_{\wv}}} \overline{\psi}_\wv$.
\item $\overline{\sigma}_{\wv, 2}|_{I_{F_\wv}} \otimes \omega(\wv) \circ \Art_{F_\wv}^{-1}$ is trivial. (In other words, $\overline{\sigma}_{\wv, 2}$ is the twist of an unramified representation by a ramified quadratic character.)
\end{itemize}
We recall that $\cC_\cO$ denotes the category of complete Noetherian local $\cO$-algebras with residue field $\cO / \varpi = k$.
\begin{lemma}\label{lem_decomposition_of_deformation}
Let $R \in \cC_\cO$ and let $\rho_\wv : G_{F_\wv} \to \GL_n(R)$ be a continuous lift of $\overline{\rho}_\wv$ (i.e.\ a continuous homomorphism such that $\rho_\wv \text{ mod }\ffrm_R = \overline{\rho}_\wv$). Then there are continuous lifts $\sigma_{\wv, i} : G_{F_\wv} \to \GL_{n_i}(R)$ of $\overline{\sigma}_{\wv, i}$ ($i = 1, 2$) with the property that $\sigma_{\wv, 1} \oplus \sigma_{\wv, 2}$ is $1 + M_n(\ffrm_R)$-conjugate to $\rho_\wv$. Moreover, each $\sigma_{\wv, i}$ is itself unique up to $1 + M_{n_i}(\ffrm_R)$-conjugacy.
\end{lemma}
\begin{proof}
The splitting exists and is unique because the groups $H^i(F_\wv, \Hom(\overline{\sigma}_{\wv, 1}, \overline{\sigma}_{\wv, 2}))$ and $H^i(F_\wv, \Hom(\overline{\sigma}_{\wv, 2}, \overline{\sigma}_{\wv, 1}))$ vanish for $i = 0, 1$. Compare \cite[Lemma 2.3]{shottonGLn}.
\end{proof}
Let $R_\wv^\square \in \cC_\cO$ denote the universal lifting ring, i.e.\ the representing object of the functor of all continuous lifts of $\overline{\rho}_\wv$. We write $R(\wv, \Theta, \overline{\rho}_\wv)$ for the quotient of $R_\wv^\square$ associated by \cite[Definition 3.5]{shottonGLn} to the inertial type $\tau_\wv : I_{F_\wv} \to \GL_n(\overline{\bQ}_p)$, $\tau_\wv = \oplus_{i=1}^{n_1} (\iota^{-1} \Theta^{q_\wv^{i-1}} \circ \Art_{F_\wv}^{-1}) \oplus (\omega(\wv) \circ \Art_{F_\wv}^{-1})^{\oplus n_2}$.
We record the following properties of $R(\wv, \Theta, \overline{\rho}_\wv)$.
\begin{proposition}\label{prop_properties_of_lifting_ring}
\begin{enumerate}
\item The ring $R(\wv, \Theta, \overline{\rho}_\wv)$ is reduced, $p$-torsion-free, and is supported on a union of irreducible components of $R_\wv^\square$. In particular, $\Spec R(\wv, \Theta, \overline{\rho}_\wv)$ is $\cO$-flat and equidimensional of dimension $1 + n^2$.
\item Let $x : R_\wv^\square \to \overline{\Q}_p$ be a homomorphism, and let $\rho_x : G_{F_\wv} \to \GL_n(\overline{\bQ}_p)$ be the pushforward of the universal lifting, with its associated direct sum decomposition $\rho_x \cong \sigma_{x, 1} \oplus \sigma_{x, 2}$. Then $x$ factors through $R(\wv, \Theta, \overline{\rho}_\wv)$ if and only if there is an isomorphism $\sigma_{x, 1} \cong \Ind_{G_{F_{\wv, n_1}}}^{G_{F_\wv}} \psi_x$ for a character $\psi_x : G_{F_{\wv, n_1}} \to \overline{\bQ}_p^\times$ such that $\psi_x|_{I_{F_{\wv, n_1}}} \circ \Art_{F_{\wv, n_1}} = \iota^{-1} \Theta$ and $\sigma_{x, 2}|_{I_{F_\wv}} \otimes \omega(\wv) \circ \Art_{F_\wv}^{-1}$ is trivial.
\item Let $\sigma_{\wv, 1} : G_{F_\wv} \to \GL_{n_1}(R(\wv,\Theta,\overline{\rho}_\wv))$ be the representation associated to the universal lifting by Lemma \ref{lem_decomposition_of_deformation}. There exists $\alpha_\wv \in R(\wv,\Theta,\overline{\rho}_\wv) / (\varpi)$ such that for any Frobenius lift $\phi_\wv \in G_{F_\wv}$, $\det(X - \sigma_{\wv,1}(\phi_\wv^{n_1})) \equiv (X - \alpha_\wv)^{n_1} \text{ mod }\varpi$.
\item Let $L_\wv / F_\wv$ be a finite extension such that $\tau_\wv|_{I_{L_\wv}}$ is trivial, and let $R_{L_\wv}^\square$ denote the universal lifting ring of $\overline{\rho}|_{G_{L_\wv}}$. Then the natural morphism $R_{L_\wv}^\square \to R(\wv, \Theta, \overline{\rho}_\wv)$ (classifying restriction of the universal lifting to $G_{L_\wv}$) factors over the quotient $R_{L_\wv}^\square \to R_{L_\wv}^{ur}$ that classifies unramified liftings of $\overline{\rho}_\wv|_{G_{L_\wv}}$.
\end{enumerate}
\end{proposition}
\begin{proof}
The first two properties follow from \cite[Proposition 3.6]{shottonGLn}. For the third, let $\phi_\wv$ be a Frobenius lift. We note that $\det(X - \sigma_{\wv, 1}(\phi_\wv)) = X^{n_1} + (-1)^{n_1} \det \sigma_{\wv, 1}(\phi_\wv)$. Indeed, this can be checked at $\overline{\bQ}_p$-points, at which $\sigma_{\wv, 1}$ is irreducible, induced from a character of $G_{F_{\wv, n_1}}$ which extends $\iota^{-1} \Theta \circ \Art_{F_{\wv, n_1}}^{-1}$. Reducing modulo $\varpi$ and applying Hensel's lemma, we find that there is an element $\alpha'_\wv \in 	R(\wv, \Theta, \overline{\rho}_\wv) / (\varpi)$ such that $\det(X - \sigma_{\wv, 1}(\phi_\wv)) 
	\equiv 
	\prod_{i=1}^{n_1}(X - q_\wv^{i-1} \alpha'_\wv ) \text{ mod 
	}\varpi$. If $\alpha_\wv = (\alpha'_\wv)^{n_1}$ then $\det(X - \sigma_{\wv, 1}(\phi^{n_1}_\wv)) = (X - \alpha_\wv)^{n_1}$. 
	For the fourth part of the lemma, we need to show that the universal lifting is unramified on restriction to $G_{L_\wv}$. Since $R(\wv,\Theta,\overline{\rho}_\wv)$ is reduced, it suffices to check this at each geometric generic point. At such a point $\sigma_{\wv, 1}$ is irreducible, induced from a character of $G_{F_{\wv, n_1}}$, while $\sigma_{\wv, 2}$ is a quadratic ramified twist of an unramified representation. The result follows.
\end{proof}
\subsection{Algebraic modular forms}\label{subsec:alg_mod_forms}

Finally, we define notation for algebraic modular forms on the group $G$. 
Retaining our standard assumptions, fix a coefficient field $E \subset 
\overline{\Q}_p$ containing the image of each embedding $F \to 
\overline{\Q}_p$, with ring of integers $\cO$, and let $\widetilde{I}_p$ denote 
the set of embeddings $\tau : F \to E$ inducing a place of $\widetilde{S}_p$. 
Given $\lambda = (\lambda_\tau)_\tau \in (\Z^n_+)^{\widetilde{I}_p}$, we write 
$V_{\lambda}$ for the $E[\prod_{v \in S_p} \GL_n(F_\wv)]$-module denoted 
$W_\lambda$ in \cite[Definition 2.3]{ger}; it is the restriction to 
$\GL_n(F_\wv)$ of a tensor product of highest weight representations of 
$\GL_n(E)$. We write $\cV_\lambda \subset V_\lambda$ for the $\cO[\prod_{v \in 
S_p}\GL_n(\cO_{F_\wv})]$-submodule denoted $M_\lambda$ in \emph{loc. cit.}; it 
is an $\cO$-lattice. 

In this paper we will only consider algebraic modular forms with respect to open compact subgroups $U \subset G(\A_{F^+}^\infty)$ which decompose as a product $U = \prod_v U_v$, and such that for each $v \in S_p$, $U_v \subset \iota_\wv^{-1} \GL_n(\cO_{F_\wv})$. Given such a subgroup, together with a finite set $\Sigma$ of finite places of $F^+$ and a smooth $\cO[U_\Sigma]$-module $M$, finite as $\cO$-module, we define $S_\lambda(U, M)$ to be the set of functions $f : G(F^+) \backslash G(\A_{F^+}^\infty) \to \cV_\lambda \otimes_\cO M$ such that for each $u \in U$ and $g \in G(\A_{F^+}^\infty)$, $u \cdot f(gu) = f(g)$. (Here $U$ acts on $\cV_\lambda \otimes_\cO M$ via projection to $U_p \times U_\Sigma$.) If $\lambda = 0$, we drop it from the notation and simply write $S(U, M)$.

We recall the definition of some useful open compact subgroups and Hecke operators (see \cite[\S 2.3]{ger} for more details):
\begin{itemize}
	\item For any place $v$ of $F^+$ which splits $v = w w^c$ in $F$, the maximal compact subgroup $\GL_n(\cO_{F_w})$. If $v \not\in \Sigma \cup S_p$, $U_v = \iota_w^{-1} \GL_n(\cO_{F_w})$, and $1 \leq j
	\leq n$, then the unramified Hecke operator $T_w^j$ given by the double coset operator \[T_w^j = \left[ \iota_w^{-1}\left( \GL_n(\cO_{F_w})\begin{pmatrix}
	\varpi_w \mathrm{Id}_j & 0 \\ 0 & \mathrm{Id}_{n-j}
	\end{pmatrix} \GL_n(\cO_{F_w}) \right) \times U^v \right]\] acts on $S_\lambda(U, M)$.
		\item For any place $v$ of $F^+$ which splits $v = w w^c$ in $F$, the Iwahori subgroup $\Iw_w\subset\GL_n(\cO_{F_w})$ of matrices which are upper-triangular modulo $\varpi_w$.
	\item For any place $v \in S_p$ and $c \geq b \geq 0$ with $c \geq 1$, the 
	subgroup $\Iw_\wv(b, c) \subset 
	\GL_n(\cO_{F_\wv})$ of matrices which are upper-triangular $\varpi_\wv^c$ 
	and unipotent upper-triangular modulo $\varpi_\wv^b$. If $U_v = \iota_\wv^{-1} \Iw_\wv(b, 
	c)$ for each $v \in S_p$ and $1 \leq j \leq n$, then the re-normalised 
	Hecke operator $U_{\wv, \lambda}^j$ of \cite[Definition 2.8]{ger} acts on 
	$S_\lambda(U, M)$. (This Hecke operator depends on our choice of 
	uniformizer $\varpi_\wv$. However, the ordinary part of $S_\lambda(U, M)$, defined below using these operators, is independent of choices.)
	\item For any place $v$ of $F^+$ which splits $v = w w^c$ in $F$, the principal congruence subgroup $K_\wv(1) = \ker(  \GL_n(\cO_{F_\wv}) \to \GL_n(k(\wv)))$. 
\end{itemize}
When $U_v = \iota_\wv^{-1} \Iw_\wv(b, c)$ for each $v \in S_p$, there is a canonical direct sum decomposition $S_\lambda(U, M) = S_\lambda^{ord}(U, M) \oplus S_\lambda^{n-ord}(U, M)$ with the property that $S_\lambda^{ord}(U, M)$ is the largest submodule of $S_\lambda(U, M)$ where each operator $U_{\wv, \lambda}^j$ ($v \in S_p$, $j = 1, \dots, n$) acts invertibly (\cite[Definition 2.13]{ger}).

We recall some basic results about the spaces $S_\lambda(U, M)$. We say that $U$ is sufficiently small if for $g \in G(\A_{F^+}^\infty)$, the group $G(F^+) \cap g U g^{-1}$ is trivial. We have the following simple lemma (cf. \cite[p. 1351]{ger}):
\begin{lemma}\label{lem_sufficiently_small}
	Suppose that $U$ is sufficiently small and that $M$ is $\cO$-flat. Then for 
	any $c \geq 1$, the natural map $S_\lambda(U, M) \otimes_\cO \cO/\varpi^c 
	\to S_\lambda(U, M / (\varpi^c))$ is an isomorphism.
\end{lemma}
After fixing an isomorphism $\iota : \overline{\Q}_p \to \C$, we can describe the spaces $S_\lambda(U, M)$ in classical terms (\cite[Lemma 2.5]{ger}):
\begin{lemma}\label{lem_algebraic_automorphic_forms_are_classical}
	Let $\iota : \overline{\Q}_p \to \C$ be an isomorphism. Then there is an isomorphism
	\[ S_\lambda(U, M) \otimes_{\cO, \iota} \C \cong \oplus_\sigma m(\sigma) 
	\Hom_U((M \otimes_{\cO, \iota} \C)^\vee, \sigma^\infty)
	\]
	respecting the action of Hecke operators at finite places away from $\Sigma \cup S_p$, where the sum runs over automorphic representations $\sigma$ of 
	$G(\A_{F^+})$ such that for each embedding $\tau : F \to \C$ inducing a 
	place $v$ of $F^+$, $\sigma_v$ is the restriction to $G(F^+_v)$ of the dual 
	of the irreducible algebraic representation of $\GL_n(\C)$ of highest 
	weight $\lambda_{\iota^{-1} \tau}$. 
\end{lemma}

\section*{Part I: Analytic continuation of functorial liftings}
The first part of this paper (\S\S \ref{sec_rigid_geometry} -- \ref{sec_ping_pong}) is devoted to the proof of Theorem \ref{thm:intro_propagation} from the introduction, which shows that automorphy of the $n$\textsuperscript{th} symmetric power for one level $1$ cuspidal Hecke eigenform implies automorphy of the $n$\textsuperscript{th} symmetric power for all level $1$ cuspidal Hecke eigenforms. 

As described in the introduction, the proof has two main ingredients. The first, which is the main result of \S\ref{sec_rigid_geometry}, is that automorphy of symmetric powers can be propagated along irreducible components of the Coleman--Mazur eigencurve. The second ingredient, which is explained in \S\ref{sec_ping_pong}, uses the main result of \cite{Buz05} and has already been sketched in the introduction.

Here we make some further introductory remarks on \S\ref{sec_rigid_geometry}. By making a suitable (in particular, soluble) base change to a CM field, we translate ourselves to the setting of definite unitary groups. We start from a classical point $z_0$ of an eigenvariety for a rank $2$ unitary group, $\cE_2$, such that the $n$\textsuperscript{th} symmetric power of the associated Galois representation is known to be automorphic. We use Emerton's construction of eigenvarieties (involving his locally analytic Jacquet functor), and our point of view on eigenvarieties and Galois representations is particularly influenced by those of \cite{bellaiche_chenevier_pseudobook} and \cite{Bre17}. Like the authors of \cite{Bre17}, we rely in an essential way on the results of \cite{Ked14}, which make it possible to spread out pointwise triangulations to global triangulations. We consider the diagram:

\[ \xymatrix{ \cE_2 \ar[r]^-{i_2}& \cX_{ps,2} \times \cT_2\ar[d]^{\Sym^n} \\ \cE_{n+1}\ar[r]^-{i_{n+1}} & \cX_{ps,n+1} \times \cT_{n+1}}\] Here, $\cE_n$ is an eigenvariety for a rank $n+1$ unitary group, $\cX_{ps,d}$ is a certain rigid space of $d$-dimensional $p$-adic Galois pseudocharacters and $\cT_d$ is a rigid space parameterising characters of a $p$-adic torus. Our eigenvarieties come equipped with maps to these character varieties as part of their construction; combining this with the existence of a family of Galois pseudocharacters over the eigenvariety interpolating the global Langlands correspondence at classical points gives the closed immersions $i_d$ appearing in the diagram. The map $\Sym^n$ corresponds to taking the $n^\text{th}$ symmetric power of the $2$-dimensional pseudocharacter. 

Our task is to show that if $\mathcal{C}$ is an irreducible component of $\cE_2$ containing $z_0$, then $\Sym^n(i_2(C))$ is contained in the image of $i_n$. A classicality result (Lemma \ref{lem_generic_implies_classical}) will then be used to show that for another classical point $z_1$ of $\mathcal{C}$, its symmetric power $\Sym^n(i_2(z_1))$ is actually the image of a \emph{classical} point of $\cE_{n+1}$. 

To show that $\Sym^n(i_2(C))$ is indeed contained in the image of $i_n$, we combine a simple lemma in rigid geometry (Lemma \ref{lem_inclusion_of_irreducible_components}) with information coming from the local geometry of a certain natural locally closed neighbourhood of $\Sym^n(i_2(z_0))$ in $\cX_{ps,n+1}\times\cT_{n+1}$ which contains open subspaces of both $\cE_{n+1}$ and $\Sym^n(i_2(C))$. This subspace is essentially the trianguline variety, but since we work with spaces of pseudocharacters instead of representations we restrict to open neighbourhoods in which our pseudocharacters are absolutely irreducible and hence naturally lift to representations. Our results on the vanishing of adjoint Selmer groups \cite{newton2020adjoint} are used to compare $\cE_{n+1}$ and the trianguline variety. We proceed in a similar way to the proof of \cite[Corollary 7.6.11]{bellaiche_chenevier_pseudobook}, which shows that vanishing of an adjoint Selmer group implies that $i_{n+1}$ induces an isomorphism between completed local rings of the eigenvariety and the trianguline variety.

\section{Trianguline representations and eigenvarieties}\label{sec_rigid_geometry}

Throughout this section, we let $p$ be a prime and let $E \subset \overline{\Q}_p$ be a coefficient field. We write $\C_p$ for the completion of $\overline{\bQ}_p$. If $\cX$ is a quasi-separated 
$E$-rigid space we let $\cX(\Qpbar) = \bigcup_{E' \subset \Qpbar}\cX(E')$, 
where the union is 
over finite extensions of $E$. We can naturally view $\cX(\Qpbar)$ as  a subset 
of the set of closed points of the rigid space $\cX_{\Cp}$ (where base 
extension of a quasi-separated rigid space is as defined in \cite[\S 
9.3.6]{Bos84}, see also \cite[\S3.1]{Con99}).

\subsection{An `analytic continuation' lemma}\label{subsec_prototype_argument}

Suppose given a 
diagram of $E$-rigid spaces
\[ \xymatrix@1{ \ar[r]_\beta \cY & \cG & \cX \ar[l]^\alpha}, \]
where $\alpha$ is a closed immersion. We identify $\cX$ with a subspace of $\cG$. Let $x \in \cY$ be a point such that 
$\beta(x) \in \cX$. 
\begin{lemma}\label{lem_inclusion_of_irreducible_components}
	Suppose that $\beta^{-1}(\cX)$ contains an affinoid open neighbourhood of 
	 $x$. Then for each irreducible component $\cC$ of $\cY$ containing  $x$, 
	 we 
	have $\beta(\cC) \subset \cX$.
\end{lemma}
\begin{proof}
	We observe that $\beta^{-1}(\cX) \cap \cC$ is a Zariski closed subset of 
	$\cC$ which contains a non-empty affinoid open subset. This forces 
	$\beta^{-1}(\cX) \cap \cC = \cC$ (apply \cite[Lemma 2.2.3]{Con99}),  hence 
	$\beta(\cC) \subset \cX$.
\end{proof}

\subsection{A Galois deformation space}

Let $F$, $S$, $p$ be as in our standard assumptions (\S \ref{sec_definite_unitary_groups}). We assume that $E$ contains the image of every 
embedding $\tau: F \to \Qpbar$.

\subsubsection{Trianguline deformations -- infinitesimal geometry}

This section has been greatly influenced by works of Bella\"iche and Chenevier 
\cite{bellaiche_chenevier_pseudobook, Che11}. We use the formalism of families 
of $(\varphi, \Gamma_{F_\wv})$-modules, as in \cite{Ked14}. Thus if $v \in S_p$ and $X$ is an 
$E$-rigid space, one can define the Robba ring $\cR_{X, F_\wv}$; if $\cV$ is a 
family of representations of $G_{F_\wv}$ over $X$, then the functor 
$D^\dagger_{rig}$ of \cite[Theorem 2.2.17]{Ked14} associates to $\cV$ the 
family $D^\dagger_{rig}(\cV)$ of $(\varphi, \Gamma_{F_\wv})$-modules over $X$ 
which is, locally on $X$, finite free over $\cR_{X, F_\wv}$. We refer to 
\cite[\S 2]{Hel16} for the definitions of these objects, as well as more 
detailed references. 
If $X = \operatorname{Sp} A$, 
where $A$ is an $E$-affinoid algebra, we write $\cR_{X, F_\wv} = \cR_{A, 
F_\wv}$. If $\delta:F_\wv^\times \to A^\times$ is a continuous character, we 
have a rank one $(\varphi, \Gamma_{F_\wv})$-module $\cR_{A, 
F_\wv}(\delta_v)$ defined by \cite[Construction 6.2.4]{Ked14}. We will also have cause to mention the $(\varphi,\Gamma)$-cohomology groups $H^*_{\varphi,\gamma_{F_{\wv}}}(-)$ which are defined in \cite[\S2.3]{Ked14}.

Let $v \in S_p$, and let $\rho_v : G_{F_\wv} \to \GL_n(E)$ be a continuous 
representation. If $\delta_v = (\delta_{v, 1}, \dots, \delta_{v, n}) : 
(F_\wv^\times)^n \to E^\times$ is a continuous character, we call a 
triangulation of $\rho_v$ of parameter $\delta_v$ an increasing filtration of 
$D^\dagger_{\text{rig}}(\rho_v)$ by direct summand $(\varphi, 
\Gamma_{F_\wv})$-stable $\cR_{E, F_\wv}$-submodules such that the successive 
graded pieces are isomorphic to $\cR_{E,F_\wv}(\delta_{v, 1}), \dots, \cR_{E, 
F_\wv}(\delta_{v, n})$. We say that $\rho_v$ is trianguline of parameter 
$\delta_v$ if it admits a triangulation of parameter $\delta_v$. If $\delta_v$ 
satisfies $\delta_{v,i}(\varpi_\wv) \in \cO^\times$ for each $i$, then we say 
that $\delta_v$ is an ordinary parameter. Equivalently, $\delta_v$ is ordinary 
if $\delta_{v,i}\circ \Art_{F_\wv}^{-1}$ extends to a continuous character of 
$G_{F_\wv}$ for each $i$. For an ordinary parameter 
$\delta_v$, 
$\rho_v$ is trianguline of parameter $\delta_v$ if and only if $\rho_v$ has a 
filtration with successive graded pieces isomorphic to $\delta_{v,1}\circ 
\Art_{F_\wv}^{-1}, \dots, \delta_{v,n}\circ \Art_{F_\wv}^{-1}$. 

We say that the character $\delta_v$ is regular if for all $1 \leq i < j \leq 
n$, we have $\delta_{v, i} / \delta_{v, j} \neq x^{a_v}$ for any $a_v = (a_{v, 
\tau})_\tau \in \bbZ_{\geq 0}^{\Hom_{\bbQ_p}(F_\wv, E)}$, where by definition 
$x^{a_v}(y) = \prod_\tau \tau(y)^{a_{v, \tau}}$. Note that the characters $x^{a_v}$ satisfy $|x^{a_v}(p)|_p = p^{-\sum_{\tau}a_{v,\tau}}$, so there is an affinoid cover of the rigid space $\Hom(F_\wv^\times, \bG_m)$ with each open containing only finitely many $x^{a_v}$. 

We define $\cT_v = 
\Hom( (F_\wv^\times)^n, \bG_m)$, a smooth rigid space over $E$, 
and write $\cT_v^{reg} \subset \cT_v$ for the Zariski open subspace of regular 
characters (Zariski open by the finiteness observation in the preceding paragraph). We define $\cW_v = \Hom((\cO_{F_\wv}^\times)^n, \bG_m)$ and write 
$r_v 
: \cT_v \to \cW_v$ for the natural restriction map. 
\begin{lemma}\label{lem_uniqueness_of_triangulation_for_regular_parameter}
	Let $\rho_v : G_{F_\wv} \to \GL_n(E)$ be a continuous representation. Then 
	for any $\delta_v \in \cT_v^{reg}(E)$, $\rho_v$ admits at most one 
	triangulation of parameter $\delta_v$. If such a 
	triangulation exists, then $\rho_v$ is strictly trianguline of parameter $\delta_v$ in the sense of 
	\cite[Definition 6.3.1]{Ked14}.
\end{lemma}
\begin{proof}
Suppose $\rho_v$ admits a triangulation of parameter $\delta_v$, so $D^\dagger_{\text{rig}}(\rho_v)$ is equipped with an increasing filtration $\Fil_\bullet$. Following \cite[Definition 6.3.1]{Ked14}, we need to show that for each $0 \le i \le n$ the cohomology group $H^0_{\varphi,\gamma_{F_{\wv}}}\left(\left(D^\dagger_{\text{rig}}(\rho_v)/\Fil_i\right) (\delta_{v,i+1})^{-1}\right)$ is one-dimensional. It follows from \cite[Proposition 6.2.8]{Ked14} that $H^0_{\varphi,\gamma_{F_{\wv}}}\left(\gr_j\left(D^\dagger_{\text{rig}}(\rho_v)\right) (\delta_{v,i})^{-1}\right)$ vanishes when $i < j$ and is one-dimensional when $i = j$. The vanishing holds precisely because $\delta_v$ is regular. A d\'{e}vissage completes the proof. 
\end{proof}
\begin{defn}
	If $\delta: F_\wv^\times \to E^\times$ is a continuous character (hence 
	locally $\Qp$-analytic) we let the tuple $(wt_\tau(\delta))_{\tau \in 
		\Hom_{\bbQ_p}(F_\wv,E)}$ be such that the derivative of $\delta$ is the map 
	\begin{align*}F_\wv & \to E \\ x &\mapsto 
	\sum_{\tau \in 
		\Hom_{\bbQ_p}(F_\wv,E)} -wt_\tau(\delta)\tau(x). \end{align*} 
\end{defn}
We can extend this discussion to Artinian local rings. Let $\cC'_E$ denote the 
category of Artinian local $E$-algebras with residue field $E$. If $A \in 
\cC'_E$, then $\cR_{A, F_\wv} = \cR_{E, F_\wv} \otimes_E A$. If $\rho_v : 
G_{F_\wv} \to \GL_n(A)$ is a continuous representation, then 
$D^\dagger_{\text{rig}}(\rho_v)$ is a free $\cR_{A, F_\wv}$-module. If $\delta_v 
\in \cT_v(A)$, we call a triangulation of $\rho_v$ of parameter $\delta_v$ an 
increasing filtration of $D^\dagger_{\text{rig}}(\rho_v)$ by direct summand 
$(\varphi, \Gamma_{F_\wv})$-stable $\cR_{A, F_\wv}$-submodules such that the 
successive graded pieces are isomorphic to $\cR_{A,F_\wv}(\delta_{v, 1}), 
\dots, \cR_{A, F_\wv}(\delta_{v, n})$.

If $\rho_v : G_{F_\wv} \to \GL_n(E)$ is a continuous representation and $\cF_v$ 
is a triangulation of parameter $\delta_v \in \cT_v(E)$, then we write 
$\cD_{\rho_v, \cF_v, \delta_v} : \cC'_E \to \mathrm{Sets}$ for the functor which 
associates to any $A$ the set of equivalence classes of triples $(\rho_v', 
\cF_v', \delta_v')$, where:
\begin{itemize}
	\item $\rho_v' : G_{F_\wv} \to \GL_n(A)$ is a lifting of $\rho_v$, continuous with respect to the $p$-adic topology on $A$.
	\item $\delta_v' \in \cT_v(A)$ is a lifting of $\delta_v$.
	\item $\cF_v'$ is a triangulation of $\rho_v'$ of parameter $\delta_v'$ 
	which lifts $\cF_v$  (note that there is a canonical isomorphism $D^\dagger_{rig}(\rho_v') \otimes_{A} E \cong D^\dagger_{rig}(\rho_v)$, as $D^\dagger_{rig}$ commutes with base change).
\end{itemize}
Triples $(\rho_v', \cF_v', \delta_v')$, $(\rho_v'', \cF_v'', \delta_v'')$ are 
said to be equivalent if there exists $g \in 1 + M_n(\ffrm_A)$ which conjugates 
$\rho_v'$ to $\rho_v''$ and takes $\cF_v'$ to $\cF_v''$. 

We write $\cD_{\rho_v}$ for the functor of equivalence classes of lifts 
$\rho'_v : G_{F_\wv} \to \GL_n(A)$. Thus forgetting the triangulation 
determines a natural transformation $\cD_{\rho_v, \cF_v, \delta_v} \to 
\cD_{\rho_v}$.
\begin{proposition}\label{prop_representability_of_trianguline_def_functor}
	Suppose that $\delta_v \in \cT_v^{reg}(E)$. Then the natural transformation 
	$\cD_{\rho_v, \cF_v, \delta_v} \to \cD_{\rho_v}$ is relatively 
	representable, and injective on $A$-points for every $A \in \cC'_E$. If 
	$\rho_v$ is absolutely irreducible, then both functors are 
	pro-representable, in which case there is a surjective morphism $R_{\rho_v} 
	\to R_{\rho_v, \cF_v, \delta_v}$ of (pro-)representing objects.
\end{proposition}
\begin{proof}
	If $F_\wv = \bbQ_p$, this is contained in \cite[Proposition 
	2.3.6]{bellaiche_chenevier_pseudobook} and \cite[Proposition 
	2.3.9]{bellaiche_chenevier_pseudobook}. The general case is given by \cite[Lemma 2.35, Proposition 2.37, Corollary 2.38]{nakamura-deformations} (noting that Nakamura works with Berger's category of $B$-pairs, which is equivalent to the category of $(\varphi,\Gamma)$-modules over the Robba ring).
\end{proof}
A consequence of Proposition 
\ref{prop_representability_of_trianguline_def_functor} is that when $\delta_v$ 
is regular, $\cD_{\rho_v, \cF_v, \delta_v}(E[\epsilon])$ can be identified with 
a subspace of $\cD_{\rho_v}(E[\epsilon]) = H^1(F_{\wv}, \ad \rho_v)$. Since 
$\cF_v$ is moreover uniquely determined by $\delta_v$ (Lemma 
\ref{lem_uniqueness_of_triangulation_for_regular_parameter}), this subspace 
depends only on $\delta_v$, when it is defined. We write $H^1_{tri, 
\delta_v}(F_{\wv}, \ad \rho_v)$ for this subspace. We observe that there is a natural transformation $\cD_{\rho_v, \cF_v, \delta_v} \to \operatorname{Spf} \widehat{\cO}_{\cW_v, r_v(\delta_v)}$, which sends a triple $(\rho_v', 
\cF_v', \delta_v')$ to the character $r_v(\delta_v')$. Evaluating on $E[\epsilon]$-points, we obtain an $E$-linear map
\[ H^1_{tri, 
	\delta_v}(F_{\wv}, \ad \rho_v) \to T_{r_v(\delta_v)}\cW_v, \]
where $T_{r_v(\delta_v)}\cW_v$ denotes the Zariski tangent space of $\cW_v$ at the point $r_v(\delta_v)$. This map appears in the statement of the following lemma:
\begin{lemma}\label{lem_weight_map_etale_for_non_critical_refinement}
	Let $\rho_v : G_{F_\wv} \to \GL_n(E)$ be a continuous representation. 
	Suppose that:
	\begin{enumerate}
		\item $\rho_v$ is de Rham.
		\item $\rho_v$ is trianguline of parameter $\delta_v \in 
		\cT_v^{reg}(E)$.
		\item\label{item_non_critical} For each $\tau \in \Hom_{\bbQ_p}(F_\wv, E)$, we have 
		\[ wt_\tau(\delta_{v, 1}) < wt_\tau(\delta_{v, 2}) < \dots < 
		wt_\tau(\delta_{v, n}). \] 
	\end{enumerate}
We note that the labelled weights $wt_\tau$ coincide 
with the 
labelled Hodge--Tate weights of $\rho_v$ (cf. \cite[Lemma 6.2.12]{Ked14}).

	Then the natural map 
	\[ \ker\left(H^1_{tri, \delta_v}(F_{\wv}, \ad \rho) \to 
	T_{r_v(\delta_v)}\cW_v 
	\right)\to H^1(F_\wv, \ad \rho) \]
	has image contained in 
	\[ H^1_g(F_\wv, \ad \rho) = \ker(H^1(F_\wv, \ad \rho) \to H^1(F_\wv, \ad \rho \otimes_{\bQ_p} B_{dR})). \]
\end{lemma}
\begin{proof}
	We must show that if $(\rho'_v, \cF_v', \delta'_v) \in \cD_{\rho_v, \cF_v, 
	\delta_v}(E[\epsilon])$ is an element in the kernel of the map to 
	$T_{r_v(\delta_v)} \cW_v$, then $\rho'_v$ is de Rham. When $F_\wv = \Q_p$, this follows from \cite[Proposition 
	2.3.4]{bellaiche_chenevier_pseudobook}; in general it follows from modifying their argument as in \cite[Proposition 2.6]{Hel16} (the coefficients in this latter result are assumed to be a field, whilst we need coefficients $E[\epsilon]$, but the same proof works with any Artin local $E$-algebra as coefficient ring). 
\end{proof}
We say that a triangulation of a representation $\rho_v : G_{F_\wv} \to 
\GL_n(\overline{\bQ}_p)$ of parameter $\delta_v$ is non-critical if for each 
$\tau \in \Hom_{\bbQ_p}(F_\wv, E)$, the labelled weights are an increasing 
sequence of integers:
\[ wt_\tau(\delta_{v, 1}) < wt_\tau(\delta_{v, 2}) < \dots < wt_\tau(\delta_{v, 
	n}). \]
In other words, if $\delta_v$ satisfies condition (\ref{item_non_critical}) of Lemma
\ref{lem_weight_map_etale_for_non_critical_refinement}. 

We now give a criterion for a de Rham representation to have a triangulation 
satisfying this condition. This generalizes \cite[Lemma 2.9]{Hel16}, which treats the crystalline case.
\begin{lemma}\label{lem_numerically_non_critical_implies_non_critical}
	Let $v \in S_p$, and let $\rho_v : G_{F_\wv} \to \GL_n(\overline{\bQ}_p)$ 
	be a de Rham representation satisfying the following conditions:
	\begin{enumerate}
		\item There exists an increasing filtration of the associated 
		Weil--Deligne representation $\mathrm{WD}(\rho_v)$ (by 
		sub-Weil--Deligne representations) with associated gradeds given by 
		characters $\chi_{v, 1}, \dots, \chi_{v, n} : W_{F_\wv} \to 
		\overline{\bQ}_p^\times$. 
		\item For each embedding $\tau : F_\wv \to E$, the $\tau$-Hodge--Tate weights of $\rho_v$ are distinct.
		\item For each embedding $\tau : F_\wv \to E$, let $k_{\tau, 1} < \dots 
		< k_{\tau, n}$ denote the strictly increasing sequence of 
		$\tau$-Hodge--Tate weights of $\rho_v$. Then we have for all $\tau \in 
		\Hom_{\bbQ_p}(F_\wv, E)$:	\[   v_p( \chi_{v, 1}(p) ) < k_{\tau, 2} + 
		\sum_{\tau' \neq \tau} k_{\tau', 1}  \]
		and for all $i = 2, \dots, n-1$:
		\[  v_p( (\chi_{v, 1}  \dots \chi_{v, i})(p)  
		) < k_{\tau, i+1} + \sum_{\tau' \neq \tau} k_{\tau', i} + \sum_{\tau'} 
		\sum_{j=1}^{i-1} k_{\tau', j}. \]
	\end{enumerate}
	Then $\rho_v$ is trianguline of parameter $\delta_v$, where for each $i = 1, 
	\dots, n$, $\delta_{v, i} : F_\wv^\times \to \Qpbar^\times$ is defined by the 
	formula $\delta_{v, i}(x) = (\chi_{v, i} 
	\circ \Art_{F_\wv}(x))\prod_{\tau}\tau(x)^{-k_{\tau, 
			i}} $. In 
	particular, the pair $(\rho_v, \delta_v)$ satisfies condition (3) of 
	Lemma \ref{lem_weight_map_etale_for_non_critical_refinement}.
\end{lemma}
\begin{proof}
	The filtration of 	$\mathrm{WD}(\rho_v)$ determines an increasing 
	filtration $ 0 =M_0 \subset M_1 \subset M_2 \subset \dots \subset M_n = 
	D_{pst}(\rho_v)$ of $D_{pst}(\rho_v)$ by sub-$(\varphi, N, 
	G_{F_\wv})$-modules (via the equivalence of categories of 
	\cite[Proposition~4.1]{MR2359853}). The main result of \cite{Ber08} states 
	that there is 
	an equivalence of 
	tensor categories between the category of filtered $(\varphi, N, 
	G_{F_\wv})$-modules and a certain category of $(\varphi, 
	\Gamma_{F_\wv})$-modules (restricting to the usual equivalence between 
	weakly 
	admissible filtered $(\varphi, N, 
	G_{F_\wv})$-modules and $(\varphi, \Gamma_{F_\wv})$-modules associated to 
	de Rham representations). We 
	thus obtain a triangulation 
	of the associated $(\varphi, \Gamma_{F_\wv})$-module of $\rho_v$. 
	
	A filtered $(\varphi, N, G_{F_\wv})$-module $M$ of rank 1 is determined up 
	to isomorphism by its associated character $\chi : W_{F_\wv} \to 
	\overline{\bQ}_p^\times$ and the (unique) integers $a_\tau$ such that 
	$\gr^{a_\tau} (M_{F_{\wv}} \otimes_{F_{\wv}\otimes_{\Qp}\overline{\bQ}_p, 
	\tau\otimes\mathrm{id}} \overline{\bQ}_p) \neq 0$. The 
	corresponding rank-1 $(\varphi, \Gamma_{F_\wv})$-module is the one 
	associated to the character $\delta : F_\wv^\times \to 
	\overline{\bQ}_p^\times$ given by the formula $\delta = x^{-a_v} (\chi 
	\circ \Art_{F_\wv})$ (cf. \cite[Example 6.2.6(3)]{Ked14}).	What we 
	therefore need to verify is that if $a_{\tau, i} \in \bZ$ are the integers 
	for which $\gr^{a_{\tau, i}} (M_i / M_{i-1} \otimes_{F_{\wv, 0}, \tau} 
	\overline{\bQ}_p) \neq 0$, then $a_{\tau, i} = k_{\tau, i}$. 
	
	This follows from hypothesis (3) of the lemma, together with the fact that 
	$D_{pst}(\rho_v)$ is a weakly admissible filtered $(\varphi, N, 
	G_{F_\wv})$-module, as we now explain. We show by induction on $i$ that 
	the jumps in the induced Hodge--de Rham filtration of $M_i$ are as 
	claimed. For $M_1$, if these jumps are $k_{\tau,j_\tau}$ then we have for 
	each 
	$\tau$
	\begin{multline*} \frac{1}{[F_\wv: \bQ_p]} \left( k_{\tau, 2} + 
	\sum_{\tau' \neq \tau} k_{\tau', 1} \right)  > \frac{1}{[F_{\wv, 0} : 
		\bQ_p]}  v_p( \chi_{v, 1}(\varpi_v) ) \\ = t_N( M_1 )  \geq t_H(M_1) = 
	\frac{1}{[F_\wv: \bQ_p]} \sum_{\tau'} k_{\tau',j_{\tau'}}. \end{multline*}
	Since the sequences $k_{\tau, i}$ are strictly increasing, this is 
	possible only if $j_\tau = 1$ for each $\tau$. In general, if the jumps of 
	$M_{i-1}$ are as expected and $M_i / M_{i-1}$ has jumps $k_{\tau,j_\tau}$ 
	then 
	we have for each $\tau$
	\begin{multline*}  \frac{1}{[F_\wv: \bQ_p]} \left(k_{\tau, i+1} + 
	\sum_{\tau' \neq \tau} k_{\tau', i} + \sum_{\tau'} \sum_{j=1}^{i-1} 
	k_{\tau', j} \right) > t_N(M_i) \\ \geq t_H(\Fil_i) = 
	\frac{1}{[F_\wv: \bQ_p]} \left(\sum_{\tau'} k_{\tau', j_{\tau'}} + 
	\sum_{\tau'} \sum_{j=1}^{i-1} k_{\tau', j}\right). \end{multline*}
	Once again this is possible only if $j_\tau = i$ for each $\tau$. 
\end{proof}

\begin{defn}\label{defn:numerically_non_critical}
	We say that a character $\delta_v \in \cT_v(\overline{\bQ}_p)$ is numerically 
	non-critical if it satisfies the following conditions: \begin{enumerate}
		\item For each $\tau \in 
		\Hom_{\bbQ_p}(F_\wv, E)$, the labelled weights $wt_\tau(\delta_{v,1}), \ldots,wt_\tau(\delta_{v,n})$ are an increasing sequence of integers. 
		
		\item For each $\tau \in 
	\Hom_{\bbQ_p}(F_\wv, E)$, and for each $i = 1, \dots, n-1$, we have
	\[v_p ( (\delta_{v, 1} \dots \delta_{v, i})(p) ) < 
	wt_{\tau}(\delta_{v,i+1}) - wt_{\tau}(\delta_{v,i}). \]\end{enumerate}
\end{defn}

Following \cite[Remark 2.4.6]{bellaiche_chenevier_pseudobook}, we may 
reformulate Lemma \ref{lem_numerically_non_critical_implies_non_critical} as 
follows: let $\rho_v : G_{F_\wv} \to \GL_n(\overline{\bQ}_p)$ be a Hodge--Tate regular de 
Rham representation, and suppose that $\WD(\rho_v)$ is equipped with an 
increasing filtration such that the associated gradeds are given by characters 
$\chi_{v, 1}, \dots, \chi_{v, n} : W_{F_\wv} \to \overline{\bQ}_p^\times$. Let 
$k_{\tau, 1} < \dots < k_{\tau, n}$ be the strictly increasing sequences of 
$\tau$-Hodge--Tate weights, and let $\delta_v \in \cT(\overline{\bQ}_p)$ be the 
character defined by the formula $\delta_{v, i}(x) = 
(\chi_{v, i} 
\circ \Art_{F_\wv}(x))\prod_{\tau}\tau(x)^{-k_{\tau, 
		i}} $. Then if $\delta_v$ is numerically non-critical, the 
representation $\rho_v$ admits a non-critical triangulation with parameter 
$\delta_v$.

The most important case for us is that of 2-dimensional de Rham representations of $G_{\Q_p}$, and their symmetric powers. In this case the possible triangulations admit a particularly explicit description (cf. \cite{colmez-tri}; this description can be easily justified, including the case $p = 2$, using the results of \cite{Ber08}):
\begin{example}\label{ex:gl2qptriangulations}Let $\rho: G_{\Qp}\to 
	\GL_2(\Qpbar)$ be a de Rham 
	representation with Hodge--Tate weights $k_1 < k_2$ such that $\mathrm{WD}(\rho)^{ss} = 
	\chi_1 \oplus \chi_2$ with $\chi_i: W_{\Qp} \to \Qpbarx$. Assume for simplicity 
	that $\chi_1 \ne \chi_2$. 
	Then we have the following 
	possibilities:\begin{enumerate}
		\item If $\rho$ is not potentially crystalline, then we can choose $\chi_1, \chi_2$ so that $\chi_1 = 
		\chi_2(|\cdot|_p\circ\Art^{-1}_{\Qp})$. In this case
		$\rho$ has a unique triangulation. It is non-critical, of parameter \[\delta = 
		(x^{-k_1}\chi_1\circ\Art_{\Qp},x^{-k_2}\chi_2\circ\Art_{\Qp}). \]
		\item If $\rho$ is potentially crystalline and irreducible, or reducible and 
		indecomposable, 
		$\rho$ has two triangulations. Both of these are non-critical, and their respective parameters are \[\delta = 
		(x^{-k_1}\chi_1\circ\Art_{\Qp},x^{-k_2}\chi_2\circ\Art_{\Qp})\] and  
		\[\delta = 
		(x^{-k_1}\chi_2\circ\Art_{\Qp},x^{-k_2}\chi_1\circ\Art_{\Qp}).\]
		
		\item If $\rho$ is decomposable, we can assume $\rho \cong \psi_1\oplus\psi_2$ 
		where $\psi_i$ has Hodge--Tate weight $k_i$ and $\mathrm{WD}(\psi_i) = \chi_i$. In this case $\rho$ admits two triangulations. The non-critical triangulation has parameter \[\delta = 
		(x^{-k_1}\chi_1\circ\Art_{\Qp},x^{-k_2}\chi_2\circ\Art_{\Qp})\] and the critical triangulation has parameter
		\[\delta = 
		(x^{-k_2}\chi_2\circ\Art_{\Qp},x^{-k_1}\chi_1\circ\Art_{\Qp})\]
		(see for example \cite[Example 3.7]{bergdall-paraboline} for the crystalline 
		case).
	\end{enumerate}
\end{example}
We now consider the global situation. We define $\cT = \prod_{v \in S_p} 
\cT_v$, $\cT^{reg} = \prod_{v \in S_p} \cT_v^{reg}$, and $\cW = \prod_{v \in 
S_p} \cW_v$. We write $r = \prod_{v \in S_p} r_v :  \cT \to \cW$ for the 
product of restriction maps. Let $\cG_n = (\GL_n \times \GL_1) \rtimes \{ \pm 1 \}$ denote the 
group scheme defined in \cite[\S 2.1]{cht}, $\nu_{\cG_n} : \cG_n \to \GL_1$ its 
character, and suppose given a continuous homomorphism $\rho : G_{F^+, S} \to 
\cG_n(E)$ such that $\nu_{\cG_n} \circ \rho = \epsilon^{1-n} \delta_{F / 
F^+}^n$.
We write $\ad \rho$ for the $E[G_{F^+, S}]$-module given by adjoint action of 
$\cG_n$ on the Lie algebra of $\GL_n$. We write $\cD_\rho : \cC'_E \to 
\operatorname{Sets}$ for the functor which associates to each $A \in \cC'_E$ the 
set of $\ker(\GL_n(A) \to \GL_n(E))$-conjugacy classes of lifts $\rho' : 
G_{F^+, S} \to \cG_n(A)$ of $\rho$ such that $\nu_{\cG_n} \circ \rho' = 
\epsilon^{1-n} \delta_{F / F^+}^n$.

If $\Delta \subset G_{F, S}$ is a subgroup, then $\rho(\Delta) \subset 
\cG_n^\circ(E) = \GL_n(E) \times \GL_1(E)$, and we follow \cite{cht} in writing 
$\rho|_{\Delta} : \Delta \to \GL_n(E)$ for the projection to the first factor. 
If $v \in S$, then there is a natural functor $\cD_\rho \to 
\cD_{\rho|_{G_{F_\wv}}}$, given by restriction $\rho' \mapsto 
\rho'|_{G_{F_\wv}}$. 

Let $\delta = (\delta_v)_{v \in S_p} \in \cT^{reg}(E)$ be such 
that for each $v \in S_p$, $\rho_v = \rho|_{G_{F_\wv}}$ is trianguline of 
parameter $\delta_v$. We define a functor 
\[ \cD_{\rho, \cF, \delta} = \cD_\rho \times_{\prod_{v \in S_p} \cD_{\rho_v}} 
\prod_{v \in S_p} \cD_{\rho_v, \cF_v, \delta_v} \]
and $H^1_{tri, \delta}(F_S / F^+, \ad \rho) = \cD_{\rho, \cF, 
\delta}(E[\epsilon]) \subset H^1(F_S / F^+, \ad \rho)$.
\begin{proposition}\label{prop_global_tangent_space_in_the_non-critical_case}
	Keeping assumptions as above, suppose further that the following conditions 
	are satisfied:
	\begin{enumerate}
		\item For each $v \in S_p$, $\rho_v$ is de Rham.
		\item For each $v \in S$, $\mathrm{WD}(\rho|_{G_{F_\wv}})$ is generic, in the sense of \cite[Definition 1.1.2]{All16}.
		  \item For each $v \in S_p$ and for each $\tau \in 
		  \Hom_{\bbQ_p}(F_\wv, E)$, we have 
		 \[ wt_\tau(\delta_{v, 1}) < wt_\tau(\delta_{v, 2}) < \dots < 
		 wt_\tau(\delta_{v, n}). \] In other words, the triangulation of 
		 $\rho_v$ with parameter $\delta_v$ is non-critical.
		\item $H^1_f(F^+, \ad \rho) = 0$.
	\end{enumerate}
Then $\dim_E H^1_{tri, \delta}(F_S / F^+, \ad \rho) \leq  \dim_E \cW$.
\end{proposition}
\begin{proof}
	Our assumption that  $\mathrm{WD}(\rho|_{G_{F_\wv}})$ is generic means that 
	for each $v \in S_p$, $H^1_f(F_\wv, \ad \rho) = H^1_g(F_\wv, \ad \rho)$ and 
	for each $v \in S - S_p$, $H^1_f(F_\wv, \ad \rho) = H^1(F_\wv, \ad \rho)$ 
	(see \cite[Remark~1.2.9]{All16}).	
	Lemma \ref{lem_weight_map_etale_for_non_critical_refinement} then implies 
	that the map $H^1_{tri, \delta}(F_S / F^+, \ad \rho) \to T_{r(\delta)} \cW$ 
	is injective (its kernel being contained in $H^1_f(F^+, \ad \rho) = 0$). 
\end{proof}

\subsubsection{Trianguline representations -- global 
geometry}\label{subsec_trianguline_reps_global_geometry}

We fix a continuous pseudocharacter $\overline{\tau} : G_{F, S} \to k$ of 
dimension $n \geq 1$ which is conjugate self-dual, in the sense that 
$\overline{\tau} \circ c = \overline{\tau}^\vee \otimes \epsilon^{1-n}$. (We define pseudocharacters following Chenevier \cite{chenevier_det}, where they are called determinants. For a summary of this theory, including what it means for a pseudocharacter of $G_{F, S}$ to be conjugate self-dual, see \cite[\S 2]{newton2020adjoint}.) Let 
$R_{ps}$ denote the universal pseudodeformation 
ring representing the functor of lifts of 
$\overline{\tau}$ to conjugate self-dual pseudocharacters over objects of $\cC_\cO$ (cf. \cite[\S 2.19]{newton2020adjoint}). If $v \in S_p$, 
let $R_{ps, v}$ denote the 
pseudodeformation ring of $\overline{\tau}|_{G_{F_\wv}}$. We write $\cX_{ps}$ 
for the rigid generic fibre of $R_{ps}$, and $\cX_{ps, v}$ for the rigid 
generic fibre of $R_{ps, v}$. Then there is a natural morphism $\cX_{ps} \to 
\cX_{ps, p} = \prod_{v \in S_p} \cX_{ps, v}$ of rigid spaces over $E$. We 
recall that to any representation $\rho : G_{F, S} \to 
\GL_n(E)$ such that $\tr \overline{\rho} = \overline{\tau}$, and which is 
conjugate self-dual in the sense that $\rho^c \cong \rho^\vee \otimes \epsilon^{1-n}$, 
is associated a closed point $\tr \rho \in \cX_{ps}(E)$. 
Conversely, if $t \in \cX_{ps}(E)$, then there exists a semi-simple conjugate 
self-dual representation $\rho : G_{F, S} \to \GL_n(\overline{\bQ}_p)$ such 
that $\tr \rho = t$, and this representation is unique up to isomorphism. 

If $\rho : G_{F, S} \to \GL_n(E)$ is an absolutely irreducible representation 
such that $\rho^c \cong \rho^\vee \otimes \epsilon^{1-n}$, then there is a 
homomorphism $\rho_1 : G_{F^+, S} \to \cG_n(E)$ such that $\rho_1|_{G_{F, S}} = 
\rho$ and $\nu_{\cG_n} \circ \rho_1 = \epsilon^{1-n} \delta_{F / F^+}^a$. The 
integer $a \in \{ 0, 1 \}$ is uniquely determined by $\rho$, and any two such 
extensions are conjugate by an element of $\GL_n(\overline{\Q}_p)$. (See 
\cite[Lemma 2.1.4]{cht}.) The following lemma extends extends this to objects 
of $\cC_E$:
\begin{lemma}\label{lem_deforming_conjugate_self_dual_representations}
	Let $\rho : G_{F, S} \to \GL_n(E)$ be an absolutely irreducible 
	representation such that $\rho^c \cong \rho^\vee \otimes \epsilon^{1-n}$, 
	and fix an extension $\rho_1$ as in the previous paragraph. Let $A \in 
	\cC_E$. Then the following sets are in canonical bijection:
	\begin{enumerate}
		\item The set of $\ker(\GL_n(A) \to \GL_n(E))$-conjugacy classes of liftings $\rho' : G_{F, S} \to \GL_n(A)$ of $\rho$ such that $\tr \rho' \circ c = \tr (\rho')^\vee \otimes \epsilon^{1-n}$.
		\item The set of $\ker(\GL_n(A) \to \GL_n(E))$-conjugacy classes of liftings $\rho_1' : G_{F^+, S} \to \cG_n(A)$ of $\rho_1$ such that $\nu \circ \rho_1' = \epsilon^{1-n} \delta_{F / F^+}^a$.
	\end{enumerate}
\end{lemma}
\begin{proof}
	There is an obvious map sending $\rho_1'$ to $\tr \rho_1'|_{G_{F, S}}$. We 
	need to check that this is bijective (at the level of conjugacy classes). 
	To check injectivity, let $\rho_1', \rho_1'' : G_{F^+, S} \to \cG_n(A)$ be 
	two such homomorphisms and suppose that $\rho_1'|_{G_{F, S}}$, 
	$\rho_1''|_{G_{F, S}}$ are $\ker(\GL_n(A) \to \GL_n(E))$-conjugate. We must 
	show that $\rho_1'$, $\rho_1''$ are themselves conjugate. We may suppose 
	that in fact $\rho_1'|_{G_{F, S}} = \rho_1''|_{G_{F, S}}$. In this case 
	Schur's lemma (cf. \cite[Lemma 2.1.8]{cht}), applied to 
	$\rho_1'(c)^{-1}\rho_1''(c)$, shows that $\rho_1'$, 
	$\rho_1''$ are equal.
	
	Now suppose given $\rho' : G_{F, S} \to \GL_n(A)$ lifting $\rho$, and such that $\tr \rho' \circ c = \tr (\rho')^\vee \otimes \epsilon^{1-n}$. Let $J \in \GL_n(E)$ be defined by $\rho_1(c) = (J, - \nu_{\cG_n} \circ \rho_1(c)) \jmath$ (cf. \cite[Lemma 2.1.1]{cht}), so that $\rho(\sigma^c) = J \rho(\sigma)^{-t} J^{-1} \otimes \epsilon^{1-n}$ for all $\sigma \in G_{F, S}$. Then \cite[Example 3.4]{chenevier_det} implies the existence of a matrix $J' \in \GL_n(A)$ lifting $J$ such that $\rho'(\sigma^c) = J' (\rho'(\sigma))^{-t} (J')^{-1}$ for all $\sigma \in G_{F, S}$. By \cite[Lemma 2.1.1]{cht}, this implies the existence of a homomorphism $\rho_1' : G_{F^+, S} \to \cG_n(A)$ lifting $\rho_1$ and such that $\rho_1'|_{G_{F, S}} = \rho'$. This completes the proof.
\end{proof}

There is a Zariski open subspace $\cX_{ps, v}^{v-irr} \subset \cX_{ps, v}$ 
consisting of those points at which the universal pseudocharacter is 
absolutely irreducible. We write $\cX_{ps, p}^{p-irr} = \prod_{v \in S_p} \cX_{ps, 
v}^{v-irr}$ and $\cX_{ps}^{p-irr}$ for the pre-image of $\cX_{ps, p}^{p-irr}$. 
Thus again there is a canonical morphism $\cX_{ps}^{p-irr} \to \cX_{ps, 
p}^{p-irr}$. According to \cite[\S 4.2]{chenevier_det}, there exists an Azumaya 
algebra $\mathcal{A}$ over $\cX_{ps, v}^{v-irr}$ and a homomorphism $\rho_v^u : 
G_{F_\wv} \to \mathcal{A}^\times$ such that $\tr \rho_v^u$ is the 
universal pseudocharacter. 
\begin{lemma}\label{lem_locally_split}
	Let $\rho_v : G_{F_\wv} \to \GL_n(E)$ be an absolutely irreducible 
	representation, corresponding to a point $z = \tr \rho_v \in \cX_{ps, 
	v}^{v-irr}(E)$. Then there exists an affinoid open neighbourhood $z \in \cU 
	\subset \cX_{ps, v}^{v-irr}(E)$ and an isomorphism $\mathcal{A}|_\cU \cong 
	M_n(\cO_\cU)$.
\end{lemma}
\begin{proof}
	Let $\cU$ be an open affinoid neighbourhood of $z$. The stalk $\cO_{\cU, 
	z}$ is a Henselian local ring (\cite[Proposition~7.1.8]{fresnel-vdp}). 
	Thus the stalk $\mathcal{A}_z$ is an Azumaya algebra over a Henselian local 
	ring which is split over the closed point; it is therefore split, i.e.\ 
	there exists an isomorphism $\mathcal{A}_z \cong \M_n(\cO_{\cU, z})$. After 
	shrinking $\cU$, this extends to an isomorphism $\mathcal{A}(\cU) \cong 
	M_n(\cO(\cU))$, as desired. 
\end{proof}
\begin{lemma}\label{lem_character_spaces}
	\begin{enumerate}
		\item Let $z \in \cX_{ps, p}^{p-irr}(E)$ be the closed point 
		corresponding to the isomorphism class of a tuple $(\rho_v)_{v \in 
		S_p}$ of absolutely irreducible representations $\rho_v : G_{F_\wv} \to 
		\GL_n(E)$. 
		Then there is a canonical isomorphism
		\[ T_z \cX_{ps, p}^{p-irr} \cong \oplus_{v \in S_p} H^1(F_\wv, \ad 
		\rho_v). \]
		\item Let $z \in \cX_{ps}^{p-irr}(E)$ be the closed point determined by 
		a representation $\rho : G_{F^+, S} \to \cG_n(E)$ such that for each $v 
		\in S_p$, $\rho|_{G_{F_\wv}}$ is absolutely irreducible. Then there is 
		a canonical isomorphism
		\[ T_z \cX_{ps}^{p-irr} \cong H^1(G_{F^+, S}, \ad \rho), \]
		which has the property that the diagram
		\[ \xymatrix{  T_z \cX_{ps}^{p-irr} \ar[r] \ar[d] & H^1(G_{F^+, S}, \ad 
		\rho) \ar[d] \\ T_z \cX_{ps, p}^{p-irr} \ar[r] & \oplus_{v \in S_p} 
		H^1(F_\wv, \ad \rho_v)} \]
		commutes.
	\end{enumerate}
\end{lemma}
\begin{proof}
	The first part follows from \cite[\S 4.1]{chenevier_det}, which states that the completed local ring of $\cX_{ps, v}^{v-irr}$ at the $E$-point corresponding to an absolutely irreducible representation $\rho_v : G_{F_\wv} \to \GL_n(E)$ pro-represents the functor $D_{\rho_v}$. The second follows from this and Lemma \ref{lem_deforming_conjugate_self_dual_representations}.
\end{proof}
\begin{prop}\label{prop_affinoid_local_trianguline_space}
	Let $v \in S_p$, and let $\rho_v : G_{F_\wv} \to \GL_n(E)$ be an absolutely 
	irreducible representation which is trianguline of parameter 
	$\delta_v \in \cT_v^{reg}(E)$. Let $z \in \cX_{ps, v}^{v-irr} \times 
	\cT_v^{reg}$ be the closed point corresponding to the pair $(\rho_v, \delta_v)$. Then:
	\begin{enumerate}
		\item There exists an affinoid open neighbourhood $\cU_v \subset 
		\cX_{ps, v}^{v-irr} \times \cT_v^{reg}$ of $z$ over which there exists 
		a universal representation $\rho_v^u : G_{F_\wv} \to \GL_n(\cO(\cU_v))$. 
		Let $\cV_v \subset \cU_v$ denote the set of points $(\rho_v', 
		\delta_v')$ such that $\rho_v'$ is trianguline of parameter $\delta_v'$, and let 
		$\cZ_v \subset \cU_v$ denote the Zariski closure of $\cV_v$. Then 
		$\cV_v$ is the set of points of a Zariski open subspace of $\cZ_v$.
		\item The Zariski tangent space of $\cZ_v$ at $z$ is contained in the 
		subspace $H^1_{tri, \delta_v}(F_\wv, \ad \rho_v)$ of the Zariski tangent space of $\cX_{ps, 
		v}^{v-irr} \times \cT_v^{reg}$ at $z$.
	\end{enumerate}
\end{prop}
\begin{proof}
		By Lemma \ref{lem_locally_split}, there is an affinoid neighbourhood 
		$\cU_v \subset \cX_{ps, v}^{v-irr} \times \cT_v^{reg}$ of $z$ over 
		which there exists a universal representation $\rho_v^u : G_{F_\wv} \to 
		\GL_n(\cO(\cU_v))$. We can assume without loss of generality that 
		$\cU_v$ is connected. By \cite[Corollary 6.3.10]{Ked14}, there is a reduced 
		rigid space $\cZ'$ over $E$ and a proper birational morphism $f : \cZ' 
		\to \cZ_v$ having the following properties:
	\begin{itemize}
		\item For every point $z' \in \cZ'$, the absolutely irreducible 
		representation $\rho_{f(z')}$ is trianguline. 
		\item There is an increasing filtration $0 = \cF_0 \subset \cF_1 
		\subset 
		\dots \subset \cF_n = D^\dagger_{rig}(f^\ast(\rho^u)))$ by coherent 
		$(\varphi, 
		\Gamma_{F_\wv})$-stable $\cR_{\cZ',F_{\wv}}$-submodules.
		\item There exists a Zariski closed subset $\cZ'_b \subset \cZ'$ such 
		that $\cZ'_b \cap f^{-1}(\cV_v) = \emptyset$ and for each $z' \in \cZ' - 
		\cZ'_b$, the pullback of $\cF_\bullet$ to 
		$D^\dagger_{rig}(\rho_{f(z')})$ is a triangulation of parameter 
		$\delta_{f(z')}$. 
		\item Over $\cZ' - \cZ'_b$, $\cF_\bullet$ is in fact a filtration by 
		local direct summand $\cR_{\cZ', F_\wv}$-submodules. 
		\item The map $f$ factors through a proper birational morphism $\tilde{f}: \cZ' \to \widetilde{\cZ}_v$, where $\widetilde{\cZ}_v$ is the normalisation of $\cZ_v$. Moreover, $\tilde{f}$ factors as the composition of a sequence of proper birational morphisms between normal rigid spaces \[\cZ' = \cZ_m \to \cZ_{m-1} \to \cdots \cZ_1 = \widetilde{\cZ}_v\] where each morphism $\cZ_i \to \cZ_{i-1}$ is glued, locally on the target, from analytifications\footnote{Here we mean the relative analytification defined by K\"{o}pf \cite{kopf}, see also \cite[Example 2.2.11]{conrad-ample}.} of birational projective schemes over $\Spec(A)$, with $\Sp(A) \subset \cZ_{i-1}$ an affinoid open.
	\end{itemize}
	Note that the final point is a consequence of the construction in the proof of \cite[Theorem 6.3.9]{Ked14}. The third point actually implies that $\cZ' = \cZ_b' \sqcup f^{-1}(\cV_v)$, 
	hence $\cZ_v= f(\cZ'_b) \sqcup \cV_v$. Since $f$ is proper this shows that 
	$\cV_v$ is Zariski open in $\cZ_v$.

	Let $\tilde{z}_1, \dots, \tilde{z}_m \in \widetilde{\cZ}_v$ be the closed points of the normalisation with image in $\cZ_v$ equal to $z$. For each $1 \le j \le m$, let $z'_j$ be a closed point of the preimage of $\tilde{z}_j$ in $\cZ'$. We denote by $z_j$ the image of $z'_j$ in any of the $\cZ_i$, for $1 \le i \le m$. We claim that the map $\widehat{\cO}_{\widetilde{\cZ}_v , \tilde{z}_j} \to \widehat{\cO}_{\cZ' , z'_j}$ on completed local rings is injective; 
	indeed, it follows from the final point in the itemized list above that we need to show injectivity for each map $\widehat{\cO}_{\cZ_i , z_j} \to \widehat{\cO}_{\cZ_{i+1} , z_j}$. Each of these maps coincides with the map on complete local rings $\widehat{A}_{x} \to \widehat{\cO}_{X,x}$ associated with a (projective, birational) morphism of schemes $X \to \Spec(A)$, where $A$ is the ring of functions on an open affinoid neighbourhood of $z_j \in \cZ_i$, $x \in \Spec(A)$ is the maximal ideal given by $z_j$ and $x' \in X$ is a closed point mapping to $x$. The complete local ring $\widehat{A}_{x}$ is a domain (by normality and excellence of $A_x$) and the map $A_x \to \cO_{X,x}$ is injective (by dominance of $X \to \Spec(A)$), so \cite[Corollaire 3.9.8]{EGAI} gives the desired injectivity. The map $\widehat{\cO}_{\cZ_v,z} \to \prod_{i=1}^m \widehat{\cO}_{\widetilde{\cZ}_v , \tilde{z}_i}$ is the normalisation of $\widehat{\cO}_{\cZ_v,z}$, so it is also injective. Putting everything together, we have shown that the map $\widehat{\cO}_{\cZ_v,z} \to \prod_{i=1}^m \widehat{\cO}_{\cZ'_v,z'_i}$ is injective. 
	
	After possibly extending $E$, we can assume that all of the points $z_i'$ in $\cZ'$ have residue field $E$. The existence of a global triangulation over 
	$\cZ' - \cZ'_b$ implies that for each $i = 1, \dots, m$, there is a 
	classifying map $R_{\rho_v, \cF_v, \delta_v} \to \widehat{\cO}_{\cZ', 
	z'_i}$, where $\cF_v$ is the unique triangulation of $\rho_v$ of parameter 
	$\delta_v$. This implies the existence of a commutative diagram
	\[ \xymatrix{ \widehat{\cO}_{\cU_v, z} \ar[r] \ar[d]  & 
	\widehat{\cO}_{\cZ_v, z} \ar[d] \\
		R_{\rho_v, \cF_v, \delta_v} \ar[r] & \prod_{i=1}^m\widehat{\cO}_{\cZ', 
		z'_i}, } \] 
	where the left vertical arrow is surjective (Proposition \ref{prop_representability_of_trianguline_def_functor}) and the top horizontal 
	arrow is surjective. We have already noted that the right vertical arrow is injective.	These facts together imply that the top horizontal arrow factors through a surjective 
	map $R_{\rho_v, \cF_v, \delta_v}  \to \widehat{\cO}_{\cZ_v, z}$. This 
	implies the desired result at the level of Zariski tangent spaces.	
\end{proof}
\begin{remark}
	Prompted by a referee, we note that the definition of `Zariski dense' given in \cite[Definition 6.3.2]{Ked14} is somewhat non-standard. In this paper (and in other references we cite such as \cite{Bre17}), a subset $Z$ of a rigid space $X$ is called Zariski dense if the smallest closed analytic subset of $X$ which contains $Z$ is $X$. In \cite[Definition 6.3.2]{Ked14} the stronger condition is imposed that $Z$ is Zariski dense (in the usual sense) in each member of some admissible affinoid cover of $X$.
	
When we apply \cite[Corollary 6.3.10]{Ked14} in the above proof, we have a Zariski dense subset of an affinoid, so there is no discrepancy between the definitions in this case. Eugen Hellmann has explained to us that the crucial result \cite[Theorem 6.3.9]{Ked14} does in fact hold with the weaker, standard definition of Zariski dense. Since it may be of interest, we sketch their argument.

We start with $X, \delta, M$ as in \cite[Theorem 6.3.9]{Ked14} and suppose we have a Zariski dense (in the usual sense) subset $X_{\mathrm{alg}} \subset X$ satisfying the assumptions of \emph{loc.~cit.} We may assume that $X$ is normal and connected, and will show that $X_{\mathrm{alg}}$ can be enlarged to a subset which is Zariski dense in the stronger sense of \cite{Ked14}. 

There are coherent sheaves $H^i_{\varphi,\gamma_K}(M^\vee(\delta))$, $H^i_{\varphi,\gamma_K}(M^\vee(\delta)/t_\sigma)$ on $X$, which are locally free over a non-empty (hence dense) Zariski open subset $U \subset X$. At points $z$ in the Zariski dense subset $X_{\mathrm{alg}}\cap U$, the fibre $H^0_{\varphi,\gamma_K}(M^\vee(\delta))\otimes_{\cO_X}k(z) \cong H^0_{\varphi,\gamma_K}(M^\vee_z(\delta_z))$ has dimension one and the map $M_z \to \cR_{k(z)}(\pi_K)(\delta_z)$ dual to a non-zero element of this fibre is surjective. The latter condition is equivalent to non-vanishing of the map $H^0_{\varphi,\gamma_K}(M^\vee(\delta))\otimes_{\cO_X}k(z) \to  H^0_{\varphi,\gamma_K}(M^\vee(\delta)/t_\sigma)\otimes_{\cO_X}k(z)$ for every $p$-adic embedding $\sigma$. These conditions hold over a Zariski open subset $U' \subset U$. Since $U'$ contains $X_{\mathrm{alg}}\cap U$, it is also Zariski dense in $X$. Moreover, $U'$ contains a Zariski dense subset of every affinoid open $V \subset X$. Indeed, the intersection $V\cap U$ with the Zariski open and dense subset $U$ contains an affinoid open subset of $V$. Repeating this step, the intersection $V \cap U'$ also contains an affinoid open subset of $V$. We have shown that we obtain the desired enlargement of $X_{\mathrm{alg}}$ by adjoining $U'$. 
\end{remark}

\subsection{The unitary group eigenvariety}\label{sec_unitary_group_eigenvariety}

Now let $F, S, p, G = G_n$ be as in our standard assumptions (\S \ref{sec_definite_unitary_groups}). We continue to assume that $E$ 
contains the 
image of every embedding $\tau: F \to \Qpbar$. In particular, the reductive group
$\Res_{F^+/\Q}G_n$ splits over $E$.

Let $U_n = \prod_v U_{n, v} \subset 
G_n(\bA_{F^+}^\infty)$ be an open compact subgroup such that for every finite place $v\not\in S$ of $F^+$, $U_{n, v}$ is 
hyperspecial maximal compact subgroup of $G_n(F^+_v)$. We 
define $\bT^S_n = \cO[T_w^1, \dots, T_w^n, (T_{w}^n)^{-1}] 
\subset 
\cO[U_n^S \backslash G_n(\bA_{F^+}^{\infty, S}) / U_n^S]$ to be the algebra 
generated by the unramified Hecke operators at split places $v = w w^c$ of $F^+$ not lying in $S$. These operators were defined in \S \ref{subsec:alg_mod_forms}.

We write $T_n \subset B_n = T_n N_n \subset \GL_n$ for the usual maximal torus 
and upper triangular Borel subgroup, and define $E$-rigid 
spaces 
\[ \cW_n = 
\Hom(\prod_{v \in S_p} 
T_n(\cO_{F_\wv}), \bG_m) \]
 and 
 \[ \cT_n = \Hom(\prod_{v \in S_p} T_n(F_\wv), 
\bG_m). \]
Restriction of characters determines a morphism $r : 
\cT_n \to \cW_n$ 
of rigid spaces. Note that the spaces 
$\cT_n, \cW_n$ may be canonically 
identified with the spaces $\cT, \cW$ of the previous section. 

We fix a choice of isomorphism $\iota : \overline{\bbQ}_p \to \bC$. If $\pi$ is 
an automorphic representation of $G_n(\bA_{F^+})$ with $\pi^{U_n} \ne 0$, 
there is a corresponding semisimple Galois representation $r_{\pi, \iota} : 
G_{F, S} \to \GL_n(\overline{\bQ}_p)$ (cf. Corollary 
\ref{cor:galois_existence}), which satisfies local-global compatibility at each 
place of $F$. The space $\iota^{-1} 
(\pi^\infty)^{U_n}$ is naturally an isotypic $\bT_n^S$-module, which therefore 
determines a homomorphism $\psi_\pi : \bT_n^S \to \overline{\bQ}_p$. We call an 
\emph{accessible refinement} of $\pi$ a choice $\chi = (\chi_v)_{v \in S_p}$ 
for each 
$v \in S_p$ of  a (necessarily smooth) character $\chi_v : T_n(F_\wv) \to 
\overline{\bbQ}_p^\times$ which appears as a subquotient of the 
normalised Jacquet module $\iota^{-1}r_{N_n}( \pi_v) = \iota^{-1} (\pi_{v, 
N_n(F_\wv)} \delta_{B_n}^{-1/2})$; equivalently, for which there is an embedding of 
$\pi_v$ into the normalised induction $i_{B_n}^{\GL_n} \iota \chi_v$.
Note that $\chi \in \cT_n(\overline{\bQ}_p)$.

\begin{lemma}
	Let $\pi$ be an automorphic representation of $G_n(\bA_F^+)$, and let $\chi = (\chi_v)_{v \in S_p}$ be an accessible refinement of $\pi$. Then for each $v \in S_p$, there is an increasing filtration of $\rec^T_{F_\wv}(\iota^{-1}\pi_v)$ by sub-Weil--Deligne 
	representations with graded pieces
	\[ \chi_{v, 1}| \cdot |^{(1-n)/2} \circ \Art_{F_\wv}, \dots, \chi_{v, n}| \cdot |^{(1-n)/2} \circ \Art_{F_\wv}. \]
\end{lemma}
\begin{proof}
Since $\pi_v$ admits an accessible refinement, it is a subquotient of a 
principal series representation. Suppose that $\rec_{F_\wv}(\pi_v) = 
\oplus_{i=1}^k \operatorname{Sp}_{n_i}(\psi_i| \cdot |^{(n_i - 1)/2})$ for 
some characters $\psi_i : F_\wv^\times \to \C^\times$. By the Langlands 
classification, $\pi_v$ is isomorphic to a subquotient of the normalised 
induction
\[ \Pi = i_P^{\GL_n} \operatorname{St}_{n_1}(\psi_1) \otimes \dots \otimes 
\operatorname{St}_{n_k}(\psi_k), \] where $P \subset \GL_n$ is the 
standard parabolic subgroup corresponding to the partition $n = n_1 + n_2 
+ \dots + n_k$. It will therefore suffice to show the stronger statement 
that if $\alpha = \alpha_1 \otimes \dots  \otimes \alpha_n$ is a 
subquotient of the normalised Jacquet module of $\Pi$, then there is an 
increasing filtration of $\oplus_{i=1}^k \operatorname{Sp}_{n_i}(\psi_i | 
\cdot |^{(n_i-1)/2})$ by sub-Weil--Deligne representations with graded 
pieces given by $\alpha_1 \circ \Art_{F_\wv}, \dots, \alpha_n \circ 
\Art_{F_\wv}$. We recall that each $\operatorname{Sp}_{n_i}(\psi_i | 
\cdot |^{(n_i-1)/2})$ comes with a standard basis $e_1,e_2,\ldots,e_{n_i}$. We concatenate and relabel these bases so that $e_1,e_2,\ldots,e_n$ is a basis for $\oplus_{i=1}^k \operatorname{Sp}_{n_i}(\psi_i |\cdot|^{(n_i-1)/2})$ with $e_{1+\sum_{i=1}^{j-1}n_i},\ldots e_{\sum_{i=1}^j n_i}$ the standard basis for $\operatorname{Sp}_{n_j}(\psi_j | 
	\cdot |^{(n_j-1)/2})$.
	 
We first treat the case $k = 1$, $n_1 = n$. After twisting we can assume that 
$\psi = 1$. Then the normalised Jacquet module of $\operatorname{St}_n$ equals 
$| \cdot |^{(n-1)/2} \otimes \dots \otimes | \cdot |^{(1-n)/2}$, while there is 
a unique invariant flag of $\operatorname{Sp}_n(| \cdot |^{(n-1)/2})$ given by 
$\Fil_i = \operatorname{span}(e_1, \dots, e_i)$ ($i = 1, \dots, n$) which has 
the desired graded pieces. 
	 
Now we return to the general case.	 Using \cite[Theorem 1.2]{Zel80}, we see 
that the irreducible subquotients of the normalised Jacquet module of $\Pi$ are 
precisely the characters $\beta_{w^{-1}(1)} \otimes \dots \otimes 
\beta_{w^{-1}(n)}$, where $w \in S_n$ is any permutation which is increasing on 
each of the sets $\{ 1, \dots, n_1 \}, \{ n_1 +1, \dots, n_1 + n_2 \}, \dots, 
\{ n_1 + \dots + n_{k-1} + 1, \dots, n_1 + \dots + n_k \}$, and $(\beta_1, 
\dots, \beta_n)$ is the concentation of the tuples $(\psi_i | \cdot 
|^{(n_i-1)/2}, \dots, \psi_i | \cdot |^{(1-n_i) / 2})$ for $i = 1, \dots, k$.

	 We see that the increasing filtration of 
	 \[ \oplus_{i=1}^k \operatorname{Sp}_{n_i}(\psi_i | 
	 \cdot |^{(n_i-1)/2}) \]
	  given by $\Fil_j = \operatorname{span}(e_{w^{-1}(1)}, \dots, 
	  e_{w^{-1}(j)})$ is a filtration by sub-Weil--Deligne representations 
	  which has the desired property. This completes the proof.
\end{proof}
If $\chi$ is an accessible refinement of $\pi$, then we write $\nu(\pi, \chi) 
\in \cT_n(\overline{\bQ}_p)$ for the character
\begin{equation}\label{eqn:nuforRAn} \nu(\pi, \chi) = \kappa(\pi) \cdot ( \chi_v \iota^{-1}\delta_{B_n}^{-1/2})_{v \in 
S_p}, \end{equation}
where $\kappa(\pi) \in \cT_n(E)$ is the ($B_n$-dominant) 
$\Q_p$-algebraic character which is the highest weight of $\iota^{-1} \pi_\infty$. If 
$\kappa(\pi)_v = (\kappa_{\tau,1} \ge \kappa_{\tau,2} \ge \cdots \ge 
\kappa_{\tau,n})_{\tau: F_\wv \to \Qpbar}$ then the labelled Hodge--Tate 
weights of $r_{\pi,\iota}|_{G_{F_\wv}}$, in increasing order, are 
$(-\kappa_{\tau,1} < 
-\kappa_{\tau,2}+1 < \cdots <
-\kappa_{\tau,n}+n-1)_{\tau: F_\wv \to \Qpbar}$. 

We 
write $\jmath_n : \cT_n \to \cT_n$ for the map defined by the formula
\[ \jmath_n(\nu)_v = \nu_v \cdot (1, \epsilon^{-1}\circ \Art_{F_{\wv}}, \dots, 
\epsilon^{1-n}\circ \Art_{F_{\wv}}). \]
The reason for introducing this map is that if $\pi$ is an automorphic 
representation of $G_n(\bA_{F^+})$ and $\chi$ is an accessible refinement, then 
the parameter $\delta$ associated to $\chi|\cdot|^{(1-n)/2}$ by the formula of 
Lemma 
\ref{lem_numerically_non_critical_implies_non_critical} satisfies 
$\jmath_n(\nu(\pi , \chi)) = \delta$. 
We call the accessible refinement $\chi$ 
numerically non-critical or ordinary if $\delta$ is. Note that this property 
depends on the pair $(\pi, \chi)$ and not just on $\chi$.

\subsubsection{Emerton's eigenvariety 
construction}\label{sssec:emertonconstruction} We now describe the 
construction, following Emerton \cite{MR2207783}, of the (tame level $U_n$) 
eigenvariety for $G_n$. We use Emerton's construction because we do not want to 
restrict to considering $\pi$ with Iwahori-fixed vectors at places in $S_p$ (as 
is done, 
for example, in \cite{bellaiche_chenevier_pseudobook}) and it seems to us that 
Emerton's representation-theoretic viewpoint is the most transparent way to 
handle this level of generality.

We recall the set-up of \S \ref{subsec:alg_mod_forms}, so for each 
dominant weight $\lambda$ we have a module 
$S_{\lambda}(U_n,\cO/\varpi^n)$ of algebraic modular forms, which has a natural 
action of $\TT_n^S$. When $\lambda$ is trivial we omit it from the notation. 

We define \[\widetilde{S}(U_n^p,\cO) := \varprojlim_s\left(
\varinjlim_{U_{p}}S(U_n^pU_p, \cO/\varpi^s)\right) \]
and
\[ \widetilde{S}(U_n^p,E):= \widetilde{S}(U_n^p,\cO)\otimes_{\cO}E,\] so 
$\widetilde{S}(U_n^p,E)$ is an $E$-Banach space (with unit ball 
$\widetilde{S}(U_n^p,\cO)$), equipped with an 
admissible continuous representation of $G_n(F^+_p)$. For dominant weights 
$\lambda$, we can 
consider the space of locally $V_\lambda^\vee$-algebraic vectors 
$\widetilde{S}(U_n^p,E)^{V_\lambda^\vee-alg}$. We have a $(G_n(F^+_p)\times 
\TT_n^S)$-equivariant isomorphism 
\[\varinjlim_{U_p}S_\lambda(U_n^pU_p,\cO)\otimes_\cO V_\lambda^\vee \cong 
\widetilde{S}(U_n^p,E)^{V_\lambda^\vee-alg}\]
(see \cite[Corollary 2.2.25]{MR2207783}). We can also consider the space of 
locally $\Qp$-analytic vectors 
$\widetilde{S}(U_n^p,E)^{an}$, and apply Emerton's locally analytic Jacquet 
functor $J_{B_n}$ to this locally analytic representation of $G_n(F^+_p)$. We 
thereby obtain an essentially admissible locally analytic representation 
$J_{B_n}\widetilde{S}(U_n^p,E)^{an}$ of $\prod_{v \in S_p}T_n(F_\wv)$, and by 
duality a coherent sheaf $\cM_n$ on $\cT_n$, equipped with an action of 
$\TT_n^S$. We denote by $\mathcal{A}_n \subset \End(\cM_n)$ the coherent 
$\cO_{\cT_n}$-algebra subsheaf generated by  $\TT_n^S$. Now we can define the 
eigenvariety, an $E$-rigid space, as a relative rigid analytic spectrum \[\cE_n 
:= 
\mathrm{Sp}_{\cT_n}\mathcal{A}_n \xrightarrow{\nu'} \cT_n\] 
equipped with the canonical finite morphism $\nu'$. 

We define another finite 
morphism $\nu: 
\cE_n \to \cT_n$ by twisting $\nu'$ by $\delta_{B_n}^{-1}$ 
(see Remark \ref{nncvsncs}). By construction, we also have a ring homomorphism 
$\psi: \TT_n^S \to \cO(\cE_n)$, so we obtain a map on points: 
\[ \psi^* \times 
\nu: \cE_n(\Qpbar) \to \Hom(\T_n^S,\Qpbar) \times \cT_n(\Qpbar). \]
For $E'/E$ finite (with $E' \subset \Qpbar$), a point $(\psi_0, \nu_0) \in 
\Hom(\T_n^S,E') \times \cT_n(E')$ is in 
the 
image of $\psi^* \times \nu$ if and only if the eigenspace \[ 
J_{B_n}\left(\widetilde{S}(U_n^p,E')^{an}\right)[\psi_0,\nu_0\delta_{B_n}]\] is non-zero, 
or in 
other words if there is a non-zero $\prod_{v \in S_p}T(F_\wv)$-equivariant 
map \[\nu_0\delta_{B_n}\to 
J_{B_n}\left(\widetilde{S}(U_n^p,E')^{an}\right)[\psi_0].\] We define the subset $Z_n 
\subset 
\cE_n(\Qpbar)$ of classical 
points to be those for which there is moreover a non-zero map to the Jacquet 
module of the locally algebraic vectors: \[\nu_0\delta_{B_n}\to 
J_{B_n}\left(\widetilde{S}(U_n^p,E')^{alg}\right)[\psi_0].\] 

\begin{lemma}\label{lem:eigenspaces}
For any characters $\psi: \T_n^S \to E$ and $\chi: \prod_{v \in S_p}T(F_\wv) \to E^\times$, we have \[\Hom_{\prod_{v \in S_p}T(F_\wv)}\left(\chi,J_{B_n}\left(\widetilde{S}(U_n^p,E)^{an}\right)[\psi]\right)  = \Hom_{\prod_{v \in S_p}T(F_\wv)}\left(\chi,J_{B_n}\left(\widetilde{S}(U_n^p,E)^{an}[\psi]\right)\right).\] 
\end{lemma}
\begin{proof}
This can be seen using Emerton's canonical lift \cite[Proposition 3.4.9]{MR2292633}, which identifies both sides of the equality with the same eigenspace in $\widetilde{S}(U_n^p,E)^{an}$. Alternatively, we can use the left exactness of the Jacquet functor. In the latter argument we need to use the fact that $\T_n^S$ acts on $\widetilde{S}(U_n^p,E)^{an}$ via a Noetherian ring and we then deduce that passing to an eigenspace for a (finitely generated) ideal in this ring commutes with the Jacquet functor.
\end{proof}

We now relate the classical points $Z_n$ to refined automorphic representations. 
Let $\mathcal{A}_n$ denote the set of automorphic 
representations $\pi$ of $G_n(\bA_{F^+})$ such that $(\pi^\infty)^{U_n^{S_p}} 
\neq 0$, let $\mathcal{RA}_n$ denote the set of pairs $(\pi, \chi)$ where 
$\pi \in \mathcal{A}_n$ and $\chi$ is an accessible refinement of $\pi$, 
and let $\mathcal{Z}_n \subset \Hom(\bT_n^S, \overline{\bbQ}_p) \times 
\cT_n(\overline{\bbQ}_p)$ denote the set of points of the form $(\psi_\pi, 
\nu(\pi, \chi))$, where $(\pi, \chi) \in \mathcal{RA}_n$. We note in particular 
the existence of the surjective map $\gamma_n : 
\mathcal{RA}_n \to \cZ_n$. 
\begin{lemma}The map
$\psi^\ast \times \nu$ restricts to a 
bijection $Z_n \to \mathcal{Z}_n$.\end{lemma}
\begin{proof}
If $z \in Z_n \cap \cE_n(E')$ is a classical point defined over 
	$E'$, it follows from the above discussion on locally algebraic vectors 
	that 
	$z$ arises arises from a non-zero map \[\nu(z)\delta_{B_n} \to 
	J_{B_n}\left(\varinjlim_{U_p}S_\lambda(U_n^pU_p,E')[\psi^*(z)]\otimes_{E'} 
	V_\lambda^\vee\right)\] for some dominant weight $\lambda$. It 
	follows 
	from \cite[Prop.~4.3.6]{MR2292633} that such maps correspond bijectively 
	with 
	non-zero maps \[\nu(z)\delta_{B_n}(\lambda^\vee)^{-1} \to 
	J_{B_n}\left(\varinjlim_{U_p}S_\lambda(U_n^pU_p,E')[\psi^*(z)]\right),\] 
	where $\lambda^\vee$ is the highest weight of $V_\lambda^\vee$.
	
	By Lemma \ref{lem_algebraic_automorphic_forms_are_classical}, we have 
	$(\pi,\chi) \in \mathcal{RA}_n$ where $\pi_\infty$ has highest weight 
	$\iota\lambda^\vee$, $\psi_\pi = 
	\psi^*(z)$ and $\nu(z) = \lambda^\vee 
	\chi\delta_{B_n}^{-1/2} 
	= \nu(\pi,\chi)$. 
	This shows that we do indeed have an induced map $Z_n \to \cZ_n$, and it is 
	easy to see that this is a bijection.
\end{proof}	

\begin{remark}\label{nncvsncs}
	An accessible refinement $\chi$ is numerically non-critical if and only if 
	for every $v\in S_p$ the character $\nu(\pi,\chi)_v\delta_{B_n} = 
	\kappa(\pi)_v \chi_v \delta_{B_n}^{1/2}$ has non-critical slope, in the sense of \cite[Defn.~4.4.3]{MR2292633}. The renormalisation (replacing $\chi_v 
	\delta_{B_n}^{-1/2}$ with $\chi_v \delta_{B_n}^{1/2}$) appears in Emerton's 
	eigenvariety construction because $\chi_v \delta_{B_n}^{1/2}$ is a smooth 
	character appearing in the (non-normalised) Jacquet module 
	$\iota^{-1}\pi_{v,N_n(F_\wv)}$, whilst Bella\"{i}che--Chenevier normalise 
	things to be compatible with the Hecke action on Iwahori-fixed vectors (see 
	\cite[Prop.~6.4.3]{bellaiche_chenevier_pseudobook}).
\end{remark} 

Our next task is to recall some well known properties of the eigenvariety $\cE_n$ (cf.~\cite[\S7]{Bre15}), variants of which are established by numerous authors in slightly different contexts (e.g.~\cite{chenevier-GLn,Buz07,MR2207783,loeffler-ocaaf}). We follow the exposition of \cite{Bre17} which establishes these properties for the patched eigenvariety. In order to at least sketch the proofs of these properties in our context, we first introduce a `spectral variety' which will turn out to be a Fredholm hypersurface over $\cW_n$. 

We fix the element $z =(z_v)_{v \in S_p} \in \prod_{v\in S_p} T_n(F_{\wv})$ with $z_v = \diag(\varpi_\wv^{n-1},\ldots,\varpi_{\wv},1)$, and let $Y$ be the closed subgroup of $\prod_{v\in S_p} T_n(F_{\wv})$ generated by $\prod_{v\in S_p} T_n(\cO_{F_\wv})$ and $z$. The rigid space $\widehat{Y} = \Hom(Y,\bG_m)$ is then identified with $\cW_n \times \bG_m$. As in \cite[\S3.3]{Bre17}, it follows from \cite[Proposition 3.2.27]{MR2292633} that $J_{B_n}\widetilde{S}(U_n^p,E)^{an}$ has dual equal to the space of global sections of a coherent sheaf $\mathcal{N}_n$ on $\cW_n \times \bG_m$. We define $\cY_z$ to be the schematic support (cf.~above D\'{e}finition 3.6 in \cite{Bre17}) of $\mathcal{N}_n(\delta_{B_n}^{-1})$. This rigid space comes equipped with a closed immersion $\cY_z \hookrightarrow \cW_n \times \bG_m$. The twist in the definition of $\cY_z$ is there to ensure that this closed immersion is compatible with the map $\nu$. Indeed, the map from $\cE_n$ given by composing $\nu$ with the restriction map to $\cW_n\times\bG_m$ factors through a finite map $f: \cE_n \to \cY_z$, giving us a commutative diagram: 

\[ \xymatrix{ \cE_n \ar[r]^{\nu}\ar[d]^{f}& \cT_n\ar[d] \\ \cY_z\ar@{^{(}->}[r] & \cW_n\times \bG_m} \]

Now we state our proposition summarising the key properties of the eigenvariety $\cE_n$.
\begin{proposition}\label{prop:eigenvariety}
The tuple $(\cE_n, \psi, \nu, Z_n)$ has the following 
properties: 

\begin{enumerate}
	\item\label{reduced} $\cE_n$ is a reduced $E$-rigid space, equipped with a 
	finite morphism 
	$\nu : \cE_n \to \cT_n$. We write $\kappa$ for the induced map $\kappa : 
	\cE_n \to \cW_n$.
	\item\label{dense} $Z_n \subset \cE_n(\overline{\bbQ}_p)$ is a Zariski 
	dense subset 
	which accumulates at every point of $Z_n$ (in other words, each point of 
	$Z_n$ admits a basis of affinoid neighbourhoods $V$ such that $V \cap Z_n$ 
	is Zariski dense in $V$), and the map 
	$\psi^\ast \times \nu : \cE_n(\overline{\bbQ}_p) \to \Hom(\bT_n^S, 
	\overline{\bbQ}_p) \times \cT_n(\overline{\bbQ}_p)$ restricts to a 
	bijection $Z_n \to \mathcal{Z}_n$.
	\item\label{hecke} For any affinoid open $V \subset \cT_n$, the map 
	$\bT^S_n 
\otimes \cO(V) \to \cO(\nu^{-1} V)$ is surjective.
	\item\label{equidim} $\cE_n$ is equidimensional of dimension equal to $\dim 
	\cW_n$. 
	For any irreducible component $\cC \subset \cE_n$, $\kappa(\cC)$ is a Zariski open subset of $\cW_n$.
	\item\label{classicality} Let $z \in \cE_n(\Qpbar)$ be a point, and suppose 
	that $\delta = 
	\jmath_n(\nu(z))$ factors as $\delta = \delta_{alg} \delta_{sm}$, where 
	$\delta_{alg}$ is a strictly dominant algebraic character and $\delta_{sm}$ 
	is smooth, and that $\delta$ is numerically non-critical. Then $z \in Z_n$.
	\item\label{bounded} $\psi$ takes values in the subring $\cO(\cE_n)^{\le 
	1}$ of 
	bounded elements. 
\end{enumerate}

\end{proposition}
\begin{proof}
	First we note that a tuple satisfying the first three properties is unique - we will not actually use this fact, but it can be proved in the same way as \cite[Proposition 7.2.8]{bellaiche_chenevier_pseudobook} (our context 
	is slightly different, as we equip our eigenvarieties with a map to $\cT_n$ 
	instead of $\cW_n\times \bG_m$). We also note that it is not essential for our purposes to show that $\cE_n$ is reduced (this is the most delicate of the listed properties); we could instead replace $\cE_n$ with its underlying reduced subspace.
	
	Now we summarise how to verify these 
	properties. Property (\ref{classicality}) follows from Emerton's 
	`classicality criterion' for his Jacquet functor \cite[Theorem 
	4.4.5]{MR2292633} (cf.~Remark \ref{nncvsncs}). Property (\ref{hecke}) holds 
	by construction.
	
	Property (\ref{equidim}) can be established as in \cite[\S3.3]{Bre17} using the spectral variety $\cY_z$. More precisely, (the proof of) Lemma 3.10 in this reference shows that the closed analytic subset of $\cW_n \times \bG_m$ underlying $\cY_z$ is a Fredholm hypersurface, and $\cY_z$ has an admisible cover by affinoids $(U'_i)_{i\in I}$ on which the map to $\cW_n$ is finite and surjective with image an open affinoid $W_i \subset \cW_n$. Moreover, each $U'_i$ is disconnected from its complement in the inverse image of $W_i$ and $\Gamma(U'_i,\mathcal{N}_n)$ is a finite projective $\cO_{\cW_n}(W_i)$-module. 
	
	Having established the existence of a good affinoid cover of the spectral variety, we set $U_i = f^{-1}(U'_i)$. Since $f$ is a finite map, $(U_i)_{i\in I}$ is an admissible affinoid cover of $\cE_n$. It can then be shown, as in \cite[Proposition 3.11]{Bre17}, that each affinoid $\cO_{\cE_n}(U_i)$ is isomorphic to a $\cO_{\cW_n}(W_i)$-algebra of endomorphisms of the finite projective $\cO_{\cW_n}(W_i)$-module $\Gamma(U_i,\cM_n)$. We can now prove Property (4) as in \cite[Corollaire 3.12]{Bre17}: this shows that $\cE_n$ is equidimensional of dimension equal to $\cW_n$, without embedded components, and each irreducible component maps surjectively to an irreducible component of $\cY_z$. Since irreducible components of Fredholm hypersurfaces are again Fredholm hypersurfaces, the image of such an irreducible component is Zariski open in $\cW_n$ (cf.~\cite[Corollaire 3.13]{Bre17}). 
	
	Now to establish property (\ref{dense}), using property (\ref{classicality}), it suffices to show that points $z \in \cE_n(\Qpbar)$ with numerically non-critical $\delta = \jmath_n(\nu(z))$ accumulate at any point $z_0$ with $\kappa(z_0)$ locally algebraic (cf.~\cite[Th\'{e}or\`eme 3.19]{Bre17}). Using the good affinoid cover described in the previous paragraph, we have an affinoid neighbourhood $U$ of $z_0$ which is a finite cover of an affinoid $W \subset \cW_n$. In fact, such $U$ form a neighbourhood basis at $z_0$ (cf.~\cite[Theorem 2.1.1, Lemma 2.1.2]{taibi-eigenvariety}). The valuations $v_p(\delta_{v,i}(p))$ are bounded as $z$ varies in $U$ (with $\delta = \jmath_n(\nu(z))$). It follows from the description in Definition \ref{defn:numerically_non_critical} that there is a subset $\Sigma$ of $\cW_n$ accumulating at $\kappa(z_0)$ such that $\kappa^{-1}(\Sigma)\cap U$ consists entirely of points with numerically non-critical $\delta$. The subset $\kappa^{-1}(\Sigma)\cap U$ is Zariski dense in $U$. 
	
	Finally, to establish property (\ref{reduced}) it remains to prove that $\cE_n$ is reduced. Since we showed that $\cE_n$ is without embedded components, it suffices to prove that every irreducible component of $\cE_n$ contains a reduced point. Using (\ref{equidim}) and the Zariski density of algebraic characters in $\cW_n$, it suffices to show that $\cE_n$ is reduced at every point $z_0$ with $\kappa(z_0)$ algebraic. 
	We use a good affinoid neighbourhood $U = \Sp(B)$ of $z_0$ as in the previous paragraph, with $W = \kappa(U) = \Sp(A)$. The finite $A$-algebra $B$ is identified with a sub-$A$-algebra of $\End_A(M)$, where $M = \Gamma(U,\cM_n)$ is a finite projective $A$-module. As in the proof of \cite[Proposition 3.9]{chenevier-JL}, it now suffices to show that for $w$ in a Zariski dense subset of $W$, the Hecke algebra $\T_n^S$ and $\prod_{v \in S_p}T_n(F_{\wv})$ act semisimply on the fibre $M\otimes_A k(w)$ --- we use the fact that an endomorphism of a projective $A$-module which vanishes in the fibres at a Zariski dense subset of points in $W$ necessarily vanishes. The proof of \cite[Corollaire 3.20]{Bre17} shows that we can achieve this by choosing $w$ so that their pre-images in $U$ have \emph{tr\`{e}s classique} associated characters $\delta$ \cite[D\'{e}finition 3.17]{Bre17} (this is a condition on characters with algebraic image in $\cW_n$ which can be guaranteed by a `numerical' condition as in the proof of \cite[Th\'{e}or\`eme 3.19]{Bre17}, in particular it gives a Zariski-dense and self-accumulating subset of $\cE_n$). We can replace the reference to \cite{CEGGPSBreuilSchneider} in the proof with 
	the well-known assertion that the Hecke algebra $\T_n^S$ acts semisimply on $\varinjlim_{U_p}S_\lambda(U_n^pU_p,E)$ for dominant $\lambda$. Finally, property (\ref{bounded}) follows from the fact that the $\T_n^S$-action stabilizes the unit ball $\widetilde{S}(U_n^p,\cO)\subset \widetilde{S}(U_n^p,E)$.
\end{proof}

The properties established in Proposition \ref{prop:eigenvariety} imply the 
existence of a 
conjugate self-dual Galois pseudocharacter $T_n : G_{F, S} \to \cO(\cE_n)$ with 
the property that for any point $z \in Z_n$ corresponding to a pair $(\pi, 
\chi)$, $T_{n, z} = \tr r_{\pi, \iota}$.  This is proved as in \cite[Proposition 7.5.4]{bellaiche_chenevier_pseudobook} and \cite[Proposition 7.1.1]{chenevier-GLn}. The key points are that $\cO(\cE_n)^{\le 1}$ is compact \cite[Lemma 7.2.11]{bellaiche_chenevier_pseudobook} and the map $\cO(\cE_n)^{\le 1} \rightarrow \prod_{z \in Z_n} \C_p$ given by the evaluation maps at each $z \in Z_n$ is a continuous injection (by Zariski density of $Z_n$ and reducedness of $\cE_n$). Then \cite[Example 2.32]{chenevier_det} is used to glue together the pseudocharacters $\tr r_{\pi, \iota}$ to form the continuous pseudocharacter $T_n$. 

The pseudocharacter $T_n$ determines an admissible cover $\cE_n = 
\sqcup_{\overline{\tau}_n} \cE_n(\overline{\tau}_n)$ as a disjoint union of 
finitely many open subspaces indexed by $G_k$-orbits of pseudocharacters $\overline{\tau}_n : 
G_{F, S} \to \overline{\bF}_p$, over which which the residual pseudocharacter 
satisfies the condition $\overline{T}_{n, z} = \overline{\tau}_n$ (cf. 
\cite[Theorem 3.17]{chenevier_det}).

Fix a pseudocharacter $\overline{\tau}_n : G_{F, S} \to \overline{\bF}_p$. 
Extending $E$ if necessary, we may assume that $\overline{\tau}_n$ takes values 
in $k$. We 
recall some of the $E$-rigid spaces of Galois representations defined in \S \ref{subsec_trianguline_reps_global_geometry}, now 
decorated with 
$n$ subscripts. Thus $\cX_{ps, n}$ is the space of conjugate self-dual 
deformations of $\overline{\tau}_n$, $\cX_{ps, n}^{p-irr}$ is its Zariski open 
subspace of pseudocharacters which are irreducible at the $p$-adic places. We 
also have the subspace $\cX_{ps, n}^{irr}$ of pseudocharacters which are 
(globally) irreducible.  The 
existence of $T_n$ determines a morphism $\lambda : \cE(\overline{\tau}_n) \to 
\cX_{ps, n}$, and the morphism $i_n = \lambda \times (\jmath \circ \nu) : \cE_n(\overline{\tau}_n) 
\to \cX_{ps, n} \times \cT_n$ is a closed immersion, by point (\ref{hecke}) in 
the list 
of defining properties of $\cE_n$.

Now assume that $n \geq 3$, and let $\overline{\tau}_2 : G_{F, S} \to 
\overline{\bF}_p$ be a conjugate self-dual pseudocharacter of dimension 2. Let 
$\overline{\tau}_n = \Sym^{n-1} \overline{\tau}_2$; then $\overline{\tau}_n$ is 
a conjugate self-dual pseudocharacter of $G_{F, S}$ of dimension $n$. Taking 
symmetric powers of pseudocharacters determines a morphism $\sigma_{n, g} : 
\cX_{ps, 2} \to \cX_{ps, n}$. On the other hand, we can define a map 
$\sigma_{n, p} : \cT_2 \to \cT_n$ by the formula 
\[ ( (\delta_{v, 1}, \delta_{v, 2}))_{v \in S_p} \mapsto ( (\delta_{v, 
1}^{n-1}, \delta_{v, 1}^{n-2} \delta_{v, 2}, \dots, \delta_{v, 2}^{n-1}) )_{v 
\in S_p}. \]
We write $\sigma_n  = \sigma_{n, g} \times \sigma_{n, p} : \cX_{ps, 2} \times 
\cT_2 \to \cX_{ps, n} \times \cT_n$ for the product of these two morphisms. 
We have constructed a diagram
\[ \xymatrix@1{ \cE_2(\overline{\tau}_2) \ar[r]^-{\sigma_n \circ i_2}& \cX_{ps, 
n} \times \cT_n & \ar[l]_-{i_n} \cE_n(\overline{\tau}_n).} \]
Compare Lemma \ref{lem_inclusion_of_irreducible_components}.
\begin{defn}\label{defn_n_regular}
	Let $\pi$ be an automorphic representation of $G_2(\bA_{F^+})$, let $\chi 
	= (\chi_v)_{v \in S_p}$ be an accessible refinement of $\pi$, and let $n 
	\ge 2$. We say that 
	$\chi$ is $n$-regular if for each $v \in S_p$ the character $\chi_v = 
	\chi_{v,1}\otimes\chi_{v,2}$ satisfies $(\chi_{v,1}/\chi_{v,2})^i \ne 1$ 
	for $1 \le i \le n-1$.
\end{defn}
\begin{theorem}\label{thm_geometric_eigenvariety_argument}
	Let $(\pi_2, \chi_2) \in \mathcal{RA}_2$ satisfy $\tr \overline{r}_{\pi_2, 
	\iota} = \overline{\tau}_2$, and let $z_2 = \gamma_2(\pi_2, \chi_2) \in 
	\cE_2(\overline{\tau}_2)(\Qpbar)$. Suppose that:
	\begin{enumerate}
		\item The refinement $\chi_2$ is numerically non-critical and 
		$n$-regular. 
		\item There exists $(\pi_n, \chi_n) \in \mathcal{RA}_n$ such that 
		$(\sigma_n \circ i_2)(z_2) = i_n(z_n)$, where $z_n = \gamma_n(\pi_n, 
		\chi_n)$.
		\item\label{biglocalimage} For each $v \in S_p$, the Zariski closure of $r_{\pi_2, 
		\iota}({G_{F_\wv}})$ (in $\GL_2/\Qpbar$) contains $\SL_2$.
	\end{enumerate}
	 Then each irreducible component $\cC$ of $\cE_2(\overline{\tau}_2)_{\Cp}$ 
	 containing $z_2$ satisfies 
	 $(\sigma_n \circ i_2)(\cC) \subset 
	 i_n(\cE_n(\overline{\tau}_n)_{\Cp})$.
\end{theorem}
\begin{proof}
	Extending $E$ (the field over which $\cE_2$ is defined) if necessary, we 
	may assume that $z_2 \in 
	\cE_2(\overline{\tau}_2)(E)$ and $r_{\pi_2,\iota}$ takes values in 
	$\GL_2(E)$. By \cite[Theorem 3.4.2]{Con99} (which says that an irreducible 
	component of $\cE_2(\overline{\tau}_2)_{\Cp}$ is contained in the base 
	change of an irreducible component of $\cE_2(\overline{\tau}_2)$), it 
	suffices 
	to show that each irreducible component $\cC$ of $\cE_2(\overline{\tau}_2)$ 
	containing $z_2$ satisfies $(\sigma_n \circ i_2)(\cC) \subset 
	i_n(\cE_n(\overline{\tau}_n))$.
	By Lemma 
	\ref{lem_inclusion_of_irreducible_components}, it is enough to 
	show that $(\sigma_n \circ i_2)^{-1}(i_n(\cE_n(\overline{\tau}_n)))$ 
	contains an 
	affinoid open neighbourhood of $z_2$. To prove this, we will use a number 
	of the results established so far.

	Since $\chi_2$ is a numerically non-critical 
	refinement, the parameter $\delta_2$ of the associated triangulation is 
	non-critical, in the sense that for each $\tau \in \Hom(F, E)$, the 
	sequence of $\tau$-weights of $\delta_2$ is strictly increasing (Lemma 
	\ref{lem_numerically_non_critical_implies_non_critical}). Passing to 
	symmetric powers, we see that $r_{\pi_n, \iota}$ is trianguline of 
	parameter $\delta_n = \sigma_{n, p}(\delta_2)$, and that $\delta_n$ is 
	non-critical (although it is not necessarily numerically non-critical). 
	
	The $n$-regularity of $\chi_2$ 
	implies that $\delta_n \in \cT_n^{reg}(E)$.
	We are going to apply Proposition 
	\ref{prop_global_tangent_space_in_the_non-critical_case} 
	to conclude that $\dim_E H^1_{tri, \delta_n}(F_S / F^+, \ad 
	r_{\pi_n,\iota}) \leq \dim 
	\cW_n = \dim \cE_n(\overline{\tau}_n)$. Note that if $v \in S$ then
	$\mathrm{WD}(r_{\pi_n, \iota}|_{G_{F_\wv}})$ is generic, because it is 
	pure: the base change of $\pi_n$ to $\GL_n(\bA_F)$ exists (for example, by 
	\cite[Corollaire 5.3]{labesse}) and is cuspidal (because $r_{\pi_n, \iota}$ 
	is irreducible), so we can appeal to the main theorem of 
	\cite{Caraianilnotp} (which establishes the general case; under various additional hypotheses, purity was 
	established in \cite{ht,ty,shin,MR3010691}).
	If $v \in S_p$, let us write $f_v :\cX_{ps, n} \times \cT_n \to \cX_{ps, n, 
	v} \times \cT_v$ for the natural restriction map. By Proposition 
	\ref{prop_affinoid_local_trianguline_space} and Lemma 
	\ref{lem_character_spaces}, we can find for each $v \in S_p$ an affinoid 
	open neighbourhood $\cU_v \subset \cX_{ps, n, v} \times \cT_v$ of the point 
	$f_v(z_n)$ such that the following properties hold:
	\begin{itemize}
		\item In fact, $\cU_v \subset  \cX_{ps, n, v}^{v-irr} \times 
		\cT_v^{reg}$ and there exists a universal representation $\rho_v^u : 
		G_{F_\wv} \to \GL_n(\cO(\cU_v))$ over $\cU_v$.
		\item Let $\cZ_v \subset \cU_v$ denote the Zariski closure of the set 
		$\cV_v \subset \cU_v$ of points corresponding to pairs $(\rho_v, 
		\delta_v)$ such that $\rho_v$ is trianguline of parameter $\delta_v$. 
		Then the Zariski tangent space of $\cZ_v$ at $f_v(z_n)$ is contained in 
		$H^1_{tri, 
		\delta_v}(F_\wv, \ad r_{\pi_n,\iota}|_{G_{F_\wv}})$. 
	\end{itemize}
	We can then find an affinoid open neighbourhood $\cU \subset \cX_{ps, n} 
	\times \cT_n$ of the point $z_n$ such that the following properties hold:
	\begin{itemize}
		\item $\cU \subset \cap_{v \in S_p} f_v^{-1}(\cU_v)$ and there exists a 
		universal representation $\rho^u : G_{F, S} \to \GL_n(\cO(\cU))$ over 
		$\cU$.
		\item Let $\cZ = \cU \cap (\cap_{v \in S_p} f_v^{-1}(\cZ_v))$. Then $\cZ$ 
		is a closed analytic subset of $\cU$ and the Zariski tangent space of 
		$\cZ$ at the point $z_n$ is contained in $H^1_{tri, \delta}(F_S / F^+, 
		\ad r_{\pi_n,\iota})$. By the main theorem of \cite{newton2020adjoint}, we have 
	$H^1_f(F^+, 
	\ad r_{\pi_n, \iota}) = 0$, so Proposition \ref{prop_global_tangent_space_in_the_non-critical_case} implies that the Zariski tangent space of $\cZ$ at point $z_n$ has dimension at most $\dim \cW_n$.
	\end{itemize}
	Let $\cU' = \cE_n(\overline{\tau}_n) \cap \cU$. Then $\cU'$ is an affinoid 
	open neighbourhood of $z_n$ in $\cE_n(\overline{\tau}_n)$. We note that if 
	$z'_n = \gamma_n(\pi'_n, \chi'_n) \in \cU'$, where $\chi'_n$ is a numerically 
	non-critical refinement, then $z'_n \in \cZ$ (by Lemma 
	\ref{lem_numerically_non_critical_implies_non_critical}, and the definition 
	of $\cZ$). Such points accumulate at $z_n$, implying that every irreducible 
	component of $\cU'$ containing the point $z_n$ is contained in $\cZ$. In 
	particular, $\cZ$ contains an affinoid open neighbourhood of $z_n$ in 
	$\cU'$, so we have $\dim \cZ \geq \dim \cE_n(\overline{\tau}_n) = 
	\dim \cW_n$. It follows that $\dim \cZ = \dim \cW_n$, that $\widehat{\cO}_{\cZ, 
	z_n}$ is a regular local ring, and that $\cZ$ is smooth at the point 
	$z_n$. 	Consequently, $\cU'$ and $\cZ$ are locally isomorphic at $z_n$, 
	$\cE_n(\overline{\tau}_n)$ is smooth at the 
point $z_n$, and $\cE_n(\overline{\tau}_n)$ has a unique irreducible 
component passing through $z_n$. 
	Applying Lemma \ref{lem_inclusion_of_irreducible_components}, we can also 
	deduce that the unique irreducible component $\cZ'$ of $\cZ$ containing 
	$z_n$ is 
	contained in $\cU'$.

	Now let $\cU'' = (\sigma_n \circ i_2)^{-1}(\cU)$, and let $g = (\sigma_n 
	\circ 
	i_2)|_{\cU''} : \cU'' \to \cU$. Then $\cU''$ is an admissible open of 
	$\cE_2(\overline{\tau}_2)$, and $g^{-1}(\cZ) \subset \cU''$ is a non-empty 
	closed analytic subset. Let $(\pi'_2, \chi'_2) \in \mathcal{RA}_2$ be a 
	pair such that $\chi'_2$ 
	satisfies the analogue of property (1) in the theorem. Arguing again as in 
	the second paragraph of the proof, we see that if the point $(\sigma_n 
	\circ i_2)(\gamma_2(\pi'_2, \chi'_2))$ lies in $\cU$, then it in fact lies in 
	$\cZ$. Since such points accumulate at $z_2$, we see that $g^{-1}(\cZ')$ 
	contains each irreducible component of $\cU''$ which passes through $z_2$ 
	(and hence contains an affinoid open neighbourhood of $z_2$). Since $\cZ' 
	\subset \cU'$ we deduce that $(\sigma_n \circ 
	i_2)^{-1}(\cE_n(\overline{\tau}_n))$ 
	contains an affinoid open neighbourhood of $z_2$. This completes the proof.
\end{proof}
\begin{remark}
Note that assumption (\ref{biglocalimage}) on the image of the local Galois representation ensures that all symmetric powers remain locally irreducible. We need this to apply the results of \S\ref{subsec_trianguline_reps_global_geometry}. The authors expect that, with some effort, this material could be adjusted to allow locally reducible (but globally irreducible) families of Galois representations.
\end{remark}

We also prove a version of this result in the ordinary case. We first note a 
well-known consequence of Hida theory:
\begin{lemma}\label{lem:ord_locus_Hida}
The Zariski closure of the classical points with ordinary refinements 
$\cE_n(\overline{\tau}_n)^{ord} \subset \cE_n(\overline{\tau}_n)$ is a union of 
connected components of 
$\cE_n(\overline{\tau}_n)$ which are finite over $\cW_n$, and every classical 
point of $\cE_n(\overline{\tau}_n)^{ord}$ has an ordinary refinement. All 
points of $\cE_n(\overline{\tau}_n)^{ord}$ with dominant locally algebraic image in 
$\cW_n$ are classical. 
\end{lemma}
\begin{proof}
We can identify $\cE_n(\overline{\tau}_n)^{ord}$ with the generic fibre of the 
formal spectrum of Hida--Hecke algebra (a localization of the ring denoted by 
$\widetilde{\mathbb{T}}^{S,\ord}(U_n(\mathfrak{p}^\infty),\cO)$ in 
\cite[\S2]{ger}), since this is naturally a Zariski closed subspace of 
$\cX_{ps,n} \times \cT_n$ in which the classical points with ordinary 
refinements are Zariski dense. We deduce from Hida theory that 
$\cE_n(\overline{\tau}_n)^{ord}$ is finite over $\cW_n$ and equidimensional of 
dimension $\dim\cW_n$. Moreover, the map $\nu:\cE_n(\overline{\tau}_n)^{ord} 
\to \cT_n$ 
factors through the open subspace $\cT_n^\circ \subset \cT_n$ classifying 
unitary characters of $\prod_{v \in S_p}T_n(F_\wv)$. 

On the other hand, we claim that every point of  
$\cE_n(\overline{\tau}_n)\times_{\nu,\cT_n}\cT_n^\circ$ is 
contained in $\cE_n(\overline{\tau}_n)^{ord}$. Assuming this, these (reduced) 
subspaces of $\cE_n(\overline{\tau}_n)$ are equal and $\cE_n(\overline{\tau}_n)^{ord}$ is an open and closed 
subspace of 
$\cE_n(\overline{\tau}_n)$. The final part of the lemma follows from the classicality 
theorem in Hida theory \cite[Lemma 2.25]{ger}.

It remains to show the claimed inclusion of $\cE_n(\overline{\tau}_n)\times_{\nu,\cT_n}\cT_n^\circ$ in $\cE_n(\overline{\tau}_n)^{ord}$. Suppose $z$ is an $E$-point of $\cE_n(\overline{\tau}_n)\times_{\nu,\cT_n}\cT_n^\circ$ (extending scalars deals with the general case). The character $\nu(z)\delta_{B_n}$ then appears in the eigenspace $\left(J_{B_n}\widetilde{S}(U_n^p,E)^{an}\right)[\psi^*(z)]$. This character therefore also appears in $J_{B_n}\left(\widetilde{S}(U_n^p,E)^{an}[\psi^*(z)]\right)$, by Lemma \ref{lem:eigenspaces}. Applying \cite[Corollary 
6.4]{sorensen-ordjac} to $\widetilde{S}(U_n^p,\cO)[\psi^*(z)]$  (note that Sorensen's Jacquet modules are twisted by 
$\delta_{B_n}^{-1}$ compared to ours), we deduce that the unitary character $\nu(z)$ appears in the ordinary part $\mathrm{Ord}_{B_n}\widetilde{S}(U_n^p,\cO)[\psi^*(z)]$. This shows that $z$ is a point of $\cE_n(\overline{\tau}_n)^{ord}$.
\end{proof}

\begin{theorem}\label{thm_geometric_argument_ord}
	Let $(\pi_2, \chi_2) \in \mathcal{RA}_2$ satisfy $\tr \overline{r}_{\pi_2, 
		\iota} = \overline{\tau}_2$, and let $z_2 = \gamma_2(\pi_2, \chi_2) \in 
	\cE_2(\overline{\tau}_2)(\Qpbar)$. Suppose that:
	\begin{enumerate}
		\item The refinement $\chi_2$ is ordinary. 
		\item There exists $(\pi_n, \chi_n) \in \mathcal{RA}_n$ such that 
		$(\sigma_n \circ i_2)(z_2) = i_n(z_n)$, where $z_n = \gamma_n(\pi_n, 
		\chi_n)$.
		\item The Zariski closure of $r_{\pi_2, 
			\iota}({G_{F}})$ contains $\SL_2$.
	\end{enumerate}
	Then each irreducible component $\cC$ of $\cE_2(\overline{\tau}_2)_{\Cp}$ 
	containing $z_2$ satisfies 
	$(\sigma_n \circ i_2)(\cC) \subset 
	i_n(\cE_n(\overline{\tau}_n)_{\Cp})$.
\end{theorem}
\begin{proof}
	Extending $E$ if necessary, we 
	may assume that $z_2 \in 
	\cE_2(\overline{\tau}_2)(E)$ and $r_{\pi_2,\iota}$ takes values in 
	$\GL_2(E)$. We denote 
	by $\cT_n^{HT-reg} \subset \cT_n$ the Zariski open subset where for each $v 
	\in S_p$ and $\tau \in \Hom_{\Qp}(F_\wv,E)$ 
	the labelled weights $\wt_{\tau}(\delta_{n,v,i})$ are distinct for $i= 1, 
	\dots,n$. 
	
	By Lemma 
	\ref{lem_deforming_conjugate_self_dual_representations} and (a global 
	variant of) Lemma \ref{lem_locally_split}, there is an open affinoid 
	neighbourhood \[z_n = (\tr r_{\pi_n,\iota},\delta_n) \in \cU \subset 
	\cX^{irr}_{ps,n} \times \cT_n^{reg}\] and a 
	universal representation $\rho^u: G_{F,S} \to \GL_n(\cO(\cU))$ such that 
	the 
	induced representation $(\rho^u)\hat{}_{z_n}: G_F \to 
	\GL_n(\cO(\cU)\hat{}_{z_n})$ with coefficients in the completed local ring 
	at $z_n$ extends to a homomorphism 
	$(\rho^u)\hat{}_{z_n}: G_{F^+,S} \to 
	\G_n(\cO(\cU)\hat{}_{z_n})$ with $\nu_{\cG_n} \circ (\rho^u)\hat{}_{z_n} = 
	\epsilon^{1-n} \delta_{F /F^+}^n$.
	
	Since $\chi_2$ is ordinary, the parameter $\delta_n = 
	\sigma_{n,p}(\delta_2)$ is ordinary. We denote by 
	$\mathcal{FL}_{p}(\cO_\cU^n) \xrightarrow{\alpha} \cU$ the rigid space 
	(equipped with a proper map to $\cU$) classifying $S_p$-tuples $(\cF_v)_{v 
	\in S_p}$ of full flags in $\cO_\cU^n$. We consider the closed subspace 
	\[\cZ^{ord} \subset \mathcal{FL}_{p}(\cO_\cU^n)\] whose points $z$ 
	correspond 
	to flags $\cF_v$ which are $G_{F_\wv}$-stable (under the $\rho^u$-action) 
	for each $v \in S_p$ 
	and the action of $G_{F_\wv}$ on $\gr^i(\cF_v)$ is given by 
	$\delta_{z,v,i}\circ \Art_{F_{\wv}}^{-1}$ where $\delta_{z}$ is the 
	parameter of $\alpha(z)$. Since our parameters lie in $\cT_n^{HT-reg}$, 
	$\cZ^{ord} \to \cU$ is a closed immersion (it is a proper monomorphism). 

Using the existence of $(\rho^u)\hat{}_{z_n}$, we can view the tangent space 
$T_{z_n}\cZ^{ord}$ as a subspace of $H^1(G_{F^+,S},\ad r_{\pi_n,\iota})$. By a 
similar argument to Proposition 
\ref{prop_global_tangent_space_in_the_non-critical_case}, it follows from e.g.\ \cite[Lemma 3.9]{ger} (which gives the analogue of Lemma 
\ref{lem_weight_map_etale_for_non_critical_refinement} in the ordinary case) 
 and the main theorem of \cite{newton2020adjoint} that the map 
$T_{z_n}\cZ^{ord} \to 
T_{r(\delta_n)}\cW_n$ is injective. On the other hand, 
$\cE_n(\overline{\tau}_n)^{ord}\cap 
\cU$, a subspace of $\cZ^{ord}$ containing $z_n$, is equidimensional of 
dimension $\dim\cW_n$. We deduce that $\cZ^{ord}$ is smooth at $z_n$, and that 
$\cE_n(\overline{\tau}_n)$ is locally isomorphic to $\cZ^{ord}$ at $z_n$. We 
complete the proof in the same 
way as Theorem \ref{thm_geometric_eigenvariety_argument}. \end{proof}
We restate Theorem \ref{thm_geometric_eigenvariety_argument} and Theorem \ref{thm_geometric_argument_ord} in a way that does not make explicit reference to $\cE_n$.
\begin{cor}\label{cor_main_application_for_U(2)}
	Let $(\pi_2, \chi_2), (\pi'_2, \chi'_2) \in \mathcal{RA}_2$, and let $z_2, 
	z_2' \in \cE_2(\Qpbar)$ be the corresponding points of the eigenvariety. 
	Suppose 
	that one of the following two sets of conditions are satisfied:
	\begin{enumerate}
		\item The refinement $\chi_2$ is numerically non-critical and 
		$n$-regular.
		\item For each $v \in S_p$, every triangulation of $r_{\pi_2', 
		\iota}|_{G_{F_\wv}}$ is non-critical. The refinement $\chi_2'$ is 
		$n$-regular.
		\item For each $v \in S_p$, the Zariski closures of the images of 
		$r_{\pi_2, \iota}|_{G_{F_\wv}}$ and $r_{\pi'_2, \iota}|_{G_{F_\wv}}$ 
		contain $\SL_2$.
		\item There exists an automorphic representation $\pi_n$ of 
		$G_n(\bA_{F^+})$ such that 
		\[ \Sym^{n-1} r_{\pi_2, \iota} \cong r_{\pi_n, \iota}. \]
		\item The points $z_2, z_2'$ lie on a common irreducible component of 
		$\cE_{2,\Cp}$;\footnote{By \cite[Theorem 3.4.2]{Con99}, this assumption 
		is equivalent to requiring that there is a finite extension of 
		coefficient fields $E'/E$ such that $z_2, z_2'$ lie on a common 
		geometrically irreducible component of $\cE_{2,E'}$.}
	\end{enumerate}	or
	\begin{enumerate}
	\item[(1\textsuperscript{ord})] The refinement $\chi_2$ is ordinary.
	\item[(2\textsuperscript{ord})] The Zariski closure of the image of 
	$r_{\pi_2, \iota}|_{G_{F}}$ contains $\SL_2$.
	\item[(3\textsuperscript{ord})] There exists an automorphic representation 
	$\pi_n$ of 
	$G_n(\bA_{F^+})$ such that 
	\[ \Sym^{n-1} r_{\pi_2, \iota} \cong r_{\pi_n, \iota}. \]
	\item[(4\textsuperscript{ord})] The points $z_2, z_2'$ lie on a common 
	irreducible component of 
	$\cE_{2,\Cp}$ (this implies that the refinement $\chi_2'$ is also ordinary, 
	by Lemma \ref{lem:ord_locus_Hida}).
\end{enumerate}	
	Then there exists an automorphic representation $\pi'_n$ of 
	$G_n(\bA_{F^+})$ such that 
	\[ \Sym^{n-1} r_{\pi'_2, \iota} \cong r_{\pi'_n, \iota}. \]
\end{cor}
\begin{proof}
Choose $U_n \subset G_n(\bA_{F^+}^\infty)$ so that $(\pi_n^\infty)^{U_n^{S_p}} \neq 
0$ and take $\overline{\tau}_2 = \tr \overline{r}_{\pi_2, \iota}$. Then 
$(\sigma_n 
\circ i_2)(z_2) \in i_n(\cE_n(\overline{\tau}_n)(\Qpbar))$. We claim that setting $\chi_{n,v} = \chi_{2,v,1}^{n-1}\otimes \chi_{2,v,1}^{n-2}\chi_{2,v,2}\otimes \cdots \otimes \chi_{2,v,2}^{n-1}$  for $v \in S_p$ defines an accessible refinement $\chi_n$ of $\pi_n$. Fix $v \in S_p$. To temporarily simplify notation, we write $\chi = \chi_1\otimes \chi_2$ for $\chi_{2,v}$. The representation $\pi_{2,v}$ is isomorphic to either $\St_2(\iota\chi_{1}|\cdot|^{-1/2})$ or to an irreducible parabolic induction $i_{B_2}^{\GL_2}\iota\chi$. In the first case, \[\rec_{F^+_v}^T(\iota^{-1}\pi_{n,v}) \cong \Sym^{n-1}\rec_{F^+_v}^T(\iota^{-1}\pi_{2,v}) \cong \Sp_n(\chi_{1}^{n-1}|\cdot|^{(1-n)/2})\] and $\chi_{n,v}$ is the unique accessible refinement of $\pi_{n,v}$. In the second case, 
\[\rec_{F^+_v}^T(\iota^{-1}\pi_{n,v}) \cong \Sym^{n-1}\rec_{F^+_v}^T(\iota^{-1}\pi_{2,v}) \cong \bigoplus_{i=0}^{n-1} \chi_{1}^{n-1-i}\chi_2^i|\cdot|^{(1-n)/2}\circ\Art_{F^+_v}^{-1}.\] Note that $\pi_{n,v}$ is generic (the base change of $\pi_n$ to $\GL_n(\bA_F)$ is cuspidal, since $r_{\pi_n,\iota}$ is irreducible). We can now use the characterisation of generic representations in the Bernstein--Zelevinsky classification \cite[Theorem 9.7]{Zel80} and the compatibility with local Langlands \cite[\S4.4]{rodier}. It follows that no pair of characters in the above direct sum decomposition have ratio equal to the norm character, so the parabolic induction $i_{B_n}^{\GL_n} \iota\chi_{n,v}$ is irreducible and isomorphic to $\pi_{n,v}$. 
 In particular, $\chi_{n,v}$ is an accessible refinement of $\pi_{n,v}$. 

Taking the above discussion into account, it is straightforward to see that $(\sigma_n\circ i_2)(z_2)$ is associated 
to the pair $(\pi_n, \chi_n) \in \mathcal{RA}_n$. Thus the hypotheses of Theorem 
\ref{thm_geometric_eigenvariety_argument} or \ref{thm_geometric_argument_ord} 
are satisfied, and for any $(\pi'_2, 
\chi'_2)$ as in the statement of the corollary there exists a point $z'_n \in 
\cE_n(\overline{\tau}_n)(\Qpbar)$ such that $T_{n, z'} = 
\tr \Sym^{n-1} r_{\pi_2', \iota}$. It remains to show that $z'_n \in Z_n$, or 
in other words that $z'_n$ is associated to a classical automorphic 
representation. 

In the ordinary case, this follows from Lemma \ref{lem:ord_locus_Hida}. In the 
remaining case, Lemma 
\ref{lem_symmetric_power_preserves_genericity} shows that 
for each $v \in S_p$, every triangulation of $\Sym^{n-1} r_{\pi'_2, 
\iota}|_{G_{F_\wv}}$ is non-critical. It follows from Lemma 
\ref{lem_generic_implies_classical} that $z_n' \in Z_n$.
\end{proof}
\begin{lemma}\label{lem_symmetric_power_preserves_genericity}
Let $v \in S_p$, and let $\rho_v : G_{F_\wv} \to \GL_2(\overline{\bQ}_p)$ be a 
continuous, regular de Rham representation such that $\WD(\rho_v)$ has two 
distinct characters $\chi_1, \chi_2$ as Jordan--H\"older factors, which satisfy 
$(\chi_1 / \chi_2)^i \neq 1$ for each $i = 1, \dots, n-1$. Suppose moreover 
that every triangulation of $\rho_v$ is non-critical. Then every 
triangulation of $\Sym^{n-1} \rho_v$ is non-critical.
\end{lemma}
\begin{proof}
We begin by describing the data of the triangulation of $\rho_v$ in a bit more 
detail. Let $K = F_\wv$ and let $L / K$ be a Galois extension over which 
$\rho_v$ becomes semi-stable. Let $L_0$ be the  maximal unramified extension of $L / \bQ_p$.  After enlarging $E$, we can assume that every 
embedding of $L$ in $\overline{\bQ}_p$ lands in $E$, and that $\rho_v$ is 
defined over $E$. The filtered $(\varphi,N,\Gal(L/K))$-module $D$ associated to $\rho_v$ consists of 
the following data:
\begin{enumerate}
\item A free $L_0 \otimes_{\bQ_p} E$-module $D$ of rank 2, equipped with a 
$\sigma \otimes 1$-semilinear endomorphism $\varphi$. 
\item An $L_0 \otimes_{\bQ_p} E$-linear endomorphism $N$ of $D$ satisfying the 
relation $N \varphi = p \varphi N$. 
\item An $L_0$-semilinear, $E$-linear action of the group $\Gal(L / K)$ on $D$ 
that commutes with the action of both $\varphi$ and $N$. 
\item A decreasing, $\Gal(L/K)$-stable, filtration $\Fil_\bullet D_L$ of $D_L = D \otimes_{L_0} L$.
\end{enumerate}
For each embedding $\tau : L \to E$, we write $l_\tau \subset D_\tau = D_L 
\otimes_{L\otimes E, \tau} E$ for the image of the rank 1 step of the filtration 
$\Fil_\bullet$. We can define an action of the group $W_K$ on $D$ by the 
formula $g \cdot v = (g \text{ mod }W_L) \circ \varphi^{-\alpha(g)}$, where 
$\alpha(g)$ is the power of the absolute arithmetic Frobenius induced by $g$ on 
the residue field of $\overline{K}$.

This action preserves the factors of the product decomposition $D = \prod_t 
D_t$, where $t$ ranges over embeddings $t : L_0 \to E$ and $D_t = D 
\otimes_{L_0\otimes E, t} E$. Moreover, the isomorphism class of the Weil--Deligne 
representation $D_t$ is independent of $t$. The data of a triangulation of 
$\rho_v$ is equivalent to the data of a choice of character appearing in some 
(hence every) $D_t$. If $N$ is non-zero on $D_t$, then there is a unique $N$-stable line in $\Sym^{n-1}D_t$. Hence there is a unique triangulation of $\Sym^{n-1}\rho_v$, induced by the unique (non-critical) triangulation of $\rho_v$, and it is also non-critical. From now on we assume that $N = 0$, and we proceed as indicated in \cite[Example 3.26]{Che11}.

We can choose a basis $e_1, e_2$ for $D$ as $L_0 
\otimes_{\bQ_p} E$-module such that the projection of the vectors $e_1, e_2$ to 
each $D_t$ is a basis of eigenvectors for the group $W_K$. 

Having made this choice of basis, each line $l_\tau$ is spanned by a linear 
combination of $e_1, e_2$. Our assumption that every triangulation of $\rho_v$ 
is non-critical is equivalent to the requirement that $l_\tau$ may be spanned 
by a vector $e_1 + a_\tau e_2$, where $a_\tau \in E^\times$ for all $\tau$. Indeed, if $l_\tau$ is spanned by $e_i$ for some $i$, then the triangulation corresponding to the submodule of $D$ spanned by $e_i$ will fail the condition required for non-criticality with respect to the embedding $\tau$.

Having made these normalisations, the condition that every triangulation of 
$\Sym^{n-1} \rho_v$ be non-critical is equivalent to the following statement: 
let $I \subset \{ 0, \dots, n-1 \}$ be a subset, and let $\sum_{i \in I} a_i 
x^i \in E[x]$ be a polynomial, which is equal to $(1 + a_\tau x)^{|I|} Q(x)$ 
for 
some polynomial $Q(x) \in E[x]$ of degree at most $n - 1 - |I|$, then $Q(x) = 0$. Polynomials of the latter form correspond to elements of the $|I|$th step of the Hodge filtration on $\Sym^{n-1}D_\tau$ and the statement implies that this Hodge filtration is in general position compared to the filtration induced by every triangulation. Replacing the variable $x$ with $-a_\tau x$, we can assume that $a_\tau = -1$. As in \cite[Example 3.26]{Che11}, the vanishing of the $|I|$ successive derivatives at $1$ of $\sum_{i \in I} a_i 
x^i$ gives a non-degenerate linear system of $|I|$ equations satisfied by the $a_i$, and therefore the $a_i$ are all zero. Non-degeneracy is checked by noticing that the determinant of the linear system is the Vandermonde determinant $\prod_{i < j \in I}(i-j)$.
\end{proof}
\begin{lemma}\label{lem_generic_implies_classical}
	Let $z \in \cE_n(\overline{\tau}_n)(\Qpbar)$ be a point with $i_n(z) = (\tr 
	r_z, 
	\delta) \in \cX_{ps, n}^{p-irr}(\Qpbar) \times \cT_n^{reg}(\Qpbar)$. 
	Suppose that 
	$\delta = \jmath_n(\nu(z)) = \delta_{alg}\delta_{sm}$ with $\delta_{alg}$ 
	algebraic and $\delta_{sm}$ smooth. Suppose moreover 
	that, 
	for each $v \in S_p$, every triangulation of $r_z|_{G_{F_\wv}}$ is 
	non-critical. 
	Then $z \in Z_n$ (in particular, $\delta_{alg}$ is strictly dominant).
\end{lemma}
\begin{proof}
	After extending $E$, we may assume that $z \in \cE_n(\overline{\tau}_n)(E)$ 
	and $r_z$ takes values in $\GL_n(E)$. We note that, since $\delta$ is 
	locally algebraic, it follows from 
	property (\ref{classicality}) of the eigenvariety that the subset of numerically 
	non-critical classical points in $i_n^{-1}(\cX_{ps, n}^{p-irr} \times 
	\cT_n^{reg})$ accumulates at $z$. It follows from \cite[Corollary 
	6.3.10]{Ked14}, applied as in Proposition 
	\ref{prop_affinoid_local_trianguline_space}, that there is a connected 
	affinoid neighbourhood $\cU$ of $z$ in $i_n^{-1}(\cX_{ps, n}^{p-irr} \times 
	\cT_n^{reg})$, over which there exist representations $\rho_v^u: 
	G_{F_\wv} \to \GL_n(\cO(\cU))$ for each $v \in S_p$ with trace equal to the restriction to $G_{F_\wv}$ of the universal pseudocharacter and a non-empty Zariski open and dense subspace $\cV \subset \cU$ 
	such that for every $z' \in \cV$ with $i_n(z') = (\tr r',\delta')$,  
	$r'$ is trianguline of parameter $\delta'$.  Now 
	we can 
	apply \cite[Lemme 
	2.11]{Bre17}\footnote{We caution the reader that the version of this paper currently available on the arXiv contains a less general result than the published version, to which we appeal here. In particular, it restricts to  Galois representations which are known in advance to be crystalline.}
	to deduce that 
	$\delta_{alg}$ is strictly dominant and $r_z$ is trianguline 
	of parameter $\delta$.  
	
	We now argue as in \cite[Prop.~3.28]{Bre17} (which is itself similar to 
	the argument of \cite[Prop.~4.2]{Che11}). The idea of the argument is to show that failure of classicality would entail the existence of a `companion point' to $z$, with the same associated Galois representation and a locally algebraic weight which is \emph{not} strictly dominant. This would contradict \cite[Lemma 2.11]{Bre17}. 
	
	Let $\eta = 
	\nu(z)\delta_{B_n} = \eta_{alg}\eta_{sm}$, with $\eta_{alg}$ 
	dominant algebraic (since $\delta_{alg}$ is strictly dominant) and 
	$\eta_{sm}$ smooth. By the construction of 
	$\cE_n$ and Lemma \ref{lem:eigenspaces} we have a non-zero space of morphisms
	
	\[0 \ne \Hom_{\prod_{v \in S_p}T_n(F_\wv)}\left( \eta, 
	J_{B_n}\left(\widetilde{S}(U_n^p,E)^{an}[\psi^*(z)]\right)\right).\] Now 
	we use some of the work of Orlik--Strauch \cite{orlik-strauch}, with 
	notation as in 
	\cite[\S2]{Bre15}. We denote by $\gog_n$ the $\Qp$-Lie algebra of $\prod_{v 
	\in S_p}G_n(F^+_v) \cong \prod_{v \in S_p} \GL_n(F_\wv)$ and denote by $\overline{\mathfrak{b}}_n \subset 
	\gog_n$ the lower 
	triangular Borel. We define a locally analytic representation of $\prod_{v 
	\in S_p}G_n(F_v^+)$ (see \cite[Thm.~2.2]{Bre15} for the definition of 
	the functor $\cF_{\overline{B}_n}^{G_n}$):
	\[\cF_{\overline{B}_n}^{G_n}(\eta\delta_{B_n}^{-1}) := 
	\cF_{\overline{B}_n}^{G_n} 
	\left(\left(U(\gog_{n,E})\otimes_{U(\overline{\mathfrak{b}}_{n,E})}\eta_{alg}^{-1}
	\right)^\vee,\eta_{sm}\delta_{B_n}^{-1}\right).\] Note that 
	$\left(U(\gog_{n,E})\otimes_{U(\overline{\mathfrak{b}}_{n,E})}\eta_{alg}^{-1}
	\right)^\vee$
	 has a unique simple submodule (isomorphic to the unique simple quotient of 
	 $U(\gog_{n,E})\otimes_{U(\overline{\mathfrak{b}}_{n,E})}\eta_{alg}^{-1}$),
	  the algebraic representation 
	$V(\eta_{alg})^\vee$ with lowest (with respect to $B_n$) weight 
	$\eta_{alg}^{-1}$. It 
	follows from 
	\cite[Thm.~2.2]{Bre15} that  
		$\cF_{\overline{B}_n}^{G_n}(\eta\delta_{B_n}^{-1})$ has a locally 
		algebraic quotient 
		isomorphic to $V(\eta_{alg})\otimes_E 
		\Ind_{\overline{B}_n}^{G_n}\eta_{sm}\delta_{B_n}^{-1}$. 
	
	By 
	\cite[Thm.~4.3]{Bre15},
there is a non-zero 
	space of morphisms \[0 \ne 
	\Hom_{\prod_{v \in 
	S_p}G_n(F^+_v)}\left(\cF_{\overline{B}_n}^{G_n}(\eta\delta_{B_n}^{-1}),
\widetilde{S}(U_n^p,E)^{an}[\psi^*(z)]\right).\] 

The Jordan--H\"{o}lder factors 
of $\cF_{\overline{B}_n}^{G_n}(\eta\delta_{B_n}^{-1})$ can be described 
using 
\cite[Thm.~2.2]{Bre15} and standard results on the Jordan--H\"{o}lder factors 
of Verma modules (see \cite[Cor.~4.6]{Bre15}). Suppose $\lambda \in \cT_n(E)$ 
is an 
algebraic character. Denote by $M_\lambda$ the unique simple submodule of the 
dual Verma module 
$\left(U(\gog_{n,E})\otimes_{U(\overline{\mathfrak{b}}_{n,E})}\lambda^{-1}\right)^\vee$.
 Then the Jordan--H\"{o}lder factors 
of $\cF_{\overline{B}_n}^{G_n}(\eta\delta_{B_n}^{-1})$ are all of the form 
\[JH(w,\pi) = \cF_{\overline{P}_n}^{G_n} 
\left(M_{w\cdot\eta_{alg}},\pi\right)\] 
with $\overline{P}_n$ a 
parabolic subgroup of $\prod_{v \in S_p}G_n(F^+_v)$ containing 
$\overline{B}_n$, $\pi$ a Jordan--H\"{o}lder factor of the 
parabolic induction of $\eta_{sm}\delta_{B_n}^{-1}$ from $\overline{B}_n$ 
to the Levi of 
$\overline{P}_n$, and $w$ an element of the Weyl group of 
$(\Res_{F^+/\Q}G_n)\times_\Q E$, acting by the `dot action' on 
$\eta_{alg}$. Here $\overline{P}_n$ is maximal for 
$M_{w\cdot\eta_{alg}}$, in the sense of \cite[\S2]{Bre15}. 

We claim that there cannot be a non-zero morphism $JH(w,\pi) \to 
\widetilde{S}(U_n^p,E)^{an}[\psi^*(z)]$ for $w \ne 1$. Suppose, for a 
contradiction, that there is such a map. It follows from 
\cite[Cor.~3.4]{BreuilSocleI} that we have $\psi^*(z) = \psi^*(z')$ (and hence 
an isomorphism of Galois representations $r_z \cong r_{z'}$) for a point 
$z' \in \cE_n(\overline{\tau}_n)(E)$ with $\jmath_n(\nu(z'))$ locally algebraic 
but \emph{not} strictly dominant (its algebraic part matches the algebraic part 
of $\jmath_n(w\cdot\eta_{alg})$). The argument in the first paragraph of 
this proof, using \cite[Lemma 2.11]{Bre17}, then gives a contradiction. 

We deduce from this that any map 
\[\cF_{\overline{B}_n}^{G_n}(\eta\delta_{B_n}^{-1}) \to 
\widetilde{S}(U_n^p,E)^{an}[\psi^*(z)]\] factors through the locally 
algebraic quotient $V(\eta_{alg})\otimes_E 
\Ind_{\overline{B}_n}^{G_n}\eta_{sm}\delta_{B_n}^{-1}$. Applying 
\cite[Thm.~4.3]{Bre15} 
again, we deduce that we have equalities 
\[ \begin{split}\Hom_{\prod_{v \in S_p}T_n(F_\wv)}&\left( \eta, 
J_{B_n}\left(\widetilde{S}(U_n^p,E)^{an}[\psi^*(z)]\right)\right) \\&  = 
\Hom_{\prod_{v \in 
		S_p}G_n(F^+_v)}\left(\cF_{\overline{B}_n}^{G_n}(\eta\delta_{B_n}^{-1}),
\widetilde{S}(U_n^p,E)^{an}[\psi^*(z)]\right) \\ & = \Hom_{\prod_{v \in 
		S_p}G_n(F^+_v)}\left(\cF_{\overline{B}_n}^{G_n}(\eta\delta_{B_n}^{-1}),
\widetilde{S}(U_n^p,E)^{alg}[\psi^*(z)]\right) \\ & = \Hom_{\prod_{v \in 
S_p}T_n(F_\wv)}\left( \eta, 
J_{B_n}\left(\widetilde{S}(U_n^p,E)^{alg}[\psi^*(z)]\right)\right).
\end{split} \]

In particular, our point $z$ arises from a non-zero map 
\[ \eta \to 
J_{B_n}\left(\widetilde{S}(U_n^p,E)^{alg}[\psi^*(z)]\right). \]
Applying 
\cite[Prop.~4.3.6]{MR2292633} and computing locally algebraic vectors as in 
\S \ref{sssec:emertonconstruction} we see that such a map corresponds to a 
non-zero 
map of smooth representations $\eta_{sm} \to J_{B_n}\left(\varinjlim_{U_p} 
S_{\eta_{alg}^\vee}(U_n^pU_p, 
E)[\psi^*(z)]\right)$ and hence a pair 
$(\pi_n,\iota\circ\eta_{sm}\delta_{B_n}^{-1/2}) \in \mathcal{RA}_n$ with 
corresponding classical point equal to $z$. We therefore have $z \in Z_n$.
\end{proof}

\subsection{Application to the eigencurve}\label{subsec_eigencurve}

Thus far in this section we have found it convenient to phrase our arguments in 
terms of automorphic forms on unitary groups. Since our intended application 
will rely on particular properties of the Coleman--Mazur eigencurve for $\GL_2$, we now show 
how to deduce what we need for the eigencurve from what we have done so far.

We first introduce the version of the eigencurve that we use. Fix an integer $N 
\geq 1$, prime to $p$. Let $\cT_0 = \Hom(\bbQ_p^\times / \bbZ_p^\times \times 
\bbQ_p^\times  , \bG_m)$; it is the $E$-rigid space parameterising characters 
$\chi_0 = \chi_{0, 1} \otimes \chi_{0, 2}$ of $(\bbQ_p^\times)^2$ such that 
$\chi_{0, 1}$ is unramified. Let 
$\cW_0 = \Hom(\bbZ_p^\times, \bG_m)$, and write $r_0 : \cT_0 \to \cW_0$ for the 
morphism given by $r_0(\chi_{0, 1} \otimes \chi_{0, 2}) = 
\chi_{0,1}/\chi_{0,2}|_{\Zpx} = \chi_{0,2}^{-1}|_{\Zpx}$.  We denote the map $r_0\circ\nu_0: \cE_0 
\to \cW_0$ by $\kappa$. Let $\bT^{pN}_0 =  
\cO[ \{ 
T_l, S_l \}_{l 
\nmid pN} ]$ denote the polynomial ring in unramified Hecke operators at primes 
not dividing $Np$. Here $T_l$ and $S_l$ are the double coset operators for the matrices $\left(\begin{smallmatrix}
l & 0\\ 0 &1
\end{smallmatrix}\right)$ and $\left(\begin{smallmatrix}
l & 0\\ 0 & l
\end{smallmatrix}\right)$. 
Let $U_1(N) =  \prod_l U_1(N)_l \subset \GL_2(\widehat{\bZ}) = \prod_l \GL_2(\bZ_l)$ be defined by
\[ U_1(N)_l = \left\{ \left( \begin{array}{cc} a & b \\ c & d \end{array}\right) \in \GL_2(\bZ_l) :  c, d - 1 \in N \bZ_l \right\}. \]
The eigencurve is a tuple $(\cE_0, \psi_0, \nu_0, Z_0)$, where:
\begin{enumerate}
	\item $\cE_0$ is a reduced $E$-rigid space, equipped with a finite morphism 
	$\nu_0 : \cE_0 \to \cT_0$.
	\item $\psi_0 : \bT^{pN}_0 \to 
	\cO(\cE_0)$ is a ring homomorphism, which takes values in the subring 
	$\cO(\cE_0)^{\leq 1}$ of bounded elements.
	\item $Z_0 \subset \cE_0(\overline{\bQ}_p)$ is a Zariski dense subset which 
	accumulates at itself.
\end{enumerate}
The following properties are satisfied:
\begin{enumerate}
	\item $\cE_0$ is equidimensional of dimension $\dim \cW_0 = 1$. For any irreducible component $\cC \subset \cE_0$, $\kappa(\cC)$ is a Zariski open subset of $\cW_0$.
	\item Let $\mathcal{A}_0$ denote the set of cuspidal automorphic 
	representations $\pi_0$ of $\GL_2(\bA_\bbQ)$ such that 
	$(\pi_0^\infty)^{U_1(N)^p} \neq 0$ and $\pi_{0, \infty}$ has the same 
	infinitesimal character as $(\Sym^{k-2} \bC^2)^\vee$ for some $k \geq 2$ (in which case we say $\pi_{0}$ has weight $k$), and let $\mathcal{RA}_0$ denote the set of pairs $(\pi_0, \chi_0)$, where $\pi_0 \in \mathcal{A}_0$ and $\chi_0 = \chi_{0, 1} \otimes \chi_{0, 2}$ is an accessible refinement of $\pi_{0, p}$ such that $\chi_{0, 1}$ is \emph{unramified}. As in the unitary case we considered above, for $(\pi_0,\chi_0) \in \mathcal{RA}_0$ we have a homomorphism $\psi_{\pi_0}: \bT^{pN}_0 \to \Qpbarx$ determined by the action of the Hecke operators on $\iota^{-1}(\pi_0^\infty)^{U_1(N)^p}$. There is also a character $\nu_0(\pi_0,\chi_0) \in \cT_0(\Qpbar)$ defined in exactly the same way as in the unitary case (\ref{eqn:nuforRAn}). An explicit formula appears below (\ref{eqn:nuforRA0}). Our assumption that $\chi_{0,1}$ is unramified implies that this character does indeed give a point of $\cT_0$. Now we can let $\mathcal{Z}_0 \subset \Hom(\bT^{pN}_0, \overline{\bQ}_p) \times 
	\cT(\overline{\bQ}_p)$ denote the set of points of the form $(\psi_{\pi_0}, 
	\nu_0(\pi_0, \chi_0))$, where $(\pi_0, \chi_0) \in \mathcal{RA}_0$. Then 
	the 
	map $\psi_0^\ast \times \nu_0 : 
	\cE_0(\overline{\bQ}_p) \to \Hom(\bT^{pN}_0, \overline{\bQ}_p) \times 
	\cT_0(\overline{\bQ}_p)$ restricts to a bijection $Z_0 \to \cZ_0$.
	\item For any affinoid open $\cV_0 \subset \cT_0$, the map $\bT^{pN}_0 
	\otimes \cO(\cV_0) \to \cO(\nu_0^{-1} \cV_0)$ is surjective.
\end{enumerate}
The uniqueness of the tuple $(\cE_0, \psi_0, \nu_0, Z_0)$ follows from 
\cite[Proposition 7.2.8]{bellaiche_chenevier_pseudobook}. Its existence can be 
proved in various ways. A construction using overconvergent modular forms is 
given in \cite{Buz07}. We note that in this case, in contrast to the unitary 
group case, the map $\mathcal{RA}_0 \to \mathcal{Z}_0$ is bijective -- a 
consequence of the strong multiplicity one theorem. We will therefore feel free 
to speak of the cuspidal automorphic representation $\pi_0 \in \mathcal{A}_0$ 
associated to a point lying in $Z_0$. As in the unitary group case, there is a Galois pseudocharacter $t: G_{\Q,Np} \to \cO(\cE_0)$ with the property that for $z \in Z_0$ associated to $(\pi_0,\chi_0)$, $t_z = \tr r_{\pi_0,\iota}$. 

Let us describe explicitly the link with more classical language. We are using the normalisations of \cite[\S 11]{Diamond-Im}. If $(\pi_0, 
\chi_0) \in \mathcal{A}_0$, then there is a cuspidal holomorphic modular form 
$f = q+ \sum_{n \geq 2} a_n(f) q^n$ of level $\Gamma_1(Np^r)$ (for some $r \geq 
1$) which is an eigenform for all the Hecke operators $T_l$ ($l \nmid Np$) and 
$U_p$, in their classical normalisations, and we have the formulae
\[ a_l(f) = \text{eigenvalue of }T_l\text{ on }\pi_{0,l}^{\GL_2(\Z_l)}, 
\text{ }p^{-1/2} a_p(f) = \iota\chi_{0, 1}(p). \]

Note that the central character of $\pi_0$ is a Hecke character $\psi_{\pi_0}$ 
with $\psi_{\pi_0}|_{\R_{> 0}}(z) = z^{2-k}$. So $\psi_{\pi_0}|_{\Qp^\times} = 
\iota(\chi_{0,1}\chi_{0,2})$ is a finite order twist of the character $z 
\mapsto |z|^{2-k}$. To convince the reader that these formulae are correct, we observe that if $\pi_{0,l}$ is a normalised induction $i_{B_2}^{\GL_2}\mu_1\otimes \mu_2$, then the eigenvalue of $T_l$ on $\pi_{0,l}^{\GL_2(\Z_l)}$ is $l^{1/2}(\mu_1(l) + \mu_2(l))$ \cite[(11.2.4)]{Diamond-Im}, whilst considering the central character shows that $\mu_1(l)\mu_2(l)$ has (complex) absolute value $l^{k-2}$. This is compatible with the fact that $a_l(f)$ is a sum of numbers with absolute values $l^{(k-1)/2}$. 

We can define a map $s : \cE_0(\overline{\Q}_p) \to \R$, called the slope, 
by composing the projection to $\cT_0(\overline{\bbQ}_p)$ with the map
$\chi \mapsto v_p\left(\chi\left(\begin{smallmatrix}
p & 0 \\ 0 & 1
\end{smallmatrix} \right)\right)$. Note that if $(\pi_0, \chi_{0,1}\otimes 
\chi_{0,2}) 
\in \mathcal{RA}_0$ we have \begin{equation}\label{eqn:nuforRA0}\nu_0(\pi_0,\chi_0)\left(\begin{smallmatrix}t_1 & 
0 \\ 0 
& t_2\end{smallmatrix}\right) = 
t_2^{2-k}\chi_{0,1}(t_1)\chi_{0,2}(t_2)\left|\frac{t_1}{t_2} 
\right|^{-1/2}_p\end{equation} so the slope map sends $(\pi_0, \chi_{0,1}\otimes 
\chi_{0,2}) \in \mathcal{RA}_0$ to $1/2 + v_p( 
\chi_{0, 1}(p))$. 

In particular, at a point $z_0 \in Z_0$ corresponding to a classical 
holomorphic modular form 
$f$, $s(z_0)$ equals the $p$-adic valuation of $\iota^{-1} a_p(f)$. Note 
that the corresponding pair $(\pi_0, \chi_0)$ is numerically non-critical 
exactly when $s(z_0) < k-1$ and ordinary exactly when $s(z_0) = 0$. The classicality criterion of Coleman \cite{coleman-classicality1,coleman-classicality2} shows that a point $z \in \cE_0(\Qpbar)$ with $\kappa(z)$ restricting to $t \mapsto t^{k-2}$ on a finite index subgroup of $\Zpx$ and $s(z) < k-1$ is necessarily in $Z_0$.  

Let $Z_0^{pc} \subset Z_0$ denote the subset of points corresponding to pairs 
$(\pi, \chi)$ where $\pi_p$ is not a twist of the Steinberg representation 
($pc$ stands for \emph{potentially crystalline}). We 
now define a `twin' map $\tau : Z_0^{pc} \to Z_0^{pc}$. Let $(\pi_0, \chi_0)$ be 
the pair corresponding to a point $z \in Z_0^{pc}$. Write $\chi_0 = \chi_{0, 1} 
\otimes \chi_{0, 2}$. Since $\pi_{0,p}$ is not a twist of the Steinberg 
representation, $\pi_{0,p}$ equals the full normalised induction $i_{B_2}^{\GL_2} 
\iota \chi_0$, 
which is irreducible. Let $\psi : \bbQ^\times \backslash \bA_\bQ^\times \to \Zpbarx$ be the unique finite order character which is unramified outside $p$ and 
such that $\psi|_{\bbZ_p^\times} = \chi_{0, 
2}|_{\bbZ_p^\times}^{-1}$. Then the 
character $\chi_{0, 2} \psi|_{\bbZ_p^\times}$ is 
unramified and $\chi_0' =  
\chi_{0, 2} \psi|_{\bbZ_p^\times} \otimes \chi_{0, 1} \psi|_{\bbZ_p^\times}$ is an accessible refinement of the 
twist $\pi_0 \otimes \iota\psi$. We therefore have a point $\tau(z) \in Z_0^{pc}$ corresponding to the pair 
$\tau(\pi_0, \chi_0) = (\pi_0 \otimes \iota \psi, \chi_0')$, that we call the 
twin of $z$. Note that $\tau^2 = 1$ and if $\pi_{0, p}$ is unramified then 
$\tau(z)$ is the usual companion point appearing in the Gouvea--Mazur 
construction of the infinite fern \cite[\S 18]{Maz97}. The following lemma is an easy computation.
\begin{lemma}\label{lem_slope_of_twin_point}
	Let $z \in Z_0^{pc}$, and let $z' = \tau(z)$. Let $s, s'$ denote the slopes 
	of these two points, and $\kappa(z), \kappa(z') \in 
	\cW_0(\overline{\bQ}_p)$ their images in weight space. Then $s + s' = k-1$ 
	and $v_p(\kappa(z)(1+q) - 1) = v_p(\kappa(z')(1+q) - 1)$, where $q = p$ if 
	$p$ is odd and $q = 4$ if $p$ is even.
\end{lemma}
Here is the main result of \S \ref{sec_rigid_geometry}.
\begin{thm}\label{thm_propogration_along_components_of_eigencurve}
	Let $(\pi_0, \chi_0), (\pi_0', \chi'_0) \in \mathcal{RA}_0$ and let $n \geq 
	2$. Let $z_0, z_0' \in Z_0$ be the corresponding points. Suppose that one 
	of following two sets of conditions are satisfied:
	\begin{enumerate}
		\item The refinement $\chi_0$ is numerically non-critical and 
		$n$-regular.
		\item The refinement $\chi'_0$ is $n$-regular.
		\item The Zariski closures of $r_{\pi_0, \iota}(G_{\Qp})$ and 
		$r_{\pi'_0, \iota}(G_{\Qp})$ contain 
		$\SL_2$.
		\item $\Sym^{n-1} r_{\pi_0, \iota}$ is automorphic; 
	\end{enumerate}
or
	\begin{enumerate}
	\item[(1\textsuperscript{ord})] The refinement $\chi_0$ is ordinary.
	\item[(2\textsuperscript{ord})] $\pi_0$ and $\pi_0'$ are not CM (so 
	the Zariski closures of $r_{\pi_0, \iota}(G_{\Q})$ and $r_{\pi_0', 
	\iota}(G_{\Q})$ contain $\SL_2$).
	\item[(3\textsuperscript{ord})] $\Sym^{n-1} r_{\pi_0, \iota}$ is 
	automorphic.
\end{enumerate}
	If the points $z_0, z_0'$ lie on a common irreducible component of 
	$\cE_{0,\Cp}$, then $\Sym^{n-1} 
	r_{\pi'_0, \iota}$ is also automorphic.
\end{thm}
\begin{proof}
	We want to apply Corollary 
	\ref{cor_main_application_for_U(2)}. We first need to specify suitable data 
	$F, S, G_2, U_2$. Let $F' / \bbQ$ be an abelian CM extension satisfying the 
	following conditions:
	\begin{itemize}
		\item Each prime dividing $Np$ splits in $F'$. 
		\item $[(F')^+ : \bQ]$ is even.
		\item The extension $F' / (F')^+$ is everywhere unramified.
	\end{itemize}
	
	After extending $E$, we may assume that $z_0, z_0' \in \cE_0(E)$ and that 
	there is an irreducible component $\cC \subset \cE_0$ containing the points 
	$z_0, z_0'$. Moreover, by the first part of \cite[Theorem 3.4.2]{Con99}, we 
	may assume that $\cC$ is geometrically irreducible. Let $W$ denote the 
	unique
	connected component of $\cW_0$ containing $\kappa(\cC)$. We can find a 
	character $\chi : G_\bbQ \to \cO(W)^\times$ such that the determinant of 
	the universal pseudocharacter over $\cC$ equals $\epsilon^{-1} 
	\chi$ ($\chi$ is the product of a finite order $p$-unramified character 
	and the composition of $\epsilon$ with the universal character 
	$\Zp^\times \to \cO(\cW_0)$). By 
	Lemma 
	\ref{lem_existence_of_global_characters}, we can find a 
	finite \'etale morphism $\eta : \widetilde{W} \to W$ and a character $\psi 
	: G_{F'} \to \cO(\widetilde{W})^\times$, unramified almost everywhere, such 
	that $\psi \psi^c = \chi|_{G_{F'}}$, and such that for each place $v | p$ of $(F')^+$, there is a place $\wv | v$ of $F'$ such that $\psi|_{G_{F'_\wv}}$ is unramified. We now let $F/\Q$ be a soluble, Galois, 
	CM extension, containing $F'$, such that:
	\begin{itemize}
	\item Each prime dividing $Np$ splits in $F$. 
	\item The extension $F / F^+$ is everywhere unramified.
	\item The character $\psi|_{G_F}$ is unramified away from $p$.
\end{itemize}	
	
	Let $S$ denote the set of places of $F^+$ dividing $Np$. Fix as usual a 
	set of factorisations $v = \wv \wv^c$ for $v \in S$. Fix the unitary group 
	$G_2$ as in our standard assumptions (\S 
	\ref{sec_definite_unitary_groups}). Then for each $v \in S$, there is an 
	isomorphism $\iota_\wv: 
	G_2(F_v^+) \to \GL_2(F_{\wv})$. We let $U_2 = \prod_v U_{2, v} \subset 
	G(\bA_{F^+}^\infty)$ be an open subgroup with the property that $U_{2, v}$ 
	is hyperspecial maximal compact if $v \not\in S$, and $U_{2, v}$ is 
	the pre-image under $\iota_\wv$ of the subgroup $U_{1}(N)_l$ of 
	$\GL_2(\bbQ_l)$ if $v \in S$ 
	has residue characteristic $l$ (in which case $F_{\wv}=\bbQ_l$). 
	
	We recall that $\cT_2 = \prod_{v \in S_p} \Hom((F_\wv^\times)^2, \bG_m)$. 
	Let $\widetilde{\cT} = \cT_0 \times_{\cW_0} \widetilde{W}$, with 
	$\tilde{\kappa}: \widetilde{\cT} \to \widetilde{W}$ the projection map. If 
	$\chi^u = \chi^u_1 \otimes \chi^u_2 \in \cT_0(\cT_0)$ denotes the universal 
	character, then the tuple of characters \[ 
((\chi^u_1 \circ \mathbf{N}_{F_\wv / \bbQ_p} \cdot \psi^{-1}|_{G_{F_\wv}} \circ \Art_{F_\wv}) \otimes (\chi^u_2 \circ \mathbf{N}_{F_\wv / \bbQ_p} \cdot \psi^{-1}|_{G_{F_\wv}} \circ \Art_{F_\wv}) )_{v \in S_p} \]
in $\cT_2(\widetilde{\cT})$ determines a morphism $b_p : \widetilde{\cT} \to 
\cT_2$. Writing $\cX_{0,ps}$ for the rigid space of 2-dimensional 
pseudocharacters of $G_\bbQ$, unramified outside $Np$, there is a base change 
morphism $b : \cX_{0,ps} \times \widetilde{\cT} \to \cX_{2,ps} \times \cT_2$ 
covering $b_p$ and sending a pair $(\tau, \mu)$ to $(\tau|_{G_F} \otimes 
\psi^{-1}_{\tilde{\kappa}(\mu)}, b_p(\mu))$. This leads to a diagram of rigid 
spaces
\[ \xymatrix{ \cX_{0,ps} \times \widetilde{\cT} \ar[r]^b & \cX_{2,ps} \times \cT_2 \\ \cE_0 \times_{\cT_0} \widetilde{\cT} \ar[u]^{\widetilde{i}} & \cE_2 \ar[u]_{i_2}  }. \]
Let $\widetilde{\cC} = \cC \times_W \widetilde{W} = \cC \times_{\cT_0} 
\widetilde{\cT}$. Then the morphism $\widetilde{\cC} \to \cC$ is finite \'etale. In 
particular, each irreducible component of $\widetilde{\cC}$ maps surjectively to 
$\cC$. Choose $E'/E$ so that the irreducible components of $\widetilde{\cC}_{E'}$ 
are geometrically irreducible (we apply \cite[Theorem 3.4.2]{Con99} again). 
Since $\cC$ is geometrically irreducible, we 
still know that each irreducible component of $\widetilde{\cC}_{E'}$ maps 
surjectively to $\cC_{E'}$. Consequently, we can find points $z_1, z_1'$ of 
$\widetilde{\cC}_{E'}$ lifting $z_0, z_0'$ and lying on a common geometrically 
irreducible component $\widetilde{\cC}'$ of $\widetilde{\cC}_{E'}$. We next wish to 
show that $b 
\circ \widetilde{i}(\widetilde{\cC}') \subset i_2 (\widetilde{\cE}_{2,E'})$, or 
equivalently that $(b \circ \widetilde{i})^{-1}(i_2(\cE_{2,E'}))$ contains 
$\widetilde{\cC}'$. Since $i_2$ is a closed immersion, it suffices to show that 
$\widetilde{Z}'_0$, the pre-image of $Z_0$ in $\widetilde{\cC}'$, satisfies $b 
\circ \widetilde{i}(\widetilde{Z}'_0) \subset i_2(\cE_2(\Qpbar))$ (the 
accumulation 
property of $\widetilde{Z}'_0$ in $\widetilde{\cC}'$ is inherited from the 
corresponding property of the subset $Z_0 \cap \cC \subset \cC$). 

To see this, we note that for any $(\pi, \chi) \in \cZ_0$, with lift 
$\widetilde{z}' \in \widetilde{Z}'_0$, the base change 
$\pi_{F}$ (which exists since $F / \bQ$ is soluble) is still cuspidal. 
Indeed, if not then $r_{\pi, \iota}|_{G_F}$ would be reducible, implying 
that $\pi$ was automorphically induced from a quadratic imaginary subfield 
$K / \bbQ$ of $F / \bbQ$. This is a contradiction, since we chose $F$ so that all primes 
dividing $N p$ split in $F$, yet $K$ must be ramified at at least one such 
prime. The descent of $\pi_{F} \otimes \iota \psi_{\widetilde{z}'}^{-1}$ to 
$G_2$ (which exists, by \cite[Th\'eor\`eme 5.4]{labesse}) gives (together with 
$b_p(\nu_0(\pi, \chi))$) a point of $\cE_2$ which equals the image of 
$\widetilde{z}'$ under the map $b \circ \widetilde{i}$.
	
We can now complete the proof. Indeed, the points $b \circ \widetilde{i}(z_1)$, 
$b \circ \widetilde{i}(z_1')$ lie on a common geometrically irreducible 
component of $\cE_{2,E'}$, 
by construction. They satisfy the conditions of Corollary 
\ref{cor_main_application_for_U(2)} (in particular, Example 
\ref{ex:gl2qptriangulations} shows that our assumption on 
$r_{\pi_0',\iota}(G_{\Qp})$ in the non-ordinary case implies that all of its 
triangulations are 
non-critical). We therefore conclude the existence of an 
automorphic representation $\pi'_n$ of $G_n(\bA_{F^+})$ such that $\Sym^{n-1} 
r_{\pi_0', \iota}|_{G_F} \cong r_{\pi_n', \iota}$. Our assumptions (cf. Lemma 
\ref{lem_local_image_contains_SL_2}(2)) imply that $\Sym^{n-1} r_{\pi_0', 
\iota}|_{G_F}$ is irreducible, and therefore that the base change of $\pi'_n$ 
is a cuspidal automorphic representation of $\GL_n(\bA_F)$. Soluble descent for 
$\GL_n$ now implies that $\Sym^{n-1} r_{\pi_0', \iota}$ is itself automorphic. 
\end{proof}
\begin{lemma}\label{lem_existence_of_global_characters}
	Let $F$ be a CM number field. Suppose that each $p$-adic place of $F^+$ 
	splits in $F$, and let $\widetilde{S}_p$ be a set of $p$-adic places of $F$ 
	such that $\widetilde{S}_p \sqcup \widetilde{S}_p^c$ is the set of all 
	$p$-adic places of $F$. Let $W$ be a connected $E$-rigid space, and let 
	$\chi : G_{\bbQ} \to \cO(W)^\times$ be a continuous character (continuity 
	defined by demanding that the induced characters with values in 
	$\cO(U)^\times$ are continuous for all affinoid admissible opens $U \subset 
	W$,  as in \cite[\S 2]{Buz04}), unramified almost everywhere. Then we 
	can find a finite \'etale morphism $\eta : \widetilde{W} \to W$ and a 
	continuous character $\psi : G_F \to \cO(\widetilde{W})^\times$ such that 
	the following properties hold:
	\begin{enumerate}
		\item $\psi$ is unramified almost everywhere.
		\item For each $\wv \in \widetilde{S}_p$, $\psi|_{G_{F_\wv}}$ is 
		unramified.
		\item $\psi \psi^c = \eta^\ast(\chi)|_{G_F}$.
	\end{enumerate}
\end{lemma}
\begin{proof}
	We first claim that we can find a finite \'etale morphism $W' \to W$ and a continuous character $\lambda : G_F \to \cO(W')^\times$ with the following properties:
	\begin{itemize}
		\item $\lambda$ is unramified almost everywhere.
		\item $\chi|_{G_F} \lambda \lambda^c$ has finite order.
	\end{itemize}
	Indeed, let $L : \prod_{w | p} \cO_{F_w}^\times \to \cO(W)^\times$ be defined by the formula 
	\[ L( (u_w)_w) = \prod_{\wv \in \widetilde{S}_p^c} \chi|_{G_F}^{-1} \circ \Art_{F_\wv}(u_\wv). \]
	 Then $L$ is continuous, and trivial on a finite index subgroup of 
	 $\cO_{F}^\times$ (it is trivial on the norm 1 units in 
	 $\cO_{F^+}^\times$). It follows from Chevalley's theorem \cite[Th\'{e}or\`{e}me 1]{Chevalley-GL1CSP} that there is a compact open subgroup $U^p$ of $\prod_{w\nmid p}\cO^\times_{F_w}$ such that $L$ is trivial on $\Gamma(U^p) := \left(U^p \times \prod_{w | p} \cO_{F_w}^\times\right)\cap \cO_{F}^\times$.
	 
	 Note that if $H$ is a product of a finite abelian group and a finite 
	 $\bbZ_p$-module, and $H' \subset H$ is a 
	 finite index subgroup, then the natural map $\Hom(H, \bG_m) \to \Hom(H', 
	 \bG_m)$ of rigid spaces is finite \'etale. Maps of rigid spaces $W \to \Hom(H, \bG_m)$ biject with continuous characters $H \to \cO(W)^\times$. 
	 
	 It follows that we may extend 
	 $L$ to a continuous character $L' : F^\times \backslash 
	 \bA_F^{\infty, \times} \to \cO(W')^\times$, for some finite \'etale 
	 morphism $W' \to W$. Indeed, we apply the preceding remark with $H'$ the quotient of $\prod_{w | p} \cO_{F_w}^\times$ by the closure of $\Gamma(U^p)$ and $H$ the quotient of  $F^\times \backslash 
	 \bA_F^{\infty, \times}$ by the closure of the image of $U^p$ (cf.~the discussion in \cite[\S 2]{Buz04}). 
	 
	 We define $\lambda$ by $\lambda \circ \Art_F = L'$. 
	 The character $\chi|_{G_F} \lambda \lambda^c$ has finite order because it 
	 factors through the Galois group of an abelian extension of $F$ which is 
	 unramified at all but finitely many places and unramified at the primes 
	 above $p$.
	 
	 Replacing $W'$ by a connected component, we may suppose that $W'$ is 
	 connected, in which case the character $\chi|_{G_F} \lambda \lambda^c$ is 
	 constant (i.e. pulled back from a morphism $W' \to \operatorname{Sp} E'$, 
	 for a finite extension $E' / E$). Applying \cite[Lemma A.2.5]{BLGGT}, we 
	 may find a finite extension $E'' / E'$ and a continuous character $\varphi 
	 : G_F \to (E'')^\times$ of finite order such that $\chi|_{G_F} \lambda 
	 \lambda^c = \varphi \varphi^c$. The proof is complete on taking 
	 $\widetilde{W} = W'_{E''}$ and $\psi = \varphi \lambda^{-1}$. 
\end{proof}
We conclude this section with a lemma that will be used in \S \ref{sec_higher_levels}. It uses the existence of the universal pseudocharacter $t$ over $\cE_0$.
\begin{lemma}\label{lem_bad_locus_of_eigencurve_zariski_closed}
Fix $n \geq 1$, and let $\cZ \subset \cE_0$ denote the set of points $x$ satisfying one of the following conditions:
\begin{enumerate}
\item $t_x$ is absolutely reducible;
\item $t_x = \tr \rho_x$ for an absolutely irreducible representation $\rho_x : G_\bQ \to \GL_2(\overline{\bQ}_p)$, and the Zariski closure of the image of $\rho_x$ does not contain $\SL_2$.
\item There exists a prime $l | N$ such that $t_x|_{G_{\bQ_l}} = \chi_1 + \chi_2$ for characters $\chi_i : G_{\bQ_l} \to \overline{\bQ}_p^\times$ such that $(\chi_1 / \chi_2)^i = 1$ for some $i = 1, \dots, n-1$.
\end{enumerate}
Then $\cZ$ is Zariski closed.
\end{lemma}
\begin{proof}
The discussion in \cite[\S 4.2]{chenevier_det} shows that the locus where $t_x$ is absolutely reducible is Zariski closed. If $\rho_x : G_\bQ \to \GL_2(\overline{\bQ}_p)$ is irreducible, then the Zariski closure of the image of $\rho_x$ contains $\SL_2$ if and only if $\Sym^6 \rho_x$ is irreducible. Indeed, the Zariski closure of the image of $\rho_x$ contains $\SL_2$ if and only if the Zariski closure $G_x$ of the image of the associated projective representation $\operatorname{Proj} \rho_x : G_\bQ \to \PGL_2(\overline{\bQ}_p)$ is $\PGL_2$. There are two possibilities for the group $G_x$, which is a (possibly disconnected) reductive group: the first is that it is finite, hence either dihedral or conjugate to one of $A_4$, $S_4$, or $A_5$. In any of these cases $\Sym^6 \rho_x$ is reducible. The next is that $G_x$ has a non-trivial identity component, which therefore contains a maximal torus of $\PGL_2$. The only possibilities are therefore that either $G_x$ equals the normaliser of this maximal torus (in which case $\Sym^6 \rho_x$ is again reducible) or that $G_x = \PGL_2$ (in which case $\Sym^6 \rho_x$ is irreducible). 

This shows that the set $\cZ_{12}$ of points satisfying conditions (1) or (2) of the Lemma is Zariski closed. Finally, if $l | N$ and $i = 1, \dots, n-1$, let $\cZ_{3, l, i}$ denote the set of points $x$ such that $t_x|_{G_{\bQ_l}} = \chi_1 + \chi_2$ for some characters $\chi_1, \chi_2$ such that $(\chi_1 / \chi_2)^i = 1$. It remains to show that $\cZ_{3, l, i}$ is Zariski closed. Its complement is the set of points such that either $t_x|_{G_{\bQ_l}}$ is absolutely irreducible, or $t_x|_{G_{\bQ_l}}$ is absolutely reducible and there exists $g \in G_{\bQ_l}$ such that the discriminant of the characteristic polynomial of $g^i$ under the pseudocharacter $t_x$ is non-zero. This is a union of Zariski open sets.
\end{proof}
\section{Ping pong}\label{sec_ping_pong}

In this section we use the rigid analytic results of \S 
\ref{sec_rigid_geometry} to prove the following theorem. We 
recall that we say that an automorphic representation $\pi$ of $\GL_2(\bA_\bQ)$ has 
``weight $k$'' 
for an integer $k \geq 2$ if $\pi_\infty$ has the same infinitesimal character 
as the dual of the algebraic representation $\Sym^{k-2} \bC^2$.
\begin{thm}\label{thm_propogration_of_automorphy_of_symmetric_powers}
	Fix an integer $n \geq 2$. Let $\pi_0$ be a cuspidal automorphic 
	representation of $\GL_2(\bA_\bQ)$ which is everywhere unramified and of 
	weight 
	$k$, for some $k \geq 2$. Suppose that $\Sym^{n-1} r_{\pi_0, \iota}$ is 
	automorphic for some (equivalently, any) prime $p$ and isomorphism $\iota : 
	\overline{\bQ}_p \to \bC$. Then for any everywhere unramified cuspidal 
	automorphic representation $\pi$ of $\GL_2(\bA_\bQ)$ of weight $l \geq 2$, 
	$\Sym^{n-1} r_{\pi, \iota}$ is automorphic.
\end{thm}
To prove Theorem \ref{thm_propogration_of_automorphy_of_symmetric_powers}, we 
will use the properties of the eigencurve $\cE_0$, as defined in \S 
\ref{subsec_eigencurve}. More precisely, we henceforth let $p = 2$, $N = 1$, 
and let $\cE_0$ denote the eigencurve defined with respect to this particular 
choice of parameters. We fix an isomorphism $\iota: \Qbar_2 \to \C$. $\cE_0$ is 
supported on the connected component $\cW_0^+ 
\subset \cW_0$ defined by $\chi(-1) = 1$. We write $\chi^u: \Z_2^\times \to 
\cO(\cW_0)$ for the universal character. We have the following explicit result 
of Buzzard and Kilford on the geometry of the morphism $\kappa : \cE_0 \to 
\cW_0^+$ and the slope map $s 
: \cE_0(\overline{\bQ}_p) \to \R$:
\begin{thm}\label{thm_buzzard_kilford}
	Let $w \in \cO(\cW_0)$ denote the function $\chi^u(5) - 1$. Then:
	\begin{enumerate}
		\item $w$ restricts to an isomorphism between $\cW_0^+$ and the open 
		unit disc $\{ | w | < 1 \}$.
		\item Let $\cW_0(b) \subset \cW_0^+$ denote the 
		open subset where $ | 8 
		| 
		< | w | < 1$, and let $\cE_0(b) = \kappa^{-1}(\cW_0(b))$. Then there is 
		a decomposition
		$\cE_0(b) = \sqcup_{i=1}^\infty X_i$ of $\cE_0(b)$ as a 
		countable disjoint union of admissible open subspaces such that for each $i \geq 1$, $\kappa|_{X_i} : 
		X_i \to \cW_0(b)$ is an isomorphism.
		\item For each $i = 1, 2, \ldots$, the map $s \circ 
		\kappa|_{X_i}^{-1} : \cW_0(b)(\overline{\bQ}_p) \to 
		X_i(\overline{\bQ}_p) \to \bR$ equals the 
		map $i v_p \circ w$. 
	\end{enumerate}
\end{thm}
\begin{proof}
	This is almost the main theorem of \cite{Buz05}, except that here we are 
	using the cuspidal version of the eigencurve. However, if $\cE_1$ denotes 
	the full eigencurve used in \cite{Buz05}, then there is a decomposition 
	$\cE_1 = \cE_1^{ord} \sqcup \cE_1^{non-ord}$ as a union of open and closed 
	subspaces. This follows from the fact that the ordinary locus 
	$\cE_1^{ord}$ in the eigencurve can also be constructed using Hida theory 
	(see \cite[\S 6]{pilloni-oc}), so is finite over 
	$\cW_0$. Since $\cE_1$ is separated the open immersion $\cE_1^{ord} 
	\hookrightarrow \cE_1$ is therefore also finite, hence a closed 
	immersion. In our particular case $(p=2,N=1)$, we have $\cE_1^{ord} \cong 
	\cW_0^+$ (the unique ordinary family is the family of Eisenstein series) 
	and 
	therefore $\cE_1^{non-ord} = \cE_0$, giving the 
	statement we have here. See Lemma 7.4 of the (longer) arXiv version of 
	\cite{BC-arxiv} for an alternative argument. 
	
	We note as 
	well that in our 
	normalisation, 
	the trivial character in $\cW_0$ corresponds to forms of weight 2, whereas 
	in the notation of \cite{Buz05}, the character $x^2$ corresponds to 
	forms of weight 2. However, this renormalisation does not change the region 
	$\cW_0(b)$.
\end{proof}
Before giving the proof of Theorem 
\ref{thm_propogration_of_automorphy_of_symmetric_powers}, we record some useful 
lemmas.
\begin{lemma}\label{lem_ping_pong}
	Let $z \in Z_0^{pc} \cap \cE_0(b)$, and suppose $z \in X_i$. Let $z' = 
	\tau(z)$ be the twin of $z$. Then $z \in X_{i'}$, where $i'$ satisfies the 
	relation $i + i' = (k-1) / v_p(w(z))$.
\end{lemma}
\begin{proof}
	By Lemma \ref{lem_slope_of_twin_point}, $z'$ lies in $\cE_0(b)$, so in 
	$X_{i'}$ for a unique integer $i' \geq 1$. Writing $s, s'$ for the slopes 
	of these two points and $k$ for their weights, we have $k-1 = s + s' = i 
	v_p(w(z)) + i' v_p(w(z))$, hence $i + i' = (k-1) / v_p(w(z))$. 
\end{proof}
As a sanity check, we observe that in the context of the proof of Lemma 
\ref{lem_ping_pong}, $(k-1) / v_p(w(z))$ is always an integer. Indeed, 
$\kappa(z)$ satisfies $\kappa(z)(5) = 5^{k-2} \zeta_{2^m}$ for some $k \geq 2$ 
and $m \geq 0$. This weight lies in $\cW_0(b)$ if and only if either $k$ is 
odd, or $k$ is even and $m \geq 1$. If $m \geq 1$, then $v_p(w(z)) = 2^{1-m}$. 
If $m = 0$ and $k$ is odd, then $v_p(w(z)) = 2$. In either case we see that 
$(k-1) / v_p(w(z))$ is an integer. 

\begin{lemma}\label{lem:level1regularity}
	Let $\pi$ be an everywhere unramified cuspidal automorphic representation 
	of $\GL_2(\A_\Q)$ of weight $k \ge 2$. Every accessible 
	refinement of $\pi$ is numerically non-critical and $n$-regular for every 
	$n \ge 2$ (recall that we have fixed $p=2$ and these notions refer to the 
	local factor at $2$, $\pi_2$).
\end{lemma}
\begin{proof}
	Numerical non-criticality of every refinement is immediate from the fact 
	that there are no cusp forms of level $1$ that are ordinary at $2$. For 
	regularity, if we fix a refinement $\chi = \chi_1\otimes\chi_2$ then 
	$\alpha = p^{1/2}\iota\chi_1(p)$ and $\beta = p^{1/2}\iota\chi_2(p)$ are 
	the roots 
	of the polynomial $X^2-a_2X + 2^{k-1}$, with $a_2$ the $T_2$-eigenvalue of 
	the level $1$ weight $k$ normalised eigenform $f$ associated to $\pi$. We 
	need to show that $\alpha/\beta$ is not a root of unity.

Suppose $\alpha/\beta = \zeta$ is a root of unity. If we fix $\iota_5: \Qbar_5 \cong \C$, the semisimplified mod $5$ Galois 
	representation $\rbar_{f,\iota_5}$ arises up to twist from a level 1 
	eigenform of weight $\le 6$ (i.e.~the level $1$ Eisenstein series of weight $4$ or $6$). This shows that $\iota_5^{-1}(\zeta) \equiv 2^3 \text{ or }2^5 \text{ mod } \m_{\Zbar_5}$, and therefore $\zeta$ is the product of a $5$-power root of unity and $\pm i$ (since $2$ has order $4$ in $\F_5^\times$). Applying a similar argument at the prime $7$ with $\iota_7: \Qbar_7 \cong \C$, we see that $\iota_7^{-1}(\zeta) \equiv 2^3, 2^5 \text{ or } 2^7 \text{ mod } \m_{\Zbar_7}$ and therefore $\zeta$ is the product of a $7$-power root of unity and a cube root of unity.  This gives the desired contradiction.	(We thank Fred Diamond for 
	pointing out this argument to us, and thank an anonymous referee for explaining how to avoid using Hatada's congruence which appeared in the first version of this argument.) 
\end{proof}

\begin{lemma}\label{lem_local_image_contains_SL_2}
	Let $\pi$ be a cuspidal automorphic representation of $\GL_2(\bA_\bQ)$ of 
	weight $k \geq 2$. We temporarily let $p$ be an arbitrary prime. Then:
	\begin{enumerate}
		\item $r_{\pi, \iota}|_{G_{\bbQ_p}}$ is reducible if and only if $\pi$ 
		is $\iota$-ordinary.
		\item Suppose either that $r_{\pi, \iota}|_{G_{\bbQ_p}}$ is irreducible 
		and $\pi_p$ admits a $3$-regular refinement, or that $k > 2$ 
		and $r_{\pi, \iota}$ is not potentially 
		crystalline. Then the Zariski closure of $r_{\pi, \iota}(G_{\bbQ_p})$ 
		(in $\GL_2/\Qpbar$) contains $\SL_2$.
		\item Suppose again that $p = 2$, and that $\pi$ is everywhere 
		unramified. Then the Zariski 
		closure of $r_{\pi, \iota}(G_{\bbQ_2})$ contains $\SL_2$.
	\end{enumerate}
\end{lemma}
\begin{proof}
	For the first part, $\iota$-ordinarity implies reducibility by local-global compatibility at $p$, as in \cite[Theorem 
	2.4]{jackreducible}. For the converse, if $r_{\pi, \iota}|_{G_{\bbQ_p}}$ is 
	reducible, then its Jordan--H\"older factors are de Rham characters of $G_{\bbQ_p}$ and therefore have the form $\psi_i \epsilon^{-k_i}$, where the $\psi_i$ are $\Zpbarx$-valued characters with finite order restriction to inertia and the $k_i$ are the Hodge--Tate weights (we can assume $k_1 = 0$ and $k_2 = k-1$ in our situation). The Weil representation part of $\WD(r_{\pi, 
		\iota}|_{G_{\bbQ_p}})$ is therefore equal to $\psi_1 \oplus \psi_2|\cdot|^{1-k}\circ\Art_{\Qp}^{-1}$. Since $\WD(r_{\pi, 
	\iota}|_{G_{\bbQ_p}}) = \rec^T_{\bbQ_p}(\iota^{-1}\pi_p)$, $\pi_p$ is a subquotient of the normalised induction $i_{B_2}^{\GL_2}(\iota\psi_1\circ\Art_{\Qp}|\cdot|^{1/2}\otimes \iota\psi_2\circ\Art_{\Qp}|\cdot|^{3/2-k})$. It follows from \cite[Lemma 2.3]{jackreducible} that $\pi$ is $\iota$-ordinary.
	
	For the second part, we note that $r_{\pi, \iota}|_{G_{\bQ_p}}$ is irreducible. Indeed, if  $r_{\pi, \iota}|_{G_{\bbQ_p}}$ is reducible then the first part of the lemma shows that $\pi$ is $\iota$-ordinary, and \cite[Lemma 2.3]{jackreducible} implies that $\pi_p$ is a subquotient of $i_{B_2}^{\GL_2}\chi_1\otimes\chi_2$ with $v_p\left(\iota^{-1}(\chi_1/\chi_2 (p))\right) = 1-k$. If $r_{\pi, \iota}$ is not potentially crystalline then $\WD(r_{\pi, \iota}|_{G_{\bbQ_p}})$ has $N \ne 0$ and local-global compatibility implies that if $\pi_p$ is a subquotient of $i_{B_2}^{\GL_2}\chi_1\otimes\chi_2$, then $\chi_1/\chi_2 = |\cdot|^{\pm 1}$. This is a contradiction if $k > 2$.
		
	Thus $r_{\pi, \iota}|_{G_{\bbQ_p}}$ is 
	irreducible and the Zariski closure $H$ of its image is a reductive 
	subgroup of $\GL_2$. Let $T$ be a maximal torus of $H$. Since $r_{\pi, \iota}$ is 
	Hodge--Tate regular, $T$ is regular in $\GL_2$ (i.e. its centralizer is a maximal torus 
	of $\GL_2$) by \cite[Theorem 1]{Sen}. If $H$ does not contain $\SL_2$, then 
	it is contained in the 
	normaliser of a maximal torus of $\GL_2$ and $r_{\pi, \iota}|_{G_{\bbQ_p}}$ 
	is induced from a character of an index two subgroup. This forces $\WD(r_{\pi, \iota}|_{G_{\bbQ_p}})$ to likewise be induced, so any refinement $\chi = \chi_1\otimes\chi_2$ of $\pi_p$ 
	satisfies $\chi_1^2 = \chi_2^2$, and this Weil--Deligne representation has $N = 0$. This is a contradiction, since we are assuming either that there exists a 3-regular refinement or that $N$ is non-zero. 
	
	For the third part, we have already observed (see the proof of Theorem \ref{thm_buzzard_kilford}) that there are no cusp forms 
	of level 1 that are ordinary at 2, so $r_{\pi, \iota}|_{G_{\bbQ_2}}$ is 
	irreducible. Suppose that $r_{\pi, \iota}|_{G_{\bbQ_2}}$ is induced. Then 
	$\WD(r_{\pi, \iota}|_{G_{\bbQ_2}}) = \Ind_{W_K}^{W_{\bbQ_2}} \psi$ for some 
	quadratic extension $K / \bbQ_2$ and character $\psi : W_K \to 
	\overline{\bQ}_2^\times$. Since this Weil--Deligne representation must be 
	unramified, we see that $K$ and $\psi$ are both unramified, and therefore 
	that $\psi$ extends to a character $\psi : W_{\bbQ_2} \to 
	\overline{\bQ}_2^\times$ such that $\WD(r_{\pi, \iota}|_{G_{\bbQ_2}}) = 
	\psi \oplus (\psi \otimes \delta_{K / \bbQ_2})$. In particular, the 
	$T_2$-eigenvalue (which equals the trace of Frobenius in this 
	representation) is $0$, but, as shown in the proof of Lemma 
	\ref{lem:level1regularity}, this is impossible. 
\end{proof}
\begin{lemma}\label{lem_propogation_through_annuli}
	Let $i \geq 1$ be an integer, and let $z \in Z_0 \cap X_i(\overline{\Q}_p)$ be the 
	point 
	corresponding to a pair $(\pi, \chi)$. Suppose that $\chi$ is $n$-regular. 
	Let $z' \in Z_0 \cap X_i(\overline{\Q}_p)$ be 
	any other point, corresponding to a pair $(\pi', \chi')$, with $\chi'$ 
	$n$-regular. If 
	$\Sym^{n-1} r_{\pi, \iota}$ is automorphic, then $\Sym^{n-1} r_{\pi', 
	\iota}$ is also.
\end{lemma}
\begin{proof}
	We can assume $n \geq 3$. We apply Theorem \ref{thm_propogration_along_components_of_eigencurve} 
	(note that $X_{i,\Cp} \cong \cW_0(b)_{\Cp}$ is irreducible): 
	since there are no ordinary points in $\cE_0$, $\chi$ is numerically 
	non-critical and Lemma \ref{lem_local_image_contains_SL_2} implies that 
	the 
	Zariski closures of $r_{\pi,\iota}(G_{\Qp})$ and $r_{\pi',\iota}(G_{\Qp})$ 
	contain $\SL_2$ (the first part of the lemma shows that we are in the locally irreducible and $3$-regular case of the second part of the lemma).
\end{proof}
\begin{lemma}\label{lem_first_step}
	Let $\pi$ be a cuspidal, everywhere unramified automorphic representation 
	of $\GL_2(\bA_\bQ)$ of weight $k \geq 2$, and let $\chi$ be a choice of 
	accessible refinement. Then there exists an integer $m_\pi \geq 1$ such that for any integer $m \geq m_\pi$, we can find a cuspidal automorphic 
	representation $\pi'$ of $\GL_2(\bA_\bQ)$ satisfying the following 
	conditions:
	\begin{enumerate}
		\item $\pi'$ is unramified outside 2.
		\item $\pi'$ admits two accessible refinements of distinct slopes (in 
		particular these refinements are $n$-regular for every $n \ge 2$). 
		\item There is an accessible refinement $\chi'$ of $\pi'$ such that 
		$(\pi, \chi)$ and $(\pi', \chi')$ define points $z, z'$ on the same 
		irreducible component of $\cE_{0,\Cp}$.
		\item $\kappa(z') \in \cW_0(b)$. In particular, $z' \in 
		X_i$ for some 
		$i \geq 1$ (notation as in the statement of Theorem 
		\ref{thm_buzzard_kilford}). 
		\item Set $z'' = \tau(z')$, and let $(\pi'', \chi'') \in Z_0^{pc}$ be 
		the associated pair.\footnote{Bouncing from $z'$ to its twin point $z''$ reminded the authors of a game of ping pong, whence the section title. Earlier versions of the argument involved longer rallies!}%
		Then $z'' \in X_{2^m - 1}$.
		\item The automorphy of any one of the three representations 
		\[\Sym^{n-1} r_{\pi, \iota},\Sym^{n-1} r_{\pi', \iota},\Sym^{n-1} 
		r_{\pi'', \iota}\] implies automorphy of all three.
	\end{enumerate}
\end{lemma}
\begin{proof}
	We use Theorem \ref{thm_buzzard_kilford}. Extending $E$ if necessary, we 
	may assume that $z \in \cE_0(E)$ and every irreducible component of $\cE_0$ 
	containing $z$ is geometrically irreducible. Fix one of these irreducible 
	components and fix $i$ such that this irreducible component contains $X_i$ 
	(such an $i$ exists, because every irreducible component of $\cE_0$ 
	has Zariski open image in $\cW_0^+$, hence intersects $\cE_0(b)$ and 
	therefore contains a non-empty union of irreducible components of 
	$\cE_0(b)$). We define $m_\pi$ to be least integer $m_\pi \geq 1$ satisfying the inequality
	\[ (2i + 2^{m_\pi + 1} - 3) / 2 > 2i. \]
	Given $m \geq m_\pi$, we choose $z' \in X_i(\overline{\bbQ}_p)$ to be the point such 
	that 
	$\kappa(z')(5) = 5^{k' - 2}$, where $k' = 2 i + 2^{m+1} - 1$. Then $(k'-2)/2 > 2i = s(z')$. By Coleman's classicality criterion, $z' \in Z_0$. If $z'$ was not in $Z_0^{pc}$, its slope would be $(k'-2)/2$. So $z' \in Z_0^{pc}$ and if $(\pi',	
	\chi')$ denotes the corresponding pair, then the two accessible refinements 
	of $\pi'$ have distinct slopes ($2i$ and $k'-1 - 2i$).
	
	Let $z'' = \tau(z')$ denote the twin point, and $(\pi'', \chi'')$ the 
	corresponding pair. Then $z''$ lies on $X_{i'}$, where $i' = (k'-1)/2 - i = 
	2^m - 1$, by Lemma \ref{lem_ping_pong}. We're done: the first 5 properties 
	of $(\pi', \chi')$ follow by construction, whilst the 6th follows
	from Theorem 
	\ref{thm_propogration_along_components_of_eigencurve}, Lemma 
	\ref{lem:level1regularity} and Lemma \ref{lem_local_image_contains_SL_2}.
\end{proof}
We can now complete the proof of Theorem 
\ref{thm_propogration_of_automorphy_of_symmetric_powers}.
\begin{proof}[Proof of Theorem 
\ref{thm_propogration_of_automorphy_of_symmetric_powers}]
	Let $\pi_0, \pi_0'$ be everywhere unramified cuspidal automorphic representations
	of $\GL_2(\bA_\bQ)$ of weights $k_0, k_0' \geq 2$, respectively. Define integers $m_{\pi_0}, m_{\pi'_0}$ as in Lemma \ref{lem_first_step} and fix an integer $m \geq \max(m_{\pi_0}, m_{\pi'_0})$. Combining Lemma \ref{lem_first_step} and Lemma \ref{lem_propogation_through_annuli} (applied with $i = 2^m - 1$) we see that the automorphy of $\Sym^{n-1} r_{\pi_0, \iota}$ implies that of $\Sym^{n-1} r_{\pi'_0, \iota}$. Since $\pi'_0$ was arbitrary, this completes the proof.
\end{proof}

\section*{Part II: Raising the level}

Most of the remainder of this paper (\S\S \ref{sec:automorphic_level_raising} --  \ref{sec_existence_of_seed_points}) is devoted to the proof of Theorem \ref{introthm_existence_of_single_symmetric_power} from the introduction, namely the existence for each $n \geq 2$ of a single regular algebraic, cuspidal, everywhere unramified automorphic representation $\pi$ of $\GL_2(\A_\Q)$ such that $\Sym^{n-1} \pi$ exists. As a guide to what follows, we now give an expanded sketch of the proof of this theorem. 

Fix, for the sake of argument, a regular algebraic, cuspidal, everywhere unramified automorphic representation $\pi$ of $\GL_2(\A_\Q)$. We will try to establish the existence of $\Sym^{n-1} \pi$ by proving the automorphy of one of the Galois representations $\Sym^{n-1} r_{\pi, \iota}$ associated to a choice of prime $p$ and isomorphism $\iota : \overline{\bQ}_p \to \bC$, using an automorphy lifting theorem. First, if $K / \bQ$ is an imaginary quadratic extension then we can find (using e.g.\ \cite[Lemma A.2.5]{BLGGT}) a (de Rham) character $\omega : G_K \to \overline{\bQ}_p^\times$ such that $\omega \omega^c = (\det r_{\pi, \iota} \epsilon^{-1})^{n-1}$. Then the representation $\rho = \omega \otimes \Sym^{n-1} r|_{G_K}$ satisfies $\rho^c \cong \rho^\vee \otimes \epsilon^{1-n}$, so has the potential to be associated to a RACSDC automorphic representation of $\GL_n(\A_K)$. This means we can use an automorphy lifting theorem adapted to such automorphic representations. (The automorphy of $\rho$ will imply that of $\Sym^{n-1} r_{\pi, \iota}$ by quadratic descent.)

We need to select $\pi$ and  $\iota$ so that the residual representation $\overline{\rho}$ is automorphic. For ``most'' $\iota$ (say, for all but finitely many primes $p$) the image of $\overline{r}_{\pi, \iota}$ will contain a conjugate of $\SL_2(\bF_p)$ and $\Sym^{n-1} \overline{\rho}$ will be irreducible, and it is not clear how to proceed. We therefore want to avoid this generic case. Here we choose $\pi$ and $\iota$ so that there is an isomorphism $\overline{r}_{\pi, \iota} \cong \Ind_{G_K}^{G_\bQ} \overline{\chi}$ for some imaginary quadratic extension $K / \bQ$ and character $\overline{\chi} : G_K \to \overline{\bF}_p^\times$. Then there is an isomorphism
\[ \overline{\rho} \cong \oplus_{i=1}^n \overline{\omega} \overline{\chi}^{n-i} (\overline{\chi}^c)^{i-1}. \]
In particular, this residual representation is highly reducible, being a sum of $n$ characters. Most automorphy lifting theorems in the literature require the residual representation to be irreducible; we will apply \cite[Theorem 1.1]{All19}, an automorphy lifting theorem that does not have this requirement, but that does have some other stringent conditions. These conditions include the requirement that there exist a RACSDC automorphic representation $\Pi$ of $\GL_n(\A_K)$ such that $\overline{r}_{\Pi, \iota} \cong \overline{\rho}$, and satisfying the following:
\begin{itemize}
\item $\Pi$ is $\iota$-ordinary (and so is $\pi$).
\item There exists a prime $l \neq p$ and a place $v | l$ of $K$ such that both $\pi_l$ and $\Pi_v$ are twists of the Steinberg representation (of $\GL_2(\bQ_l)$ and $\GL_n(K_v)$, respectively).
\end{itemize}
It is easy to arrange that the first requirement be satisfied, by choosing $p$ to be a prime which splits in $K$. The second is more difficult. First it requires that $\pi$ is ramified at $l$, whereas we have to this point asked for $\pi$ to be everywhere unramified. We will thus first find a ramified $\pi$ for which $\Sym^{n-1} \pi$ exists, and eventually remove the primes of ramification using the $l$-adic analytic continuation of functoriality results proved in the first part of the paper. The main problem is then to find a $\Pi$ verifying the residual automorphy of $\overline{\rho}$ such that $\Pi_v$ is a twist of the Steinberg representation. This is what will occupy us in \S\S \ref{sec:automorphic_level_raising} --  \ref{sec:galois_level_raising} below. The above argument is then laid out carefully in \S \ref{sec_existence_of_seed_points} in order to finally prove Theorem \ref{introthm_existence_of_single_symmetric_power}.

Here is how we get our hands on $\Pi$. By choosing an appropriate lift of the character $\overline{\chi}$, we can choose characters $X_1, \dots, X_n : K^\times \backslash \A_K^\times \to \bC^\times$ such that $\Pi_0 = X_1 \boxplus \dots \boxplus X_n$ is a regular algebraic and conjugate self-dual (although not cuspidal!) automorphic representation of $\GL_n(\A_K)$ whose associated residual representation is $\overline{\rho}$. If $G$ is a definite unitary group in $n$ variables associated to the extension $K / \bQ$, quasi-split at finite places, then we can hope that $\Pi_0$ transfers to an automorphic representation of $G(\A_\bQ)$. There is a slight wrinkle here in that such a group $G$ does not exist if $n$ is even, and even in the case that $n$ is odd there is a potential obstruction to the existence of this transfer given, at least conjecturally, by Arthur's multiplicity formula. Both of these obstacles can be avoided by replacing $\bQ$ with a suitable soluble totally real extension $F / \bQ$. In order to avoid introducing additional notation in this sketch, we pretend they can be dealt with already in the case $F = \bQ$. (Actually, we will find it convenient to take $\Pi_0$ to be the box sum of two cuspidal automorphic representations of $\GL_2(\A_K)$ and $\GL_{n-2}(\A_K)$, respectively. This means that the final form of the proof of Theorem \ref{introthm_existence_of_single_symmetric_power} will be a kind  of induction on $n$.)

We thus find ourselves with an automorphic representation $\Sigma_0$ of $G(\bA_\bQ)$, whose base change (in the sense of Theorem \ref{thm_base_change}) is $\Pi_0$. Say for the sake of argument that $l$ splits in $K$, so that we can identify $G(\bQ_l)$ with $\GL_n(K_v)$. If we can find another automorphic representation $\Sigma_1$ of $G(\A_\bQ)$, congruent modulo $p$ to $\Sigma_0$ and such that $\Sigma_{1, v}$ is a twist of the Steinberg representation, then we will have solved our problem. We can therefore now focus on this problem of level-raising for the definite unitary group $G$.

There are differing approaches to this problem in the literature. First there is the purely automorphic approach, pioneered by Ribet for $\GL_2(\bA_\bQ)$ \cite{Rib84}. Some generalisations to higher rank groups of this statement do exist (see for example \cite{Tho14}), but nothing that is applicable in the level of generality considered here. Then there is the purely Galois theoretic approach, based on the powerful automorphy lifting theorems which are now available for Galois representations in arbitrary rank (see for example \cite{gee061}). We can not directly apply such results here because the only automorphy lifting theorems applicable in the residually reducible case (namely those of \cite{All19}) require the existence of at least one place at which the starting automorphic representation is sufficiently non-degenerate. 

We solve the problem here by combining aspects of these approaches. A similar combination of techniques is used in the paper \cite{ctiii}: the idea is to first use an automorphic technique to replace $\Sigma_0$ with a representation $\Sigma_0'$ such that $\Sigma'_{0, v}$ is so ramified that, in conjunction with other conditions in place, $\Sigma_0'$ is forced to be stable (i.e. its base change is cuspidal). (The possibility of doing this is the reason for choosing $\Pi_0$ to in fact be a box sum of two cuspidal factors, as mentioned above.) This situation is reflected in the deformation theory, where one finds that (in the big ordinary Galois deformation ring) the locus of reducible deformations is small enough that something like the techniques of \cite{gee061} can be applied to construct an automorphic lift of $\rhobar$ with the required local properties. These Galois theoretic arguments are carried out in \S \S \ref{sec_finiteness_result}, \ref{sec:galois_level_raising}. 

What remains to be explained then is the automorphic level-raising technique developed in \S \ref{sec:automorphic_level_raising}. The approach to creating congruences here is based on types. We recall (using the language of \cite{MR1643417}) that if $\mathfrak{s}$ is an inertial equivalence class of $G(\bQ_l)$ (i.e.\ a supercuspidal representation of a Levi factor of $G(\bQ_l)$, up to unramified twist), then an $\mathfrak{s}$-type is a pair $(U, \tau)$, where $U$ is an open compact subgroup of $G(\bQ_l)$ and $\tau$ is an irreducible finite-dimensional representation of $U$ such that for each irreducible admissible representation $\sigma_v$ of $G(\bQ_l)$, the supercuspidal support of $\sigma_v$ is in class $\mathfrak{s}$ if and only if $\sigma_v|_U$ contains $\tau$. 

It can sometimes happen that two inertial equivalence classes $\mathfrak{s}, \mathfrak{s}'$ admit types $(U, \tau)$ and $(U', \tau')$ with the property that $U = U'$, the reduction modulo $p$ $\overline{\tau}$ of $\iota^{-1} \tau$ is irreducible, and the reduction modulo $p$ $\overline{\tau}'$ of $\iota^{-1} \tau'$ contains $\overline{\tau}$ as a Jordan--H\"older factor. This situation might be called a congruence of types. If this is the case then the theory of algebraic modular forms implies that any automorphic representation $\Sigma$ of $G(\bA_\bQ)$ such that $\Sigma_l$ is of type $\mathfrak{s}$ is congruent to another $\Sigma'$ such that $\Sigma'_l$ is of type $\mathfrak{s}'$. The existence of such global congruences is explained in \cite[\S 3]{Vig01}. It gives an efficient way to construct congruences between automorphic representations $\Sigma, \Sigma'$ such that $\Sigma_l$, $\Sigma'_l$ are in different inertial equivalence classes, although it is not usually possible to change the Levi subgroup underlying the inertial equivalence class. Since $G(\bQ_l) \cong \GL_n(K_v)$ and the initial representation $\Sigma_0$ is certainly not supercuspidal at $l$, it is not immediately clear how to use this.

We therefore instead introduce an auxiliary imaginary quadratic extension $E / \bQ$ in which $l$ is inert, as well as an associated definite unitary group $G'$, and carry out the first step of the automorphic part of the level-raising argument using algebraic modular forms on $G'$. The importance of the group $G'$ is that there are conjugate self-dual irreducible admissible representations of $\GL_3(E_l)$ which are not supercuspidal, but for which the associated $L$-packets of representations of  $U_3(\bQ_l)$ contain supercuspidal elements. For carefully chosen local data, we can find use the method of types to find congruences to supercuspidal representations of $U_3(\bQ_l)$ whose base change to $\GL_3(E_l)$ is supercuspidal. We have already constructed such congruences of types  in \S \ref{subsec_local_endoscopic_classification}. In terms of automorphic representations of $\GL_n(\A_E)$, this will allow us to change the Levi subgroup underlying the inertial equivalence class at $l$ from the maximal torus of $\GL_n(E_l)$ to the group $\GL_3 \times \GL_1^{n-3}$. This will be enough for our intended application.

\section{Raising the level -- automorphic forms}\label{sec:automorphic_level_raising}

Let $n = 2k + 1 \geq 3$ be an odd integer, and let $F, p, S, G = G_n$ etc.\ be as in our standard assumptions (see \S \ref{sec_definite_unitary_groups}). Suppose given cuspidal, conjugate self-dual automorphic representations $\pi_2$ of $\GL_2(\A_F)$ and 
$\pi_{n-2}$ of $\GL_{n-2}(\A_F)$ with the following properties:
\begin{enumerate}
	\item $\pi = \pi_2 \boxplus\pi_{n-2}$ is regular algebraic.
\end{enumerate}
 Consequently, $\pi_2  | \cdot |^{(2-n)/2}$ and $\pi_{n-2}$ are regular algebraic and the representations $\overline{\rho}_2 = \overline{r}_{\pi_2| \cdot |^{(2-n)/2}, \iota}$, $\overline{\rho}_{n-2} = \overline{r}_{\pi_{n-2}|\cdot|^{-1}, \iota}$ are defined. We set $\overline{\rho} = \overline{r}_{\pi, \iota} = \overline{\rho}_2 \oplus \overline{\rho}_{n-2}$. Moreover, $\pi_2$ and $\pi_{n-2}$ are both tempered, cf. the remark after Corollary \ref{cor:galois_existence}.
 \begin{enumerate} 
 \setcounter{enumi}{1}
	\item $\pi$ is $\iota$-ordinary.
	\item We are given disjoint, non-empty sets $T_1, T_2, T_3$ of places of $F^+$  with the following properties:
	\begin{enumerate}
	\item For all $v \in T = T_1 \cup T_2 \cup T_3$, $v \not\in S$ and $q_v$ is odd. The representation $\pi$ is unramified outside $S \cup T$. If $\wv$ is a place of $F$ lying above a place in $T$ then, as in \S \S \ref{subsec_local_endoscopic_classification}, \ref{subsec_local_theory_GL_n}, we write $\omega(\wv) : k(\wv)^\times \to \bC^\times$ for the unique quadratic character.
	\item For each $v \in T_1$, $v$ is inert in $F$, $q_v \text{ mod }p$ is a primitive $6^\text{th}$ root of unity, and the characteristic of $k(v)$ is greater than $n$. There are characters $\chi_\wv, \chi_{\wv, 0}, \chi_{\wv, 1}, \dots, \chi_{\wv, 2k-2} : F_\wv^\times \to \bC^\times$ such that $\chi_\wv, \chi_{\wv, 0}$ are unramified and for each $i = 1, \dots, 2k-2$, $\chi_{\wv, i}|_{\cO_{F_\wv}^\times} = \omega(\wv)$. We have $\pi_{2, \wv} \cong \St_2(\chi_\wv)$ and $\pi_{n-2, \wv} \cong \boxplus_{i=0}^{2k-2} \chi_{\wv, i}$.
	\item For each $v \in T_2$, $v$ splits $v = \wv \wv^c$ in $F$, $q_v \text{ mod }p$ is a primitive $2^\text{nd}$ root of unity, $\pi_{2, \wv}|_{\GL_{2}(\cO_{F_\wv})}$ contains $\widetilde{\lambda}(\wv, \Theta_\wv)$ (for some order $p$ character $\Theta_\wv$ as in \S \ref{subsec_local_theory_GL_n}, with $n_1 = 2$), and $\pi_{n-2, \wv}|_{\GL_{n-2}(\cO_{F_\wv})}$ contains $\omega(\wv) \circ \det$. Thus $\pi_\wv$ satisfies the equivalent conditions of Proposition \ref{prop_types_for_general_linear_group}, and $\pi_\wv|_{\q_\wv}$ contains $\widetilde{\lambda}(\wv, \Theta_\wv, n)$.
	\item  For each $v \in T_3$, $v$ splits $v = \wv \wv^c$ in $F$, $q_v \text{ mod }p$ is a primitive $(n-2)^\text{th}$ root of unity, $\pi_{n-2, \wv}|_{\GL_{n-2}(\cO_{F_\wv})}$ contains $\widetilde{\lambda}(\wv, \Theta_\wv)$ (for some order $p$ character $\Theta_\wv$ as in \S \ref{subsec_local_theory_GL_n}, with $n_1 = n-2$), and $\pi_{2, \wv}|_{\GL_{2}(\cO_{F_\wv})}$ contains $\omega(\wv) \circ \det$. Thus $\pi_\wv$ satisfies the equivalent conditions of Proposition \ref{prop_types_for_general_linear_group}, and $\pi_\wv|_{\q_\wv}$ contains $\widetilde{\lambda}(\wv, \Theta_\wv, n)$.
	\end{enumerate}
\end{enumerate}
Let $\widetilde{T} = \{ \wv \mid v \in T \}$. We fix for each $v \in T_1$ a character $\theta_v : C(k(v)) \to \bC^\times$ of order $p$ (notation as in Proposition \ref{prop_supercuspidals_on_U_3}). In the rest of this section, we will prove the following theorem. 
\begin{theorem}\label{thm_automorphic_level_raising}
	With hypotheses as above, let $L^+ / F^+$ be a totally real $S \cup T$-split quadratic extension, and let $L = L^+F$. 
	Then there exists a RACSDC automorphic representation $\Pi$ of 
	$\GL_n(\A_{L})$ with the following properties:
	\begin{enumerate}
		\item $\Pi$ is $\iota$-ordinary, and unramified at any place not dividing $S \cup T$.
		\item $\overline{r}_{\Pi, \iota} \cong \overline{r}_{\pi, \iota}|_{G_L}$.
		\item For each place $v \in T_{1, L}$, $\Pi_\wv|_{\mathfrak{r}_\wv}$ contains the representation $\widetilde{\lambda}(\wv, \widetilde{\theta}_{v|_{F^+}}, n)|_{\mathfrak{r}_\wv}$ (thus satisfying the equivalent conditions of Proposition \ref{prop_cuspidal_types_for_general_linear_group}).
	\item For each place $v \in T_{2, L} \cup T_{3, L}$, $\Pi_\wv|_{\q_\wv}$ contains the representation $\widetilde{\lambda}(\wv, \Theta_{\wv|_F}, n)$ (thus satisfying the equivalent conditions of Proposition \ref{prop_types_for_general_linear_group} with $n_2 = 2$ if $v \in T_{2, L}$ and $n_2 = n-2$ if $v \in T_{3, L}$).
	\end{enumerate}
\end{theorem}
\begin{remark}
	The places $v \in T_2 \cup T_3$ play a role in ensuring that $\Pi$ is 
	cuspidal (using Lemma \ref{lem_cuspidal_parts} below). Our set-up is 
	adapted to the proof of Proposition \ref{prop:levelraisingoddn}, which uses 
	an induction on the dimension to construct automorphic representations with 
	an unramified twist of Steinberg local factor which are congruent to some 
	very special odd-dimensional symmetric powers.
\end{remark}

We begin with two important observations.
\begin{lemma}\label{lem_local_representations}
Let $v$ be a finite place of $F^+$ which is inert in $F$. Then $\pi_\wv \in \mathcal{A}^\theta_t(\GL_n(F_\wv))_+$.
\end{lemma}
\begin{proof}
The representation $\pi_\wv$ is tempered because both $\pi_{2, \wv}$ and $\pi_{n-2, \wv}$ are tempered. By the main theorem of \cite{belchen}, $r_{\pi, \iota}$ extends to a homomorphism $r : G_{F^+} \to \cG_n(\overline{\bQ}_p)$ such that $\nu \circ r = \epsilon^{1-n} \delta_{F / F^+}^n$. Restricting to $W_{F^+_v}$ and twisting by an appropriate character, we see that the Langlands parameter of $\pi_\wv$ extends to a parameter $W_{F^+_v} \times \SL_2 \to {}^L G$.
\end{proof}
\begin{remark}A consequence of this lemma is that for $v \in T_1$ the character 
$\chi_{\wv}$ is non-trivial 
quadratic and the character $\chi_{\wv,0}$ is trivial.\end{remark}

\begin{lemma}\label{lem_cuspidal_parts}
Suppose given a partition $n = n_1 + \dots + n_r$ and cuspidal, conjugate self-dual automorphic representations $\pi'_1, \dots, \pi'_r$ of $\GL_{n_i}(\bA_F)$ such that $\pi' = \pi'_1 \boxplus \dots \boxplus \pi'_r$ is regular algebraic. Suppose moreover that the following conditions are satisfied:
\begin{enumerate}
\item There is an isomorphism $\overline{r}_{\pi', \iota} \cong \overline{r}_{\pi, \iota}$.
\item If $v \in T_2 \cup T_3$ then $\pi'_\wv|_{\q_\wv}$ contains $\widetilde{\lambda}(\wv, \Theta_\wv, n)$.
\end{enumerate}
Then one of the following two statements holds:
\begin{enumerate}
\item  We have $ r= 1$, $n_1 = n$, and so $\pi'$ is cuspidal.
\item After re-ordering we have $r = 2$, $n_1 = n-2$, $n_2 = 2$. If $v \in T_2$ 
then $\pi'_{1, \wv}|_{\GL_{n-2}(\cO_{F_\wv})}$ contains $\omega(\wv) \circ 
\det$ and $\pi'_{2, \wv}|_{\GL_2(\cO_{F_\wv})}$ contains 
$\widetilde{\lambda}(\wv, \Theta_\wv)$, while if $v \in T_3$ then $\pi'_{1, 
\wv}|_{\GL_{n-2}(\cO_{F_\wv})}$ contains $\widetilde{\lambda}(\wv, \Theta_\wv)$ 
and $\pi'_{2, \wv}|_{\GL_2(\cO_{F_\wv})}$ contains $\omega(\wv) \circ \det$. We 
have isomorphisms of semisimplified residual representations 
$\overline{r}_{\pi'_1|\cdot|^{-1}, \iota}\cong \overline{\rho}_{n-2}$ and 
$\overline{r}_{\pi'_2| \cdot |^{(2-n)/2}, \iota} \cong \rhobar_2$.
\end{enumerate}
\end{lemma}
\begin{proof}
Before beginning the proof, we observe that the representations $\overline{\rho}_2$, $\overline{\rho}_{n-2}$ have the following properties:
\begin{itemize}
\item If $v \in T_2$, then $\overline{\rho}_2|_{G_{F_\wv}}^{ss}$ is unramified and $\overline{\rho}_{n-2}|_{G_{F_{\wv}}}^{ss}$ is unramified after twisting by a ramified quadratic character. (The character $\Theta_\wv$ has order $p$.)
\item If $v \in T_3$, then $\overline{\rho}_{2}|_{G_{F_{\wv}}}^{ss}$ is unramified after twisting by a ramified quadratic character and $\overline{\rho}_{n-2}|_{G_{F_{\wv}}}^{ss}$ is unramified.
\end{itemize}
We can suppose without loss of generality that $r \geq 2$. Fix places $v_2 \in T_2, v_3 \in T_3$. By Proposition 
\ref{prop_types_for_general_linear_group}, we can assume after relabelling that $\pi'_{1, \wv_3}|_{\GL_{n_1}(\cO_{F_{\wv_3}})}$ contains $\widetilde{\lambda}(\wv_3, \Theta_{\wv_3}, n_1)$ (and in particular, $n_1 \geq n-2$), in which case $(\pi'_{2, \wv_3} \boxplus \dots \boxplus \pi'_{r, \wv_3})|_{\GL_{n - n_1}(\cO_{F_{\wv_3}})}$ contains $\omega(\wv_3) \circ \det$. There is an isomorphism 
\[ \overline{r}_{\pi', \iota}  = \oplus_{i=1}^r \overline{r}_{\pi'_i | 
\cdot|^{(n_i -  n)/2}, \iota} \cong \overline{\rho}_2 \oplus 
\overline{\rho}_{n-2}, \]
 hence
\[ \oplus_{i=1}^r \overline{r}_{\pi'_i | \cdot|^{(n_i -  n)/2}, 
\iota}|_{G_{F_{\wv_3}}}^{ss} \cong \overline{\rho}_2|_{G_{F_{\wv_3}}}^{ss} 
\oplus \overline{\rho}_{n-2}|_{G_{F_{\wv_3}}}^{ss}. \]
Since $\oplus_{i=2}^r \overline{r}_{\pi'_i | \cdot|^{(n_i -  n)/2}, 
\iota}|_{G_{F_{\wv_3}}}^{ss}$ contains no unramified subrepresentation, we 
conclude that $\overline{r}_{\pi'_1 | \cdot|^{(n_1 -  n)/2}, \iota}$ contains 
$\overline{\rho}_{n-2}$ as a subrepresentation.

We now look at the place $v_2$. There are two possibilities for the 
representation $\pi'_{1, \wv_2}$: either $\pi'_{1, 
\wv_2}|_{\GL_{n_1}(\cO_{F_{\wv_2}})}$ contains $\widetilde{\lambda}(\wv_2, 
\Theta_{\wv_2}, n_1)$, or it contains $\omega(\wv_2) \circ \det$. We claim that 
the first possibility does not occur. Indeed, in this case arguing as above 
shows that $\oplus_{i=2}^r \overline{r}_{\pi'_i | \cdot|^{(n_i -  n)/2}, 
\iota}|_{G_{F_{\wv_2}}}^{ss}$ contains no unramified subrepresentation, and 
therefore that $\overline{\rho}_2$ is a subrepresentation of 
$\overline{r}_{\pi'_1 | \cdot|^{(n_1 -  n)/2}, \iota}$. This forces $n_1 = n$ 
and $r = 1$, a contradiction. Therefore we must have $r = 2$, $n_2 = 2$,  and 
$\pi'_{2, \wv_2}|_{\GL_{2}(\cO_{F_{\wv_2}})}$ contains 
$\widetilde{\lambda}(\wv_2, \Theta_{\wv_2})$. Since $v_2, v_3$ were arbitrary, 
this  completes the proof.
\end{proof}
We now commence the proof of the theorem. Let $U = \prod_v U_v \subset G(\A_{F^+}^\infty)$ be an open compact subgroup with the following properties:
\begin{itemize}
	\item For each $v \in S$, $\pi_\wv^{\iota_\wv(U_v)} \neq 0$.
	\item If $v \not\in S \cup T$, then $U_v = G(\cO_{F^+_v})$.
	\item If $v \in S_p$, then $U_v = \iota_\wv^{-1} \Iw_{\wv}(c, c)$ for some 
	$c \geq 1$ such that $\pi_\wv^{\Iw_\wv(c, c),ord} \neq 0$ (notation 
	as in \cite[\S 5.1]{ger}) and $U_v$ contains no non-trivial torsion elements (note this implies that $U$ is sufficiently small).
	\item If $v \in T_1$ then $U_v = \iota_v^{-1} \p_v$ (notation as in \S \ref{subsec_local_endoscopic_classification}).
	\item If $v \in T_2 \cup T_3$ then $U_v = \iota_\wv^{-1} \q_\wv$ (notation as in \S \ref{subsec_local_theory_GL_n}, defined with $n_1 = 2$ if $v \in T_2$ and $n_1 = n-2$ if $v \in T_3$).
\end{itemize}
We define $\tau_g = \otimes_{v \in T_1} \tau(v, n) \otimes_{v \in T_2 \cup T_3} \widetilde{\lambda}(\wv, \Theta_\wv, n)$, where $\tau(v, n)$, $\widetilde{\lambda}(\wv, \Theta_{\wv}, n)$ are the representations of $\p_v$, $\q_\wv$ defined in \S \ref{subsec_local_endoscopic_classification}, \S \ref{subsec_local_theory_GL_n} respectively. Thus $\tau_g$ is an irreducible $\bC[U_T]$-module, which we view as a $\bC[U]$-module by projection to the $T$-component. Similarly we define $\lambda_g = \otimes_{v \in T_1} \lambda(v, \theta_v, n) \otimes_{v \in T_2 \cup T_3} \widetilde{\lambda}(\wv, \Theta_\wv, n)$. Fixing a sufficiently large coefficient field, we can choose $\cO$-lattices $\mathring{\tau}_g$ and $\mathring{\lambda}_g$ in $\iota^{-1} \tau_g$ and $\iota^{-1} \lambda_g$, respectively.

If $L^+ / F^+$ is an $S \cup T$-split totally real quadratic extension, then we define an open compact subgroup $U_L = \prod_v U_{L, v} \subset G(\A_{L^+}^\infty)$ and representations $\tau_{g, L}$, $\lambda_{g, L}$ by the same recipe (where we now replace the sets $S$, $T_i$ by their lifts $S_L$, $T_{i, L}$ to $L^+$).
\begin{proposition}\label{prop_automorphic_descent} Let $L^+ / F^+$ be an $S \cup T$-split totally real quadratic extension and let $L = L^+ F$. Then either there exists an automorphic representation $\sigma$ of $G(\A_{F^+})$ with the following properties:
	\begin{enumerate}
		\item $\pi$ is the base change of $\sigma$ (cf. Theorem \ref{thm_base_change});
		\item For each place $v\not\in T$, $\sigma_v^{U_v} \neq 0$;
		\item $\sigma_T|_{U_T}$ contains $\tau_g$.
	\end{enumerate}
	or there exists an automorphic representation $\sigma$ of $G(\A_{L^+})$ with the following properties:
	\begin{enumerate}
		\item Let $\pi_L$ denote the base change of $\pi$ with respect to the quadratic extension $L / F$. Then $\pi_L$ is the base  change of $\sigma$;
		\item For each place $v\not\in T_L$, $\sigma_v^{U_{L, v}} \neq 0$;
		\item $\sigma_{T_L}|_{U_{L, T_L}}$ contains $\tau_{g, L}$.
	\end{enumerate}
\end{proposition}
\begin{proof}
By \cite[Th\'eor\`eme 5.1]{labesse}, there is an identity
\begin{equation}\label{eqn_labesse_unstable_trace} T^G_{disc}(f) = \sum_H \iota(G, H) T_{disc}^{\widetilde{M}^H}(\widetilde{f}^H), 
\end{equation}
for any $f = f^\infty \otimes f_\infty \in C_c^\infty(G(\A_{F^+}))$ such that $f_\infty$ is a pseudocoefficient of discrete series. Here the sum on the right-hand side is over representatives for equivalence classes of endoscopic data for $G$, represented here by the associated endoscopic group $H$ (recall that we have fixed representative endoscopic triples in \S \ref{sec_definite_unitary_groups}). The coefficients $\iota(G, H)$ are given in \cite[Proposition 4.11]{labesse}, while the expression $T_{disc}^{\widetilde{M}^H}(\widetilde{f}^H)$ is given in \cite[Proposition 3.4]{labesse} as a formula
\begin{equation}\label{eqn_labesse_twisted_trace} \sum_{L \in \cL^0 / W^{M^H}} 
\sum_{s \in W^{\widetilde{M}^H}(L)_{reg}} \sum_{\widetilde{\pi}^L \in 
\Pi_{disc}(\widetilde{L}_s)} (|\det(s - 1 \mid \mathfrak{a}_L / 
\mathfrak{a}_{M^H})| | W^{M^H}(L) |)^{-1}\tr 
I_Q(\widetilde{\pi}^L)(\widetilde{f}^H), 
\end{equation}
where (summarizing the notation of \emph{op. cit.}):
\begin{itemize}
\item $\widetilde{M}^H$ is a twisted space on a Levi of $\Res_{F/F^+}\GL_n$, 
as in \S\ref{subsec:endo};
\item $\cL^0$ is the set of standard Levi subgroups of $M^H$;
\item $W^{\widetilde{M}^H}(L)_{reg}$ is the quotient by the Weyl group $W^L$ of 
the set of elements $s$ in the twisted Weyl group $W^{\widetilde{M}^H} = 
W^{M^H} \rtimes \theta_{M^H}$ which normalise $L$ and such that $\det(s - 1 
\mid \mathfrak{a}_L / \mathfrak{a}_{M^H}) \neq 0$, where $\mathfrak{a}_?$ 
denotes the Lie algebra of the maximal $\bQ$-split subtorus of the centre of a 
reductive group.
\item $\Pi_{disc}(\widetilde{L}_s)$ is the set of isomorphism classes of 
irreducible representations of the twisted space $\widetilde{L}_s(\bA_{F^+})$ 
which appear as subrepresentations of the discrete spectrum of $L$. 
\item $I_Q(\widetilde{\pi}^L)(\widetilde{f}^H)$ is a certain intertwining 
operator, with $Q$ a parabolic subgroup with Levi $L$.
\end{itemize}
We fix our choice of $f_\infty$ so that it only has non-zero traces on 
representations of $G(F^+_\infty)$ whose infinitesimal character is related, by 
twisted base change, to that of $\pi$. The argument of \cite[Proposition 
4.8]{shin} then shows that for each $L \in \cL^0 / W^{M^H}$, there is at most 
one element $s \in W^{\widetilde{M}^H}(L)_{reg}$ for which the corresponding 
summand in (\ref{eqn_labesse_twisted_trace}) can be non-zero (and a representative for $s$ can be chosen which acts as 
conjugate inverse transpose on each simple factor of $L$). Using Proposition 
\ref{prop_transfer_at_finite_places}, linear independence of characters, the 
description of the discrete spectrum of general linear groups \cite{Moe89}, and 
the Jacquet--Shalika theorem \cite{MR623137}, we can combine 
(\ref{eqn_labesse_unstable_trace}) and (\ref{eqn_labesse_twisted_trace}) to 
obtain a refined identity
\begin{equation}
\sum_i m(\sigma_i) \sigma_i(f) = \frac{1}{2}\left( 
\widetilde{\pi}(\widetilde{f}) + (\pi_{n-2} \otimes (\pi_2 \otimes \mu^{-1} 
\circ \det))^\sim (\widetilde{f}^{U_{n-2}\times U_{2}})\right),
\end{equation}
where:
\begin{itemize}
\item The sum on the left-hand side is over the finitely many automorphic representations $\sigma_i$ of $G(\bA_{F^+})$ which are unramified at all places below which $\pi$ is unramified, have infinitesimal character related to that of $\pi_\infty$ by twisted base change, and which are related to $\pi$ by (either unramified or split) base change at places $v \not\in T_1$ of $F^+$, each occurring with its multiplicity $m(\sigma_i)$.
\item The twisted traces on the right-hand side are Whittaker-normalised. (These two terms arise from $H = U_n, L = \Res_{F / F^+} \GL_{n-2} \times \GL_2$ and $H = U_{n-2} \times U_2$, $L = M^H$, respectively. The same argument as in \cite[Proposition 3.7]{labesse} shows that Arthur's normalisation of the twisted trace, implicit in the term $I_Q(?)$ of (\ref{eqn_labesse_twisted_trace}), agrees with the Whittaker normalisation on the corresponding terms.)
\end{itemize}
We remark that the representations $\pi_2$, $\pi_{n-2}$ are tempered and that 
the representations $\sigma_i^{T_1}$ (i.e.\ prime to $T_1$-part) are 
isomorphic.  If $v \in T_1$, then we can find (combining Proposition 
\ref{prop_L-packets_via_character_identities} for $U_{n-2} \times U_2$ and 
e.g.\ \cite[Proposition 4.6]{Hir04}) a finite set $\{ \lambda_{v, i} \}$ of 
irreducible admissible representations of $G(F_v^+)$ and scalars $d_{v, i} \in 
\bC$ such that $(\pi_{n-2, v} \otimes (\pi_{2, v} \otimes \mu_v^{-1} \circ 
\det))^\sim(\widetilde{f}_v^{U_{n-2} \times U_2}) = \sum_i d_{v, i} \lambda_{v, 
i}(f_v)$. By Proposition \ref{prop_archimedean_signs}, Proposition 
\ref{prop_transfer_at_finite_places} and Proposition 
\ref{prop_L-packets_via_character_identities}, we therefore have an identity: 
\[\begin{split} \sum_i m(\sigma_i) \sigma_{i, T_1}(f_{T_1}) & = \frac{1}{2} \left( \prod_{v | \infty} \epsilon(v, U_n, \varphi_{U_n}) \prod_{v \in T_1} \sum_{\tau \in \Pi(\pi_v)} c_\tau \tau(f_v) \right. \\ & \left. + \prod_{v | \infty} \epsilon(v, U_{n-2} \times U_2, \varphi_{U_{n-2} \times U_2}) \prod_{v \in T_1} \sum_{i} d_{v, i} \lambda_{v, i}(f_v) \right), \end{split} \]
Choose for each $v \in T_1$ a representation $\tau_v \in \Pi(\pi_v)$ such that $\tau_v|_{U_v}$ contains $\tau(v, n)$ (this is possible by Corollary \ref{cor_descent_L_packet_contains_tau} and Proposition \ref{prop_moy_prasad}). We can assume that for each $v \in T_1$, $\lambda_{v, 1} = \tau_v$ (possibly with $d_{v, 1} = 0$). We conclude that there is at most one automorphic representation $\sigma$ of $G(\A_{F^+})$ with the following properties:
\begin{itemize}
\item $\sigma$ is unramified outside $S \cup T$, and is related to $\pi$ by split or unramified base change at all places $v \not\in T_1$;
\item If $v \in T_1$, then $\sigma_v \cong \tau_v$.
\end{itemize}
The representation $\sigma$ occurs with multiplicity 
\[ m(\sigma) = \frac{1}{2}\left(\prod_{v | \infty} \epsilon(v, U_n, \varphi_{U_n}) \prod_{v \in T_1} c_{\tau_v} + \prod_{v | \infty} \epsilon(v, U_{n-2} \times U_2, \varphi_{U_{n-2} \times U_2})  \prod_{v \in T_1} d_{v, 1}\right). \] We note that the numbers $c_{\tau_v}$ are all non-zero, by Proposition \ref{prop_L-packets_via_character_identities}. If $m(\sigma)$ is non-zero, then we're done (we are in the first case in the statement of the proposition). Otherwise, $\prod_{v \in T_1} c_{\tau_v}^2 = \prod_{v \in T_1} d_{v, 1}^2$, which we now assume.

In this case, let $L^+ / F^+$ be a totally real quadratic $S \cup T$-split extension, let $L = L^+ F$, and let $\pi_L$ denote the base change of $\pi$ with respect to the quadratic extension $L / F$. If $v \in T_{1, L}$, let $\tau_v = \tau_{v|_{F^+}}$. Then repeating the same argument shows that there is at most one automorphic representation $\sigma$ of $G(\A_{L^+})$ with the following properties:
\begin{itemize}
\item $\sigma$ is unramified outside $S_L \cup T_L$, and is related to $\pi_L$ by split or unramified base change at all places $v \not\in T_{1, L}$;
\item If $v \in T_{1, L}$, then $\sigma_v \cong \tau_v$.
\end{itemize}
Using the remark after Proposition \ref{prop_archimedean_signs}, we see that the representation $\sigma$ occurs with multiplicity 
\[ m(\sigma) =  \frac{1}{2}\left(\prod_{v \in T_{1, L}} c_{\tau_v} + \prod_{v \in T_{1, L}} d_{v, 1}\right) = \prod_{v \in T_1} c_{\tau_v}^2. \]
This is non-zero, so we're done in this case also (and we are in the second case of the proposition).
\end{proof}
We now show how to complete the proof of Theorem \ref{thm_automorphic_level_raising}, assuming first that we are in the first case of Proposition \ref{prop_automorphic_descent}. We let $\sigma$ be the automorphic representation of $G(\A_{F^+})$ whose existence is asserted by Proposition \ref{prop_automorphic_descent}. Let $\lambda \in (\Z^n_+)^{\Hom(F, \overline{\Q}_p)}$ be such that $\sigma$ contributes to $S_\lambda^{ord}(U, \iota^{-1} \tau_g)$ under the isomorphism of Lemma \ref{lem_algebraic_automorphic_forms_are_classical}. Let $\T \subset \End_\cO(S_\lambda^{ord}(U, \iota^{-1} \tau_g))$ be the commutative $\cO$-subalgebra generated by unramified Hecke operators $T_w^j$ at split places $v = w w^c \not\in S$ of $F^+$, and let $\ffrm \subset \T$ be the maximal ideal determined by $\sigma$. 

Then $S^{ord}_\lambda(U, \mathring{\tau}_g \otimes_\cO k)_\ffrm$ is non-zero, 
by Lemma \ref{lem_sufficiently_small}, hence (using the exactness of 
$S^{ord}_\lambda(U, -)$ as a functor on $k[U]$-modules, together with 
Proposition \ref{prop_congruence_of_types}) $S^{ord}_\lambda(U, 
\mathring{\lambda}_g \otimes_\cO k)_\ffrm \neq 0$, hence $S^{ord}_\lambda(U, 
\mathring{\lambda}_g)_\ffrm \neq 0$. Applying Lemma 
\ref{lem_algebraic_automorphic_forms_are_classical} once again, we conclude the 
existence of an automorphic representation $\Sigma$ of $G_n(\A_{F^+})$ with the 
following 
properties:
\begin{itemize}
	\item $\overline{r}_{\Sigma, \iota} \cong \overline{r}_{\pi, \iota}$.
	\item $\Sigma_T|_{U_T}$ contains $\lambda_g$.
	\item $\Sigma$ is $\iota$-ordinary and is unramified outside $S \cup T$.
\end{itemize}
Let $\Pi$ denote the base change of $\Sigma$, let $L^+ / F^+$ be a quadratic totally real extension as in the statement of Theorem \ref{thm_automorphic_level_raising}, and let $\Pi_L$ denote base change of $\Pi$ with respect to the extension $L / F$.  We claim that $\Pi_L$ satisfies the requirements of Theorem \ref{thm_automorphic_level_raising}. The only points left to check are that $\Pi_L$ is cuspidal and that if $v \in T_{1, L}$ then $\Pi_{L, v}$ satisfies condition (3) in the statement of Theorem \ref{thm_automorphic_level_raising}. In fact, it is enough to show that $\Pi$ is cuspidal and that if $v \in T_1$ then $\Pi_\wv|_{\mathfrak{r}_\wv}$ contains $\widetilde{\lambda}(\wv, \widetilde{\theta}_v, n)|_{\mathfrak{r}_\wv}$. We first show that $\Pi$ is cuspidal. If $\Pi$ is not cuspidal, then Lemma \ref{lem_cuspidal_parts} shows that $\Pi =\Pi_{n-2} \boxplus  \Pi_2$ where $\Pi_{n-2}, \Pi_{2}$ are cuspidal, conjugate self-dual automorphic representations of $\GL_{n-2}(\A_F)$, $\GL_{2}(\A_F)$, respectively. Arguing as in proof of Proposition \ref{prop_automorphic_descent}, we obtain an identity 
\begin{equation}\label{eqn_base_change_trace_identity} \sum_i m(\Sigma_i) 
\Sigma_i(f) = \frac{1}{2}\left( \widetilde{\Pi}(\widetilde{f}) + ( \Pi_{n-2} 
\otimes ( \Pi_2 \otimes \mu^{-1} \circ \det) 
)^\sim(\widetilde{f}^{U_{n-2}\times U_2}) \right), 
\end{equation}
where the sum on the left-hand side is over the finitely many automorphic representations $\Sigma_i$ of $G(\bA_{F^+})$ which are unramified at all places below which $\Pi$ is unramified, have infinitesimal character related to that of $\Pi_\infty$ by twisted base change, and which are related to $\Pi$ by (either unramified or split) base change at places $v \not\in T_1$ of $F^+$.

Fix $v \in T_1$, and consider a test function of the form $f = f_v \otimes f_\infty \otimes f^{v, \infty}$, where:
\begin{itemize}
\item $f_\infty$ is a coefficient for $\Sigma_\infty$.
\item $f_v$ is the test function denoted $\phi$ in the statement of Proposition \ref{prop_kazhdan_varshavsky}.
\item $f^{v, \infty}$ is the characteristic function of an open compact subgroup of $G(\bA_{F^+}^\infty)$.
\item $\Sigma(f) \neq 0$.
\end{itemize}
Then $\Sigma_i(f)$ is non-negative for any $i$, and the left-hand side of 
(\ref{eqn_base_change_trace_identity}) is non-zero. We conclude that at least 
one of the terms $\widetilde{\Pi}(\widetilde{f})$ and $(\Pi_{n-2} \otimes 
(\Pi_2 \otimes \mu^{-1} \circ \det) )^\sim(\widetilde{f}^{U_{n-2}\times U_2})$ 
is non-zero. In either case Proposition \ref{prop_kazhdan_varshavsky} implies 
that the cuspidal support of $\Pi_\wv$, and therefore $\Pi_{n-2, \wv}$, 
contains a supercuspidal representation $\Psi$ of $\GL_3(F_\wv)$ such that the 
semisimple residual representation attached to $\rec^T_{F_\wv}(\iota^{-1}\Psi)$ 
is unramified. This contradicts Lemma \ref{lem_cuspidal_parts}, which implies 
that $\overline{r}_{\Pi_{n-2} | \cdot |^{-1}, \iota}|_{G_{F_\wv}}^{ss}$ is the 
sum of an unramified character and the
twist of an unramified representation by a quadratic ramified character.

Therefore $\Pi$ is cuspidal, and a similar argument now gives an identity
\begin{equation}\label{eqn_base_change_trace_identity_2} \sum_i m(\Sigma_i) 
\Sigma_i(f) = \widetilde{\Pi}(\widetilde{f}).
\end{equation}
With the same choice of test function we have $\widetilde{\Pi}(\widetilde{f}) \neq 0$, so another application of  Proposition \ref{prop_kazhdan_varshavsky} shows that $\Pi_\wv$ has the required property. This completes the proof of Theorem \ref{thm_automorphic_level_raising}, assuming that the first case of Proposition \ref{prop_automorphic_descent} holds. If the second case holds, the argument is very similar, except that there is no need to replace $\Pi$ by its base change with respect to a quadratic extension $L/F$. In either case, this completes the proof.

\section{A finiteness result for Galois deformation rings}\label{sec_finiteness_result}

In this section we prove that certain Galois deformation rings are finite over the Iwasawa algebra (Theorem \ref{thm_finiteness_of_deformation_ring}), and use this to give a criterion for a given deformation to have an irreducible specialization with useful properties (Theorem \ref{thm_existence_of_generic_primes}). These technical results form the basis for the arguments in \S \ref{sec:galois_level_raising}, where we will apply our criterion to the Galois representation valued over a big ordinary Hecke algebra.

The novelty of the results proved in this section is that we assume that the 
residual representation is reducible (in fact, to simplify the exposition we 
assume that this representation is a sum of characters). The main tools are the 
automorphy lifting theorems proved in \cite{All19} and the idea of potential 
automorphy, for which we use \cite{BLGGT} as a reference. The notation and 
definitions we use for Galois deformation theory in the ordinary case are 
summarized in \cite[\S 3]{All19}, and we refer to that paper in particular for 
the notion of local and global deformation problem, and the definitions of the 
particular local deformation problems used below.

Before getting stuck into the details, we record a useful lemma. If $\Gamma$ is a profinite group, $k$ is a field with the discrete topology, and $\overline{\rho} : \Gamma \to \GL_n(k)$ is a continuous representation, we say that $\overline{\rho}$ is primitive if it is not isomorphic to a representation of the form $\Ind_{\Gamma'}^{\Gamma} \overline{\sigma}$ for some finite index proper closed subgroup $\Gamma' \subset \Gamma$ and representation $\overline{\sigma} : \Gamma' \to \GL_{n / [\Gamma : \Gamma']}(k)$. This condition appears as a hypothesis in the automorphy lifting theorem proved in \cite{All19}.
\begin{lemma}\label{lem_primitive_representations}
Suppose that $\overline{\rho} = \overline{\chi}_1 \oplus \dots \oplus \overline{\chi}_n$, for some continuous characters $\overline{\chi}_i : \Gamma \to k^\times$ such that for each $i \neq j$, $\overline{\chi}_i / \overline{\chi}_j$ has order greater than $n$. Then $\overline{\rho}$ is primitive.
\end{lemma}
\begin{proof}
Suppose that there is an isomorphism $\overline{\rho} \cong \Ind_{\Gamma'}^{\Gamma} \overline{\sigma}$. Then Frobenius reciprocity implies that $\overline{\sigma}$ contains each character $\overline{\chi}_i|_{\Gamma'}$. These $n$ characters are distinct: if $\overline{\chi}_i|_{\Gamma'} = \overline{\chi}_j|_{\Gamma'}$, then $(\overline{\chi}_i / \overline{\chi}_j)^{[\Gamma : \Gamma']} = 1$, which would contradict our assumption that $\overline{\chi}_i / \overline{\chi}_j$ has order greater than $n$ if $i \neq j$. Thus $\overline{\sigma}$ must have dimension at least $n$, implying that $\Gamma = \Gamma'$. It follows that $\overline{\rho}$ is primitive. 
\end{proof}
Now let $n \geq 2$ and let $F, S, p$ be as in our standard assumptions (see \S 
\ref{sec_definite_unitary_groups}), and let $E \subset \overline{\Q}_p$ be a 
coefficient field.  We recall the definition 
of the Iwasawa algebra $\Lambda$. If $v \in S_p$, then we write $\Lambda_v = 
\cO \llbracket 
(I^\text{ab}_{F_\wv}(p))^n \rrbracket$, where $I^\text{ab}_{F_\wv}(p)$ denotes 
the inertia subgroup of the Galois group of 
the maximal abelian pro-$p$ extension of $F_\wv$. 
We set $\Lambda = \widehat{\otimes}_{v \in S_p} \Lambda_v$, the completed tensor product 
being over $\cO$. For each $v \in S_p$ and $i=1,\ldots,n$ there is a universal 
character $\psi_v^i : I^\text{ab}_{F_\wv}(p) \to \Lambda_v^\times$. At times we will need to introduce Iwasawa algebras also for extension fields $F' / F$ and for representations of degree $n' \neq n$, in which case we will write e.g.\ $\Lambda_{F', n'}$ for the corresponding Iwasawa algebra, dropping a subscript when either $F' = F$ or $n' = n$.

Let $\mu : G_{F^+, S} \to \cO^\times$ be a continuous character which is de 
Rham and such that $\mu(c_v) = -1$ for each place $v | \infty$ of 
$F^+$. Fix an integer $n \geq 2$, and suppose given characters 
$\overline{\chi}_1, \dots, \overline{\chi}_n : G_{F, S} \to k^\times$ such that 
for each $i = 1, \dots, n$, $\overline{\chi}_i \overline{\chi}_i^c = 
\overline{\mu}|_{G_{F, S}}$. We set $\overline{\rho}  =\oplus_{i=1}^n 
\overline{\chi}_i$; then $\overline{\rho}$ extends to a homomorphism 
$\overline{r} : G_{F^+, S} \to \cG_n(k)$ such that $\nu_{\cG_n} \circ 
\overline{r} = \overline{\mu}$, by setting $\overline{r}(c) = (1_n, 1) \jmath \in \cG_n(k)$. We suppose that for each $v \in S_p$, 
$\overline{r}|_{G_{F_\wv}}$ is trivial.

Let $\Sigma$ be a set of finite places of $F^+$ split in $F$ and disjoint from $S$, and $\widetilde{\Sigma}$ a lift of $\Sigma$ to $F$. If for each $v \in \Sigma$, $q_v \equiv 1 \text{ mod }p$ and $\overline{r}|_{G_{F_\wv}}$ is trivial, then we can define the global deformation problem
\[ \cS_\Sigma = (F / F^+, S \cup \Sigma, \widetilde{S} \cup \widetilde{\Sigma}, \Lambda, \overline{r}, \mu, \{ R_{v}^\Delta \}_{v \in S_p} \cup \{ R_v^\square \}_{v \in S - S_p} \cup \{ R_{v}^{St} \}_{v \in \Sigma} ). \]
(For the convenience of the reader, we summarize the notation from \cite[\S 3]{All19}. Thus the local lifting ring $R_{v}^\Delta$ represents the functor of  ordinary, variable weight liftings; $R_v^\square$ the functor of all liftings; and $R_{v}^{St}$ the functor of Steinberg liftings.) If $\overline{r}$ is Schur, in the sense of \cite[Definition 2.1.6]{cht}, then the corresponding global deformation functor is represented by an object $R_{\cS_\Sigma} \in \cC_\Lambda$. If $\Sigma$ is empty, then we write simply $\cS = \cS_{\emptyset}$. 
\begin{theorem}\label{thm_finiteness_of_deformation_ring}
	Suppose that the following conditions are satisfied:
	\begin{enumerate}
		\item $p > 2n$.
		\item For each $1 \leq i < j \leq n$, $\overline{\chi}_i / 
		\overline{\chi}_j|_{G_{F(\zeta_p)}}$ has order greater than $2n$. (In 
		particular, 
		$\overline{r}$ is Schur.)
		\item $[F(\zeta_p) : F] = p-1$.
		\item $\Sigma$ is non-empty.
	\end{enumerate}
	Then $R_{\cS_\Sigma}$ is a finite $\Lambda$-algebra. 
\end{theorem} 
\begin{proof}
	We will compare $R_{\cS_\Sigma}$ with a deformation ring for Galois 
	representations to $\cG_{2n}$. First, fix a place $v_q$  of $F$ prime to $S 
	\cup \Sigma$, lying above a rational prime $q > 2n$ which splits in $F$. 
	After possibly enlarging $k$, we can find a character $\overline{\psi} : 
	G_{F} \to k^\times$ satisfying the following conditions:
	\begin{itemize}
		\item $\overline{\psi} \overline{\psi}^c = \epsilon^{1-2n} \overline{\mu}|_{G_F}^{-1}$.
		\item For each $v \in S \cup \Sigma$, $\overline{\psi}|_{G_{F_\wv}}$ is unramified.
		\item $q$ divides the order of $\overline{\psi} / 
		\overline{\psi}^c(I_{F_{v_q}})$.
	\end{itemize}
	Using the formulae in \cite[\S 1.1]{BLGGT}, we can write down a character 
	\[ (\overline{\psi}, \epsilon^{1-2n} \overline{\mu}^{-1} \delta_{F / F^+}) : G_{F^+} \to \cG_1(k), \]
	 the tensor product representation 
	 \[ \overline{r} \otimes (\overline{\psi}, \epsilon^{1-2n} \overline{\mu}^{-1} \delta_{F / F^+}) : G_{F^+} \to \cG_n(k), \]
	  which has multiplier $\epsilon^{1-2n}$, and the representations
	\[ \overline{r}_1 = I(\overline{r} \otimes (\overline{\psi}, \epsilon^{1-2n} \overline{\mu}^{-1} \delta_{F / F^+}) ) : G_{F^+} \to \GSp_{2n}(k) \]
	and 
	\[ \overline{r}_2 = \widehat{(\overline{r}_1)}_{G_F} : G_{F^+} \to \cG_{2n}(k).
	\]
	These representations have the following properties:
	\begin{itemize}
		\item The multiplier character of $\overline{r}_1$ equals $\epsilon^{1-2n}$.
		\item The multiplier character $\nu_{\cG_{2n}} \circ \overline{r}_2$ 
		equals $\epsilon^{1-2n}$. 
		\item The representations $\overline{r}_1|_{G_{F}}$ and 
		$\overline{r}_2|_{G_F}$ are both conjugate in $\GL_{2n}(k)$ to 
		$\overline{\rho} \otimes \overline{\psi} \oplus \overline{\rho}^c 
		\otimes \overline{\psi}^c$.
	\end{itemize}
	Let $\overline{\rho}_2 = \overline{r}_2|_{G_{F}}$. We observe that the following conditions are satisfied:
	\begin{itemize}
		\item $\zeta_p \not\in \overline{F}^{\ker \ad \overline{\rho}_2}$ and $F \not\subset F^+(\zeta_p)$.
		\item $\overline{\rho}_2$ is primitive.
		\item The irreducible constituents of $\overline{\rho}_2|_{G_{F(\zeta_p)}}$ occur with multiplicity 1.
	\end{itemize}
	Indeed, the condition $F \not\subset F^+(\zeta_p)$ holds because $[F(\zeta_p) : F] = p-1$. We have $\overline{F}^{\ker \ad \overline{\rho}_2} \subset F(\{ \overline{\chi}_i / \overline{\chi}_j \}_{i \neq j}, \overline{\psi} / \overline{\psi}^c, \{ \overline{\chi}_i \overline{\chi}_j \overline{\mu}^{-1} \}_{i, j} ) = M$, say, and $c$ acts on $\Gal(M / F)$ as $-1$. It follows that $F(\zeta_p) \cap M$ has degree at most 2 over $F$, showing that $\zeta_p \not\in M$. To see that $\overline{\rho}_2$ is primitive, it is enough (by Lemma \ref{lem_primitive_representations}) to show that $\overline{\chi}_i / \overline{\chi}_j$ has order greater than $2n$ if $i \neq j$ and that $(\overline{\chi}_i \overline{\psi}) / (\overline{\chi}_j^c \overline{\psi}^c)$ has order greater than $2n$ for any $i, j$. These properties hold by hypothesis in the first case and since $q > 2n$ in the second. Finally, the constituents of $\overline{\rho}_2$ are, with multiplicity, $\overline{\chi}_1 \otimes \overline{\psi}, \dots, \overline{\chi}_n \otimes \overline{\psi}$, $\overline{\chi}_1^c \otimes \overline{\psi}^c, \dots, \overline{\chi}_n^c \otimes \overline{\psi}^c$. Our hypotheses include the condition that $\overline{\chi}_i\otimes\overline{\psi}|_{G_{F(\zeta_p)}} \neq \overline{\chi}_j \otimes\overline{\psi}|_{G_{F(\zeta_p)}}$ if $i \neq j$. If $\overline{\chi}_i \otimes\overline{\psi}|_{G_{F(\zeta_p)}} = \overline{\chi}^c_j\otimes \overline{\psi}^c|_{G_{F(\zeta_p)}}$ then $\overline{\psi} / \overline{\psi}^c|_{I_{F_{v_q}}}$ is trivial, a contradiction.
	
	Fix an isomorphism $\iota : \overline{\bQ}_p \to \C$. By \cite[Theorem 3.1.2]{BLGGT}, we can find a Galois totally real extension $L^+ / F^+$ and a regular algebraic, self-dual, cuspidal, automorphic representation $\pi$ of $\GL_{2n}(\A_{L^+})$, with the following properties:
	\begin{itemize}
	
		\item Let $L = F L^+$. Then $L / F$ is linearly disjoint from the extension of $F(\zeta_p)$ cut out by $\overline{\rho}_2|_{G_{F(\zeta_p)}}$. In particular, $[L(\zeta_p) : L] = p-1$, $\zeta_p \not\in \overline{L}^{\ker \ad \overline{\rho}_2|_{G_L}}$, and $L \not\subset L^+(\zeta_p)$.
			\item There is an isomorphism $\rbar_{\pi,\iota} \cong \rbar_1|_{G_{L^+}}$.
		\item $\pi$ is $\iota$-ordinary. More precisely, $\pi$ is of weight 0 and for each place $v | p$ of $L^+$, $\pi_v$ is an unramified twist of the Steinberg representation.
		\item $\overline{\rho}_2|_{G_L}$ is primitive.
		\item The irreducible constituents of $\overline{\rho}_2|_{G_{L(\zeta_p)}}$ occur with multiplicity 1.
		\item For each place $v$ of $L^+$ lying above a place of $\Sigma$, $\pi_v$ 
		is an unramified twist of the Steinberg representation.
	\end{itemize}
	More precisely, \cite[Theorem 3.1.2]{BLGGT} guarantees the existence of 
	$L^+$ satisfying the first condition and an essentially self-dual $\pi$ 
	satisfying all the remaining conditions (except possibly the last one). The 
	last paragraph of the proof notes that the $\pi$ constructed is in fact 
	self-dual and $\pi_v$ is an unramified twist of the Steinberg 
	representation for each place $v |p$ of $L^+$. We can moreover ensure that 
	$\pi$ is Steinberg at the places of $L^+$ 
	lying above $\Sigma$ by inserting the condition ``$t(P) < 0$ for all 
	places $v | \Sigma$ of $L^+$'' in the first list of conditions on \cite[p. 
	549]{BLGGT}. 
	
	After possibly adjoining another soluble totally real extension of $F^+$ to $L^+$, we can assume that the following further conditions are satisfied:
	\begin{itemize}
		\item $\overline{\rho}_2|_{G_L}$ is unramified at those finite places not dividing $S_L \cup \Sigma_L$.
		\item Each place of $L$ at which $\pi_L$  is ramified is split over $L^+$.
		\item For each place $v \in S_L \cup \Sigma_L$, $\overline{\psi}|_{G_{L_\wv}}$ is trivial.
		\item For each place $v \in S_{p, L}$,  $[L_\wv : \bQ_p] > 2n(2n-1)/2 + 1$ and $\overline{\rho}_2|_{G_{L_\wv}}$ is trivial.
	\end{itemize}
 Here $\pi_L$ denotes the base change of $\pi$. It is a RACSDC automorphic representation of $\GL_{2n}(\A_L)$. By construction, then, $\pi_L$ satisfies the hypotheses of \cite[Theorem 6.2]{All19}. Therefore, if we define the global deformation problem
	\begin{multline*} \cS' = (L / L^+, S_L \cup \Sigma_L, \widetilde{S}_L \cup \widetilde{\Sigma}_L, \Lambda_{L, 2n}, \overline{r}_2|_{G_{L^+}}, \epsilon^{1-2n} , \\\{ R_v^\triangle \}_{v \in S_{p, L}} \cup \{ R_v^\square \}_{v \in S_L - S_{p, L}} \cup \{ R_v^{St} \}_{v \in \Sigma_L} \}), \end{multline*}
	then $R_{\cS'}$ is a finite $\Lambda_{L, 2n}$-algebra. (Here we have written $\Lambda_{L, 2n}$ to distinguish from $\Lambda = \Lambda_{F, n}$ used above.) 
	
	We now need to relate the rings $R_{\cS'}$ and $R_{\cS_\Sigma}$. In fact, it will be enough to construct a commutative diagram
	\[ \xymatrix{ R_{\cS'} / (\varpi) \ar[r] &  R_{\cS_\Sigma} / (\varpi) \\
	\Lambda_{L, 2n}/(\varpi) \ar[r] \ar[u] & \Lambda_{F, n} / (\varpi)  \ar[u] }\]
	where the top horizontal morphism is finite. We first specify the map $\Lambda_{L, 2n}/(\varpi) \to \Lambda_{F, n}/(\varpi)$. It is the map that for each place $w \in \widetilde{S}_{p, L}$ lying above a place $\wv$ of $F$ classifies the tuple of characters
	\[ (\psi^v_1|_{I_{L_w}}, \dots, \psi^v_n|_{I_{L_w}}, \psi^v_n|^{-1}_{I_{L_w}} , \dots, \psi^v_1|^{-1}_{I_{L_w}} ). \]
	This endows the ring $R_{\cS_\Sigma} / (\varpi)$ with the structure of $\Lambda_{L, 2n}$-algebra. To give a map $R_{\cS'} / (\varpi) \to R_{\cS_\Sigma} / (\varpi)$, we must give a lifting of $\overline{r}_2|_{G_{L^+}}$ over $R_{\cS_\Sigma} / (\varpi)$ which is of type $\cS'$. To this end, let $r$ denote a 
	representative of the universal deformation (of $\overline{r}$) to $R_{\cS_\Sigma} / (\varpi)$, 
	and let $r' = I(r \otimes (\overline{\psi}, \epsilon^{1-2n} \overline{\mu}^{-1} \delta_{F / F^+}))_{G_{F}} ^{\wedge}|_{G_{L^+}}$ 
	(notation as in \cite[\S 1.1]{BLGGT}).
Then $r'$ is a lift of 
	$\overline{r}_2$ and $r'|_{G_L}$ is the restriction of $r|_{G_F} \otimes 
	\overline{\psi} \oplus r^c|_{G_F} \otimes \overline{\psi}^c$ to $G_L$. We need to check that for each $v \in S_{p, L}$, $r'|_{G_{L_\wv}}$ is of type $R_v^\triangle$; and that for each $v \in \Sigma_L$, $r'|_{G_{L_\wv}}$ is of type $R_v^{St}$. These statements can be reduced to a universal local computation. 
	
	It follows that $r'$ is of type $\cS'$, and so determines a 
	morphism $R_{\cS'} / (\varpi) \to  R_{\cS_\Sigma} / (\varpi)$. To complete the proof, it will be enough to show that this is a finite ring map. We can enlarge the above commutative diagram to a diagram
		\[ \xymatrix{ R_{\cS'} / (\varpi) \ar[r] &  R_{\cS_\Sigma} / (\varpi) \\
	Q_{{\overline{t}_2|_{G_L}}}/(\varpi) \widehat{\otimes}_\cO \Lambda_{L, 2n}/(\varpi) \ar[r] \ar[u] & Q_{{\overline{t}}}/(\varpi) \widehat{\otimes}_\cO\Lambda_{F, n} / (\varpi),  \ar[u] }\]
	where $Q_{\overline{t}_2|_{G_L}}$ is the complete Noetherian local $\cO$-algebra classifying pseudocharacters of $G_{L, S_L \cup \Sigma_L}$ lifting the restriction of $\overline{t}_2 = \tr {\overline{r}_2|_{G_F}}$ to $G_L$,  $Q_{\overline{t}}$ is defined similarly with respect to the pseudocharacter $\overline{t} = \tr \overline{r}|_{G_F}$ of $G_{F, S \cup \Sigma}$, and the map $Q_{{\overline{t}_2|_{G_L}}}/(\varpi) \to Q_{{\overline{t}}}/(\varpi)$ is the one classifying the natural transformation sending a pseudocharacter $t$ lifting $\overline{t}$ to the pseudocharacter $(t|_{G_L} \otimes \overline{\psi}) + (t^c|_{G_L} \otimes \overline{\psi}^c)$. We deduce from \cite[Proposition 3.29]{jackreducible} that the vertical arrows are finite ring maps. The map $\Lambda_{L, 2n} \to \Lambda_{F, n}$ is also finite, so it's enough finally to show that the map $Q_{{\overline{t}_2|_{G_L}}}/(\varpi) \to Q_{{\overline{t}}}/(\varpi)$ is finite. This map can in turn be written as a composite
	\[ Q_{{\overline{t}_2|_{G_L}}}/(\varpi) \to Q_{{\overline{t}_2|_{G_F}}}/(\varpi) \to Q_{{\overline{t}}}/(\varpi), \]
	where the first map classifies restriction of pseudocharacters from $G_F$ to $G_L$. Since $\overline{r}_2|_{G_F}$ is  multiplicity free, \cite[Proposition 2.5]{All19} (specifically, the uniqueness of the expression as a sum of pseudocharacters) implies that the second map is in fact surjective. We finally just need to show that the first map is finite, and this follows from the following general lemma.
\end{proof}
\begin{lemma}\label{lem_finiteness_for_psuedodef_rings}
Let $\Gamma$ be a topologically finitely generated profinite group, let $\Sigma$ be a closed subgroup of finite index, and let $\overline{t}$ be a pseudocharacter of $\Gamma$ with coefficients in $k$ of some dimension $n$. Let $Q_{\overline{t}}$ be the complete Noetherian local $\cO$-algebra classifying lifts of $\overline{t}$. Then the map $Q_{\overline{t}|_\Sigma} \to Q_{\overline{t}}$ classifying restriction to $\Sigma$ is a finite ring map.
\end{lemma}
\begin{proof}
It suffices to show that $Q_{\overline{t}} / (\ffrm_{Q_{\overline{t}|_\Sigma}})$ is Artinian. If not, we can find a prime ideal $\p$ of this ring of dimension 1; let $A$ be its residue ring (which is a $k$-algebra), and let $t_A$ be the induced pseudocharacter of $\Gamma$ with coefficients in $A$. Let $N = [\Gamma : \Sigma]$. If $\gamma \in \Gamma$ then $\gamma^{N!} \in \Sigma$. If we factor the characteristic polynomial of $X - \gamma$ under $t$ as $\prod_{i=1}^n(X - \alpha_i)$ for some elements $\alpha_i$ in the algebraic closure of $\operatorname{Frac} A$, then the characteristic polynomial of $\gamma^{N!}$ under $t$, namely $\prod_{i=1}^n(X - \alpha_i^{N!})$, lies in $k[X]$ and equals the characteristic polynomial of $\gamma^{N!}$ under $\overline{t}$. This shows that the elements $\alpha_i$ are in fact algebraic over $k$, and thus (using \cite[Corollary 1.14]{chenevier_det}) that $t_A$ can be defined over $k$, and must in fact equal $\overline{t}$. This is a contradiction.
\end{proof}
\begin{cor}\label{cor_existence_of_lifts}
	With hypotheses as in Theorem \ref{thm_finiteness_of_deformation_ring}, fix $\lambda \in (\Z^n_+)^{\Hom(F, \overline{\Q}_p)}$ such that for each $i = 1, \dots, n$ and $\tau \in \Hom(F, \overline{\Q}_p)$, we have $\lambda_{\tau c, i} = - \lambda_{\tau, n+1-i}$. Suppose further that for each $v \in S_p$, $[F_\wv : \Q_p] > n(n-1)/2 + 1$. Then there exists a homomorphism $r : G_{F^+, S} \to \cG_n(\overline{\Z}_p)$ lifting $\overline{r}$ such that $r|_{G_{F, S}}$ is ordinary of weight $\lambda$, in the sense of \cite[Definition 2.5]{jackreducible}.
\end{cor}
\begin{proof}
	We observe that \cite[Proposition 3.9, Proposition 3.14]{jackreducible} show that for each 
	minimal prime $Q \subset R_{\cS_\Sigma}$, $\dim R_{\cS_{\Sigma}} / Q = \dim 
	\Lambda$; consequently, there is a minimal prime $Q_\Lambda$ of $\Lambda$ 
	and a finite injective algebra morphism $\Lambda / Q_\Lambda \to 
	R_{\cS_{\Sigma}} / Q$. The corollary follows on choosing any prime of 
	$R_{\cS_{\Sigma}} / Q[1/p]$ lying above a maximal ideal of $\Lambda / 
	Q_\Lambda[1/p]$ associated to the weight $\lambda$ as in \cite[Definition 2.24]{ger}.
\end{proof}
\begin{cor}\label{cor:modouttosteinberg}
	With hypotheses (1) -- (3) of Theorem 
	\ref{thm_finiteness_of_deformation_ring}, choose a place $v_0 \not\in S$ of 
	$F^+$ split in $F$ and a lift $\wv_0$ to $F$ such that $q_{v_0} \equiv 1 
	\text{ mod }p$ and $\overline{r}|_{G_{F_{\wv_0}}}$ is trivial. Consider the 
	quotient 
	\[ A = R_{\cS_\emptyset} / (\varpi, \{ \tr r_{\cS_\emptyset}(\Frob_{\wv_0}^i) - n \}_{i = 1, \dots, n} ). \]
	 Then $A$ is a finite $\Lambda$-algebra. Consequently, $\dim 
	 R_{\cS_\emptyset}/(\varpi) \leq n[F^+ : \Q] + n$.
\end{cor}
\begin{proof}
	It suffices to verify that the quotient of $R_{\cS_\emptyset}$ where the 
	characteristic 
	polynomial of $\Frob_{\wv_0}$ equals $\prod_{i=1}^n(X - q_{v_0}^{1-i})$ is 
	a quotient of $R_{\cS_{ \{ 
	v_0 \}}}$. This in turn means checking that the quotient $A_{v_0}$ of the 
	local unramified lifting ring $R_{v_0}^{ur}$ where the characteristic 
	polynomial of Frobenius equals $\prod_{i=1}^n(X - q_{v_0}^{1-i})$ is a 
	quotient of $R_{v_0}^{St}$. Since $A_{v_0}$ is flat over $\cO$, this 
	follows from the definition of $R_{v_0}^{St}$ (see \cite[\S 3]{tay}).
\end{proof}
For the statement of the next proposition, suppose given a surjection $R_{\cS} / (\varpi) \to A$ in $\cC_\Lambda$, where $A$ is a domain, and let $r : G_{F^+, S} \to \cG_n(A)$ denote the pushforward of (a representative of) the universal deformation. Suppose given the following data:
\begin{itemize}
	\item A decomposition $r = r_1 \oplus r_2$, where the $r_i : G_{F^+, S} \to 
	\cG_{n_i}(A)$ satisfy $\nu_{\cG_{n_i}} \circ r_i = \nu_{\cG_n} \circ r$. (In other words, $r|_{G_F} = r_1|_{G_F} \oplus r_2|_{G_F}$ and if $r_i(c) = (A_i, 1) \jmath$ then $r(c) = \diag(A_1, A_2) \jmath$.)
	\item A subset $R \subset S - S_p$ (consisting of places of odd residue characteristic) with the following property: for each $v 
	\in R$ we are given an integer $1 \leq n_\wv \leq n$ such that $q_\wv \text{ mod }p$ is a primitive 
	$n_\wv^\text{th}$ root 
	of unity and there is a decomposition $\overline{r}|_{G_{F_\wv}} = 
	\overline{\sigma}_{\wv, 1} \oplus \overline{\sigma}_{\wv, 2}$, where 
	$\overline{\sigma}_{\wv, 1} = \Ind_{G_{F_{\wv, n_\wv}}}^{G_{F_\wv}} 
	\overline{\psi}_\wv$ with $F_{\wv, n_\wv} / F_\wv$ the unramified extension of degree $n_\wv$
	and $\overline{\psi}_\wv$ an unramified character of $G_{F_{\wv, n_\wv}}$, and 
	$\overline{\sigma}_{\wv, 2}$ is the twist of an unramified representation 
	of $G_{F_\wv}$ of dimension $n-n_\wv$ by a ramified quadratic character. 
	\item An isomorphism $\iota : \overline{\bQ}_p \to \bC$ and for each $v \in R$, a character $\Theta_\wv : \cO_{F_{\wv, n_\wv}}^\times \to \bC^\times$ of order $p$. Thus the lifting ring $R(\wv,  \Theta_\wv, n)$ is defined (notation as in \S \ref{subsec_local_theory_GL_n}).
\end{itemize}
\begin{proposition}\label{prop_bound_on_dimension_of_reducibility_locus}
	With the above assumptions on $R_{\cS} / (\varpi) \to A$, suppose that the following additional conditions are satisfied:
	\begin{enumerate}
		\item $p > 2n$.
		\item For each $1 \leq i < j \leq n$, $\overline{\chi}_i / \overline{\chi}_j|_{G_{F(\zeta_p)}}$ has order greater than $2n$. (In particular, this character is non-trivial and $\overline{r}$ is Schur.)
		\item For each $v \in S_p$, $[F_\wv : \Q_p] > n(n-1)/2 + 1$.
		\item $[F(\zeta_p) : F] = p-1$.
		\item For each $v \in R$, both $\overline{r}_1|_{G_{F_\wv}}$ and 
		$\overline{r}_2|_{G_{F_\wv}}$ admit a non-trivial unramified 
		subquotient and the composite map $R_v^{\square} \to R_{\cS} \to A$ 
		factors 
		over 
		$R(\wv, \Theta_\wv, n)$.
	\end{enumerate}
	Let $L_{S_p}$ denote the maximal abelian pro-$p$ extension of $F$ 
	unramified outside $S_p$, and let $\Delta = \Gal(L_{S_p} / F) / (c+1)$. Let 
	$d_R$ denote the $\Z_p$-rank of the subgroup of $\Delta$ generated by the 
	elements $\Frob_\wv$, $v \in R$. Then $\dim A \leq n[F^+ : \Q] + n - d_R$.
\end{proposition}
\begin{proof}
Fix a place $\wv_0$ of $F$ split over $F^+$, prime to $S$, and such that $q_{\wv_0} \equiv 1 \text{ mod 
}p$ and $\overline{r}|_{G_{F_{\wv_0}}}$ is trivial. Let $I$ denote the ideal of $A$ generated by the coefficients of the polynomial $\det(X - r|_{G_{F, S}}(\Frob_{\wv_0})) - (X - 1)^n$. Then $\dim A / I \geq \dim A - n$. After replacing $A$ by $A / \p$, where $\p \subset A$ is a prime ideal minimal among those containing $I$, we can assume that $r|_{G_{F, S}}(\Frob_{\wv_0})$ is 
unipotent,
and must show $\dim A \leq 
n[F^+ : \Q] - d_R$. Consider the deformation problems $(i = 1, 2)$:
\[ \cS_i = (F / F^+, S \cup \{ v_0 \}, \widetilde{S} \cup \{ \wv_0 \}, \Lambda_{n_i}, \overline{r}_i, \mu, \{ R_{v}^\Delta \}_{v \in S_p} \cup \{ R_v^\square \}_{v \in S - S_p} \cup \{ R_{v_0}^{St} \}). \]
Let $K = \operatorname{Frac} A$. We now repeat the argument of \cite[Lemma 3.6]{All19}: if $v \in S_p$, then (since we assume $[F_\wv : \bQ_p] > n(n-1)/2+1$) we can appeal to \cite[Corollary 3.12]{jackreducible}, which implies the existence of an increasing filtration
\[ 0 \subset \Fil_v^1 \subset \Fil_v^2 \subset \dots \subset \Fil_v^n = \overline{K}^n \]
of $r|_{G_{F_\wv}} \otimes_A \overline{K}$ by $G_{F_\wv}$-invariant subspaces, such that each $\gr^i \Fil_v^\bullet = \Fil_v^i / \Fil_v^{i-1}$ ($i = 1, \dots, n$) is 1-dimensional, and such that the character $I_{F_\wv}(p) \to \overline{K}^\times$ afforded by $\gr^i \Fil_v^\bullet$ agrees with the pushforward of the universal character $\psi_v^i : I_{F_\wv} \to \Lambda_v^\times$. Using the decomposition $r = r_1 \oplus r_2$, we obtain induced filtrations $\Fil_v^\bullet \cap  (\overline{K}^{n_1} \oplus 0^{n_2})$ of $r_1|_{G_{F_\wv}} \otimes_A \overline{K}$  and $\Fil_v^\bullet \cap  (0^{n_1} \oplus \overline{K}^{n_2})$ of  $r_2 \otimes_A \overline{K}$ and, applying \cite[Corollary 3.12]{jackreducible} once more, we see that we can find an isomorphism $\Lambda_{ n_1} \widehat{\otimes} \Lambda_{n_2} \cong \Lambda = \Lambda_n$ such that, endowing $A$ with the induced $\Lambda_i$-algebra structure, $r_i$ is a lifting of $\overline{r}_i$ of type $\cS_i$ for each $i = 1, 2$. We deduce the existence of a surjective $\Lambda$-algebra homomorphism $R_{\cS_1} \widehat{\otimes}_\cO R_{\cS_2} \to A$. We observe that Theorem 
\ref{thm_finiteness_of_deformation_ring} applies to the deformation problems 
$\cS_1$ and $\cS_2$, showing that $\dim R_{\cS_i} / (\varpi) \leq n_i [F^+ : 
\Q]$.

Let $\psi_i : G_{F, S} \to \cO^\times$ denote the Teichm\"uller lift of $\overline{\psi}_i = \det \overline{r}_i|_{G_{F, S}}$, and let $R_{\cS_i}^{\psi_i}$ denote the quotient of $R_{\cS_i}$ over which the determinant of the universal deformation equals $\psi_i$. Then \cite[Lemma 3.36]{jackreducible} states that there is an isomorphism $R_{\cS_i} \cong R_{\cS_i}^{\psi_i} \widehat{\otimes}_\cO \cO \llbracket \Delta \rrbracket$. In particular, $\dim R_{\cS_i}^{\psi_i} / (\varpi) \leq (n_i - 1)[F^+ : \Q]$. To complete the proof, it is enough to show that if $A' =  A / (\ffrm_{R_{\cS_1}^{\psi_1}}, \ffrm_{R_{\cS_2}^{\psi_2}})$, then $\dim A' \leq 2[F^+ : \Q] - d_R = \dim k \llbracket \Delta \times \Delta \rrbracket - d_R$.

To this end, we observe that by construction there is a surjection $k 
\llbracket \Delta \times \Delta \rrbracket \to A'$. If $\Psi_1, \Psi_2 : G_{F, 
S} \to k\llbracket \Delta \times \Delta \rrbracket^\times$ are the two 
universal characters, then the third part of Proposition \ref{prop_properties_of_lifting_ring} 
(together with 
our assumption that both $\overline{r}_1|_{G_{F_\wv}}$ and 
$\overline{r}_2|_{G_{F_\wv}}$ admit an unramified subquotient) implies that the 
relation $\Psi_1(\Frob_\wv)^{n_\wv} = \Psi_2(\Frob_\wv)^{n_\wv}$ holds in $A'$ for each  $v  \in R$. Since 
$\Delta$ is a pro-$p$ group, this implies that $\Psi_1(\Frob_\wv) = 
\Psi_2(\Frob_\wv)$ in $A'$, and hence that the map $k \llbracket \Delta \times 
\Delta \rrbracket \to A'$ factors over the completed group algebra of the 
quotient of $\Delta \times \Delta$ by the subgroup topologically generated by 
the elements $(\Frob_\wv, - \Frob_\wv)_{v \in R}$. This completes the proof.
\end{proof}
We are now in a position to prove the main theorem of this section, which 
guarantees the existence of generic primes in sufficiently large quotients of a 
certain deformation ring. For the convenience of the reader, we state our 
assumptions from scratch. 

Thus we take $F, S, p$ as in our standard assumptions (see \S \ref{sec_definite_unitary_groups}). We assume that $[F(\zeta_p) : F] = (p-1)$. We let $E$ be a coefficient field, and suppose given an isomorphism $\iota : \overline{\bQ}_p \to \bC$ and a continuous character  $\mu : G_{F^+, S} \to \cO^\times$ which is de Rham and such that $\mu(c_v) = -1$ for each place $v | \infty$ of $F^+$. We fix an integer $2 \leq n < p / 2$ and characters $\overline{\chi}_1, \dots, \overline{\chi}_n : G_{F, S} \to k^\times$ such that for each $i = 1, \dots, n$, $\overline{\chi}_i \overline{\chi}_i^c = \overline{\mu}|_{G_{F, S}}$. We set $\overline{\rho}  =\oplus_{i=1}^n \overline{\chi}_i$; then $\overline{\rho}$ naturally extends to a homomorphism $\overline{r} : G_{F, S} \to \cG_n(k)$ such that $\nu_{\cG_n} \circ \overline{r} = \overline{\mu}$. We suppose for each $1 \leq i < j \leq n$, $\overline{\chi}_i / \overline{\chi}_j|_{G_{F(\zeta_p)}}$ has order greater than $2n$. This implies that $\overline{r}$ is Schur. We suppose that for each $v \in S_p$, $\overline{r}|_{G_{F_\wv}}$ is trivial and $[F_\wv : \Q_p] > n(n-1)/2 + 1$.

We suppose given a subset $R = R_1 \sqcup R_2 \subset S - S_p$ (consisting of places of odd residue characteristic) and integers $1 \leq n_\wv \leq n$ ($v \in R$) such that for each $v \in R$, $q_\wv \text{ mod }p$ is a primitive $n_\wv^\text{th}$ root of unity, and there is a decomposition $\overline{r}|_{G_{F_\wv}} = \overline{\sigma}_{\wv, 1} \oplus \overline{\sigma}_{\wv, 2}$, where $\overline{\sigma}_{\wv, 1} = \Ind_{G_{F_{\wv, n_\wv}}}^{G_{F_\wv}} \overline{\psi}_\wv$ is induced from an unramified character of the unramified degree $n_\wv$ extension of $F_{\wv}$, and $\overline{\sigma}_{\wv, 2}$ is the twist of an unramified representation of dimension $n - n_\wv$ by a ramified quadratic character. We fix for each $v \in R$ a character $\Theta_\wv : \cO_{F_{\wv, n_\wv}}^\times \to \bC^\times$ of order $p$. 

Assuming (as we may) that $E$ is large enough, we may then (re-)define the global deformation problem
\[ \cS = (F / F^+, S, \widetilde{S}, \Lambda, \overline{r}, \mu, \{ R_v^\triangle \}_{v \in S_p} \cup \{ R(\wv, \Theta_\wv, n) \}_{v \in R} \cup \{ R_v^\square \}_{v \in S - (S_p \cup R)}). \]
Following \cite[Definition 3.7]{All19}, we say that a prime $\p \subset R_{\cS}$ of dimension 1 and characteristic $p$ is generic at $p$ if it satisfies the following conditions:
\begin{itemize}
	\item Let $A = R_{\cS} / \p$, and let $r_\p : G_{F^+, S} \to \cG_n(A)$ be the pushforward of (a representative of) the universal deformation. Then for each $v \in S_p$, the (pushforwards from $\Lambda$ of the) universal characters $\psi^v_1, \dots, \psi^v_n : I^{ab}_{F_\wv}(p) \to A^\times$ are distinct.
	\item There exists $v \in S_p$ and $\sigma \in I^{ab}_{F_\wv}(p)$ such that the elements $\psi^v_1(\sigma), \dots, \psi^v_n(\sigma) \in A^\times$ satisfy no non-trivial $\Z$-linear relation. 
\end{itemize}
We say that $\p$ is generic if it is generic at $p$ and if $r_{\p}|_{G_{F, S}} \otimes_{A} \operatorname{Frac} A$ is absolutely irreducible.
\begin{theorem}\label{thm_existence_of_generic_primes}
	With assumptions as above, let $R_{\cS} \to B$ be a surjection in 
	$\cC_\Lambda$, where $B$ is a finite $\Lambda / (\varpi)$-algebra. Let 
	$L_{S_p}$ denote the maximal abelian pro-$p$ extension of $F$ unramified 
	outside $S_p$, and let $\Delta = \Gal(L_{S_p} / F) / (c+1)$. Let $d_{R_i}$ 
	denote the $\Z_p$-rank of the subgroup of $\Delta$ topologically generated 
	by the elements $\Frob_\wv$, $v \in R_i$. 
	
	Suppose that the following conditions are satisfied:
	\begin{enumerate}
		\item Each irreducible component of $\Spec B$ has dimension strictly greater than $\sup(\{ n [F^+ : \Q] + n - d_{R_i} \}_{i=1,2}, \{ n [F^+ : \Q] -  [F_\wv : \Q_p] \}_{v \in S_p} )$. 
		\item For each direct sum decomposition $\overline{r} = \overline{r}_1 \oplus \overline{r}_2$ with $\overline{r}_j : G_{F^+, S} \to \cG_{n_j}(k)$ ($j = 1, 2$) and $n_1 n_2 \neq 0$, there exists $i \in \{ 1, 2 \}$ such that for each $v \in R_i$, both $\overline{r}_1|_{G_{F_\wv}}$ and $\overline{r}_2|_{G_{F_\wv}}$ admit a non-trivial unramified subquotient. 
	\end{enumerate}
	Then there exists a prime $\p \subset R_{\cS}$ of dimension 1 and characteristic p, containing the kernel of $R_{\cS} \to B$, which is generic. 
\end{theorem}
\begin{proof}
	After passage to a quotient by a minimal prime, we can assume that $B$ is a 
	domain. The argument is now very similar to that of \cite[Lemma 
	3.9]{All19}. Indeed, by \cite[Lemma 3.8]{All19}, we can find countable collection $(I_i)_{i \geq 1}$ of ideals $I_i \subset \Lambda / (\varpi)$ such that for all $i \geq 1$, $\dim \Lambda / (\varpi, I_i) \leq \sup \{ n [F^+ : \bQ] -  [ F_\wv : \bQ_p ] \}_{v \in S_p}$ and if $\p \subset R_\cS$ is a prime of dimension 1 and characteristic $p$ which is not generic at $p$, then $I_i R_\cS \subset \p$ for some $i \geq 1$. Let $I_\cS^{red} \subset R_\cS$ be the reducibility ideal defined just before \cite[Lemma 3.4]{All19}, and let $I_0 = (I_\cS^{red}, \varpi) R_\cS$. Proposition 
	\ref{prop_bound_on_dimension_of_reducibility_locus} shows that $\dim R_\cS / I_0 \leq \sup \{ n [ F^+ : \bQ] + n - d_{R_i} \}_{i = 1, 2}$.
	
	Since $B$ is a finite $\Lambda / (\varpi)$-algebra, we have $\dim B / I_i \leq \sup \{ n [F^+ : \bQ] -  [ F_\wv : \bQ_p ] \}_{v \in S_p}$. We also have $\dim B / I_0 \leq  \sup \{ n [ F^+ : \bQ] + n - d_{R_i} \}_{i = 1, 2}$. The existence of a generic prime $\p$ containing the kernel of the map $R_\cS \to B$ thus follows from \cite[Lemma 1.9]{jackreducible}.
\end{proof}
We conclude this section with a result concerning the existence of automorphic 
lifts of prescribed types, under the hypothesis of residual automorphy over a 
soluble extension. It only uses the results of \cite{All19} and not the results 
proved earlier in this section, and is very similar in statement and proof to 
\cite[Theorem 5.2.1]{Bel19}. 

We begin by re-establishing notation. We therefore let $F_0$ be an imaginary CM field such that $F_0 / F_0^+$ is everywhere unramified. We fix a prime $p$ and write $S_{0,p}$ for the set of of $p$-adic places of $F_0^+$. We fix a finite set $S_0$ of finite places of $F_0^+$ containing $S_{0, p}$. We assume that each place of $S_{0, p}$ splits in $F_0$, but not necessarily that each place of $S_0 - S_{0, p}$ splits in $F_0$.  We choose for each $v \in  S_0$  a  place $\wv$ of $F_0$ lying above $v$, and write $\widetilde{S}_0 = \{ \wv \mid v \in S_0 \}$. We fix a coefficient field $E \subset \overline{\bQ}_p$. Fix an integer $n \geq 2$, and suppose given a continuous representation $\overline{\rho} : G_{F_0, S_0} \to \GL_n(k)$ satisfying the following conditions:
\begin{itemize}
\item There is an isomorphism $\overline{\rho} \cong \oplus_{i=1}^r \overline{\rho}_i$, where each representation $\overline{\rho}_i$ is absolutely irreducible and satisfies $\overline{\rho}_i^{c} \cong \overline{\rho}_i^\vee \otimes \epsilon^{1-n}$. Moreover, for each  $1 \leq  i < j \leq r$, we have $\overline{\rho}_i \not\cong \overline{\rho}_j$.
\end{itemize}
\begin{proposition}\label{prop_automorphic_lifts_of_prescribed_types}
Fix disjoint subsets $T_0, 
\Sigma_0 \subset S_0$ consisting of prime-to-$p$ places which split in $F_0$. 
We assume that for each $v \in \Sigma_0$, we have $q_\wv \equiv 1 \text{ mod 
}p$ and $\overline{\rho}|_{G_{F_\wv}}$ is trivial.  We fix for each $v \in T_0$ 
a quotient $R_v^\square \to \overline{R}_v$ of the universal lifting ring of 
$\overline{\rho}|_{G_{F_\wv}}$ corresponding to a non-empty union of 
irreducible components of $\Spec R_v^\square[1/p]$. We suppose that if $v \in 
S_0$ and $v$ is inert in $F_0$, then $\overline{\rho}(I_{F_{0, \wv}})$ is of 
order prime to $p$. Fix a weight $\lambda \in (\bZ^n_+)^{\Hom(F_0, 
\overline{\bQ}_p)}$ such that for each $i = 1, \dots, n$ and $\tau \in 
\Hom(F_0, \overline{\bQ}_p)$, $\lambda_{\tau c, i} = - \lambda_{\tau, n + 1 
-i}$, and suppose that for each $v \in S_{0, p}$, $\overline{\rho}|_{G_{F_{0, 
\wv}}}$ admits a lift to $\overline{\bZ}_p$ which is ordinary of weight 
$\lambda_\wv$, in the sense of \cite[Definition 3.8]{ger}. Suppose that there 
exists a soluble CM extension $F / F_0$ such that  the following conditions  
are satisfied:
\begin{enumerate}
\item $p > \max(n, 3)$. For each place $v | p$ of $F$, we have $[F_v : \bQ_p] > n(n-1)/2 + 1$ and $\overline{\rho}|_{G_{F_v}}$ is trivial.
\item $F(\zeta_p)$ is not contained in $\overline{F}^{\ker \ad \overline{\rho}}$ 
and $F$ is not contained in $F^+(\zeta_p)$. For each $1 \leq i < j \leq r$, 
$\overline{\rho}_i|_{G_{F(\zeta_p)}}$ is absolutely irreducible and 
$\overline{\rho}_i|_{G_{F(\zeta_p)}} \not\cong 
\overline{\rho}_j|_{G_{F(\zeta_p)}}$. Moreover, $\overline{\rho}|_{G_F}$ is 
primitive and $\overline{\rho}(G_F)$ has no quotient of order  $p$. 
\item There exists a RACSDC automorphic representation $\pi$ of $\GL_n(\A_F)$ and an  isomorphism $\iota : \overline{\bQ}_p \to \bC$ such that $\overline{r}_{\pi, \iota} \cong \overline{\rho}|_{G_F}$. Moreover, $\pi$ is $\iota$-ordinary and there exists a place $v$ of $F$ lying above $\Sigma_0$ such that $\pi_v$ is an unramified twist of the Steinberg representation.
\item If $S$ denotes the set of places of $F^+$ lying above $S_0$, then each place of $S$ splits in $F$.
\end{enumerate}
Then there exists a RACSDC automorphic representation $\pi_0$ of $\GL_n(\bA_{F_0})$ satisfying the following conditions:
\begin{enumerate}
\item $\pi_0$ is unramified outside $S_0$ and there is an isomorphism 
$\rbar_{\pi_0, \iota} \cong \overline{\rho}$.
\item $\pi_0$ is $\iota$-ordinary of weight $\iota \lambda$.
\item For each place $ v \in T_0$, $r_{\pi_0, \iota}|_{G_{F_{0, \wv}}}$ defines a point of $\overline{R}_v$.
\item For  each place  $v \in \Sigma_0$, $\pi_{0, \wv}$ is an unramified twist of the Steinberg representation.
\item For each place $v \in S_0$ which is inert in $F_0$, reduction modulo $p$ 
induces an isomorphism $r_{\pi_0, \iota}(I_{F_0, \wv}) \to \overline{r}_{\pi_0, 
\iota}(I_{F_0, \wv})$. 
\end{enumerate}
\end{proposition}
\begin{proof}
Let $v \in S_{0, p}$, and let $\rho_v : G_{F_{0, \wv}} \to \GL_n(\overline{\bZ}_p)$ be the lift of $\overline{\rho}|_{G_{F_\wv}}$ which is ordinary of weight $\lambda_\wv$ and which exists by assumption. Thus $\rho_v$ is conjugate over $\overline{\bQ}_p$ to an upper-triangular representation with the property that if $\alpha_{v, 1}, \dots, \alpha_{v, n} : G_{F_{0, \wv}} \to \overline{\bQ}_p^\times$ are the characters appearing on the diagonal, then for each $i = 1, \dots, n$ the character $\alpha_{v, i}$ is equal, on restriction to some open subgroup of $I_{F_{0, \wv}}$, to the character 
\[ \chi_{\lambda_\wv, i} : \sigma \in I_{F_{0, \wv}} \mapsto  \epsilon(\sigma)^{1-i} \prod_{\tau : F_{0, \wv} \to \overline{\bQ}_p} \tau(\Art_{F_{0, \wv}}^{-1}(\sigma))^{-\lambda_{\tau, n - i + 1} }. \]
After enlarging $E$, we can assume that each character $\alpha_{v, i}$ takes values in $\cO$. We use the restricted characters $\alpha_{v, i}|_{I_{F_{0, \wv}}(p)} : I_{F_{0, \wv}}(p) \to \cO^\times$ ($v \in S_{0, p}, i = 1, \dots, n$) to define a homomorphism $\Lambda_{F_0} \to \cO$.

Let $\beta_{\wv, i} = \alpha_{v, i} \chi_{\lambda_\wv, i}^{-1}$. Then $\tau_\wv =\oplus_{i=1}^n \beta_{\wv, i}$ is an inertial type and the type $\tau_\wv$, Hodge type $\lambda_\wv$ lifting ring $R_\wv^{\lambda_\wv, \tau_\wv}$ is defined and equidimensional of dimension $1 + n^2 + n(n-1)[F_{0, \wv} : \bQ_p] / 2$ (see \cite[Theorem 3.3.4]{kisindefrings}). When $\tau_\wv$ is trivial, \cite[Lemma 3.10]{ger} shows that there is a minimal prime ideal of $R_\wv^{\lambda_\wv, \tau_\wv}$ such that, writing $R_v$ for the corresponding quotient, the following properties are satisfied:
\begin{itemize}
\item $R_v$ is $\cO$-flat of dimension $1 + n^2 + n(n-1)[F_{0, \wv} : \bQ_p] / 2$.
\item The map $R_\wv^{\lambda_\wv, \tau_\wv} \to \overline{\bQ}_p$ determined by $\rho_v$ factors through $R_v$.
\item For every homomorphism $R_v \to \overline{\bQ}_p$, the corresponding Galois representation $G_{F_{0, \wv}} \to \GL_n(\overline{\bQ}_p)$ is ordinary of weight $\lambda_\wv$, in the sense of \cite[Definition 3.8]{ger}.
\item The homomorphism $R^\square_v \widehat{\otimes}_\cO \Lambda_v \to R_v$ (completed tensor product of the tautological quotient map $R^\square_v \to R_v$ and the composite $\Lambda_v \to \cO \to R_v$) factors over the quotient $R^\square_v \widehat{\otimes}_\cO \Lambda_v \to R_v^\triangle$ (defined in e.g.\ \cite[\S 3.3.2]{jackreducible}).  
\end{itemize}
In fact, the same proof shows that these properties hold also in the case that $\tau_\wv$ is non-trivial. 

Our hypotheses imply that we can extend $\overline{\rho}$ to a continuous  
homomorphism $\overline{r} : G_{F_0^+, S_0} \to \cG_n(k)$ with the property 
that $\nu \circ \overline{r} = \epsilon^{1-n} \delta_{F_0 / F_0^+}^n$. Let 
$R^{\text{univ}}$ denote the deformation ring, defined as in \cite[Corollary 
5.1.1]{Bel19}, of $\epsilon^{1-n} \delta^n_{F_0 / F_0^+}$-polarised 
deformations of $\overline{r}$, where the quotients of the local lifting rings 
for $v \in S_0$ are specified as follows:
\begin{itemize}
\item If $v \in S_{0, p}$, we take the quotient $R_v$ defined above. 
\item If $v \in T_0$, take $\overline{R}_v$.
\item If $v \in \Sigma_0$, take the Steinberg lifting ring $R_v^{St}$.
\item If $v \in S_0$ and $v$ is inert in $F_0$, take the component 
corresponding to the functor of lifts $r$ of $\overline{r}|_{G_{F_{0, v}^+}}$ 
such that the reduction map induces an isomorphism $r(I_{F^+_{0, v}}) \to 
\overline{r}(I_{F^+_{0, v}})$. 
\end{itemize}
We can invoke \cite[Corollary 5.1.1]{Bel19} to conclude that $R^\text{univ}$ 
has Krull dimension at least 1. We remark that this result includes the 
hypothesis that $\overline{r}|_{G_{F_0(\zeta_p)}}$ is irreducible, but this is 
used only to know that the groups $H^0(F_0^+, \ad  \overline{r})$ and 
$H^0(F_0^+, \ad \overline{r}(1))$ vanish, which is true under the weaker 
condition that $\overline{r}|_{G_{F_0^+(\zeta_p)}}$ is Schur, which follows 
from our hypotheses. (The vanishing of these groups implies that the 
deformation functor is representable and that the Euler characteristic formula 
gives the correct lower bound for its dimension.) 

We consider as well the deformation problem
\[ \cS = (F / F^+, S, \widetilde{S}, \Lambda_{F}, \overline{r}|_{G_F}, \epsilon^{1-n} \delta_{F / F^+}^n, \{ R_v^\triangle \}_{v \in S_{ p}} \cup \{ R_v^{\square} \}_{v \in S  - (S_p \cup \Sigma)} \cup \{ R_v^{St} \}_{v \in \Sigma}), \]
where we define $S, T, \Sigma$ to be the sets of places of $F^+$ above $S_0, 
T_0, \Sigma_0$, respectively. Then there is a natural morphism $R_{\cS} \to 
R^\text{univ}$ of $\Lambda_F$-algebras, which is finite (apply Lemma \ref{lem_finiteness_for_psuedodef_rings} and \cite[Proposition 3.29(2)]{jackreducible}). By \cite[Theorem 
6.2]{All19}, $R_{\cS}$ is a finite $\Lambda_F$-algebra. The map $\Lambda_F \to 
R^\text{univ}$ factors through a homomorphism $\Lambda_F \to \cO$ (by construction), so $R^\text{univ}$ is a finite $\cO$-algebra (of Krull dimension at 
least 1, as we have already remarked). 

We deduce the existence of a lift $r : G_{F^+_0, S_0} \to 
\cG_n(\overline{\bZ}_p)$ of $\overline{r}$ arising from a homomorphism $R^{univ} \to  \overline{\bZ}_p$. We can 
now 
apply \cite[Theorem 
6.1]{All19} and soluble descent to conclude that $r|_{G_{F_0}}$ is automorphic, 
associated to an 
automorphic representation $\pi_0$ of $\GL_n(\A_{F_0})$ with the desired 
properties. 
\end{proof}
\section{Raising the level -- Galois theory}\label{sec:galois_level_raising}

This section is devoted to the proof of a single theorem that will bridge the 
gap between Theorem \ref{thm_automorphic_level_raising} and our intended 
applications. Let $F, S, p, G$ be as in our standard assumptions (see \S 
\ref{sec_definite_unitary_groups}), and let $n \geq 3$ be an odd integer such 
that $p > 2n$.
Fix an isomorphism $\iota : \overline{\QQ}_p \to \CC$. We suppose given a 
RACSDC automorphic representation $\pi$ of $\GL_n(\A_F)$, and that the 
following conditions are satisfied:
\begin{enumerate}
	\item $\pi$ is $\iota$-ordinary. 
	\item $[F(\zeta_p) : F] = p-1$.
	\item There exist characters $\overline{\chi}_1, \dots, \overline{\chi}_n : G_F \to \overline{\bF}_p^\times$ and an isomorphism $\overline{r}_{\pi, \iota} \cong \overline{\chi}_1 \oplus \dots \oplus \overline{\chi}_n$, where for each $i = 1, \dots, n$, we have $\overline{\chi}_i^c = \overline{\chi}_i^\vee \epsilon^{1-n}$ and for each $1 \leq i < j \leq n$, $\overline{\chi}_i / \overline{\chi}_j|_{G_{F(\zeta_p)}}$ has order strictly greater than $2n$.
	\item For each $v \in S_p$, $\overline{r}_{\pi, \iota}|_{G_{F_\wv}}$ is trivial and $[F_\wv : \Q_p] > n(n-1)/2+1$.
	\item There is a set $R = R_1 \sqcup R_2 \subset S - S_p$ with the following properties:
	\begin{enumerate}
		\item The sets $R_1$ and $R_2$ are both non-empty and for each $v \in R$, the characteristic of $k(v)$ is odd. As in \S \ref{subsec_local_theory_GL_n}, we write $\omega(\wv) : k(\wv)^\times \to \{ \pm 1 \}$ for the unique quadratic character.
		\item If $v \in R_1$, then $q_v \text{ mod }p$ is a primitive $3^\text{rd}$ root of unity, and there exists a character $\Theta_\wv : k_3^\times \to \bC^\times$ of order $p$ such that $\pi_\wv|_{\q_\wv}$ contains $\widetilde{\lambda}(\wv, \Theta_\wv, n)$ (notation as in \S \ref{subsec_local_theory_GL_n}, this representation of $\q_\wv$ defined with respect to $n_1 = 3$).
		\item If $v \in R_2$, then $q_v \text{ mod }p$ is a primitive $(n-2)^\text{th}$ root of unity, and there exists a character $\Theta_\wv : k_{n-2}^\times \to \bC^\times$ of order $p$ such that $\pi_\wv|_{\q_\wv}$ contains $\widetilde{\lambda}(\wv, \Theta_\wv, n)$ (notation as in \S \ref{subsec_local_theory_GL_n}, this representation of $\q_\wv$ defined with respect to $n_1 = n-2$).
		\item For each non-trivial direct sum decomposition $\overline{r}_{\pi, \iota} = \overline{\rho}_1 \oplus \overline{\rho}_2$, there exists $i \in  \{ 1, 2\}$ such that for each $v \in R_i$, $\overline{\rho}_1|_{G_{F_\wv}}$ and $\overline{\rho}_2|_{G_{F_\wv}}$ both admit a non-trivial unramified subquotient.
	\end{enumerate}
\end{enumerate}
(This is the situation we will find ourselves in after applying Theorem 
\ref{thm_automorphic_level_raising}. The sets of 
places $R_1$ and $R_2$ here will correspond to the sets $T_1$ and $T_3$ respectively 
from \S \ref{sec:automorphic_level_raising}, and it will be possible 
to label the characters $\overline{\chi}_i$ so that we have the following 
ramification properties:
\begin{itemize}
\item If $v \in R_1$, then $\overline{\chi}_1, \overline{\chi}_2, \overline{\chi}_3$ are unramified at $\wv$ and $\overline{\chi}_4, \dots, \overline{\chi}_n$ are ramified at $\wv$ (and the image of inertia under each of these characters has order 2).
\item If $v \in R_2$, then $\overline{\chi}_1, \overline{\chi}_2$ are ramified at $\wv$ (and the image of inertia under each of these characters has order 2) and $\overline{\chi}_3, \dots, \overline{\chi}_n$ are unramified at $\wv$.
\end{itemize}
These properties imply condition 5(d) above, which is what we actually need for the proofs in this section.)

The theorem we prove in this section is the following one:
\begin{theorem}\label{thm_level_raising_through_Galois_theory}
	With assumptions as above, fix a place $v_{St}$ of $F$ lying above $S - 
	(S_p \cup R)$ such that $q_{v_{St}} \equiv 1 \pmod{p}$ and 
	$\overline{r}_{\pi, \iota}|_{G_{F_{v_{St}}}}$ is trivial. Then we can find 
	a 
	RACSDC $\iota$-ordinary automorphic representation $\pi'$ of 
	$\GL_n(\A_{F})$ satisfying the following conditions:
	\begin{enumerate}
		\item There is an isomorphism $\overline{r}_{\pi', \iota} \cong 
		\overline{r}_{\pi, \iota}$.
		\item For each embedding $\tau : F 
		\to \overline{\bbQ}_p$, we have 
		\[ \mathrm{HT}_\tau(r_{\pi', \iota}) = 
		\mathrm{HT}_\tau(r_{\pi,\iota}). \]
		\item $\pi'_{v_{St}}$ is an unramified twist of the Steinberg representation. 
	\end{enumerate}
\end{theorem}
Our proof of this theorem follows a similar template to the proof of \cite[Theorem 5.1]{ctiii}. Briefly, we use our local conditions at the places of $R$, together with Theorem \ref{thm_existence_of_generic_primes} to show that (after a suitable base change) we can find a generic prime in the spectrum of the big ordinary Hecke algebra. This puts us in a position to use the ``$R_\p = \T_\p$'' theorem proved in  \cite{All19}, which is enough to construct automorphic lifts of $\overline{r}_{\pi, \iota}$ (or its base change) with the desired properties.

We now begin the proof. Let $F^a / F$ denote the  extension of $F(\zeta_p)$ cut out by $\overline{r}_{\pi, \iota}|_{G_{F(\zeta_p)}}$, and let $Y^a$ be a finite set of finite places of $F$ with the following properties:
\begin{itemize}
	\item For each place $v \in Y^a$, $v$ is split over $F^+$, prime to $S$, and $\pi_v$ is unramified.
	\item For each intermediate Galois extension $F^a / M / F$ such that $\Gal(M / F)$ is simple, there exists $v \in Y^a$ which does not split in $M$.
\end{itemize}
Then any $Y^a$-split finite extension $L / F$ is linearly disjoint from $F^a / F$. After conjugation, we can find a coefficient field $E$ such that $r_{\pi, \iota}$ is valued in $\GL_n(\cO)$, and extend it to a homomorphism $r : G_{F^+, S} \to \cG_n(\cO)$ such that $\nu \circ r = \epsilon^{1-n} \delta^n_{F / F^+}$. We write $\overline{r} : G_{F^+, S} \to \cG_n(k)$ for the reduction modulo $\varpi$ of $r$.
\begin{lemma}\label{lem_Schur_on_L}
	Let $L / F$ be an $Y^a$-split finite CM extension. Then:
	\begin{enumerate}
		\item $\overline{r}|_{G_{L^+(\zeta_p)}}$ is Schur. 
		\item $\overline{r}|_{G_L}$ is primitive. 
		\item Suppose moreover that $L / F$ is soluble. Then the base change of $\pi$ with respect to the extension $L /  F$ is cuspidal.
		\item More generally, suppose that $L / F_0 / F$ is an intermediate field with $L / F_0$ soluble, and let $\Pi$ be a RACSDC automorphic representation of $\GL_n(\A_{F_0})$ such that $\overline{r}_{\Pi, \iota} \cong \overline{r}|_{G_{F_0}}$. Then the base  change of $\Pi$ with respect to the extension $L / F_0$ is cuspidal.
	\end{enumerate}
\end{lemma}
\begin{proof}
	For the first part, it is enough to check that $L \not\subset L^+(\zeta_p)$ and $\overline{\chi}_i / \overline{\chi}_j|_{G_{L(\zeta_p)}}$ is non-trivial for each $1 \leq i < j \leq n$. We have $[L(\zeta_p) : L] = p-1$, which implies $L \not\subset L^+(\zeta_p)$, while $\overline{\chi}_i / \overline{\chi}_j(G_{L(\zeta_p)}) = \overline{\chi}_i / \overline{\chi}_j(G_{F(\zeta_p)})$, so this ratio of characters is non-trivial. The second part follows from Lemma \ref{lem_primitive_representations} and the fact that $\overline{\chi}_i / 
	\overline{\chi}_j|_{G_L}$ has order greater than $2n$ for any $i \neq j$, because $L / F$ is linearly disjoint from the extension $F^a / F$.
	
	The third part is a special case of the fourth part, so we just prove this. Suppose for contradiction that the base change of $\Pi$ with respect to the extension $L / F_0$ is not  cuspidal. Then we can find
	intermediate extensions $L / F_2 / F_1 / F_0$ such that there is a tower $F_1 = M_m / M_{m-1} / \dots / M_0  = F_0$, where each extension $M_{i+1} / M_i$ is cyclic of prime degree; $F_2 / F_1$ is cyclic 
	of prime degree $l$; the base change $\Pi_{F_1}$ of $\Pi$ to $F_1$  (constructed as the iterated cyclic base change with respect to the tower $M_m / M_{m-1} / \dots / M_0$) is cuspidal; but the 
	base change of $\Pi_{F_1}$ to $F_2$ is not cuspidal. We will derive a 
	contradiction. By the second part of the 
	lemma, $\overline{r}|_{G_{F_1}}$ is primitive. Let $\sigma \in \Gal(F_2 / 
	F_1)$ be a generator. Since the base change of $\Pi_{F_1}$ with respect to 
	the extension $F_2 / F_1$ is not cuspidal, there exists a cuspidal 
	automorphic representation $\Xi$ of $\GL_{n / l}(\bA_{F_2})$ such that the 
	base change of $\Pi_{F_1}$ is $\Xi \boxplus \Xi^{\sigma} \boxplus \dots 
	\boxplus \Xi^{\sigma^{l-1}}$ (see \cite[Theorem 4.2]{MR1007299}). We claim 
	that $\Xi$ is in fact conjugate self-dual. The representation $\Xi | \cdot 
	|^{(n/l -  n)/2}$ is regular algebraic (by \cite[Theorem 5.1]{MR1007299}). 
	Since $\Pi_{F_1}$ is conjugate self-dual, \cite[Proposition  
	4.4]{MR1007299})  shows that $\Xi \boxplus \Xi^{\sigma} \boxplus \dots 
	\boxplus \Xi^{\sigma^{l-1}}$ is  also  conjugate self-dual. The 
	classification of automorphic representations of $\GL_n$ then implies  
	that  there is an isomorphism $\Xi^{c, \vee} \cong \Xi^{\sigma^i}$ for some 
	$0 \leq i < l$. If $w$ is an infinite place of $F_2$, then the purity lemma 
	(\cite[Lemma 4.9]{Clo90}) implies that $\Xi_w \cong \Xi_w^{c, \vee}$, hence 
	$\Xi_w \cong \Xi_w^{\sigma^i}$. Since $\Xi \boxplus \Xi^{\sigma} \boxplus 
	\dots \boxplus \Xi^{\sigma^{l-1}}$ is regular algebraic, this is possible 
	only if $i = 0$ and $\Xi$ is indeed conjugate  self-dual.
	
Therefore $r_{\Xi | \cdot |^{(n/l - n)/2}, \iota}$ is defined and there is an 
isomorphism 
\[ r_{\Pi_{F_1}, \iota} \cong\Ind_{G_{F_2}}^{G_{F_1}} r_{\Xi | \cdot |^{(n/l -  
n)/2}, \iota}. \]
 This contradicts the second part of the lemma, which implies that $\overline{r}_{\Pi_{F_1}, \iota}$ is primitive. This contradiction completes the proof.
\end{proof}
Combining Lemma \ref{lem_Schur_on_L} and Proposition \ref{prop_automorphic_lifts_of_prescribed_types}, we see that Theorem \ref{thm_level_raising_through_Galois_theory} will follow provided we can find a soluble CM extension $L / F$ with the following properties:
\begin{itemize}
\item $L / F$ is $Y^a$-split.
\item There exists a RACSDC $\iota$-ordinary automorphic representation $\pi''$ of $\GL_n(\A_L)$ and a place $v''$ of $L$ lying above $v_{St}$ such that $\overline{r}_{\pi'', \iota} \cong \overline{r}_{\pi, \iota}|_{G_L}$ and $\pi''_{v''}$ is an unramified twist of the Steinberg representation.
\end{itemize}
After first replacing $F$ by a suitable $Y^a \cup R$-split soluble extension, 
we can assume in addition that $S = S_p \cup R \cup \{ v_{St}|_{F^+} \}$ and 
that $\pi$ is unramified outside $S_p \cup R$ (use the Skinner--Wiles base 
change trick as in \cite[Lemma 4.4.1]{cht}).
\begin{lemma}\label{lem_v_a}
	There exist infinitely many prime-to-$S$ places $\wv_a$ of $F$ with the following property: $\wv_a$ does not split in $F(\zeta_p)$ and $\overline{r}(\Frob_{\wv_a})$ is scalar. 
\end{lemma}
\begin{proof}
	By the Chebotarev density theorem, it is enough to find $\tau \in G_F$ such 
	that $\overline{r}(\tau)$ is scalar and $\barepsilon(\tau) \neq 1$. We can 
	choose any $\tau_0 \in G_F$ such that $\barepsilon(\tau_0)^2 \neq 1$, and 
	set $\tau = \tau_0 \tau_0^c$.
\end{proof}
Choose a place $\wv_a$ of $F$ as in Lemma \ref{lem_v_a} which is absolutely unramified and of odd residue characteristic. We set $S_a = \{ \wv_a|_{F^+} \}$ and $\widetilde{S}_a = \{ \wv_a \}$.

We will need to consider several field extensions $L / F$ and global deformation problems. We therefore introduce some new notation. We define a deformation datum to be a pair $\cD = (L, \{ R_v \}_{v \in X})$ consisting of the following data:
\begin{itemize}
	\item A $Y^a \cup \widetilde{S}_a$-split, soluble CM extension $L / F$.
	\item A subset $X \subset S - S_p$, which may be empty. We write  $\widetilde{X}$ for the pre-image of $X$ in $\widetilde{S}$.
	\item For each $v \in X$, one of the following complete Noetherian local rings $R_v$, representing a local deformation problem:
	\begin{itemize}
		\item For any $v \in R_L$ such that $\wv$ is split over $F$, the ring $R_v = R(\wv,\Theta_\wv,\overline{r}|_{G_{L_\wv}})$ (notation as in Proposition \ref{prop_properties_of_lifting_ring} --  we define $\Theta_\wv = \Theta_{\wv|_F}$).
		\item For any $v \in S_L$ such that $q_v \equiv 1 \text{ mod }p$ and $\overline{r}|_{G_{L_\wv}}$ is trivial, the  unipotently ramified local lifting ring $R_v = R_v^{1}$ considered in \cite[\S 3.3.3]{jackreducible}.
		\item For any $v \in S_L$ such that $q_v \equiv 1 \text{ mod }p$ and $\overline{r}|_{G_{L_\wv}}$ is trivial, the Steinberg local lifting ring $R_v = R_v^{St}$ considered in \cite[\S 3.3.4]{jackreducible}.
	\end{itemize}
\end{itemize}
If $\cD = (L, \{ R_v \}_{v \in X})$ is a deformation datum, then we can define the global deformation problem
\[ \cS_\cD = (L / L^+, S_{p, L} \cup X, \widetilde{S}_{p, L} \cup \widetilde{X}, \Lambda_L, \overline{r}|_{G_{L^+}}, \epsilon^{1-n} \delta^n_{L/ L^+}, \{ R_v^\triangle \}_{v \in S_{p, L}} \cup \{ R_v \}_{v \in X}).\]
We write $R_\cD = R_{\cS_\cD} \in \cC_{\Lambda_L}$ for the representing object of the corresponding deformation functor. 
\begin{lemma}\label{lem_lower_bound_on_dim}
	If $\cD =  (L, \{ R_v \}_{v \in X})$ is a deformation datum, then each irreducible component of $R_\cD$ has dimension at least $1 + n [L^+ : \Q]$.
\end{lemma}
\begin{proof}
	This follows from \cite[Proposition 3.9]{jackreducible}, noting that the term $H^0(L^+, \ad \overline{r}(1))$ vanishes because $\overline{r}|_{G_{L^+(\zeta_p)}}$ is Schur (cf. \cite[Lemma 2.1.7]{cht}).
\end{proof}
Given a deformation datum $\cD$, we define an open compact subgroup $U_\cD = \prod_{v \in S_L - S_{p, L}}U_{\cD, v} \subset \prod_{v \in S_L - S_{p, L}} \GL_n(\cO_{F_\wv})$ and a smooth $\cO[U_\cD]$-module $M_\cD = \otimes_{v \in S_L - S_{p, L}} M_{\cD, v}$ as follows:
\begin{itemize}
	\item If $v \not\in X$, then $U_{\cD, v} = \GL_n(\cO_{L_\wv})$ and $M_{\cD, v} = \cO$.
	\item If $v \in X \cap R_{L}$ and $R_v = R(\wv,\Theta_\wv,\overline{r}|_{G_{L_\wv}})$, then $U_{\cD, v} = \q_\wv$ and $M_{\cD, v}$ is an $\cO$-lattice in $\iota^{-1} \widetilde{\lambda}(\wv, \Theta_\wv, n)^\vee$ (notation as in \S \ref{subsec_local_theory_GL_n}).
	\item If $v \in X$ and $R_v = R_v^{1}$ or $R_v^{St}$, then $U_{\cD, v} = \Iw_\wv$ and $M_{\cD, v} = \cO$.
\end{itemize}
We now define a Hecke algebra $\T_\cD$ associated to any deformation datum $\cD = (L, \{ R_v \}_{v \in X})$. It is to be a finite $\Lambda_L$-algebra (or zero). If $c \geq 1$, let $U(\cD, c) \subset G(\A_{L^+}^\infty)$ be the open compact subgroup defined as follows: 
\[ U(\cD, c) = \prod_{v \in S_{p, L}}\iota_\wv^{-1} \Iw_\wv(c, c) \times (\prod_{v \in S_L - S_{p, L}} \iota_\wv^{-1}) U_{\cD} \times \prod_{v \in S_{a, L}} \iota_\wv^{-1} K_\wv(1) \times G(\widehat{\cO}_{L^+}^{S_L \cup S_{a, L}}) \]
(here we are using the notation for open compact subgroups established in \S \ref{subsec:alg_mod_forms}). 
Note that $U(\cD, 1)$ is sufficiently small, because of our  choice of $v_a$. 
We write \[ \T^{ord}(\cD, c) \subset \End_\cO(S^{ord}(U(\cD, c), M_\cD)) \]
for the $\cO$-subalgebra generated by the unramified Hecke operators $T_w^j$ at split places $v = w w^c \not\in S_L \cup S_{a, L}$ and the diamond operators $\langle u \rangle$ for $u \in \Iw_\wv(1, c)$ ($v \in S_{p, L}$).

Following \cite[\S 2.4]{ger}, we define $\T^{ord}(\cD) = \varprojlim_c \T^{ord}(\cD, c)$ and $H^{ord}(\cD) = \varprojlim_c \Hom(S^{ord}(U(\cD, c), M_\cD), \cO)$. We endow $\T^{ord}(\cD)$ with a $\Lambda_L$-algebra structure using the same formula as in \cite[Definition 2.6.2]{ger}. We then have the following result.
\begin{proposition}
	$H^{ord}(\cD)$ is a finite free $\Lambda_L$-module and $\T^{ord}(\cD)$ is a finite faithful $\Lambda_L$-algebra, if it is non-zero. 
\end{proposition}
\begin{proof}
	This can be proved in the same way as \cite[Proposition 2.20]{ger} and \cite[Corollary 2.21]{ger}. The proof uses that $U(\cD, 1)$ is sufficiently small.
\end{proof}
We write $\ffrm_\cD \subset \T^{ord}(\cD)$ for the ideal generated by $\ffrm_{\Lambda_L}$ and the elements $T_w^j - q_w^{j(j-1)/2} \tr \wedge^j \overline{r}(\Frob_w)$ ($w$ a split place of $L / L^+$ not lying above a place of $S_L \cup S_{a, L}$). It is either a maximal ideal with residue field $k$, or the unit ideal. In either case we set $\T_\cD = \T^{ord}(\cD)_{\ffrm_D}$, which is either a finite local $\Lambda_L$-algebra or the zero ring. (In the cases we consider, it will be non-zero, but this will require proof.)

For any deformation datum $\cD$, we write $P_\cD$ for the 
$\Lambda_L$-subalgebra of $R_{\cD}$ topologically generated by the coefficients 
of the characteristic polynomials of elements of $G_L$ in (a representative of) 
the universal deformation $r_{\cS_\cD}$. By \cite[Proposition 
3.26]{jackreducible}, the group determinant  $\det r_{\cS_\cD}|_{G_F}$ is valued in 
$P_{\cD}$, and $P_\cD$ is a complete Noetherian local $\Lambda_L$-algebra. By 
\cite[Proposition 3.29]{jackreducible}, $R_{\cD}$ is a finite $P_\cD$-algebra.
\begin{lemma}\label{lem_gal_to_hecke}
	Let $\cD = (L, \{ R_v \}_{v \in X})$ be a deformation datum, and suppose 
	that $R_v \neq R_v^{St}$ for all $v \in X$. Then there is a natural 
	surjective morphism $P_\cD \to \T_\cD$ of $\Lambda_L$-algebras.
\end{lemma}
\begin{proof}
	The proof is the essentially the same as the proof of \cite[Proposition 4.12]{jackreducible}, but we give the details for completeness. It is enough to construct maps $P_\cD \to \T^{ord}(\cD, c)_{\ffrm_D}$ which are compatible as $c \geq 1$ varies. Let $\Pi(\cD, c)$ denote the set of automorphic representations $\sigma$ of $G(\A_{L^+})$ with the following properties:
	\begin{itemize}
		\item $\overline{r}_{\sigma, \iota} \cong \overline{r}|_{G_L}$. 
		\item $\sigma_\infty$ is the trivial representation.
		\item The subspace 
		\[ \Hom_{U(\cD, c)}(M_\cD^\vee, 
		\iota^{-1}\sigma^\infty)^{ord} \subset \Hom_{U(\cD, c)}(M_\cD^\vee 
		, \iota^{-1}\sigma^\infty)\]
		 where all the Hecke operators $U_{\wv, 0}^j$ ($v \in S_p$, $j = 1, \dots, n$) act with eigenvalues which are $p$-adic units is non-zero. 
	\end{itemize}
	Then there is an injection 
	\begin{equation}\label{eqn_injection_of_hecke_algebras} \T^{ord}(\cD, c)_{\ffrm_D} \otimes_\cO \overline{\Q}_p \to \oplus_{\sigma \in \Pi(\cD, c)} \overline{\Q}_p
	\end{equation} 
	which sends any Hecke operator to the tuple of its eigenvalues on each $(\iota^{-1} \sigma^\infty)^{U(\cD, c)}$. We can find a coefficient field $E_c / E$ with ring of integers $\cO_c$ and for each $\sigma \in \Pi(\cD, c)$, a homomorphism $r_\sigma : G_{L^+, S} \to \cG_n(\cO_c)$ lifting $\overline{r}$ and such that $r_\sigma|_{G_L} \cong r_{\sigma, \iota}$ (apply \cite[Lemmas 2.1.5, 2.1.7]{cht}).
	
	Let $A_c \subset k \oplus \bigoplus_{\sigma \in \Pi(\cD, c)} \cO_c$ be the 
	subring consisting of elements $(a, (a_\sigma)_\sigma)$ such that for each 
	$\sigma$, $a_\sigma \text{ mod }\varpi_c = a$. Then $A_c$ is a local ring 
	containing the image of the map (\ref{eqn_injection_of_hecke_algebras}), 
	and the representation $\overline{r} \times (\times_{\sigma \in \Pi(\cD, 
	c)} r_\sigma)$ is valued in $\cG_n(A_c)$ and is of type $\cS_\cD$ (by our choice of deformation problems and level structures). Writing $Q_{S_L} \in \cC_\cO$ for 
	the ring classifying pseudocharacters which lift $\tr  
	\overline{r}|_{G_{L, S_L}}$, we see that there is a commutative diagram
	\[ \xymatrix{ R_\cD \ar[r] & A_c \\ Q_{S_L} \widehat{\otimes}_\cO \Lambda_L \ar[r] \ar[u] & \T^{ord}(\cD, c)_{\ffrm_D} \ar[u]. } \]
	The ring $P_\cD$ is equal to the image of the map $Q_{S_L} 
	\widehat{\otimes}_\cO \Lambda_L \to R_\cD$. The proof is thus complete on 
	noting that the right vertical arrow is injective and the bottom horizontal 
	arrow is surjective. 
\end{proof}
 We define $J_\cD = \ker(P_\cD \to \T_\cD)$; this is a proper ideal if and only if $\T_\cD \neq 0$.
 
We now fix a place $\overline{v}_{St}$ of $\overline{F}$ above $v_{St}$. If $L / F$ is a CM extension, we write $v_{St, L} = \overline{v}_{St}|_{L}$.
\begin{lemma}\label{lem_no_ramification}
	We can find a deformation datum $\cD_1 = (L_1, \emptyset)$ with the following properties:
		\begin{enumerate}
			\item $\T_{\cD_1} \neq 0$.
			\item There exists a prime ideal $\p_1 \subset R_{\cD_1}$ of dimension 1 and characteristic $p$ which is generic, such that $J_{\cD_1} \subset \p_1$, and such that $r_{\p_1}|_{G_{L_1, v_{St, L_1}}}$ is the trivial representation.
		\end{enumerate}
\end{lemma}
\begin{proof}
We first claim that we can find an $\widetilde{R} \cup Y^a \cup \{ \wv_a \}$-split soluble CM extension $L_0 / F$ with the following properties:
\begin{itemize} \item For each $i \in \{ 1, 2 \}$, let $L_{0, S_p}$ denote the maximal abelian pro-$p$ extension of $L_0$ unramified outside $S_{p, L_0}$, and let $\Delta_{L_0} = \Gal(L_{0, S_p} / L_0) / (c+1)$. Let $d_{R_{i, L_0}}$ denote the $\bZ_p$-rank of the subgroup of $\Delta_{L_0}$ topologically generated by the elements $\Frob_\wv$, $v \in R_{i, L_0}$. Then $d_{R_{i, L_0}} > n + n^2$.
	\item For each $v \in S_{p, L_0}$, $[L_{0, \wv} : \Q_p] > n^2$.
	\item For each  $v \in S_{L_0} - (S_{p, L_0} \cup R_{L_0})$, 
	$\overline{r}|_{G_{L_{0, \wv}}}$ is trivial, and $q_v \equiv 1 \text{ mod 
	}p$.
\end{itemize}
The third property is automatic, since $S_{L_0} - (S_{p, L_0} \cup R_{L_0})$ 
	consists of primes above $v_{St}|_{F^+}$. We can construct an extension satisfying the first two properties using a similar idea to the proof of \cite[Theorem 7.1]{All19}. Indeed, we can find, for any odd integer $d \geq 1$, a cyclic totally real extension $M_d / F^+$ which is $R \cup \{ v|_{F^+} \mid v \in Y^a \} \cup S_a$-split and in which each place $v \in S_p$ is totally inert. If $d > n^2$ and $L_0 = M_d \cdot F$ then $L_0 / F$ will be a $\widetilde{R} \cup Y^a \cup \{ \wv_a \}$-split soluble CM extension which also  satisfies the second point above. We need to explain how to arrange that the first point is also satisfied. By class field theory, $d_{R_{i, L_0}}$ is equal to the $\bZ_p$-rank of the subgroup of $(\cO_{L_0} \otimes_\bZ \bZ_p)^\times$ topologically generated by $(\cO_{L_0, R_{i, L_0}}^\times)^{c = -1}$. Since $\cO_{L_0, R_{i, L_0}}^{c = -1} \otimes_\bZ \overline{\bQ}_p$ decomposes as a $\overline{\bQ}_p[\Gal(L_0 / F^+)]$-module with multiplicity 1, \cite[Proposition 19]{Mai02} shows that this rank equals the $\bZ$-rank of $(\cO_{L_0, R_{i, L_0}}^\times)^{c = -1}$, which is $d|R_i|$. Choosing any $d > n + n^2$ therefore gives an extension with the desired properties.

 Let $\pi_0$ be the base change of $\pi$ with respect to the extension $L_0 / F$. It is cuspidal by Lemma \ref{lem_Schur_on_L}. Let $\cD_0 = (L_0, \{ R(\wv,\Theta_\wv,\overline{r}|_{G_{L_{0,\wv}}}) \}_{v \in R_{L_0}})$. Then $\cD_0$ is a deformation datum and the existence of $\pi_0$, together with Theorem \ref{thm_cuspidal_descent}, shows that $\T_{\cD_0} \neq 0$.
Let $B = R_{\cD_0} / (J_{\cD_0}, \ffrm_{R_{v_{St, L_0}}^{ur}})$. Then $\dim B 
\geq 
n[L_0^+ : \Q] - n^2$, and we may apply Theorem 
\ref{thm_existence_of_generic_primes} to conclude the existence of a generic 
prime $\p_0 \subset R_{\cD_0}$ of dimension 1 and characteristic $p$ which 
contains $(J_{\cD_0}, \ffrm_{R_{v_{St, L_0}}^{ur}})$.

We now make another base change. Let $L_1 / L_0$ be a CM extension with the following properties:
\begin{itemize}
	\item $L_1 / F$ is soluble and $Y^a \cup \{ \wv_a \}$-split.
	\item For each $v \in R_{L_1}$, the natural morphism 
	$R^{\square}_\wv \to R(\wv,\Theta_\wv,\overline{r}|_{G_{L_{0,\wv}}})$ factors over the unramified quotient 
	$R^\square_{\wv}\to R_\wv^{ur}$ (cf.~Proposition 
	\ref{prop_properties_of_lifting_ring}). 
	\item For each $v \in S_{L_1} - S_{p, L_1}$, $q_v \equiv 1 \text{ mod }p$ and $\overline{r}|_{G_{L_{1, \wv}}}$ is trivial.
\end{itemize}
Let $\pi_1$ denote the base change of $\pi_0$ to $L_1$. Then $\pi_1$ is a RACSDC, $\iota$-ordinary automorphic representation of $\GL_n(\A_{L_1})$ which is unramified outside $S_{p, L}$. Thus $\cD_1 = (L_1, \emptyset)$ is a deformation datum and $\T_{\cD_1} \neq 0$.

To complete the proof, it is enough to produce a commutative diagram
\[ \xymatrix{R_{\cD_1} \ar[d] & P_{\cD_1} \ar[l]\ar[d] \ar[r] & \T_{\cD_1}\ar[d] \\
	R_{\cD_0}  & P_{\cD_0} \ar[l]\ar[r]   & \T_{\cD_0}. }
	 \]
	Indeed, then we can take $\p_1$ to be the pullback of $\p_0$ along the map $R_{\cD_1} \to R_{\cD_0}$.
	The only arrow which has not already been constructed is the arrow $\T_{\cD_1} \to \T_{\cD_0}$. This may be constructed in exactly the same way as \cite[Proposition 4.18]{jackreducible}, using the construction of the Hecke algebra as an inverse limit as in the proof of Lemma \ref{lem_gal_to_hecke}, provided we can prove the following statement: for any automorphic representation $\sigma_0$ of $G(\bA_{L^+_0})$ satisfying the following conditions:
	\begin{itemize}
	\item There  exists $c \geq 1$ such that $\Hom_{U(\cD_0, c)}(M^\vee_{\cD_0}, \iota^{-1}\sigma_0^\infty)^{ord} \neq 0$;
	\item $\sigma_{0, \infty}$ is trivial;
	\item There is an isomorphism $\overline{r}_{\sigma_0, \iota} \cong \overline{r}_{\pi, \iota}|_{G_{L_0}}$;
	\end{itemize}
	(in other words, such that $\sigma_0$ contributes to $S^{ord}(U(\cD, c), M_\cD)_{\ffrm_{\cD_0}}$ for some $c \geq 1$), there exists an automorphic representation $\sigma_1$ of $G(\bA_{L^+_1})$ satisfying the following conditions:
	\begin{itemize}
		\item There exists $c \geq 1$ such that $\Hom_{U(\cD_1, c)}(M_{\cD_1}^\vee, \iota^{-1} \sigma_1^\infty)^{ord} \neq 0$;
	\item $\sigma_{1, \infty}$ is trivial;
	\item There is an isomorphism $r_{\sigma_1, \iota} \cong r_{\sigma_0, \iota}|_{G_{L_1}}$.
	\end{itemize}
	To see this, we first show that the base change of any such $\sigma_0$ to 
	$L_0$ (in the sense of Theorem \ref{thm_base_change}) must be cuspidal. We 
	will show that in fact $r_{\sigma_0, \iota}$ is irreducible. Suppose that 
	there  is a decomposition $r_{\sigma_0, \iota} \cong \rho_1 \oplus \rho_2$. 
	By assumption, there exists $v \in R_L$ such that both 
	$\overline{\rho}_1|_{G_{L_0, \wv}}$ and $\overline{\rho}_2|_{G_{L_0, \wv}}$ 
	admit an unramified  subquotient. However, local-global compatibility 
	(together with Proposition \ref{prop_types_for_general_linear_group}) shows 
	that  $r_{\sigma_0, \iota}|_{G_{L_0, \wv}} \cong \rho_1' \oplus \rho_2'$, 
	where $\rho_1' $ is an irreducible  representation of $G_{L_{0, \wv}}$ with 
	unramified residual representation and $\rho_2'$ is a representation of 
	$G_{L_{0, \wv}}$ such that $\overline{\rho}_2'$ is a sum of ramified 
	characters. This is a contradiction unless one of $\rho_1$ and $\rho_2$ is 
	the zero representation. If $\mu_0$ denotes the base change of $\sigma_0$, 
	a RACSDC $\iota$-ordinary automorphic representation of $\GL_n(\bA_{L_0})$, 
	then the base change of $\mu_0$ with respect to the soluble extension $L_1 
	/ L_0$ is also cuspidal, by Lemma \ref{lem_Schur_on_L}, and 
	the existence of $\sigma_1$ follows  from Theorem 
	\ref{thm_cuspidal_descent}. This completes the proof.
\end{proof}
We can now complete the proof of Theorem 
\ref{thm_level_raising_through_Galois_theory}. We recall that it is enough to construct a RACSDC, $\iota$-ordinary automorphic representation $\pi''$ of $\GL_n(\A_{L_1})$ such that $\overline{r}_{\pi'', \iota} \cong \overline{r}|_{G_{L_1}}$ and $\pi''_{v_{St, L_1}}$ is an unramified twist of the Steinberg representation. Let $\cD_1$ and $\p_1$ be as in 
the statement of Lemma \ref{lem_no_ramification}. Consider 
the deformation data $\cD_{1,a} = (L_1, \{ R_{v_{St, L_1}}^{St} \})$ and 
$\cD_{1, b} = (L_1, \{ R_{v_{St, L_1}}^1 \})$. Then there are surjections 
$R_{\cD_{1, b}} \to R_{\cD_{1, a}}$ and $R_{\cD_{1, b}} \to R_{\cD_1}$ and the 
prime $\p_1$ lies in intersection of $\Spec R_{\cD_1}$ and $\Spec R_{\cD_{1, 
a}}$ in $\Spec R_{\cD_{1, b}}$. We see that the hypotheses of \cite[Theorem 
4.1]{All19} are satisfied for $R_{\cD_{1, b}}$ (in the notation of loc.~cit., 
we set $R = \{v_{St,L_1}\}, \chi_{v_{St,L_1}} = 1, S(B) = \emptyset$), and 
conclude that for any 
minimal prime $Q \subset R_{\cD_{1, b}}$ contained in $\p_1$, we have 
$J_{\D_{1, b}} \subset Q$. In particular, $\dim R_{\cD_{1, b}} / Q = \dim 
\Lambda_{L_1}$. (We  remark that the essential condition for us in applying \cite[Theorem 
4.1]{All19} is that there is no ramification outside $p$; this is the reason for  proving Lemma \ref{lem_no_ramification}.)

Lemma \ref{lem_lower_bound_on_dim} and Theorem 
\ref{thm_finiteness_of_deformation_ring} show together that each minimal prime 
of $R_{\cD_{1, a}}$ has dimension equal to $\dim \Lambda_{L_1}$. Let $Q_a 
\subset R_{\cD_{1, a}}$ be a minimal prime contained in $\p_1$. Then $Q_a$ is 
also a minimal prime of $R_{\cD_{1,  b}}$, $J_{\cD_{1, b}} \subset Q_a$, and 
there exists a minimal prime $Q_0$ of $\Lambda_{L_1}$ such that $R_{\cD_{1, a}} 
/ 
Q_a$ is a finite faithful $\Lambda_{L_1} / Q_0$-algebra. If $Q_1 = Q_a \cap 
P_{\cD_{1, a}}$, then there are finite injective algebra maps
\[ \Lambda_{L_1} / Q_0 \to P_{\cD_{1, a}} / Q_1 \cong \T_{\cD_{1, b}} / Q_1 \to R_{\cD_{1, a}} / Q_a. \]
Using \cite[Lemma 
2.25]{ger} and Theorem \ref{thm_base_change}, we conclude the existence of an automorphic representation $\pi''$ of $\GL_{n}(\A_{L_1})$ with the required properties.

\section{Level 1 case}\label{sec_existence_of_seed_points}

The goal of this section is to prove Theorem \ref{introthm_existence_of_single_symmetric_power}, using the level-raising results established in the last few sections. Combining this with the results of \S\S \ref{sec_rigid_geometry} -- \ref{sec_ping_pong}, we will then be able to deduce Theorem \ref{thm_intro_main_theorem}.

Our starting point is $\sigma_0$, the cuspidal automorphic representation of 
$\GL_2(\bA_\bQ)$ of weight 5 associated to the unique normalised newform
\[ f_0(q) = q  - 4q^{2} + 16q^{4} - 14q^{5} - 64q^{8} +  \dots \]
 of 
level $\Gamma_1(4)$ and weight 5; it is the automorphic induction from the 
quadratic extension $K = \bQ(i)$ of the unique unramified Hecke character with 
$\infty$-type $(4,0)$. For any prime $p \equiv 1 \text{ mod 
}4$ and isomorphism $\iota : \overline{\bQ}_p \to \bC$, $\sigma_0$ is 
$\iota$-ordinary. We observe that ${r}_{\sigma_0, \iota} \cong 
\Ind_{G_K}^{G_\Q} \psi$ for a character ${\psi} : 
G_K \to \Qpbarx$ (which depends on $p$) and $\det{r}_{\sigma_0, \iota} = 
\delta_{K / \bQ} \epsilon^{-4}$, where $\delta_{K / \bQ} : G_\Q \to \{\pm 1\}$ is the 
quadratic character with kernel $G_K$.

The main technical result of 
this 
section is the following theorem:
\begin{theorem}\label{thm:levelraisingalln}
	Let $n \geq 3$ be an integer. Suppose given the following data:
	\begin{enumerate}
		\item A prime $p \equiv 1 \pmod{48 n!}$ and an isomorphism $\iota : \Qpbar \to \C$.
		\item A prime $q \neq p$.	
		\item A finite set $X_0$ of places of $K$, each prime to $2pq$.
		\item A de Rham character $\omega : G_K \to \overline{\Q}_p^\times$ such that $\omega \omega^c = \epsilon^3$ and $\omega|_{G_{K_v}}$ is unramified if $v \in X_0$.
	\end{enumerate}
	Then there exists a soluble CM extension $F / 
	K$ with the following properties:
	\begin{enumerate}
		\item $F / K$ is $X_0$-split.
		\item There is a RACSDC, $\iota$-ordinary automorphic representation 
		$\Pi$ of $\GL_n(\A_F)$ with the following properties:
		\begin{enumerate}
			\item There is an isomorphism 
			\[ \overline{r}_{\Pi, \iota} \cong \overline{\omega}^{n-1}|_{G_F} \otimes \Sym^{n-1} \overline{r}_{\sigma_0, \iota}|_{G_F}. \]
			\item For each embedding $\tau : F \to \overline{\bbQ}_p$, we have 
			\[ \mathrm{HT}_\tau(r_{\Pi, \iota}) = \mathrm{HT}_\tau(\omega^{n-1}|_{G_F} \otimes \Sym^{n-1} r_{\sigma_0, \iota}|_{G_F}). \]
			\item There exists a place $v | q$ of $F$ such that $\Pi_v$ is an unramified twist of the Steinberg representation.
		\end{enumerate}
	\end{enumerate}
\end{theorem}
The following lemma will be used repeatedly.
\begin{lemma}\label{lem_consequence_of_disjointness}
	Let $n \geq 3$ be an integer, and let $p \equiv 1 \pmod{48n!}$. Let $\omega : G_K \to \overline{\Q}_p^\times$ be a de Rham character such that $\omega \omega^c = \epsilon^3$. 
	
	Let $F / K$ be a finite CM extension which is linearly disjoint from the extension of $K(\zeta_p)$ cut out by $\overline{r}_{\sigma_0, \iota}|_{G_{K(\zeta_p)}}$, and set $\overline{\chi}_i = \overline{\omega}^{n-1} \otimes\overline{\psi}^{n-i} (\overline{\psi}^c)^{i-1}$, $\overline{\rho} = \overline{\chi}_1 \oplus \dots \oplus \overline{\chi}_n$, so that there is an isomorphism
	\[ \overline{\rho} \cong \overline{\omega}^{n-1} \otimes \Sym^{n-1} \overline{r}_{\sigma_0, \iota}|_{G_K}. \]
	 Then:
	\begin{enumerate}
		\item $[F(\zeta_p) : F] = p-1$.
		\item $\overline{\psi} / \overline{\psi}^c|_{G_{F(\zeta_p)}}$ has order greater than $2n(n-1)$ and for each $1 \leq i < j \leq n$, $\overline{\chi}_i / \overline{\chi}_j|_{G_{F(\zeta_p)}}$ has order  greater than $2n$.
		\item For each $1 \leq i \leq n$, $\overline{\chi}_i \overline{\chi}_i^c = \epsilon^{1-n}$.
		\item $\zeta_p \not\in F^{\ker \ad \overline{\rho}}$ and $F \not\subset F^+(\zeta_p)$.
		\item $\overline{\rho}|_{G_F}$ is primitive.
	\end{enumerate}
\end{lemma}
\begin{proof}
	We have $[F(\zeta_p) : F] = p-1$ because $F / K$ is disjoint from $K(\zeta_p) / K$. To justify the second point, let $L / K$ denote the extension cut out by $\overline{\psi} / \overline{\psi}^c$. We must show that $[L \cdot F(\zeta_p) : F(\zeta_p)] > 2n(n-1)$. We note that $[L : K] \geq (p-1)/4$, because the restriction of $\overline{\psi} / \overline{\psi}^c$ to an inertia group at $p$ has order $(p-1)/4$. Moreover, $L \cap K(\zeta_p)$ has degree at most 2 (since $c$ acts as 1 on $\Gal(K(\zeta_p) / K)$ and as $-1$ on $\Gal(L / K)$), so $[L(\zeta_p) : K] \geq (p-1)^2/8$.
	
	Since $F / K$ is supposed disjoint from $L(\zeta_p) / K$, we have $[F \cdot L(\zeta_p) : K] \geq (p-1)^2 [F : K] / 8$. Since $[F(\zeta_p) : F] = p-1$, we have $[F(\zeta_p) : K] = (p-1) [ F : K]$. Putting these together we find
	\[ [F \cdot L(\zeta_p) : F(\zeta_p)] = \frac{[F \cdot L(\zeta_p) : K]}{[F(\zeta_p) : K]} \geq (p-1)/8. \]
	Since we assume $p \equiv 1 \pmod{48n!}$, we in particular have $p - 1 \geq 
	48n!$, hence $(p-1)/8 > 2n(n-1)$.
	
	If $1 \leq i < j \leq n$ then $\overline{\chi}_i / \overline{\chi}_j = (\overline{\psi}^c / \overline{\psi})^{i-j}$, so this shows the second point of the lemma. For the third point we compute
	\[ \overline{\chi}_i \overline{\chi}_i^c = (\overline{\omega} \overline{\omega}^c|_{G_F})^{n-1} (\overline{\psi} \overline{\psi}^c|_{G_F})^{n-1} = \epsilon^{1-n}.  \]
We now come to the fourth point. To show that $\zeta_p \not\in F^{\ker \ad \overline{\rho}}$, we must find $\tau \in G_F$ such that $\overline{\rho}(\tau)$ is scalar but $\overline{\epsilon}(\tau) \neq 1$. We can choose $\tau =\tau_0 \tau_0^c$ for any $\tau_0 \in G_F$ such that $\overline{\epsilon}(\tau_0)^2 \neq 1$. Such a $\tau_0$ exists because $[F(\zeta_p) : F] = p-1$, and $\overline{\rho}(\tau)$ is scalar by the third part of the lemma. If $F \subset F^+(\zeta_p)$ then $F(\zeta_p) = F^+(\zeta_p)$ and $[F(\zeta_p) : F] = [F^+(\zeta_p) : F^+] / [F : F^+] = (p-1)/2$, contradicting the first part of the lemma. 

 For the fifth point it is enough, by Lemma \ref{lem_primitive_representations}, to show that for each $1 \leq i < j \leq n$, $\overline{\chi}_i / \overline{\chi}_j$ has order greater than $n$. This follows from the second point. 
\end{proof}
Before giving the proof of Theorem \ref{thm:levelraisingalln}, we give a corollary which establishes 
the existence of the automorphic representations necessary for the proof of Theorem 
\ref{thm_seed_points}.
\begin{cor}\label{cor:level2qseed}
	Let $n \geq 3$ be an integer. Then there exists a 
	cuspidal automorphic representation $\sigma$ of 
	$\GL_2(\A_\Q)$ of weight $5$ with the following properties:
	\begin{enumerate}
		\item $\sigma$ is unramified away from $2$ and a prime $q \equiv 3 
		\text{ mod }4$.
		\item $\sigma_2$ is 
		isomorphic to a principal series representation 
		$i_{B_2}^{\GL_2} \chi_1 \otimes 
		\chi_2$, where $\chi_1$ is unramified and $\chi_2$ has conductor 4. 
		\item $\sigma_{q}$ is an unramified twist of the Steinberg 
		representation.
		\item For any prime $p$ and any isomorphism $\iota 
		: \overline{\bQ}_p \to \bC$, $\Sym^{n-1} r_{\sigma, \iota}$ is 
		automorphic. 
	\end{enumerate}
\end{cor}
\begin{proof}
	Choose a prime $p \equiv 1 \pmod{48n!}$ and an isomorphism $\iota: \Qpbar \to \C$. It suffices to construct $\sigma$ as in the statement of the corollary such that $\Sym^{n-1} r_{\sigma, \iota}$ is automorphic for our fixed choice of $\iota$.
	
	Let $F^{avoid} / \Q$ denote the extension cut out by $\overline{r}_{\sigma_0, 
		\iota} \oplus \overline{\epsilon}$, and choose a prime $q$ satisfying 
		$q\equiv 3\text{ mod }4$ (so $a_q(f_0)=0$) and $q \equiv -1 \text{ mod 
		}p$. This implies that $\sigma_0$ 
	satisfies the level-raising congruence at $q$. By  a level-raising result 
	for $\GL_2(\A_\Q)$ (e.g.~\cite[Corollary 6.9]{diamond-ordinary}), we can 
	find an $\iota$-ordinary 
	cuspidal automorphic representation $\sigma$ of $\GL_2(\bA_\bQ)$ 
	satisfying the following conditions:
	\begin{itemize}
		\item $\sigma$ has weight 5, and $\overline{r}_{\sigma, \iota} 
		\cong \overline{r}_{\sigma_0, \iota}$.
		\item $\sigma$ is unramified at primes not dividing $2q$; $\sigma_{2}$ 
		is isomorphic to a principal series representation 
		$i_{B_2}^{\GL_2} \chi_1 
		\otimes \chi_2$, where $\chi_1$ is unramified and $\chi_2$ has 
		conductor $4$; and $\sigma_{q}$ is an unramified twist of the Steinberg 
		representation.
	\end{itemize}
	Let $\omega : G_K \to \overline{\bbQ}_p^\times$ be a character 
	crystalline at $p$ and unramified at $q$ and such that $\omega \omega^c = 
	\epsilon^{3}$. Then 
	$(\psi \omega) (\psi \omega)^c = \epsilon^{-1}$. We take $X_0$ be a set of prime-to-$2pq$ places of $K$ at which $\omega$ is unramified, and with the property that any $X_0$-split extension of $K$ is linearly disjoint from $F^{avoid} / K$.
	
	Let $F / K$ and $\pi$ be the soluble CM extension and RACSDC automorphic 
	representation of $\GL_n(\A_F)$ whose existence is asserted by Theorem 
	\ref{thm:levelraisingalln}. Thus in particular $\pi$ is $\iota$-ordinary, 
	is an unramified twist of Steinberg at some place $v | q$ of $F$, and there 
	are isomorphisms
	\[ \overline{r}_{\pi, \iota} \cong \overline{\omega}|_{G_F}^{n-1} \otimes (\oplus_{i=1}^n \overline{\psi}^{n-i} (\overline{\psi}^c)^{i-1}|_{G_F}) \cong \overline{\omega}^{n-1}|_{G_F} \otimes \Sym^{n-1} \overline{r}_{\sigma, \iota}|_{G_F}. \]	
	We now want to apply \cite[Theorem 6.1]{All19} (an automorphy lifting 
	theorem) to conclude that the 
	representation $\omega^{n-1}\otimes\Sym^{n-1}r_{\sigma,\iota}|_{G_F}$ is automorphic. This will in turn imply, by soluble descent, that $\Sym^{n-1} r_{\sigma, \iota}$ is automorphic. The hypotheses of \cite[Theorem 6.1]{All19} may be checked using Lemma \ref{lem_consequence_of_disjointness}. This concludes the proof.
\end{proof}	
We first prove Theorem \ref{thm:levelraisingalln} in the case where 
$n = 2k + 1$ is an odd integer, using the results of \S \S
\ref{sec:automorphic_level_raising} -- \ref{sec:galois_level_raising}.
\begin{prop}\label{prop:levelraisingoddn}
Theorem \ref{thm:levelraisingalln} holds when $n = 2k+1$ is odd. 
\end{prop}
\begin{proof}
 We prove the proposition by induction on odd integers $n = 2k + 1$.	Let $p$, $q$, $X_0$, $\omega$ be as in the statement of Theorem 
	\ref{thm:levelraisingalln}. Let $Z$ denote the set of rational primes below which $\omega$ is ramified, together with $2, p, q$. Let $F^{avoid} / K$ denote the extension of $K$ cut 
	out by $\overline{r}_{\sigma_0, \iota} \oplus \overline{\epsilon}$. We fix 
	a finite set $X$ of finite places of $K$ with the following properties:
	\begin{itemize}
	\item $X$ contains $X_0$.
	\item If $v \in X$ then $v$ is prime to $Z$. In particular, $\omega|_{G_{K_v}}$ is unramified.
	\item For each subextension $M / K$ of $F^{avoid} / K$ with $\Gal(M/K)$ simple and non-trivial, there exists $v \in X$ which does not split in $M$.
	\end{itemize}
	Let $q_0$ be a prime not in $Z$ and which does not split in $K$, and let $Y$ denote the set of rational primes dividing $q_0$ or an element of $X$. We make the following observations:
	\begin{itemize}
		\item If $F / K$ is a finite $X$-split extension, then $F / K$ is linearly disjoint from $F^{avoid} / K$.
		\item If $F_0 / \Q$ is a finite $Y$-split extension, then $F_0 / \Q$ is linearly disjoint from $F^{avoid} / \Q$ and $F_0 K / K$ is linearly disjoint from $F^{avoid} / K$.
	\end{itemize} Note in particular that $Y$-split extensions are linearly 
	disjoint from $K/\Q$.
	 We can find distinct rational primes $q_1, q_2, q_3$ satisfying the following conditions:
\begin{itemize}
\item For each $i = 1, 2, 3$, we have $q_i \not \in Y \cup Z$ and $q_i$ splits in $K$. In particular, $q_i$ is odd.
\item We have $q_1 > n$ and $q_1 \text{ mod }p$ is a primitive $6^\text{th}$ root of unity. The eigenvalues of $\Frob_{q_1}$ on $\Ind_{G_K}^{G_\Q} \overline{\psi}^{2k}$ have ratio $q_1^{\pm  1 } \text{ mod }p$, while the eigenvalues of $\Frob_{q_1}$ on $\Ind_{G_K}^{G_\Q} \overline{\psi}$ have ratio which is a primitive $12k^\text{th}$ root of unity in $\bF_p^\times$.
\item We have $q_2  \equiv -1 \text{ mod }p$ and the eigenvalues of  $\Frob_{q_2}$ on $\Ind_{G_K}^{G_\Q} \overline{\psi}^{2k}$ have ratio $-1$.
\item The number $q_3 \text{ mod }p$ is a primitive $(n-2)^\text{th}$ root of unity and the eigenvalues of $\Frob_{q_3}$  on $\Ind_{G_K}^{G_\bQ} \overline{\psi}$ have ratio $q_3^{\pm 1}  \text{ mod }p$.
\end{itemize}	
	To construct $q_1, q_2, q_3$ we use the Chebotarev density theorem. After conjugation, we can assume that $\overline{r}_{\sigma_0, \iota}|_{G_K} = \overline{\psi} \oplus \overline{\psi}^c$ is diagonal. Consideration of the restriction of $\overline{r}_{\sigma_0, \iota}|_{G_K}$ to the inertia groups at $p$ shows that $(\overline{r}_{\sigma_0, \iota} \oplus \overline{\epsilon})(G_K)$ contains the subgroup
	\[ \{ (\diag(a^4, (c/a)^4), c^{-1}) \mid a, c \in \bF_p^\times \} \subset 
	\GL_2(\overline{\bF}_p) \times \bF_p^\times. \]
	We assume $p \equiv 1 \pmod{48n!}$, hence in 
	particular $p \equiv 1 \pmod{96 k}$. Let $z \in \bF_p^\times$ be an element of order $96k$, and let $x_1 = z^{1+8k}$, $y_1 = z^{16k}$. Then $y_1$ is a primitive $6^\text{th}$ root of unity and $x_1^{16k} = 
	y_1^{1 + 8k}$. If the prime $q_1$ is chosen so that $q_1 > n$ and $(\overline{r}_{\sigma_0, 
	\iota} \oplus \overline{\epsilon})(\Frob_{q_1}) = (\diag(x_1^4, (y_1/x_1)^4), 
	y_1^{-1})$, then the eigenvalues of $\Frob_{q_1}$ in $\Ind_{G_K}^{G_\bQ}  \overline{\psi}$ have ratio $x_1^8 / y_1^4 = z^{8+64k - 64k} = z^8$, a  primitive $12k^\text{th}$ root of unity, while the eigenvalues of $\Frob_{q_1}$ in $\Ind_{G_K}^{G_\bQ}
	\overline{\psi}^{2k}$ have ratio
	\[ z^{16k} = y_1 = \epsilon^{-1}(\Frob_{q_1}) \equiv q_1 \pmod{p}. \]
	
	We can choose the prime $q_2$ so that $(\overline{r}_{\sigma_0, \iota} \oplus \overline{\epsilon})(\Frob_{q_2}) = (\diag(x_2^4, x_2^{-4}), 
	-1)$, where $x_2 \in \bF^\times_p$ satisfies $x_2^{16k} = -1$; and we can choose $q_3$  so that $(\overline{r}_{\sigma_0, \iota} \oplus \overline{\epsilon})(\Frob_{q_3}) = (\diag(x_3^4, (y_3/x_3)^4), 
	y_3^{-1})$, where $x_3, y_3 \in \bF^\times_p$, $y_3$ is a primitive $(n-2)^\text{th}$ root of unity, and $x_3$ is chosen  so that $x_3^8 = y_3^3$. These choices of $x_i, y_i$ are again possible because of the congruence $p \equiv 1 \text{ mod }48n!$.
	
We fix real quadratic extensions $M_i / \Q$ $(i =1, 2 ,3)$ with the following properties:
	\begin{itemize}
		\item $M_1$ is $Y \cup \{p,q, q_1, q_2 \}$-split,  and $q_3$  is  ramified in  $M_1$.
		\item $M_2$ is $Y \cup \{ p,q, q_3 \}$-split,  and $q_1, q_2$ are ramified in  $M_2$.
		\item $M_3$ is $Y \cup \{ p,q, q_1, q_3 \}$-split,  and $q_2$ is ramified in $M_3$.
	\end{itemize}
	We write $\omega_i : G_\Q \to \{ \pm 1 \}$ for the quadratic character of kernel $G_{M_i}$. 
	
	By a level-raising result for $\GL_2(\A_\Q)$ (e.g.~\cite[Theorem 
	A]{Dia94},
	 we 
	can 
	find a 
	cuspidal, regular algebraic automorphic representation $\tau$ of 
	$\GL_2(\A_\Q)$ with the following properties:
	\begin{itemize}
		\item $\tau$ is unramified outside $2, q_1, q_2$.
	\item There is an isomorphism $\overline{r}_{\tau, \iota} \cong \Ind_{G_K}^{G_\Q} \overline{\psi}^{2k}$ and $\det r_{\tau, \iota} = \det \Ind_{G_K}^{G_\Q} \psi^{2k} = \epsilon^{-8k} \delta_{K / \bQ}$.
		\item $\tau_{q_1}$ is an unramified twist of the Steinberg 
		representation,  and there is an isomorphism $\rec_{\bQ_{q_2}} 
		\tau_{q_2} \cong  \Ind_{W_{\bQ_{q_2^2}}}^{W_{\bQ_{q_2}}} \chi_{q_2}$, 
		where $\chi_{q_2}|_{I_{\bQ_q}}$ is a  character of  order $p$. In 
		particular, $\tau_{q_2}$ is supercuspidal. 

		\item $\tau$ is $\iota$-ordinary and $r_{\tau, \iota}$ has the same 
		Hodge--Tate weights as $\Ind_{G_K}^{G_\Q} \psi^{2k}$.
	\end{itemize} For the ordinary condition in the last point, note that since 
	$p > 8k$ and $r_{\tau, \iota}$ has Hodge--Tate weights $(0,8k)$ and 
	reducible local residual representation 	$\rbar_{\tau,\iota}|_{G_{\Qp}}$ 
	at $p$, $\tau$ 
	is necessarily $\iota$-ordinary \cite[Theorem 2.6]{edix:weights}. 

	By induction, there exists a soluble CM extension $F_{-1} / K$ with the following properties:
	\begin{itemize}
	\item $F_{-1} / K$ is $X$-split.
	\item There exists a RACSDC, $\iota$-ordinary automorphic representation $\pi$  of $\GL_{n-2}(\bA_{F_{-1}})$ with the following properties:
	\begin{enumerate}
	\item  There is an isomorphism
	\[  \overline{r}_{\pi,  \iota} \cong  \overline{\omega}^{n-3}|_{G_{F_{-1}}} \otimes \Sym^{n-3}  \overline{r}_{\sigma_0, \iota}|_{G_{F_{-1}}}.  \]
	\item For each embedding $\tau : F_{-1} \to \overline{\bQ}_p$, we have 
	\[ \mathrm{HT}_\tau(r_{\pi, \iota}) = \mathrm{HT}_\tau(\omega^{n-3}|_{G_E} \otimes \Sym^{n-3} r_{\sigma_0, \iota}|_{G_E}). \]
	\item There exists a place $v_{-1}|q$ of $F_{-1}$ such that $\pi_{v_{-1}}$ is an unramified twist of the Steinberg representation.
	\end{enumerate}
	\end{itemize}
	We can find a soluble CM extension $F_0 / \Q$ with the following properties:
	\begin{itemize}
		\item $F_0$ is $Y$-split.
		\item The prime $q_1$ is split in $F_0^+$, and each place of $F_0^+$ 
		above $q_1$ is inert in $F_0$. The primes $q_2, q_3$ split in $F_0$.
			\item $F_0 / F_0^+$ is everywhere unramified.
			\item For each place $v | p$ of $F_0$, $v$ is split over $F_0^+$ and $[F_{0, v} : \Q_p] > n (n-1)/2 + 1$.
			\item For each place $v | q$ of $F_0$, $v$ is split over $F_0^+$ and $q_v \equiv 1 \text{ mod }p$.
		\item There exists a crystalline character $\omega_0 : G_{F_0} \to 
		\overline{\Q}_p^\times$, unramified outside $p$, such that $\omega_0 
		\omega_0^c = \epsilon^{3} \delta_{K / \bQ}|_{G_{F_0}}$. (Use \cite[Lemma 
		A.2.5]{BLGGT}.) 
		\item For each place $v | pq $ of $F_0$, the representations $\overline{r}_{\sigma_0, \iota}|_{G_{F_{0, v}}}$ and $\overline{\omega}_0|_{G_{F_{0, v}}}$ are trivial.
	\end{itemize}Define 
		\[ \overline{\rho}_0 =\overline{\omega}_0^{n-3} \otimes \left( 
		\overline{\omega}_2 \otimes ( \oplus_{i=1}^{k-1} (\Ind_{G_K}^{G_\bQ} 
		\overline{\psi}^{2k-i-1} \overline{\psi}^{c,i-1}) ) \oplus 
		\overline{\omega}_3 \epsilon^{-4(k-1)} \right)|_{G_{F_0}}. \]
		Then for each place $v | pq$ of $F_0$, $\overline{\rho}_0|_{G_{F_{0, v}}}$ is trivial.
		
We now apply Proposition \ref{prop_automorphic_lifts_of_prescribed_types} with the following choices:
\begin{itemize}
\item $F_1 = F_{-1} \cdot F_0 \cdot M_1 \cdot M_2 \cdot M_3$. 
\item $\overline{\rho}_0$ is the residual representation defined above. 
\item $\Sigma_0$ is the set of places of $F^+_0$ lying above $q$; $T_0$ is the set of places of $F^+_0$ lying above $q_2$ or $q_3$; and $S_0$ is the set of places of $F_0^+$ lying above $p, q, q_1, q_2$, or $q_3$.
\item If $v | q_2$, then $\overline{R}_v$ is the fixed type deformation ring 
(defined as in \cite[Definition 3.5]{shottonGLn}) associated to the inertial 
type $\oplus_{i=0}^{2k-2} \omega(v) \circ \Art_{F_{0, v}}|_{\cO_{F_{0, 
v}}^\times}^{-1}$ (where as usual, $\omega(v)$ denotes the unique quadratic 
character of $k(v)^\times$ provided that $k(v)$ has odd characteristic). If $v 
| q_3$, then there is a character $\Theta_v : \cO_{F_{0, v, n-2}}^\times \to 
\bC^\times$ of order $p$ such that $\overline{R}_v$ is the fixed type 
deformation ring associated to the inertial type $\oplus_{i=0}^{2k-2}  
\iota^{-1}\Theta_v^{q_v^{i-1}} \circ \Art_{F_{0, v, n-2}}|_{\cO_{F_{0, v, 
n-2}}^\times}^{-1}$ (where $F_{0, v, n-2} / F_{0, v}$ is an unramified 
extension of degree $n-2$).
\item $\pi_1$ is the twist of the base change of $\pi$ with respect to the soluble CM extension $F_1 / F_{-1}$ by the character $\iota \omega^{n-3}_0|_{G_{F_1}} / \omega^{n-3}|_{G_{F_1}}$.
\end{itemize}
(Note that $F_1 / K$ is $X$-split, so Lemma \ref{lem_consequence_of_disjointness} may be applied to $\overline{r}_{\pi_1, \iota}$.) We conclude the existence of a RACSDC, $\iota$-ordinary automorphic representation $\pi_0$ of $\GL_{n-2}(\bA_{F_0})$ satisfying the following conditions:
\begin{itemize}
\item There is an isomorphism $\overline{r}_{\pi_0, \iota} \cong \overline{\rho}_0$.
\item $\pi_0$ is unramified outside $S_0$.
\item For each place $v | q$ of $F_0$, $\pi_{0, v}$ is an unramified twist of the Steinberg representation.
\item For each place $v | q_1$ of $F_0$, there are characters $\chi_{v, 0}, \chi_{v, 1} \dots, \chi_{v, 2k-2} : F_{0, v}^\times \to \bC$ such that $\pi_{0, v} \cong \chi_{v, 0} \boxplus \chi_{v, 1} \boxplus \dots \boxplus \chi_{v, 2k-2}$, $\chi_{v, 0}$ is unramified,and for each $i = 1, \dots, 2k-2$, $\chi_{v, i}|_{\cO_{F_0, v}^\times} = \omega(v)$.
\item For each place $v | q_2$ of $F_0$, $\pi_{0, v}|_{\GL_{n-2}(\cO_{F_{0, 
v}})}$ contains $\omega(v) \circ \det$.
\item For each place $v | q_3$ of $F_0$, $\pi_{0, v}|_{\GL_{n-2}(\cO_{F_{0, v}})}$ contains the representation $\widetilde{\lambda}(v, \Theta_v)$ (notation as in Proposition \ref{prop_depth_zero_supercuspidal_on_gl_n}).
\end{itemize}
	Let $T_i$ denote the set of places of $F_0^+$ lying above $q_i$, and let $T 
	= T_1 \cup T_2 \cup T_3$. Let $\tau_0$ denote the base change of $\tau$ to 
	$F_0$, and let $\pi_2 = \tau_0 \otimes | \cdot |^{(n-2)/2}  \iota \omega_1 
	\omega_0^{n-1}$. Let $\pi_{n-2} = \pi_0 \otimes | \cdot |^{-3} \iota 
	\omega_0^2$. We see that the hypotheses of Theorem 
	\ref{thm_automorphic_level_raising} are now satisfied, and we conclude the 
	existence of a $T$-split quadratic totally real extension $L_0^+ / F_0^+$ 
	and a RACSDC $\iota$-ordinary automorphic representation $\Pi_0$ of 
	$\GL_n(\bA_{L_0})$ satisfying the following conditions:
	\begin{itemize}
	\item The extension $L_0 / \bQ$ is soluble and $Y$-split.
	\item There is an isomorphism 
	\[ \overline{r}_{\Pi_0, \iota} \cong \overline{\omega}_0^{n-1}|_{G_{L_0}} \otimes \left( \overline{\omega}_1 \otimes \Ind_{G_K}^{G_\bQ} \overline{\psi}^{2k} \oplus \overline{\omega}_2 \otimes ( \oplus_{i=1}^{k-1} (\Ind_{G_K}^{G_\bQ} \overline{\psi}^{2{k-i}} \overline{\psi}^{c, i}) ) \oplus \overline{\omega}_3 \epsilon^{-4k} \right)|_{G_{L_0}}. \]
	\item For each place $v | q_1$ of $L_0$, there exists a character $\Theta_v : \cO_{L_{0, v, 3}}^\times \to \bC^\times$ of order $p$ such that $\Pi_{0, v}|_{\mathfrak{r}_v}$ contains $\widetilde{\lambda}(v, \Theta_v, n)|_{\mathfrak{r}_v}$ (notation as in \S \ref{subsec_local_theory_GL_n}).
	\item For each place $v | q_3$ of $L_0$, $\Pi_{0, v}|_{\q_v}$ contains the representation $\widetilde{\lambda}(v, \Theta_{v}, n)$ (where we define $\Theta_v = \Theta_{v|_{F_0}}$).
	\end{itemize}
	In fact, $\Pi_0$ has the following stronger property:
	\begin{itemize}
	\item For each place $v | q_1$ of $L_0$, $\Pi_{0, v}|_{\q_v}$ contains $\widetilde{\lambda}(v, \Theta_v, n)$.
	\end{itemize}
	To see this, it is enough to check that no two eigenvalues $\alpha, \beta 
	\in \overline{\bF}_p^\times$ of the representation $(\Sym^{n-3} 
	\Ind_{G_K}^{G_\bQ} \overline{\psi})(\Frob_{q_1})$ satisfy $(\alpha / 
	\beta)^2 = q_1^2$ (recall that if $v | q_1$ is a place of $L_0$, then 
	$L_{0, v}/ \Q_{q_1}$ is an unramified quadratic extension). Recalling the 
	numbers $x_1, y_1 \in \bF_p^\times$, we see that we must check that $(y_1^8 
	/ x_1^{16})^{i} \neq y_1^{\pm 2}$ for $i = 1, \dots, 2k-2$. However, by 
	construction $y_1^2$ is a primitive $3^\text{rd}$ root of unity and $y_1^8 
	/ x_1^{16}$ is a primitive $6k^\text{th}$ root of unity, so we cannot have 
	$(y_1^8 / x_1^{16})^{3i} = 1$ if $1 \leq i < 2k$.

	Let $L_1 = L_0 K$. Then the following conditions are satisfied:
	\begin{itemize}
	\item The extension $L_1 / K$ is soluble and $X$-split.
	\item Let $\Pi_1$ denote the base change of $\Pi_0$ with respect to the quadratic extension $L_1 / L_0$. Then $\Pi_1$ is RACSDC and $\iota$-ordinary. (It is cuspidal because $L_1 / L_0$ is quadratic and $n$ is odd, cf. \cite[Theorem 4.2]{MR1007299}.)
	\item For each place $v$ of $L_1$ of residue characteristic $q_1, q_3$, $v$ is split over $L_0$ and over $L_1^+$. (The prime $q_i$ splits in $K$.)
	\end{itemize}
Thus the hypotheses of Theorem \ref{thm_level_raising_through_Galois_theory} are satisfied with $R_1$ (resp. $R_2$) the set of places of $L_1^+$ of residue characteristic $q_1$  (resp. $q_3$), and we conclude the existence of a RACSDC  $\iota$-ordinary automorphic representation $\Pi_1'$ of $\GL_n(\bA_{L_1})$ satisfying the following conditions:
\begin{itemize}
\item There is an isomorphism   $\overline{r}_{\Pi_1', \iota} \cong \overline{r}_{\Pi_0, \iota}|_{G_{L_1}}$.
\item There exists a place $v | q$ of $L_1$  such  that $\Pi_{1, v}'$ is an unramified twist of the Steinberg representation.
\end{itemize}
Finally, let $F = L_1 M_1 M_2 M_3$, and let $\Pi'$ be the base change of 
$\Pi_1'$ with respect to the extension $F / L_1$. We see that the conclusion of 
Theorem \ref{thm:levelraisingalln} holds with $\Pi = \Pi' \otimes \iota 
(\omega|_{G_F} / \omega_0|_{G_F})^{n-1}$.
\end{proof}

\begin{proof}[Proof of Theorem \ref{thm:levelraisingalln}]
	If $n$ is odd then the statement reduces to Proposition
	\ref{prop:levelraisingoddn}. Let $m \geq 1$ be an odd integer. We will prove by induction on $r \geq 0$ that the conclusion of Theorem \ref{thm:levelraisingalln} holds for all integers of the form $n = 2^r m$.
	
	The case $r = 0$ is already known. Supposing the theorem known for a fixed $r \geq 0$ (hence $n = 2^r m$), we will now establish it for $r + 1$ (hence $n' = 2^{r+1}m = 2n$). Fix data $p, q, X_0, \omega$ as in the statement of Theorem \ref{thm:levelraisingalln}. In particular $p \equiv 1 \pmod{ 48 n'!}$. Once again we enlarge $X_0$ so that any $X_0$-split extension $F / K$ is forced to be linearly disjoint from the fixed field of $\ker(\overline{r}_{\sigma_0, \iota} \oplus \overline{\epsilon})$.
	
	By induction, we can find a soluble CM extension $F / K$ and a RACSDC automorphic representation $\pi$ of $\GL_{n}(\A_F)$ such that the following conditions are satisfied:
	\begin{itemize}
		\item $\pi$ is $\iota$-ordinary. There is an isomorphism $\overline{r}_{\pi, \iota} \cong \overline{\omega}^{n - 1}|_{G_F} \otimes \Sym^{n - 1} \overline{r}_{\sigma_0, \iota}|_{G_F}$. The representations $r_{\pi, \iota}$ and $\omega^{n - 1}|_{G_F} \otimes \Sym^{n- 1} r_{\sigma_0, \iota}$ have the same Hodge--Tate weights (with respect to any embedding $\tau : F \to \overline{\Q}_p$).
		\item There exists a place $v | q$ such that $\pi_v$ is an unramified twist of the Steinberg representation.
		\item  $F$ is $X_0$-split.
	\end{itemize}
	After possibly enlarging $F$, we can assume that the following additional conditions are satisfied:
	\begin{itemize}
		\item $q_v \equiv 1 \text{ mod }p$ and $\overline{r}_{\sigma_0, \iota}(\Frob_v)$ is trivial. 
		\item Each place of $F$ which is either $p$-adic or at which $\pi$ is ramified is split over $F^+$.
	\end{itemize}
	Let $\Omega, \Psi : K^\times \backslash \A_K^\times \to \C^\times$ be the 
	Hecke characters of type $A_0$ with $r_{\Omega, \iota} = \omega$ and 
	$r_{\Psi, \iota} = \psi$. Define $\pi_1 = \pi \otimes (| \cdot |^{n/2} 
	(\Omega \Psi \circ \mathbf{N}_{F / K})^n)$, $\pi_2 = \pi \otimes (| \cdot 
	|^{n/2} (\Omega \Psi^c \circ \mathbf{N}_{F / K})^n)$. We make the 
	following observations:
	\begin{itemize}
		\item $\pi_1$ and $\pi_2$ are cuspidal, conjugate self-dual automorphic representations of $\GL_n(\A_F)$.
		\item Let $\pi_0 = \pi_1 \boxplus \pi_2$ and define $r_{\pi_0, \iota} = r_{\pi_1 |\cdot|^{-n/2}, \iota} \oplus r_{\pi_2 | \cdot |^{-n/2}, \iota}$. Then $\pi_0$ is regular algebraic and $\iota$-ordinary. Moreover, for each finite place $v$ of $F$ there is an isomorphism $\mathrm{WD}(r_{\pi_0, \iota}|_{G_{F_v}})^{F-ss} \cong \rec_{F_v}^T(\pi_{0, v})$, and there is an isomorphism
		\[ \overline{r}_{\pi_0, \iota} \cong \overline{\omega}^{2n - 1}|_{G_F} \otimes \Sym^{2n - 1} \overline{r}_{\sigma_0, \iota}|_{G_F}. \]
		\item There are unramified characters $\xi_i : F_v^\times \to \C^\times$ such that $\pi_i \cong \St_n(\xi_i)$ and $\iota^{-1} \xi_1 / \xi_2(\varpi_v) \equiv q_v^n \text{ mod } \ffrm_{\overline{\Z}_p}$.
		\item $\overline{r}_{\pi_0, \iota}$ is not isomorphic to a twist of $1 \oplus \epsilon^{-1} \oplus \dots \oplus \epsilon^{1 - 2^{r+1} m}$.
	\end{itemize}
	We justify each of these points in turn. Since $\pi$ is conjugate 
	self-dual, the first point follows from the fact that $(\Omega \Psi)(\Omega 
	\Psi)^c = | \cdot |^{-1}$ (in turn a consequence of the identity $(\omega 
	\psi)(\omega \psi)^c = \epsilon^{-1}$). The second follows from the identity
	\[ \Sym^{2n-1} r_{\sigma_0, \iota}|_{G_K} \cong \psi^{2n-1} \oplus \psi^{2n-2} \psi^c \oplus \dots \oplus (\psi^c)^{2n-1} \cong (\psi^n \oplus (\psi^c)^n) \otimes \Sym^{n-1} r_{\sigma_0, \iota}|_{G_K}. \]
	The third point holds by construction ($q_v \equiv 1 \pmod{p}$ and $\overline{r}_{\pi_0, \iota}(\Frob_v)$ is scalar). The fourth holds since otherwise $\overline{r}_{\pi_0, \iota}|_{G_{F(\zeta_p)}}$ would be a twist of the trivial representation, contradicting part 2 of Lemma \ref{lem_consequence_of_disjointness}.
	
	We see that the hypotheses of \cite[Theorem 5.1]{Ana19} are satisfied. This theorem implies that we can find a quadratic CM extension $F' / F$ such that $F' / K$ is soluble $X_0$-split, as well as a RACSDC automorphic representation $\pi'$ of $\GL_{n'}(\A_{F'})$ satisfying the following conditions:
	\begin{itemize}
		\item $\pi'$ is $\iota$-ordinary, and there is an isomorphism
			\[ \overline{r}_{\pi', \iota} \cong \overline{\omega}^{n'- 1}|_{G_{F'}} \otimes \Sym^{n' - 1} \overline{r}_{\sigma_0, \iota}|_{G_{F'}}. \]
		\item The weight of $\pi'$ is the same as that of $\pi_0$.
		\item There exists a place $v' | v$ of $F'$ such that $\pi'_{v'}$ is an unramified twist of the Steinberg representation.
	\end{itemize}
	This existence of $F'$ and $\pi'$ completes the induction step, and therefore the proof of the theorem.
\end{proof}
\begin{remark}
	We observe that the results of \cite{Ana19} already suffice to prove Theorem 
	\ref{thm:levelraisingalln} (and hence Theorem \ref{thm:symmpowers}) when 
	$n$ is a power of two, without using the level raising results of Sections 
	\ref{sec:automorphic_level_raising}--\ref{sec:galois_level_raising}.
\end{remark}
We can now put everything together to deduce our main results on automorphy of 
symmetric powers.
\begin{theorem}\label{thm_seed_points}
	Let $n \geq 3$. Then there exists a cuspidal, everywhere unramified 
	automorphic representation $\pi$ of $\GL_2(\bA_\bQ)$ of weight $k \geq 2$  
	such that, for any isomorphism $\iota : \overline{\bQ}_p \to \bC$, 
	$\Sym^{n-1} r_{\pi, \iota}$ is automorphic. 
\end{theorem}
\begin{proof}
	By Corollary \ref{cor:level2qseed}, we can find odd 
	primes 
	$p \ne q $, with $q \equiv 3 \text{ mod }4$, and a cuspidal 
	automorphic representation $\sigma$ of $\GL_2(\bA_\bQ)$ of weight $5$ 
	satisfying the following conditions:
	\begin{itemize}
		\item $\sigma$ is unramified at primes not dividing $2q$; $\sigma_2$ is 
		isomorphic to a principal series representation $i_{B_2}^{\GL_2} \chi_1 \otimes 
		\chi_2$, where $\chi_1$ is unramified and $\chi_2$ has conductor 4; and 
		$\sigma_q$ is an unramified twist of the Steinberg representation.
		\item For any isomorphism $\iota_p : \overline{\bQ}_p \to \bC$, 
		$\Sym^{n-1} r_{\pi, \iota_p}$ is automorphic.
	\end{itemize}
 Now we choose an isomorphism $\iota_q : \overline{\bQ}_q \to 
	\bC$. By the second part of Lemma \ref{lem_local_image_contains_SL_2}, the 
	Zariski closure of $r_{\sigma, \iota_q}(G_{\bQ_q})$ contains $\SL_2$. 
	Since $\sigma_{q}$ is an unramified twist of Steinberg it has a unique 
	($q$-adic) accessible refinement, which is numerically non-critical and 
	$n$-regular.  We 
	can therefore apply Theorem 
	\ref{thm_propogration_along_components_of_eigencurve} to the point of the $q$-adic, 
	tame level 4 eigencurve associated to  $\sigma$ with its unique accessible refinement. 	
	Using the accumulation property of the 
	eigencurve to find a suitable classical point in the same (geometric) 
	irreducible component as this point, we deduce the existence of a 
	cuspidal automorphic 
	representation $\sigma'$ of $\GL_2(\bA_\bQ)$ of weight $k > 2$ satisfying 
	the following conditions:
	\begin{enumerate}
		\item $\sigma'$ is unramified outside $2$, and $\sigma'_{ 2}$ is 
		isomorphic to a principal series representation $i_{B_2}^{\GL_2} \chi_1 \otimes 
		\chi_2$, where $\chi_1$ is unramified and $\chi_2$ has conductor 4.
		\item The weight of $\sigma'$ satisfies $k \equiv 3 \text{ mod 
		}4$ (this is possible because $q \equiv 3 \text{ mod }4$, and we can 
		choose any $k \equiv 5 \text{ mod }{(q-1)q^\alpha}$ for sufficiently 
		large 
		$\alpha$).
		\item $\Sym^{n-1} r_{\sigma', \iota_q}$ is automorphic.
	\end{enumerate}
	Let $\iota : \overline{\bQ}_2 \to \bC$ be an isomorphism. These conditions 
	imply that 
	the Zariski closure of $r_{\sigma', \iota}(G_{\bbQ_2})$ must contain 
	$\SL_2$. Indeed, we have already observed in \S \ref{sec_ping_pong} that there are no 2-ordinary cusp 
	forms of tame level 1, so (invoking Lemma \ref{lem_local_image_contains_SL_2}) if this Zariski closure does not contain $\SL_2$ 
	then $r_{\sigma', \iota}|_{G_{\bbQ_2}}$ must be irreducible and induced from a quadratic extension of $\bbQ_2$, 
	implying that both refinements of $\sigma'$ at the prime 2 have slope 
	$(k-1) / 2$, an odd integer. However, Theorem \ref{thm_buzzard_kilford} 
	implies that there are no newforms of level 4 and odd slope (see 
	\cite[Corollary of Theorem B]{Buz05}); a contradiction. The same argument 
	shows that the refinement $\chi_1\otimes\chi_2$ is $n$-regular, since the 
	two refinements of $\sigma'$ have distinct slopes. 
	
	We see that $(\sigma',\chi_1\otimes\chi_2)$ satisfies the hypotheses of Theorem 
	\ref{thm_propogration_along_components_of_eigencurve}. Using the 
	accumulation property of the (tame level $1$, $2$-adic) eigencurve, we 
	deduce 
	the existence of a cuspidal, everywhere unramified automorphic 
	representation $\pi$ of $\GL_2(\A_\Q)$ such that $\Sym^{n-1} r_{\pi, 
		\iota}$ is automorphic. This completes the proof. 
\end{proof}

Combining Theorem \ref{thm_seed_points} with Theorem 
\ref{thm_propogration_of_automorphy_of_symmetric_powers}, we deduce:
\begin{theorem}\label{thm:symmpowers}
	Let $n \geq 3$, and let $\pi$ be a cuspidal, everywhere unramified 
	automorphic representation of $\GL_2(\bA_\bQ)$ of weight $k \geq 2$. Then 
	for any isomorphism $\iota : \overline{\bQ}_p \to \bC$, $\Sym^{n-1} 
	r_{\pi, \iota}$ is automorphic. 
\end{theorem}

\section{Higher levels}\label{sec_higher_levels}

In this section we extend our main theorem to higher levels as follows:
\begin{theorem}\label{thm_general_levels}
	Let $\pi$ be a cuspidal automorphic representation of $\GL_2(\A_\Q)$ of weight $k \geq 2$ satisfying the following two conditions:
	\begin{enumerate}
		\item For each prime $l$, $\pi_l$ has non-trivial Jacquet module (equivalently, $\pi_l$ admits an accessible refinement).
		\item $\pi$ is not a CM form. 
	\end{enumerate}
	Then for any $n \geq 3$ and any isomorphism $\iota : \overline{\bQ}_p \to \bC$, $\Sym^{n-1} r_{\pi, \iota}$ is automorphic. 
\end{theorem}
For example, these conditions are satisfied if $\pi$ is associated to a 
non-CM cuspidal eigenform $f$ of level $\Gamma_1(N)$ for some squarefree 
integer $N 
\geq 1$; in particular, if $k = 2$ and $\pi$ is associated to a semistable elliptic curve over $\Q$.

Fix $n \geq 3$ for the remainder of this section. We first prove the following special case of Theorem \ref{thm_general_levels}:
\begin{proposition}\label{prop_general_levels_n_regular}
	Let $\pi$ be a cuspidal automorphic representation of $\GL_2(\A_\Q)$ of weight $k \geq 2$ satisfying the following conditions:
	\begin{enumerate}
		\item For each prime $l$ such that $\pi_l$ is ramified, $\pi_l$ has an accessible refinement which is $n$-regular, in the sense of Definition \ref{defn_n_regular}.
		\item $\pi$ is not a CM form. 
	\end{enumerate}
	Then for any isomorphism $\iota : \overline{\bQ}_p \to \bC$, $\Sym^{n-1} r_{\pi, \iota}$ is automorphic. 
\end{proposition}
\begin{proof}
	We prove the proposition by induction on the number of primes $r$ dividing the conductor $N$ of $\pi$. The case $r = 0$ (equivalently, $N = 1$) is Theorem \ref{thm:symmpowers}.
	
	Suppose therefore that $r>0$ and that the theorem is known for automorphic representations of conductor divisible by strictly fewer than $r$ primes. Let $\pi$ be a cuspidal automorphic representation as in the statement of the proposition. Fix a prime $p$ at which $\pi$ is ramified, and an isomorphism $\iota : \overline{\Q}_p \to \C$. Factor $N = M p^s$, where $(M, p) = 1$.
	
	Suppose first that $r_{\pi, \iota}|_{G_{\Q_p}}$ is reducible. Then $\pi$ is 
	$\iota$-ordinary and $\pi$ admits an ordinary refinement $\chi$. After 
	twisting by a finite order character, we can assume that $(\pi,\chi) \in 
	\mathcal{RA}_0$ (here we use the notation established for the Coleman--Mazur eigencurve in \S \ref{subsec_eigencurve}). Let $\cC$ be an irreducible component of the (tame level $M$, $p$-adic) eigencurve $\cE_{0, \C_p}$ containing the point $x$ corresponding to $(\pi, \chi)$, and let $\cZ \subset \cE_0$ denote the Zariski closed set defined in Lemma \ref{lem_bad_locus_of_eigencurve_zariski_closed}. Our hypotheses imply that $x \not\in \cZ_{\C_p}$.

	We can therefore find a point $x'' \in \cC - \cZ_{\C_p}$ such that the image of $x''$ in $\cW_{0, \C_p}$ is a character of the form $y \mapsto y^{k''-2}$
	for some integer $k'' \geq 2$. Indeed, since the image of $\cC$ in $\cW_{0,\C_p}$ is Zariski open, we can find such a point in $\cC$. There is an affinoid neighbourhood $U''$ of this point which maps in a finite and surjective fashion onto an affinoid open in $\cW_{0,\Cp}$. The image of $\cZ_{\C_p}\cap U''$ in this affinoid open is Zariski closed, and we can therefore find another such point $x'' \in \cC- \cZ_{\C_p}$. (In fact, the ordinary component $\cC$ surjects onto a connected component of $\cW_{0,\C_p}$, but we will apply the same argument for a non-ordinary component.)
	
	Choosing another point in a sufficiently small affinoid neighbourhood of $x''$ in $\cC - \cZ_{\C_p}$ and applying the classicality criterion, we can find a point $x' \in \cC - \cZ_{\C_p}$ corresponding to an $\iota$-ordinary cuspidal 
	automorphic representation $\pi'$ of $\GL_2(\A_\Q)$ of weight $k' \geq 2$ 
	with the following properties:
	\begin{enumerate}
		\item Let $\chi'$ denote the ordinary refinement of $\pi'$. Then $(\pi', \chi')$ determines a point on the same irreducible component of the (tame level $M$, $p$-adic) eigencurve $\cE_{0, \C_p}$ as $(\pi, \chi)$.
				\item The level of $\pi'$ is prime to $p$. 
		\item For each prime $l | M$, each accessible refinement of $\pi'_l$ is $n$-regular.
		\item The Zariski closure of $r_{\pi', \iota}(G_\Q)$  (in $\GL_2/\Qpbar$) contains $\SL_2$. 
	\end{enumerate}
	(The latter two properties follow from the definition of the set $\cZ$ in Lemma \ref{lem_bad_locus_of_eigencurve_zariski_closed}. In fact we can take $x' = x''$, since ordinary points of classical weights are classical; however, we will repeat the same argument in the next paragraph also for a non-ordinary component of the eigenvariety, in which case two steps are required.)
	By induction, $\Sym^{n-1} r_{\pi', \iota}$ is automorphic. We may then 
	apply the ordinary case of Theorem 
	\ref{thm_propogration_along_components_of_eigencurve} to conclude that 
	$\Sym^{n-1} r_{\pi, \iota}$ is automorphic.
	
	Suppose instead that $r_{\pi, \iota}|_{G_{\Q_p}}$ is irreducible, and let 
	$\chi$ be an accessible, $n$-regular refinement. The existence of $\chi$ 
	implies that the Zariski closure of $r_{\pi, \iota}(G_{\Q_p})$ in 
	$\GL_2/\Qpbar$ contains 
	$\SL_2$, by Lemma \ref{lem_local_image_contains_SL_2}. Again, after 
	twisting by a 
	finite order 
	character, we can assume that $(\pi,\chi)\in\mathcal{RA}_0$. Repeating the same argument as in the ordinary case, we can find a cuspidal automorphic representation $\pi'$ of 
	$\GL_2(\A_\Q)$ of weight $k' \geq 2$ with the following properties:
	\begin{enumerate}
		\item $\pi'$ admits a non-ordinary refinement $\chi'$ which is 
		numerically non-critical and $n$-regular. (This again implies that the 
		Zariski closure of $r_{\pi', \iota}(G_{\Q_p})$ contains 
		$\SL_2$.)
		\item The pair $(\pi', \chi')$ determines a point on the same irreducible component of the (tame level $M$, $p$-adic) eigencurve $\cE_{0, \C_p}$ as $(\pi, \chi)$.
		\item For each prime $l | M$, each accessible refinement of $\pi'_l$ is  $n$-regular.
		\item The level of $\pi'$ is prime to $p$.
	\end{enumerate}
	By induction, $\Sym^{n-1} r_{\pi', \iota}$ is automorphic. We can then appeal to Theorem \ref{thm_propogration_along_components_of_eigencurve} to conclude that $\Sym^{n-1} r_{\pi, \iota}$ is automorphic. 
	
	In either case we are done, by induction. 
\end{proof}
To reduce the general case of Theorem \ref{thm_general_levels} to Proposition \ref{prop_general_levels_n_regular}, we establish the following intermediate result.
\begin{prop}\label{prop_TW_congruence_to_n_regular}
	Let $\pi$ be a cuspidal automorphic representation of $\GL_2(\A_\Q)$ of weight $k \geq 2$, without CM. Suppose that for each prime $l$, $\pi_l$ has non-trivial Jacquet module. Then we can find a prime $p$, an isomorphism $\iota : \overline{\Q}_p \to \C$, and another cuspidal automorphic representation $\pi'$ of $\GL_2(\A_\Q)$ of weight $k$ with the following properties:
	\begin{enumerate}
		\item $p > \max(2(n+1), (n-1)k)$.
		\item The image of $\overline{r}_{\pi, \iota}$ contains a conjugate of $\SL_2(\F_p)$.
		\item Both $\pi_p$ and $\pi'_p$ are unramified.
		\item There is an isomorphism $\overline{r}_{\pi, \iota} \cong \overline{r}_{\pi', \iota}$.
		\item For each prime $l$, $\pi'_l$ has non-trivial Jacquet module. If  $\pi'_l$ is ramified, then each accessible refinement of $\pi'_l$ is $n$-regular. 
	\end{enumerate}
\end{prop}
\begin{proof}
	We use Taylor--Wiles--Kisin patching. The idea is that if all the 
	automorphic representations congruent to $\pi \text{ mod }p$ fail to have 
	$n$-regular refinements at $l$ then the patched module will be supported on 
	a codimension one quotient of the local deformation ring at $l$, which 
	contradicts the numerology of the Taylor--Wiles--Kisin method.  
	
	Let $M$ denote the conductor of $\pi$. We can choose a
	prime 
	$p$ satisfying (1) and (2), $p > M$, such that $\pi_p$ is unramified, and satisfying 
	the following additional condition:
	\begin{itemize}
		\item For each prime $l \neq p$ such that $\pi_l$ is ramified, the universal lifting ring classifying lifts of $\overline{r}_{\pi, \iota}|_{G_{\Q_l}}$ of determinant equal to $\det r_{\pi, \iota}$ is formally smooth.
	\end{itemize}
	Indeed, it is sufficient that for each such prime $l$, the group $H^0(\Q_l, 
	\ad^0 \overline{r}_{\pi, \iota}(1))$ vanishes. Such a prime exists thanks to \cite[Proposition 3.2, Proposition 5.3]{Wes04}. 
	
 Fix an additional prime $q_a > p$ such that $\pi_{q_a}$ is unramified and such that the universal lifting ring classifying lifts of $\overline{r}_{\pi, \iota}|_{G_{\Q_{q_a}}}$ of determinant equal to $\det r_{\pi, \iota}$ is formally smooth. This is possible by e.g.\ \cite[Lemma 11]{Dia94}. 

	Fix a coefficient field $E / \bQ_p$, large enough that there is a conjugate $\overline{\rho} : G_\Q \to \GL_2(k)$ of $\overline{r}_{\pi, \iota}$ and such that $\chi = \det r_{\pi, \iota}$ takes values in $\cO$. We assume moreover that for each $\sigma \in G_\Q$, the roots of the characteristic polynomial of $\overline{\rho}(\sigma)$ lie in $k$. Let $S$ denote the set of primes at which $r_{\pi, \iota}$ is ramified (equivalently, at which $\overline{\rho}$ is ramified), together with $q_a$. We consider the global deformation problem (in the sense of \cite[Definition 5.6]{Tho16})
	\[ \cS = (\overline{\rho}, \chi, S, \{ \cO \}_{v \in S}, \{ \cD_v \}_{v \in S}), \]
	where $\cD_p$ is the functor of lifts of $\overline{\rho}|_{G_{\Q_p}}$ of determinant $\chi$ which are Fontaine--Laffaille with the same Hodge--Tate weights as $r_{\pi, \iota}$, and if $l \in S - \{ p \}$ then $\cD_l$ is the functor of all lifts of $\overline{\rho}|_{G_{\Q_l}}$ of determinant $\chi$. Since $\overline{\rho}$ is absolutely irreducible, the functor of deformations of type $\cS$ is represented by an object $R_\cS \in \cC_\cO$ (cf. \cite[Theorem 5.9]{Tho16}). We may choose a representative $\rho_\cS : G_\Q \to \GL_2(R_\cS)$ of the universal deformation. We set $H = H^1(Y_{U_1(M q_a)}, \Sym^{k-2} \cO^2)$, where $U_1(M q_a)$ is the open compact subgroup of $\GL_2(\A_{\bQ}^\infty)$ defined in \S \ref{subsec_eigencurve} and $Y_U$ is the modular curve of level $U$ (denoted $\widetilde{Y}(U)$ in \cite[\S 4.1]{MR2207783}). We write $\T^S \subset \End_\cO(H)$ for the commutative $\cO$-subalgebra generated by the unramified Hecke operators $T_l, S_l$ for $l \not\in S$. Then there is a unique maximal ideal $\ffrm \subset \T^S$ with residue field $k$ such that for each prime $l \not\in S$, the characteristic polynomial of $\overline{\rho}(\Frob_l)$ equals $X^2 - T_l X + l^{k-1} S_l \text{ mod }\ffrm$. The localization $H_\m$ is a finite free $\cO$-module, and there is a unique strict equivalence class of liftings $\rho_\ffrm : G_\Q \to \GL_2(\T^S_\m)$ of type $\cS$ such that for each prime $l \not\in S$, the characteristic polynomial of $\rho_\m(\Frob_l)$ equals the image of $X^2 - T_l X + l^{k-1} S_l$ in $\T^S_\m[X]$. (See \cite[Proposition 6.5]{Tho16} for justification of a very similar statement in the context of Shimura curves.) In particular, there is an $\cO$-algebra morphism $R_\cS \to \T^S_\m$ classifying $\rho_\m$, which is surjective.
	
	Suppose given a finite set $Q$ of primes satisfying the following conditions:
	\begin{enumerate}
		\item[(a)] $Q \cap S = \emptyset$.
		\item[(b)] For each $q \in Q$, $q \equiv 1 \text{ mod }p$ and $\overline{\rho}(\Frob_q)$ has distinct eigenvalues $\alpha_q, \beta_q \in k$.
	\end{enumerate}
	In this case we can define the following additional data:
	\begin{itemize}
		\item The group $\Delta_Q = \prod_{q \in Q} (\bbZ / q \bbZ)^\times(p)$ (i.e.\ the maximal $p$-power quotient of the product of the units in each residue field).
		\item The augmented global deformation problem
			\[ \cS_Q = (\overline{\rho}, \chi, S \cup Q, \{ \cO \}_{v \in S \cup Q}, \{ \cD_v \}_{v \in S \cup Q}), \]
			where for each $q \in Q$, $\cD_q$ is the functor all lifts of $\overline{\rho}|_{G_{\Q_q}}$ of determinant $\chi$.
		The labelling of $\alpha_q, \beta_q$ for each $q \in Q$ determines an algebra homomorphism $\cO[\Delta_Q] \to R_{\cS_Q}$ in the following way: if $\rho_{\cS_Q}$ is a representative of the universal deformation, then $\rho_{\cS_Q}|_{G_{\Q_q}}$ is conjugate to a representation of the form $A_q \oplus B_q$, where $A_q : G_{\Q_q} \to R_{\cS_Q}^\times$ is a character such that $A_q \text{ mod }\ffrm_{R_{\cS_Q}}$ is unramified and $A_q \text{ mod }\ffrm_{R_{\cS_Q}}(\Frob_q) = \alpha_q$ (and similarly for $B_q$). Then $A_q \circ \Art_{\Q_q} |_{\bbZ_q^\times}$ factors through a homomorphism $(\Z / q \Z)^\times(p) \to R_{\cS_Q}^\times$. These homomorphisms for $q \in Q$ collectively determine the algebra homomorphism $\cO[\Delta_Q] \to  R_{\cS_Q}$.
		\item The cohomology module $H_Q = H^1(Y_{U_1(M q_a) \cap U_2(Q)}, \Sym^{k-2} \cO^2)$, where we define
		\[ U_2(Q) = \left\{ \left(\begin{array}{cc} a & b \\ c & d \end{array}\right) \in \GL_2(\widehat{\bZ}) : c \equiv 0 \text{ mod }(\prod_{q \in Q} q), ad^{-1} \mapsto 1 \in \Delta_Q  \right\}, \]
		and commutative $\cO$-subalgebras $\T^{S \cup 
		Q} \subset \T^{S \cup Q}_Q \subset \End_\cO(H_Q)$. By definition, 
		$\T^{S \cup Q}$ is generated by the unramified Hecke operators $T_l, 
		S_l$ for $l \not\in S \cup Q$ and $\T^{S \cup Q}_Q$ is generated by 
		$\T^{S \cup Q}_Q$ and the operators $U_q$ for $q \in Q$. There are 
		maximal ideals $\ffrm_Q \subset \T^{S \cup Q}$ and $\ffrm_{Q, 1} 
		\subset \T^{S \cup Q}_Q$ with residue field $k$ defined as follows: 
		$\ffrm_Q$ is the unique maximal ideal such that for each prime $l 
		\not\in S \cup Q$, the characteristic polynomial of 
		$\overline{\rho}(\Frob_l)$ equals $X^2 - T_l X + l^{k-1} S_l \text{ mod 
		}\ffrm_Q$. The ideal $\ffrm_{Q, 1}$ is generated by $\ffrm_Q$ and the 
		elements $U_q - \alpha_q$ for $q \in Q$. There is a unique strict equivalence class of liftings
		$\rho_{\ffrm_Q} : G_\Q \to \GL_2(\T^{S \cup Q}_{\ffrm_Q})$ of type 
		$\cS_Q$ such that for each $l \not \in S \cup Q$, the characteristic 
		polynomial of $\rho_{\ffrm_Q}(\Frob_l)$ equals the image of $X^2 - T_l 
		X + l^{k-1} S_l$ in $\T^{S \cup Q}_{\ffrm_Q}[X]$. There is an 
		$\cO$-algebra morphism $R_{\cS_Q} \to \T_{\m_Q}^{S \cup Q}$ classifying 
		$\rho_{\ffrm_Q}$, which is surjective. Moreover, if we view $H_{Q, 
		\ffrm_{Q, 1}}$ as an $R_{\cS_Q}$-module via this map, then the two 
		$\cO[\Delta_Q]$-module structures on $H_{Q, \ffrm_{Q, 1}}$, one arising 
		from $R_{\cS_Q}$, the other arising from the action of $\Delta_Q$ via 
		Hecke operators, coincide. (These statements in turn may be justified 
		as in the proof of \cite[Lemma 6.8]{Tho16}.) Finally, $H_{Q, \ffrm_{Q, 
		1}}$ is a free $\cO[\Delta_Q]$-module and there is an isomorphism 
		$H_{Q, \ffrm_{Q, 1}} \otimes_{\cO[\Delta_Q]} \cO \cong H_{\ffrm}$ of 
		$R_{\cS_Q} \otimes_{ \cO[\Delta_Q]} \cO \cong R_{\cS}$-modules. (This 
		is again proved in a similar way to \cite[Lemma 6.8]{Tho16}, using the 
		fact that $H^i( Y_{U_1(M q_a) \cap U_0(Q)}, 
		\Sym^{k-2} (\cO / \varpi)^2)$ is Eisenstein for $i \neq 1$,
		together with \cite[Corollary 2.7]{KT}, to justify the freeness.)
\end{itemize}
		If $l \in S$, let $R_l \in \cC_\cO$ denote the universal lifting ring representing the local deformation problem $\cD_l$. By construction (if $l \neq p$) or arguing as in \cite[\S 2.4.1]{cht} (if $l = p$) $R_l$ is a formally smooth $\cO$-algebra; if $l \neq p$, then $R_l$ is formally smooth over $\cO$ of relative dimension 3, while $R_p$ has relative dimension 4. We set $T = S - \{ p, q_a \}$ and $A_\cS^T = \widehat{\otimes}_{l \in T} R_l$ (the completed tensor product being over $\cO$). The $T$-framed deformation rings $R_\cS^T$ and $R_{\cS_Q}^T$ are defined (see \cite[\S 5.2]{Tho16}) and there are canonical homomorphisms $A_{\cS}^T \to R_{\cS}^T$ and $A_{\cS}^T \to R_{\cS_Q}^T$. 
		
		By the argument of \cite[Proposition 3.2.5]{Kis09} and \cite[Proposition 5.10]{Tho16}, we can find an integer $q_0 \geq 0$ with the following property: for each $N \geq 1$, there exists a set $Q = Q_N$ of primes satisfying conditions (a), (b) above and also:
		\begin{enumerate}
			\item[(c)] $|Q_N| = q_0$.
			\item[(d)] For each $q \in Q_N$, $q \equiv 1 \text{ mod }p^N$.
			\item[(e)] The algebra map $A_{\cS}^T \to R^T_{\cS_{Q_N}}$ extends to a surjective algebra homomorphism $A_{\cS}^T\llbracket X_1, \dots, X_g \rrbracket \to R_{\cS^T_{Q_N}}$, where $g = q + |T| - 1$.
		\end{enumerate}
		We choose for each $N \geq 1$ a representative $\rho_{\cS_{Q_N}}$ of the universal deformation over $R_{\cS_{Q_N}}$ which lifts $\rho_\cS$. This choice determines an isomorphism $R_{\cS_{Q_N}}^T \cong R_{\cS_{Q_N}} \widehat{\otimes}_\cO \cT$, where $\cT$ is a power series ring over $\cO$ in $4|T| - 1$ variables. We set $H_{Q_N}^T = H_{{Q_N}, \ffrm_{{Q_N}, 1}} \widehat{\otimes}_\cO \cT$. It is a free $\cT[\Delta_{Q_N}]$-module, and there is an isomorphism $H_{Q_N}^T \otimes_{\cT[\Delta_{Q_N}]} \cO \cong H_\ffrm$ of $R_{\cS_{Q_N}^T} \otimes_{\cT[\Delta_{Q_N}]} \cO \cong R_{\cS}$-modules.
		
		We now come to the essential point of the proof. Let $l \in S - \{ p 
		\}$, and fix a Frobenius lift $\phi_l \in G_{\Q_l}$. 
\begin{lemma}
With our current assumptions, there is a principal ideal 
		$I_l \subset R_l$ with the following property: for any homomorphism $f 
		: R_l \to \overline{\Q}_p$, the resulting homomorphism $\rho_f : 
		G_{\Q_l} \to \GL_2(\overline{\Q}_p)$ has the property that the 
		eigenvalues $\alpha_l, \beta_l$ of $\rho_f(\phi_l)$ satisfy $(\alpha_l 
		/ \beta_l)^{i} = 1$ for some $i = 1, \dots, n-1$ if and only if $f(I_l) = 
		0$. Moreover, the quotient $R_l / I_l$ has dimension strictly smaller 
		than the dimension of $R_l$.
\end{lemma}		
\begin{proof}
Let $(r, N) = \rec^T_{\bQ_l}(\iota^{-1} \pi_l)$, a Weil--Deligne representation that we may assume is defined over $E$. The proof will use the fact that the Jacquet module of $\pi_l$ is non-trivial (equivalently, that the Weil--Deligne representation $(r, N)$ is reducible). 

We recall that the ring $R_l$ is a formally smooth $\cO$-algebra of relative dimension 3. Let $r_l^\text{univ} : G_{\bQ_l} \to \GL_2(R_l)$ be the universal lifting. We can take $I_l$ to be the ideal generated by the discriminant of the characteristic polynomial of $\Sym^{n-1} r_l^\text{univ}(\phi_l)$. To complete the proof of the lemma, we need to show that $\dim R_l / I_l < \dim R_l$. Since $R_l$ is an integral domain, it is equivalent to show that $I_l$ is not the zero ideal.

To show this, we split into cases. If $\pi_l$ is a twist of the Steinberg representation then the discriminant of the characteristic polynomial of $\Sym^{n-1} r(\phi_l)$ is non-zero (as the eigenvalues of $r(\phi_l)$ have eigenvalues whose ratio is a non-zero power of $l$), so we see that $I_l$ is not the zero ideal in this case. Otherwise, $N = 0$ and $r = \chi_1 \oplus \chi_2$ is a direct sum of two characters of $W_{\bQ_l}$. Let $\psi : W_{\bQ_l} \to E\llbracket T \rrbracket$ be the unramified character which sends $\psi$ to $1 + T$; then $r' = \chi_1 \psi \oplus \chi_2 \psi^{-1}$ is a deformation of $r$ to $E \llbracket T \rrbracket$ of determinant $\chi$ with the property that the discriminant of the characteristic polynomial of $\Sym^{n-1} r'(\phi_l)$ is non-zero in $E\llbracket T \rrbracket$. The existence of this deformation, together with \cite[Proposition 2.1.5]{gee061}, implies that $I_l$ cannot be the zero ideal in this case either.
\end{proof}
 We set $I = \prod_{l \in S - \{ p \}} I_l 
		A_{\cS}^T \subset A_{\cS}^T$. Then $\dim A_{\cS}^T / I = \dim 
		A_{\cS}^T - 1$.
		
		Suppose for contradiction that for each automorphic representation $\pi'$ contributing to $H_{Q_N}$ for some $N \geq 1$, there is a prime $l \in S$ such that $\pi'_l$ is ramified and there is an accessible refinement of $\pi'_l$ which is not $n$-regular. Then $I H_{Q_N}^T = 0$. On the other hand, a standard patching argument (cf. \cite[Lemma 6.10]{jack}) implies the existence of the following objects:
		\begin{itemize}
			\item A ring $S_\infty = \cT \llbracket S_1, \dots, S_{q_0} \rrbracket$ and an algebra homomorphism $S_\infty \to R_\infty = (A_{\cS}^T / I) \llbracket X_1, \dots, X_g \rrbracket$.
			\item A finite $R_\infty$-module $H_\infty$, which is finite free as $S_\infty$-module.
		\end{itemize}
		This is a contradiction. Indeed, \cite[Lemma 2.8]{KT} shows that the dimension of $H_\infty$ is the same, whether considered as $R_\infty$- or $S_\infty$-module. By freeness, its dimension as $S_\infty$-module is $\dim S_\infty = 4|T| + q_0$. On the other hand, its dimension as $R_\infty$-module is bounded above by $\dim R_\infty = \dim A_{\cS}^T - 1 + g =  4|T| + q_0 - 1$.
		
		We conclude that there exists an automorphic representation $\pi'$ 
		contributing to $H_{Q_N}$ for some $N \geq 1$ such that for each prime 
		$l \in S$ such that $\pi'_l$ is ramified, each accessible refinement of 
		$\pi'_l$ is $n$-regular. To complete the proof, we just need to explain 
		why $\pi'_q$ is $n$-regular for each prime $q \in Q_N$ such that 
		$\pi_q$ is ramified. However, our construction shows that $r_{\pi', 
		\iota}|_{I_{\Q_q}}$ has the form $C_q \oplus C_q^{-1}$, where $C_q : 
		I_{\Q_q} \to \overline{\Q}_p^\times$ has order a power of $p$. Since $p 
		> 2n$, by hypothesis, this is a fortiori $n$-regular. This completes 
		the proof.
\end{proof}
We can now finish the proof of Theorem \ref{thm_general_levels}.
\begin{proof}[Proof of Theorem \ref{thm_general_levels}]
	Let $\pi$ be a cuspidal automorphic representation of $\GL_2(\A_\Q)$ of weight $k \geq 2$, without CM, and such that each local component $\pi_l$ admits an accessible refinement. Let $p$, $\iota$, and $\pi'$ be as in the statement of Proposition \ref{prop_TW_congruence_to_n_regular}. Then $\Sym^{n-1} r_{\pi', \iota}$ is automorphic, by Proposition \ref{prop_general_levels_n_regular}.
	
	On the other hand, our assumptions imply that the residual representation $\Sym^{n-1} \overline{r}_{\pi, \iota} \cong \Sym^{n-1} \overline{r}_{\pi', \iota}$ is irreducible. We can therefore apply \cite[Theorem 4.2.1]{BLGGT} to conclude that $\Sym^{n-1} r_{\pi, \iota}$ is automorphic. This completes the proof.
\end{proof}
\bibliographystyle{amsalpha}
\bibliography{CMpatching}

\end{document}